\documentclass{book}

\RequirePackage[ansinew]{inputenc}
\usepackage[brazil]{babel}
\usepackage{amsmath,amsfonts,amsthm}
\usepackage[top=2cm,left=2cm,right=2.5cm,bottom=2cm]{geometry}
\usepackage{makeidx}
\usepackage[sc]{mathpazo}
\usepackage[pdftex]{hyperref}

\newcommand{\C}{\mathbb{C}}
\newcommand{\R}{\mathbb{R}}
\newcommand{\Z}{\mathbb{Z}}
\newcommand{\pr}{\mathbb{P}}
\newcommand{\GL}{\text{GL}}
\newcommand{\re}{\text{Re }}
\newcommand{\im}{\text{Im }}
\newcommand{\smo}{\sqrt{-1}}
\newcommand{\id}{\text{id}}
\newcommand{\zbar}{\overline{z}}
\newcommand{\del}{\partial}
\newcommand{\delbar}{\bar{\partial}}

\newcommand{\Pic}{\text{Pic}}
\newcommand{\Div}{\text{Div}}
\newcommand{\PDiv}{\text{PDiv}}
\newcommand{\Ord}{\text{ord}}
\newcommand{\End}{\text{End}}
\newcommand{\Hom}{\text{Hom}}

\newcommand{\gl}{\mathfrak{gl}}

\newcommand{\Hess}{\text{Hess}}
\newcommand{\hess}{\text{hess}}

\theoremstyle{plain}
\newtheorem{theorem}{Teorema}[chapter]
\newtheorem{proposition}[theorem]{Proposição}
\newtheorem{corollary}[theorem]{Corolário}
\newtheorem{lemma}[theorem]{Lema}

\theoremstyle{definition}
\newtheorem{definition}[theorem]{Definição}
\newtheorem{example}[theorem]{Exemplo}

\theoremstyle{remark}
\newtheorem{remark}[theorem]{Observação}

\title{\huge Fundamentos da Geometria Complexa:\\ \Large aspectos geométricos, topológicos e analíticos.}
\author{Lucas Kaufmann Sacchetto}
\date{}

\makeindex

\begin{document}
\maketitle

\tableofcontents
\renewcommand{\thechapter}{}
\renewcommand{\chaptername}{}
\chapter{Introdução}

O estudo da Geometria Complexa tem suas raízes no trabalho de Riemann e foi motivada pela Análise Complexa. No estudo de funções holomorfas de uma variável por exemplo, o conceito de continuação analítica leva naturalmente à ideia de funções multi-valoradas e foi para entender essas funções que Riemann introduziu o que hoje são chamadas Superfícies de Riemann. Na linguagem atual, uma superfície de Riemann é uma variedade complexa de dimensão $1$.

Já no caso das superfícies de Riemann, podemos observar um fenômeno muito comum na Geometria Complexa, que é a íntima relação entre as teorias topológica, analítica e geométrica. O Teorema de Uniformização afirma que toda superfície de Riemann pode ser escrita na forma $X = \widetilde{X}/\Gamma$, onde $\widetilde{X}$ é: i) a esfera de Riemann $\C_{\infty}$, ii) o plano complexo $\C$ ou iii) o disco unitário $\Delta$, e $\Gamma$ é um grupo discreto de automorfismos agindo livre e propriamente em $\widetilde{X}$. Do ponto de vista da Geometria diferencial, esses três casos podem ser caracterizados pelo fato de admitirem métricas de curvatura constante igual a $+1,0$ e $-1$ respectivamente. No caso em que $X$ é compacta, podemos diferenciar os três casos pelo genus $g$ de $X$: $g=0$, no primeiro caso, $g=1$ no segundo e $g \geq 2$ no terceiro. O genus tem também uma interpretação analítica: é a dimensão do espaço das $1$-formas holomorfas globais em $X$.\\

No caso das superfícies complexas ($\dim X = 2$) compactas também existe uma classificação, devida ao trabalho de Enriques e Kodaira, mas esta é mais complicada e ainda parcialmente incompleta. Para dimensões mais altas, uma classificação completa parece impossível. No entanto, existem algumas técnicas provindas do chamado `` Minimal Model Program".\\

Atualmente, a Geometria Complexa é uma área bastante ativa, e sua popularidade se dá possivelmente devido ao fato de estar na intersecção de vários ramos da Matemática, como a geometria diferencial, a análise complexa e a geometria algébrica. Outro motivo para o interesse está no fato de servir de ferramenta fundamental em partes importantes da Física teórica, como por exemplo na Teoria de Cordas. Essas conexões e ramificações tornam o assunto extremamente interessante e variado, e ao mesmo tempo trazem para a Geometria Complexa uma riqueza de exemplos e de pontos de vista.\\

\subsubsection{Variedades complexas}

Uma variedade complexa é um espaço topológico que é localmente modelado em abertos de $\C^n$ e cujas cartas locais diferem por transformações holomorfas.

Alguns exemplos básicos de variedades complexas são: o espaço euclideano complexo $\C^n$, o espaço projetivo complexo $\pr^n$, toros complexos $X=\C^n/L$, variedades algébricas suaves, etc.

Podemos falar de aplicações holomorfas entre variedades complexas e alguns resultados locais sobre funções holomorfas de várias variáveis se generalizam facilmente.\\

\indent \textbf{Princípio do Máximo.} Se $X$ é conexa e $f:X \to \C$ é holomorfa e não constante então $|f|$ não possui máximo local em $X$.\\
\indent \textbf{Princípio da Identidade.} Se $X$ é conexa e $f,g:X \to Y$ são holomorfas e $f=g$ em algum aberto não vazio $U \subset X$ então $f=g$ em $X$.\\

Como consequência do princípio do máximo vemos que em uma variedade compacta e conexa não existem funções holomorfas globais não constantes. Em particular, isso mostra que que nenhuma variedade compacta $X$ pode ser holomorficamente mergulhada em $\C^n$. Esse já um primeiro exemplo de como a teoria de variedades complexas difere drasticamente da teoria de variedades diferenciáveis, pois, pelo Teorema de Whitney, qualquer variedade diferenciável, compacta ou não, admite um mergulho em algum espaço euclideano $\R^N$.

Embora não admitam mergulhos no espaço afim, uma grande quantidade de variedades compactas admitem mergulhos em algum espaço projetivo complexo. Essas variedades são chamadas projetivas e são um caso particular das variedades de Kähler, das quais falaremos mais adiante.\\

Do fato de não existirem muitas funções holomorfas é útil considerarmos também as chamadas \emph{funções meromorfas} em $X$, que são ``funções'' dadas localmente como razão de funções holomorfas.\\

\subsubsection{Feixes e cohomologia}
Uma ferramenta essencial na geometria complexa é a teoria de feixes e sua cohomologia, que fornecem uma linguagem adequada para tratar de problemas formulados em termos de propriedades locais.

Um feixe sobre um espaço topológico $X$ é uma coleção de grupos abelianos $\mathcal{F}(U)$, um para cada aberto $U \subset X$, junto com homomorfismos de restrição $\mathcal{F}(U) \to \mathcal{F}(V)$, sempre que $V \subset U$. Os elementos de $\mathcal{F}(U)$ são chamados seções de $\mathcal{F}$ sobre $U$. Pedimos ainda que seções que concordam na intersecção dos abertos de uma cobertura possam ser coladas a uma seção na união desses abertos.

Alguns exemplos básicos de feixes sobre uma variedade complexa $X$ são: o feixe de funções holomorfas $\mathcal{O}_X$, o feixe de funções holomorfas que não se anulam $\mathcal{O}^*_X$, o feixe de $k$-formas $\mathcal{A}^k_X$ e os feixes localmente constantes $\underline{G}$, que associam a cada $U \subset X$ o conjunto das funções localemte constantes em $U$ com valores em um grupo abeliano $G$.

Podemos definir um morfismo de feixes $\phi: \mathcal{F}\to \mathcal{G}$ como sendo um coleção de homomorfismos $\phi_U: \mathcal{F}(U)\to \mathcal{G}(U)$ que comutam com as restrições. Podemos também definir feixes núcleo e imagem associados a $\phi$ e, com isso, definir sequências exatas de feixes. Uma sequência de feixes e morfismos $\mathcal{F} \stackrel{\phi}{\longrightarrow} \mathcal{G} \stackrel{\psi}{\longrightarrow} \mathcal{H}$ será exata em $\mathcal{G}$ se para todo elemento em $\sigma \in \ker \phi$ existe uma cobertura de $X$ por abertos $U_i$ de modo que $\sigma|_{U_i} \in \im \phi_{U_i}$. Mais geralmente uma sequência $\mathcal F^1 \to \mathcal F^2 \to \cdots$ será exata se o for em cada $\mathcal F^i$.\\

Um exemplo central na Geometria Complexa é a chamada \textit{sequência exponencial}, dada por
\begin{equation*}
0 \longrightarrow \underline{\Z} \longrightarrow \mathcal{O}_X \longrightarrow \mathcal{O}^*_X \longrightarrow 0,
\end{equation*}
onde a $\underline{\Z} \to \mathcal{O}_X$ é a inclusão e $\mathcal{O}_X \to \mathcal{O}^*_X$ é a exponencial de funções $f \mapsto \exp(2\pi \smo \cdot f)$. A exatidão da sequência é assegurada pela existência local do logaritmo.\\

\indent A ideia da cohomologia de feixes é associar a cada feixe $\mathcal{F}$ em $X$ uma sequência de grupos abelianos $H^i(X,\mathcal{F})$, $i=1,2,\ldots$, chamados grupos de cohomologia de $X$ com coeficientes em $\mathcal{F}$, de modo que morfismos de feixe induzam homomorfismos nos respectivos grupos de cohomologia e tal que para toda sequência exata curta $0 \to \mathcal{F} \to \mathcal{G} \to \mathcal{H} \to 0$ exista uma sequência exata da forma
\begin{equation*}
\begin{split}
0 &\longrightarrow H^0(X,\mathcal{F}) \longrightarrow H^0(X,\mathcal{G}) \longrightarrow H^0(X,\mathcal{H}) \longrightarrow \\
&\longrightarrow H^1(X,\mathcal{F}) \longrightarrow H^1(X,\mathcal{G}) \longrightarrow H^1(X,\mathcal{H}) \longrightarrow \cdots \\
 & \vdots \\
 & \longrightarrow H^n(X,\mathcal{F}) \longrightarrow H^n(X,\mathcal{G}) \longrightarrow H^n(X,\mathcal{H}) \longrightarrow \cdots.
\end{split}
\end{equation*}
\indent Além disso, o espaço $H^0(X,\mathcal{F})$ é naturalmente identificado com o espaço das seções globais de $\mathcal{F}$.\\

\subsubsection{Fibrados de Linha e o Grupo de Picard}
Um fibrado de linha holomorfo sobre $X$ é uma variedade complexa $L$ com uma aplicação holomorfa $\pi: L \to X$ tal que $L_x = \pi^{-1}(x)$ é um espaço vetorial complexo de dimensão $1$. Pedimos também que $L$ seja localmente trivial no sentido que existem biholomorfismos $\varphi_U: \pi^{-1}(U)\to U \times \C$ para $U \subset X$ abertos cobrindo $X$.

Equivalentemente, dada uma cobertura $\mathcal{U} = \{U_i\}$, um fibrado de linha pode ser dado por funções holomorfas que não se anulam $\varphi_{ij} \in \mathcal{O}^*_X(U_i \cap U_j)$ satisfazendo as condições de cociclo
\begin{equation*}
 \begin{split}
   \varphi_{ij} \cdot \varphi_{ji} &= I \;\; \text{ em } U_i \cap U_j \\
   \varphi_{ij} \cdot \varphi_{jk} \cdot \varphi_{ki} &= I \;\; \text{ em } U_i \cap U_j \cap U_k
 \end{split}
\end{equation*}

Um exemplo importante é o chamado \textit{fibrado canônico} $K_X$, que é o fibrado de $n$-formas holomorfas em $X$. Outro exemplo é o chamado fibrado tautológico sobre $\pr^n$, cuja fibra sobre uma reta $\ell \subset \pr^n$ é a própria reta, vista como subespaço de $\C^{n+1}$.\\

Dados dois fibrados de linha $L$ e $L'$, o produto tensorial $L \otimes L'$ é um outro fibrado de linha, e se $L^*$ denota o fibrado dual de $L$ (em que as fibras são $L_x^*$) então $L \otimes L ^* \simeq \mathcal{O}_X$, o fibrado trivial. Desta forma, o conjunto de classes de isomorfismo de fibrados de linha holomorfos formam um grupo abeliano, chamado Grupo de Picard, denotado $\Pic(X)$.\\
\indent A descrição de $L$ por cociclos nos permite identificar
\begin{equation*}
\Pic(X) \simeq H^1(X,\mathcal{O}^*_X).
\end{equation*}

Note que este grupo aparece na sequência exata longa associada à sequência exponencial:
\begin{equation*}
H^1(X,\Z) \longrightarrow H^1(X,\mathcal{O}_X) \longrightarrow H^1(X,\mathcal{O}^*_X) \longrightarrow H^2(X,\Z)
\end{equation*}

\indent A \textit{classe de Chern} de um fibrado $L \in \Pic(X)$ é sua imagem $c_1(L)$ pela aplicação $c_1:H^1(X,\mathcal{O}^*_X) \to H^2(X,\Z)$ e é um invariante importante de $L$.\\

\subsubsection{Divisores}

Um divisor em $X$ é uma soma formal $D = \sum a_i Y_i$ onde $a_i \in \Z$ e cada $Y_i \subset X$ é uma hipersuperfície irredutível. Equivalentemente, podemos definir um divisor como sendo uma seção global do feixe quociente $\mathcal{K}^*_X/\mathcal{O}^*_X$, onde $\mathcal{K}^*_X$ é o feixe das funções meromorfas sobre $X$. O grupo dos divisores em $X$ é denotado por $\Div(X)$.\\
\indent O exemplo básico de divisor é o que chamamos de divisor associado a uma função meromorfa $f \in K(X)$, definido por
\begin{equation*}
(f) = \sum \Ord (f) _Y \cdot Y,
\end{equation*}
onde a soma é sobre todas as hipersuperfícies de $X$ e $\Ord(f)_Y$ é um inteiro que mede se $f$ se anula ou se tem um singularidade ao longo de $Y$. Um divisor é dito principal se $D = (f)$ para alguma $f \in K(X)^*$. A ordem é aditiva e portanto o conjunto $\PDiv(X)$ dos divisores principais é um subgrupo de $\Div(X)$.

Da sequência exata de feixes
\begin{equation*}
0 \longrightarrow \mathcal{O}^*_X \longrightarrow \mathcal{K}^*_X \longrightarrow \mathcal{K}^*_X/ \mathcal{O}^*_X \longrightarrow 0
\end{equation*}
obtemos a sequência exata
\begin{equation*}
 \begin{matrix}
 H^0(X,\mathcal{K}^*_X) & \longrightarrow & H^0(X,\mathcal{K}^*_X/\mathcal{O}^*_X) & \longrightarrow & H^1(X,\mathcal{O}^*_X) \\
 \parallel &  & \parallel & & \parallel \\
 K(X)^* &  & \Div(X) & & \Pic(X)
 \end{matrix}
\end{equation*}
e portanto a cada divisor $D \in \Div(X)$ temos um fibrado de linha associado, denotado por $\mathcal{O}(D)$. Vemos também que o fibrado $\mathcal{O}(D)$ é trivial se e só se $D$ é um divisor principal, de onde obtemos uma inclusão
\begin{equation*}
\frac{\Div(X)}{\PDiv(X)} \hookrightarrow \Pic(X).
\end{equation*}

\subsubsection{Geometria Hermitiana}
A Geometria Hermitiana é a extensão natural da Geometria Riemanniana para as variedades complexas. Se $X$ é uma variedade complexa, a estrutura complexa de $X$ induz uma estrutura complexa em cada espaço tangente, isto é, existe um automorfismo $J:TX \to TX$ satisfazendo $J^2=-\text{id}$. Uma métrica riemanniana $g$ em é dita \textit{hermitiana} se
\begin{equation*}
g(JX,JY) = g(X,Y),
\end{equation*}
ou seja, as aplicações $J_x:T_x X \to T_x X$ são isometrias.\\
\indent A uma métrica hermitiana temos uma 2-forma associada, chamada \textit{forma fundamental}:
\begin{equation*}
\omega(X,Y) = g(JX,Y).
\end{equation*}

\indent O automorfismo $J:TX \to TX$ se estende à complexificação $T_{\C}X = TX \otimes \C$, o que fornece uma decomposição
\begin{equation*}
T_{\C}X = T^{1,0}X \oplus T^{0,1}X,
\end{equation*}
onde $T^{1,0}X$  e $T^{0,1}X$ são os autoespaços de $J$ associados aos autovalores $\smo$ e $-\smo$ respectivamente.

Essa decomposição induz uma decomposição no fibrado de $k$-formas:
\begin{equation*}
\textstyle \bigwedge^k_{\C}X \doteq \bigwedge^k(T_{\C}X) = \displaystyle \bigoplus_{p+q=k} \textstyle \bigwedge^{p,q} X, \text{ onde } \textstyle \bigwedge^{p,q} X = \textstyle \bigwedge^p (T^{1,0}X)^* \otimes \bigwedge^q (T^{0,1}X)^*.
\end{equation*}
Com isso podemos definir os operadores $\partial$ e $\bar{\partial}$ que agem nas $(p,q)$-formas tomando, respectivamente, as compontentes $(p+1,q)$ e $(p,q+1)$ da diferencial exterior $d$.

O operador $\delbar$ satisfaz $\delbar^2 = 0$ e portanto temos, para cada $p\geq 0$, um complexo de feixes
\begin{equation*}
\mathcal{A}_X^{p,0} \stackrel{\delbar}{\longrightarrow} \mathcal{A}_X^{p,1} \stackrel{\delbar}{\longrightarrow} \mathcal{A}_X^{p,2} \stackrel{\delbar}{\longrightarrow} \cdots
\end{equation*}
chamado \textit{complexo de Dolbeaut}, e os espaços de  cohomologia no nível de seções globais
\begin{equation*}
H^{p,q}_{\delbar}(X) = \frac{\ker\{\delbar: \mathcal{A}^{p,q}(X) \to \mathcal{A}^{p,q+1}(X)\}}{\im \{\delbar: \mathcal{A}^{p,q-1}(X) \to \mathcal{A}^{p,q}(X)\}}
\end{equation*}
são chamados espaços de cohomologia de Dolbeaut, que são análogos complexos dos espaços de cohomologia de de Rham.\\

\indent A partir de uma métrica hermitiana podemos definir o operador de Leftschetz $L: \alpha \mapsto \omega \wedge \alpha$ e seu adjunto $\Lambda = *^{-1} \circ L \circ *$, onde $*$ é o operador de Hodge. Juntamente com o operador de contagem $H$, que age em $\bigwedge^k X$ por multiplicaçao por $n-k$, obtemos uma representação de $\mathfrak{sl}(2,\C)$ na álgebra exterior de $X$.\\
\indent Essa representação oferece uma outra decomposição do fibrado de $k$-formas, chamada \textit{decomposição de Lefschetz}:
\begin{equation*}
\textstyle \bigwedge^k X = \displaystyle \bigoplus_{j \geq 0} L^j(P^{k-2j}(X)),\;\; P^k(X) = \ker \big(\Lambda: \textstyle \bigwedge^k X \to \bigwedge^{k-2}X \big).
\end{equation*}

\subsubsection{Variedades de Kähler}

Uma métrica hermitiana $g$ é dita de Kähler se a forma fundamental $\omega$ é fechada ($d \omega = 0)$ ou equivalentemente se o $(1,1)$-tensor $J$ é paralelo com relação à conexão de Levi-Civita de $g$ ($\nabla J=0$). Uma outra definição possível é exigir que a métrica $g$ pode ser aproximada até ordem $2$ pela métrica hermitiana padrão em $\C^n$.

Um variedade é dita de Kähler se admite uma métrica de Kähler. As variedades de Kähler formam uma classe importante de variedades complexas e contém por exemplo os toros complexos e todas as variedades projetivas.

Nem todas variedade complexa admite métricas de Kähler. Existem, por exemplo, obstruções de natureza topólogica.
\begin{proposition}
Seja $X$ uma variedade complexa compacta de dimensão $n$. Se $X$ admite uma métrica de Kähler então $H^{2k}(X,\R) \neq 0$ para $k=0,\cdots,n$.
\end{proposition}

Um exemplo importante é a chamada métrica de Fubini-Study $g_{FS}$ no espaço projetivo $\pr^n$. Ela é definida a partir de sua forma fundamental, dada localmente por
\begin{equation*}
\omega_{FS} = \frac{\smo}{2} \partial \bar{\partial} \log \bigg( 1 + \sum_{j \neq i} \left| \frac{z_j}{z_i} \right|^2 \bigg)
\end{equation*}
em $U_i = \{z_i \neq 0\}$.

Uma maneira equivalente de definir $g_{FS}$ é exigir que a submersão $S^{2n+1} \to \pr^n$ seja uma submersão riemanniana.\\

A condição de que uma métrica de Kähler pode ser aproximada pela métrica padrão permite obter relações entre operadores diferenciais em $X$ cuja definição depende apenas de derivadas primeiras da métrica, como por exemplo os operadores $\partial$ e $\bar{\partial}$, seus adjuntos $\del^*$ e $\delbar^*$ com relação a métrica $L^2$ (veja eq. \ref{eq:intro-L2-metric} adiante), os operadores de Lefschetz $L$ e $\Lambda$ e etc. Algumas dessas relações são as chamadas identidades de Kähler

\begin{theorem} Se $X$ é uma variedade complexa com uma métrica de Kähler então valem as seguintes relações de comutação
\begin{itemize}
\item[1.] $[L,\del] = [L,\delbar] = 0$ ~e~~ $[\Lambda,\del^*] = [\Lambda^*,\delbar^*] = 0$
\item[2.] $[\Lambda,\delbar] = - \smo \del^*$ ~e~~ $[\Lambda,\del] = \smo \delbar^*$,
\item[3.] $[\delbar^*,L] = \smo \del$ ~e~~ $[\del^*,L] = -\smo \delbar$.

\end{itemize}
\end{theorem}

Uma consequência das Identidades de Kähler é o fato de que os Laplacianos $\Delta_{\del} = \del \del^* + \del^* \del$ e $\Delta_{\delbar} = \delbar \delbar^* + \delbar^* \delbar$  satisfazem
\begin{equation*}
\Delta_{\del} = \Delta_{\delbar} = \frac{1}{2} \Delta,
\end{equation*}
onde $\Delta$ denota o Laplaciano usual. Em particular, o núcleo dos três laplacianos coincide, o que tem importantes consequências na Teoria de Hodge de uma variedade de Kähler.

Ainda das Identidades de Kähler obtemos o chamado \emph{Teorema ``Difícil'' de Lefschetz} que diz que a decomposiçao de Lefschetz passa para a cohomologia.
\begin{theorem}
Se $X$ é uma variedade de Kähler compacta então
\begin{equation*}
L^{n-k}:H^k(X,\R) \longrightarrow H^{2n-k}(X,\R) 
\end{equation*}
é um isomorfismo para todo $k \leq n$ e existe uma decomposição
\begin{equation*}
H^k(X,\R) = \bigoplus_{i \geq 0} L^i H^{k-2i}(X,\R)_p.
\end{equation*}
\end{theorem}

Aqui $H^k(X,\R)_p = \{c \in H^k(X,\R) : \Lambda c = 0\}$ é o espaço das classes de cohomologia primitvas.\\

\subsubsection{Teoria de Hodge}

A Teoria de Hodge é uma ferramenta analítica essencial no estudo da cohomologia de variedades de Kähler. O objetivo é buscar representantes canônicos para cada classe de cohomologia.

Uma métrica hermtiana em uma variedade compacta induz métricas nos fibrados de formas. Com isso podemos definir um produto hermitiano $L^2$ no espaço $\mathcal A^k_\C(X)$ das formas diferenciais em $X$
\begin{equation} \label{eq:intro-L2-metric}
(\alpha,\beta) = \int_X ( \alpha,\beta ) \text{vol},\;\;\; \alpha,\beta \in \mathcal{A}_{\C}^k(X).
\end{equation}

O resultado central da teoria de Hodge diz que toda forma diferencial pode ser escrita como $\alpha = \alpha_0 + \del \beta + \del^* \gamma$ com $\alpha_0$ $\del-$harmônica, e um resultado análogo para $\delbar$.

\begin{theorem} Se $X$ é uma variedade hermitiana compacta, existem duas decomposições ortogonais
\begin{equation*}
\mathcal{A}^{p,q}(X) = \mathcal{H}_{\del}^{p,q}(X) \oplus \del \mathcal{A}^{p-1,q}(X) \oplus \del^*\mathcal{A}^{p+1,q}(X)
\end{equation*}
e
\begin{equation*}
\mathcal{A}^{p,q}(X) = \mathcal{H}_{\delbar}^{p,q}(X) \oplus \delbar \mathcal{A}^{p,q-1}(X) \oplus \delbar^*\mathcal{A}^{p,q+1}(X).
\end{equation*}
\end{theorem}

Aqui $\mathcal{H}_{\del}^{p,q}(X)$ e $\mathcal{H}_{\delbar}^{p,q}(X)$ são, respectivamente, os espaços de $(p,q)$-formas $\del$-harmônicas e $\delbar$-harmônicas e  quando $X$ é de Kähler temos a igualdade $\mathcal{H}_{\del}^{p,q}(X)=\mathcal{H}_{\delbar}^{p,q}(X)$.

Como consequência, quando $X$ é de Kähler, podemos identificar o espaço de cohomologia de Dolbeaut $H^{p,q}_{\delbar}(X)$ com o espaço $H^{p,q}(X)$ das classes de cohomologia de de Rham que são representáveis por formas do tipo $(p,q)$ e com isso podemos obter uma outra decomposição na cohomologia de $X$, a chamada Decomposição de Hodge.

\begin{theorem} Se $X$ é uma variedade de Kähler compacta então existe uma decomposição
\begin{equation*}
H^k(X,\C) = \bigoplus_{p+q=k} H^{p,q}(X) ~\text{ com }~H^{p,q}(X) = \overline{H^{q,p}(X)}.
\end{equation*}
\end{theorem}

Dada a existência da decomposição de Hodge podemos considerar os espaços de cohomologia integral $H^{p,q}(X,\Z)$, definidos como sendo a intersecção de $H^{p,q}(X)$ com a imagem da aplicação natural $H^{p+q}(X,\Z) \to H^{p+q}(X,\C)$. Olhando para a sequência exata longa associada a sequência exponencial podemos ver que a primeira classe de Chern de um fibrado de linha holomorfo está sempre em $H^{1,1}(X,\Z)$. Além disso, o Teorema das $(1,1)$-classes de Lefschetz diz que todo elemento de $H^{1,1}(X,\Z)$ é da forma $c_1(L)$ para algum $L \in \Pic(X)$.
\begin{theorem}
Se $X$ é uma variedade de Kähler compacta então a aplicação $\Pic(X) \to H^{1,1}(X,\Z)$ é sobrejetora, isto é, todo elemento de $H^{1,1}(X,\Z)$ é a primeira classe de Chern de um fibrado de linha holomorfo sobre $X$.
\end{theorem}

O teorema acima responde parcialmente a famosa Conjectura de Hodge, um dos problemas do milênio propostos pelo ``Clay Mathematics Institute''. A conjectura diz que, se $X$ uma variedade projetiva, então toda classe de cohomologia em $H^{p,p}(X,\mathbb Q)$ pode ser escrita como combinação linear com coeficientes racionais de classes fundamentais de subvariedades de $X$.\\

\subsubsection{Geometria dos Fibrados Vetoriais Complexos}

Dado um fibrado vetorial complexo $E \to X$ podemos introduzir algumas estruturas geométricas em $E$. Uma estrutura hermitiana, por exemplo, é um produto hermitiano em cada fibra $E_x$ que varia diferenciavelmente com $x$.

Um outro objeto natural é o que chamamos de conexão, que é essencialmente uma maneira de derivarmos seções de $E$. Associada a uma conexão $\nabla$ temos sua curvatura $F_\nabla$, que é uma $2$-forma com valores no fibrado de endomorfismos de $E$.

Existem algumas condições de compatibilidade que podemos exigir de uma conexão. Um resultado importante nesse contexto é o seguinte.
\begin{proposition}
Seja $E$ um fibrado holomorfo com uma estrutura hermitiana. Então existe uma única conexão em $E$ que é a compatível, simultaneamente, com a estrutura hermitiana e a estrutura holomorfa de $E$. Essa conexão recebe o nome de \textbf{conexão de Chern}.
\end{proposition}

No caso de fibrados de linha holomorfos, a conexão de Chern pode ser usada para construir métricas de Kähler na variedade de base.

\begin{proposition}
Seja $X$ uma variedade complexa e $L$ um fibrado de linha holomorfo hermitiano sobre $X$. Se a conexão de Chern de $L$ tem curvatura positiva $F$ então $\omega = \smo F$ define uma métrica de Kähler em $X$. 
\end{proposition}

Um outro aspecto importante das conexões é que a partir delas (mais precisamente, a partir de sua curvatura) podemos definir invariantes cohomológicos, chamados classes características. Se $\nabla$ é uma conexão em um fibrado $E \to X$ de posto $r$, para cada polinômio simétrico invariante $P_k$ de grau $k$ em $\gl(r,\C)$ podemos definir uma classe de cohomologia $[P_k(F_\nabla)] \in H^{2k}(X,\C)$, que independe de $\nabla$.

O exemplo mais importante são as classes de Chern, obtidas tomando como polinômios invariantes os poliômios $P_k$ definidos por
\begin{equation*}
\det(I + B) = 1 +  P_1(B) +  P_2(B) + \cdots +  P_r(B).
\end{equation*}

As classes de Chern estão intimamente relacionadas com a geometria diferencial da variedade $X$. Por exemplo, se $g$ é uma métrica de Kähler em  $X$ e $r$ é o seu tensor de Ricci então a forma de Ricci, definida por $\rho = r(J \cdot, \cdot)$, satisfaz $\rho \in 2 \pi c_1(X)$, onde $c_1(X)$ é a primeira classe de Chern de $X$, definida como sendo a primeira classe de Chern de seu fibrado tangente holomorfo.

Um resultado importante nesse contexto é o celebrado Teorema de Calabi-Yau, que diz que, quando $X$ é compacta, qualquer $(1,1)$-forma real e fechada em $2\pi c_1(X)$ é a forma de Ricci de uma métrica de Kähler em $X$ com $[\omega]$ especificada. Como consequência imediata vemos que toda variedade compacta com $c_1(X) = 0$ admite uma métrica de Kähler-Einstein, i.e., uma métrica satisfazendo $\rho = \lambda \omega$ para algum $\lambda \in \R$ (neste caso temos $\lambda = 0$).\\

\subsubsection{Topologia de variedades complexas}

A existência de uma estrutura complexa tem fortes consequências sobre a topologia da variedade. Existem alguns resultados importantes nesse sentido.

No âmbito das variedades não compactas, mais precisamente das variedades de Stein (i.e., subvariedades fechadas de $\C^N$) temos o seguinte resultado.

\begin{theorem}
Seja $X \subset \C^N$ uma variedade de Stein de dimensão complexa $n$. Então os grupos de homologia de $X$ com coeficientes em $\Z$ satisfazem
\begin{equation*}
H_i(X,\Z) = 0 ~~ \text{ para } i > n 
\end{equation*}
e
\begin{equation*}
H_n(X,\Z) \text{ é livre de torção.} 
\end{equation*}
\end{theorem}

Este resultado tem consequências na topologia das variedades projetivas, como por exemplo o famoso Teorema de Hiperplanos de Lefschetz.

\begin{theorem}
Seja $X \subset \pr^m$ uma variedade algébrica de dimensão $n$. Seja $Y = X \cap W$ a intersecção de $X$ com uma hipersuperfície algébrica $W$ que contém os pontos singulares de $X$ mas não contém $X$.

Nessas condições, o homomorfismo
\begin{equation*}
H^i(X,\Z) \longrightarrow H^i(Y,\Z)
\end{equation*}
induzido pela inclusão $Y \subset X$ é um isomorfismo para $i<n-1$ e é injetor se $i=n-1$. Além disso o quociente $H^{n-1}(Y,\Z) \slash H^{n-1}(X,\Z)$ é livre de torção.
\end{theorem}

-------------------------------------------------------------------------------------------------------------------------------------------
\\

 O presente trabalho tem como objetivo apresentar uma discussão detalhada dos resultados citados acima, ilustrando-os e motivando-os através de exemplos. Todos os resultados não demonstrados são acompanhados de referências apropriadas. É assumido que o leitor tenha conhecimentos básicos de análise complexa, teoria de variedades diferenciáveis, topologia algébrica e geometria riemanniana.

O texto é fruto de um trabalho de mestrado desenvolvido pelo autor entre os anos de 2010 e 2012 sob a orientação do Professor Claudio Gorodski no Instituto de Matemática e Estatística da Universidade de São Paulo. Gostaria de prestar meus sinceros agradecimentos ao Andrew Clarke e ao Paulo Cordaro pelos apontamentos e correções sugeridas.\\

Durante a realização deste trabalho o autor recebeu apoio financeiro da FAPESP e do CNPq.
\renewcommand{\thechapter}{\arabic{chapter}}
\renewcommand{\chaptername}{Capítulo}
\chapter{Material Preliminar} \label{ch:cap0}

Neste primeiro capítulo apresentaremos os conceitos que servirão de base para o restante do trabalho. Recordaremos os principais resultados sobre funções de uma variável complexa e apresentaremos algumas definições e resultados básicos da teoria de funções de várias variáveis complexas.

Na última seção introduziremos o conceito de variedade complexa, que será o objeto central do estudo deste trabalho.

\section{Funções holomorfas de uma variável}

Uma função $f:U \to \C$ definida em um aberto $U \subset \C$ é holomorfa se é diferenciável no sentido complexo, ou seja, se o limite
\begin{equation*}
f'(z) = \lim_{w \to z} \frac{f(w)-f(z)}{w-z}
\end{equation*}
existe para todo $z \in U$.

Escrevendo $f = u + \smo v$ onde $u,v:U \to \R$, $f$ será holomorfa em $U$ se e somente se $f$ for de classe $C^1$ e satisfizer as \textit{equações de Cauchy-Riemann}
\begin{equation*}
\frac{\del u}{\del x} = \frac{\del v}{ \del y} ~~~ \text{ e } ~~~ \frac{\del u}{\del x} = -\frac{\del v}{ \del y}. 
\end{equation*}

Se definirmos os operadores diferenciais
\begin{equation*}
\frac{\del}{\del z} \doteq \frac{1}{2}\left( \frac{\del}{\del x} - \smo \frac{\del}{\del y} \right)~\text{ e }~ \frac{\del}{\del \bar z} \doteq \frac{1}{2}\left( \frac{\del}{\del x} + \smo \frac{\del}{\del y} \right),
\end{equation*}
vemos que as equações de Cauchy-Riemann para $f$ são equivalentes a equação $\frac{\del f}{\del \bar z} = 0$.

Existe ainda uma outra caracterização das funções holomorfas. Daqui em diante vamos supor que todas as funções em questão são de classe $C^1$.

Identificando $\C \simeq \R^2$ via $x+\smo y \mapsto (x,y)$, podemos ver uma função complexa como uma aplicação$f:U \to \R^2$ e portanto a diferencial de $f$ em um ponto $z \in U$ é uma aplicação $\R$-linear $df_z: T_z U \simeq \R^2 \to T_{f(z)} \R^2 \simeq \R^2$. Das equações de Cauchy-Riemann vemos então que $f$ é holomorfa se e somente se para todo $z \in U$, $df_z$ tem a forma $\begin{pmatrix} a & b \\ -b & a \end{pmatrix}$, ou seja, se e somente se
\begin{equation*}
df_z J_0 = J_0 df_z,~~ \text{ onde }~ J_0 = \begin{pmatrix} 0 & -1 \\ 1 & 0 \end{pmatrix}.
\end{equation*}

Note que, segundo a identificação $\R^2 \simeq \C$, a ação da matriz $J_0$ em $\R^2$ corresponde à multiplicação por $\smo$ em $\C$ e portanto vemos que $f$ é holomorfa se e somente se $df_z$ é $\C$-linear para todo $z \in U$.

Um resultado central na teoria de funções holomorfas de uma variável é a chamada Fórmula de Cauchy, que diz que o valor de uma função holomorfa em um ponto é determinado pelos seus valores em um círculo ao redor deste ponto.

\begin{theorem} \emph{(Fórmula de Cauchy)} 
Seja $f:U \to \C$ uma função holomorfa e suponha que $B_r(z_0) = \{w \in \C : |w-z_0|<r\} \subset U$. Então para todo $z \in B_r(z_0)$ temos que
\begin{equation*}
f(z) = \frac{1}{2\pi \smo} \int_{|w-z_0| = r} \frac{f(w)}{w-z}dw~~,
\end{equation*} 
onde o círculo $|w-z_0|=r$ é percorrido uma vez no sentido anti-horário.
\end{theorem}

A fórmula acima pode ser usada para escrever uma série de potências para $f$ na variável $z$ em torno do ponto $z_0$ com raio de convergência positivo. Em particular vemos que toda função holomorfa $f:U \to \C$ é analítica em $U$ com respeito a variável $z$. A recíproca também é válida, isto é, toda função que é analítica na variável $z$ também será holomorfa.

Recordamos a seguir os principais resultados a respeito das funções holomorfas de uma variável.

\begin{proposition} \textbf{Princípio do Máximo}.
Se $U \subset \C$ é um aberto conexo e $f:U \to \C$ é holomorfa e não constante então $|f|$ não possui máximo local em $U$.
\end{proposition}

\begin{proposition} \textbf{Princípio da Identidade}. 
Se $f,g:U \to \C$ são funções holomorfas e $f = g$ em um aberto $V \subset U$ então $f=g$ em $U$.
\end{proposition}

\begin{proposition} \textbf{Teorema de Liouville}.
Se $f:\C \to \C$ é holomorfa e limitada então $f$ é constante.
\end{proposition}

O Teorema a seguir é devido a Riemann e como veremos não se generaliza para dimensões maiores que $1$.

\begin{theorem} \emph{(Teorema da Aplicação de Riemann)}
Seja $U \subset \C$ um aberto simplesmente conexo que não é o plano complexo inteiro. Então $U$ é biholomorfo à bola unitária $B_1(0)$, isto é, existe uma função holomorfa bijetora $f:U \to B_1(0)$ tal que sua inversa também é holomorfa.
\end{theorem}

Note que a hipótese $U \neq \C$ acima é essencial pois, pelo teorema de Liouville, não existe uma função holomorfa não constante $f:\C \to B_1(0)$.

\section{Funções holomorfas de várias variáveis} \label{sec:several-complex-variables}

A definição de holomorfia para uma função de $n$ variáveis complexas é dada a partir das equações de Cauchy-Riemann para cada par de variáveis reais de $\C^n$.

\begin{definition}
Seja $U \subset \C^n$ um aberto e $f:U \to \C$ uma função de classe $C^1$. Escreva $f = u + \smo v$ onde $u,v:U \to \R$ e denote por $x_i = \re z_i$ e $y_i = \im z_i$ as coordenadas reais em $\C^n$. Dizemos que \textit{$f$ é holomorfa em $U$} se satisfaz as equações de Cauchy-Riemann 
\begin{equation*}
\frac{\del u}{\del x_i} = \frac{\del v}{\del y_i}~~ \text{ e }~~ \frac{\del u}{\del y_i} = - \frac{\del v}{\del x_i} ~~ \text{ em } U~~~~~ (i=1,\ldots,n)
\end{equation*}
\end{definition}

Definindo os operadores
\begin{equation*}
\frac{\del}{\del z_i} \doteq \frac{1}{2}\left( \frac{\del}{\del x_i} - \smo \frac{\del}{\del y_i} \right)~\text{ e }~ \frac{\del}{\del \bar z_i} \doteq \frac{1}{2}\left( \frac{\del}{\del x_i} + \smo \frac{\del}{\del y_i} \right)
\end{equation*}
vemos que $f$ é holomorfa se e somente se $\frac{\del f}{\del \bar z_i} = 0$ para $i=1,\ldots,n$.

Assim como no caso unidimensional há uma caracterização em termos da diferencial de $f$. Considere a identificação $\R^{2n} \simeq \C^n$ dada por $(x_1,\ldots,x_n,y_1,\ldots,y_n) \mapsto (x_1 + \smo y_1,\ldots,x_n + \smo y_n)$. Segundo esse isomorfismo, a multiplicação por $\smo$ em $\C^n$ corresponde à transformação linear em $\R^{2n}$ dada pela matriz
\begin{equation*}
J_0 = \begin{pmatrix} 0 & -I_n \\ I_n & 0 \end{pmatrix}
\end{equation*}
onde $I_n$ denota a matriz identidade de ordem $n$.

Das equações de Cauchy-Riemann vemos que $f:U \to \C$ é holomorfa se e somente se $df_z J_0 = J_0 df_z$ para todo $z \in U$, onde vemos $df_z$ como operador linear em $\R^{2n}$. Equivalentemente, vendo $df_z$ como um operador em $\C^n$, vemos que $f$ é holomorfa se e somente se $df_z$ é $\C$-linear para todo $z \in U$.

A fórmula de Cauchy se generaliza naturalmente para funções de várias variáveis. Antes de enunciar o resultado é conveniente definirmos os polisdicos em $\C^n$. Dada uma $n$-upla $r = (r_1,\ldots,r_n)$ com $r > 0$ e um ponto $w \in \C^n$ definimos o\textit{ polidisco centrado em $w$ de raio $r$} por
\begin{equation*}
B_r(w) = \{z \in \C^n: |z_i - w_i| < r_i,~i=1,\ldots,n\}.
\end{equation*}

Note que $B_r(w)$ é simplesmente o produto dos discos abertos em $\C$ centrados em $w_i$ e com raio $r_i$.

\begin{proposition}
Seja $f$ uma função holomorfa em um aberto $U \subset \C^n$. Seja $\xi = (\xi_1,\ldots,\xi_n) \in U$ e suponha que $B_r(\xi) \subset U$. Então para todo $z \in B_r(\xi)$ temos a fórmula
\begin{equation*}
f(z) = \left( \frac{1}{2\pi \smo} \right)^n \int_{|w_1 - \xi_1|=r_1} \cdots \int_{|w_n - \xi_n|=r_n} \frac{f(w_1,\ldots,w_n)}{(w_1 - z_1) \cdots (w_n -z_n)} dw_1 \cdots dw_n.
\end{equation*}
\end{proposition}

Assim como no caso unidimensional podemos usar a fórmula acima para escrever uma série de potências para $f$ em torno do ponto $\xi$, isto é, em $B_r(\xi)$ $f$ se escreve como uma série convergente
\begin{equation*}
f(z) = \sum_{i_1,\ldots,i_n=0}^\infty a_{i_0 \cdots i_n} (z_1 - \xi_1)^{i_1} \cdots (z_n - \xi_n)^{i_n}.
\end{equation*}

Vemos portanto que também no caso multidimesional uma função é holomorfa se e somente for analítica nas variáveis $z_1,\ldots,z_n$.

Alguns dos resultados enunciados na seção anterior, como o Princípio do Máximo, o Princípio da Identidade e o Teorema de Liouville, se generalizam facilmente para o caso de funções holomorfas de várias váriaveis. No entanto, alguns fenômenos observados no caso unidimensional, como por exemplo o Teorema da Aplicação de Riemann, não têm uma contrapartida multidimensional, como mostra o exemplo abaixo.

\begin{example}
Denote por $B = B_{(1,1)}(0) = \{z \in \C^2: |z_1|<1,|z_2|<1 \}$ o polidisco de raio $(1,1)$ e por $\mathbb D^2=\{z \in \C^2: ||z|| < 1\}$ o disco unitário em $\C^2$. Note que tanto $\mathbb D$ quanto $B$ são simplesmente conexos, mas, como mostraremos a seguir, $\mathbb D^2$ e $B$ não são biholomorfos, mostrando que o Teorema da Aplicação de Riemann não se generaliza para dimensão $2$.

Os biholomorfismos do polidisco $B$ são da forma
\begin{equation*}
\varphi(z_1,z_2) = \left( e^{\smo \theta_1} \frac{z_1 - a_1}{1 - \overline {a_1} z_1}, e^{\smo \theta_2} \frac{z_2 - a_2}{1 - \overline {a_2} z_2} \right),~~\theta_i \in [0,2\pi),~ a_i \in \mathbb D
\end{equation*}
ou da forma $\psi = \sigma \circ \varphi$ onde $\sigma(z_1,z_2) = (z_2,z_1)$ (veja por exemplo \cite{narasimhan}, cap. 5). Em particular vemos que o grupo de biholomorfismos de $B$ age transitivamente, pois escolhendo os parâmetros $a_1$ e $a_2$ corretamente podemos levar qualquer ponto de $B$ na origem.

Suponha que exista um biholomorfismo $f:\mathbb D^2 \to B$. Compondo com um biholomorfismo de $B$ que leva $f(0)$ em $0$ podemos supor, sem perda de generalidade que $f(0) = 0$. Denote por $\text{Aut}_0(\mathbb D^2)$ e $\text{Aut}_0(B)$ os grupos de biholomorfismos de $\mathbb D^2$ e $B$ que fixam a origem.

A conjugação por $f$ definiria um isomorfismo entre $\text{Aut}_0(\mathbb D^2)$ e $\text{Aut}_0(B)$. No entanto, da descrição dos biholomorfismos de $B$ acima vemos que um elemento de  $\text{Aut}_0(B)$ é da forma $\varphi (z_1,z_2)= (e^{\smo \theta_1}z_1,e^{\smo \theta_2}z_2)$ ou $\varphi (z_1,z_2)= (e^{\smo \theta_2}z_2,e^{\smo \theta_1}z_1)$ e portanto $\text{Aut}_0(B)$ é abeliano, enquanto $\text{Aut}_0(\mathbb D^2)$ não o é, pois contém o grupo unitário $\text U(2)$.

Vemos assim que $\text{Aut}_0(\mathbb D^2)$ e $\text{Aut}_0(B)$ não podem ser isomorfos e portanto o biholomorfismo $f$ não pode existir.

O exemplo acima foi apresentado pela primeira vez por Poincaré e pode ser facilmente generalizado para dimensões maiores.
\end{example}

Além de resultados que deixam de ser válidos quando passamos de funções de uma para várias variáveis há também o fenômeno oposto, isto é, resultados que só valem no caso multidimensional. Um exemplo é o Teorema de Hartogs.

\begin{theorem} \label{thm:hartogs} \emph{(Teorema de Hartogs)}
Seja $f$ uma função holomorfa em uma vizinhança de $B_r(0) \setminus B_{r'}(0) \subset \C^n$ onde $r'_i <r_i$ para $i=1,\ldots,n$ e $n \geq 2$. Então $f$ admite uma única extensão a uma função holmorfa $\widetilde f: B_r(0) \to \C$.
\end{theorem}

Vamos dar a ideia da demonstração para o caso $n=2$. O caso geral é análogo. A ideia é estender a função $f$ usando a fórmula de Cauchy.

Note que a intersecção de $B_r(0) \setminus B_{r'}(0)$ com um plano $z_1 = c$, $|c|<r_1$ é um disco $\{c\} \times \{|z_2|<r_2\}$ se $|c|>r'_1$ ou um anel $\{c\} \times \{r'_2<z_2<r_2\}$ se $|c|<r_1'$.

Defina $\widetilde f: B_r(0) \to \C$ por
\begin{equation*}
\widetilde f(z_1,z_2) = \frac{1}{2 \pi \smo} \int_{|w_2|=r_2} \frac{f(z_1,w_2)}{w_2-z_2} dw_2,~~ |z_1|<r_1,~|z_2|<r_2.
\end{equation*}

Usando o fato de que o integrando é holomorfo em $z_2$ não é difícil ver que $\widetilde f$ é holomorfa em $z_2$ (veja por exemplo o Lema 1.1.3. do capítulo 1 de \cite{huybrechts}) e  pelo fato da série de potências de $f$ convergir uniformemente para $f$ no círculo $|w_2| = r_2$ temos que $\frac{\del \widetilde f}{\del \bar z_1} = \frac{1}{2 \pi \smo} \int_{|w_2|=r_2} \frac{\del}{\del \bar z_1} \left(\frac{f(z_1,w_2)}{w_2-z_2}\right) dw_2 = 0$, pois $f$ é holomorfa em $z_1$. Assim, $\widetilde f$ também é holomorfa em $z_1$.

O fato de que $\widetilde f (z_1,z_2) = f(z_1,z_2)$ para $(z_1,z_2) \in B_r(0) \setminus B_{r'}(0)$ segue da fórmula de Cauchy em uma variável e a unicidade de $\widetilde f$ segue do Princípio da Identidade.\\

O resultado acima certamente não é valido no caso $n=1$. Considere por exemplo a função $f(z)= 1/z$ definida no anel $\{z:1<|z|<2\}$. Se $f$ admitisse uma extensão holomorfa ao disco $\{z:|z|<1\}$ esta teria que coincidir com $1/z$ no disco furado $\{z:0<|z|<1\}$, pelo Princípio da Identidade, mas isso não é possível pois $1/z$ não se estende nem continuamente à origem.

O Teorema de Hartogs também fornece algumas informações sobre as singularidades e o conjunto de zeros de funções holomorfas de mais de uma variável. Por exemplo, se $U$ é um aberto em $\C^n$, $n \geq 2$ e $S \subset U$ é um conjunto discreto, usando o Teorema de Hartogs, podemos ver que qualquer função holomorfa $f:U \setminus S \to \C$ se estende a $U$. Em particular, $f$ não possui singularidades isoladas. Esse argumento também mostra que os zeros de uma função holomorfa não são isolados, pois um zero isolado de $f$ seria uma singularidade isolada de $1/f$.\\

Tendo definido funções holomorfas de várias váriaveis podemos definir o que é uma aplicação holomorfa com valores em $\C^m$.

\begin{definition}
Uma aplicação $f:U \subset \C^n \to \C^m$ é holomorfa se cada componente $f_1,\ldots,f_m:U \to \C$ for uma função holomorfa.
\end{definition}

\section{Variedades Complexas} \label{sec:complex-manifolds}

As variedades complexas são o análogo complexo das variedades diferenciáveis. De maneira sucinta, uma variedade complexa é um espaço toplógico que é localmente modelado em abertos de $\C^n$ e cujas cartas locais diferem por transformações holomorfas.

Vamos tornar isso preciso, recordando primeiro a definição de uma variedade diferenciável.

Um \textit{atlas diferenciável} em um espaço topológico $M$ é uma coleção de pares $(U_i,\varphi_i)$ chamados de \textit{cartas} onde os $U_i$ são abertos cobrindo $M$, cada $\varphi_i:U_i \to \R^n$ é um homemorfismo sobre um aberto de $\R^n$ e, sempre que $U_i \cap U_j \neq \emptyset$, a transição
\begin{equation*}
\varphi_i \circ \varphi_j^{-1}: \varphi_j(U_i \cap U_j) \to \varphi_i(U_i \cap U_j)
\end{equation*}
é uma aplicação diferenciável.

Uma \textit{variedade diferenciável} é um espaço topológico $M$, Hausdorff e com base enumerável, equipado com um atlas diferenciável maximal. O número $n$ é chamado de dimensão de $M$.

A definição de uma variedade complexa é análoga, mas exigimos que as cartas assumam valores em $\C^n$ e que as funções de transição sejam holomorfas.

\begin{definition}
Um \textit{atlas holomorfo} em um espaço topológico $X$ é uma coleção de pares $(U_i,\varphi_i)$ chamados de \textit{cartas holomorfas} onde

\begin{itemize}
	\item [(a)] Cada $U_i$ é um aberto em $X$ e $X = \bigcup_i U_i$
  \item [(b)] Cada $\varphi_i:U_i \to \C^n$ é um homemorfismo sobre um aberto de $\C^n$
  \item [(c)] Sempre que $U_i \cap U_j \neq \emptyset$ a transição
\begin{equation*}
\varphi_{ij} = \varphi_i \circ \varphi_j^{-1}: \varphi_j(U_i \cap U_j) \to \varphi_i(U_i \cap U_j)
\end{equation*}
é uma aplicação holomorfa.
\end{itemize}

Uma \textit{variedade complexa} é um espaço topológico $X$, Hausdorff e com base enumerável, equipado com um atlas holomorfo maximal. O número $n$ é chamado de dimensão complexa de $X$.
\end{definition}

\begin{remark}
Identificando $\C^n \simeq \R^{2n}$ e usando o fato de que toda aplicação holomorfa é diferenciável fica claro da definição que uma variedade complexa de dimensão complexa $n$ é uma variedade diferenciável de dimensão $2n$.

No entanto, nem toda variedade diferenciável $M$ admite uma estrutura complexa. Uma obstrução óbvia é que a dimensão de $M$ deve ser par. Uma outra obstrução é a orientabilidade, pois toda variedade complexa é orientável\footnote{Dizemos que uma variedade diferenciável é orientável se admite um atlas difernciável de modo que o jacobiano das funções de transição tenha sempre determinante positivo.}. De fato, dado um atlas holomorfo $(U_i, \varphi_i)$ em uma variedade complexa $X$, a diferencial da mudança de coordenadas em um ponto $z \in \varphi_j(U_i \cap U_j)$, vista como uma aplicação diferenciável entre abertos de $\R^{2n}$, tem a forma
\begin{equation*}
d(\varphi_i \circ \varphi_j^{-1})_z = \begin{pmatrix} A & B \\ -B & A \end{pmatrix},~~ A,B \in \GL(n,\R)
\end{equation*}
e portanto $\det [d(\varphi_i \circ \varphi_j^{-1})_z] = (\det A)^2 + (\det B)^2 > 0$, o que mostra que $X$ é orientável.

Em geral, decidir se uma variedade real admite ou não uma estrutura complexa é uma tarefa complicada. Sobre as esferas de dimensão par por exemplo é sabido que $S^2$ possue uma única estrutura complexa e que as esferas $S^4$ e $S^{2n}$ para $n> 3$ não admitem nenhuma estrutura complexa. A existência de uma estrutura complexa em $S^6$ ainda é um problema em aberto importante.
\end{remark}

Tendo definido uma estrutura complexa podemos falar sobre funções holomorfas entre variedades complexas.
\begin{definition}
Dadas duas variedades complexas $X$ e $Y$, uma aplicação contínua $f:X \to Y$ é \textit{holomorfa} se para toda carta $(U,\varphi)$ de $X$ e toda carta $(V, \psi)$ de $Y$ com $f(U) \subset V$, a aplicação $\psi \circ f \circ \varphi^{-1}: \varphi(U) \to \psi(f(U))$ é holomorfa.
\end{definition}

Os resultados locais sobre funções holomorfas em $\C^n$ se generalizam imediatamente para as funções holomorfas em variedades complexas.

\begin{proposition} \textbf{Princípio da Identidade}. 
Sejam $X$ e $Y$ variedades complexas com $X$ conexa. Se $f,g:X \to Y$ são funções holomorfas e $f = g$ em um aberto $U \subset X$ então $f=g$ em $X$.
\end{proposition}

\begin{proposition} \textbf{Princípio do Máximo}.
Seja $X$ uma variedade complexa conexa. Se $f:X \to \C$ é uma função holomorfa não constante então $|f|$ não possui máximo local em $X$.
\end{proposition}

Como consequência do Princípio do Máximo vemos que uma variedade compacta não possui muitas funções holomorfas.
\begin{corollary} \label{cor:compact-constant}
Se $X$ é uma variedade complexa compacta e conexa então toda função holomorfa em $X$ é constante.
\end{corollary}

O resultado acima é um primeiro exemplo de como a teoria de variedades complexas difere drasticamente da teoria de variedades diferenciáveis. No caso diferenciável o espaço das funções suaves em uma variedade diferenciável (compacta ou não) é um espaço vetorial de dimensão infinita\footnote{Isso se deve ao fato de existirem partições diferenciáveis da unidade.}, ou seja, existem muitas funções suaves globais, enquanto que na categoria holomorfa (no caso compacto e conexo) só existem as constantes.

Uma consequência desse resultado é o fato de $\C^n$ não possuir subvariedades compactas.

\begin{corollary} \label{cor:compact-submanifold}
Uma variedade complexa compacta e conexa de dimensão positiva não admite um mergulho holomorfo em $\C^n$ para nenhum $n\geq 1$.
\end{corollary}

\begin{proof}
Seja $X$ uma variedade complexa compacta e conexa e suponha que existe um mergulho holomorfo $f:X \to \C^n$. As composições $f_i = f \circ z_i$ de $f$ com as funções coordenadas $z_i:\C^n \to \C$, $i=1,\cdots,n$ definem funções holomorfas globais em $X$ e portanto, do corolário \ref{cor:compact-constant}, cada $f_i$ deve ser constante. Asssim, $f$ é constante e portanto $X$ se reduz a um ponto.
\end{proof}

Este é um outro exemplo da diferença entre as teorias das variedades complexas e diferenciáveis. Em contraste com o resultado acima, o famoso Teorema de Whitney diz que toda variedade diferenciável, compacta ou não, admite um mergulho suave em algum $\R^N$. 

\subsubsection{A estrutura complexa no fibrado tangente}

Dada uma variedade complexa $X$ considere o seu fibrado tangente $TX$. 
 
 Seja $\varphi: U \to \varphi(U) \subset \C^n$ uma carta holomorfa em $U \subset X$. Se $p \in U$, a diferencial $d\varphi _p : T_p X \to T_{\varphi(p)} \C^n = \C^n$ estabelece um isomorfismo linear. Podemos então transportar a estrutura complexa padrão de $\C^n$ a $T_p X$, obtendo assim uma aplicação linear $J_p: T_p X \to T_p X$ satisfazendo $J_p ^2 = -\id$. Explicitamente
\begin{equation*}
J_p  = d (\varphi^{-1})_{\varphi (p)} \circ J_0 \circ d \varphi _p 
\end{equation*}
onde $J_0$ é a multiplicação por $\smo$ em $\C^n$.

Se $\psi: V \to \psi(U) \subset \C^n$ é um outro sistema de coordenadas com $U \cap V \neq \emptyset$ então $\psi = h \circ \varphi$ em $U \cap V$, onde $h:\varphi(U \cap V) \to \psi(U \cap V)$ é um biholomorfismo. Do fato de $h$ ser holomorfa temos que $J_0 d h_q = d h_q J_0$ para todo $q \in \varphi(U \cap V)$ (veja a seção \ref{sec:several-complex-variables}).

Assim, para $p \in U \cap V$ temos
\begin{equation*}
 \begin{split}
 d (\psi^{-1})_{\psi (p)} \circ J_0 \circ d \psi _p &=  d (\psi^{-1})_{\psi (p)} \circ J_0 \circ dh_{\varphi(p)} \circ d \varphi _p \\
    &= d (\psi^{-1})_{\psi (p)} \circ  dh_{\varphi(p)} \circ J_0 \circ d \varphi _p \\
    &= d (\varphi^{-1})_{\varphi (p)} \circ J_0 \circ d \varphi _p,
 \end{split}
\end{equation*}
o que mostra que a definição de $J_p$ independe do sistema de coordenadas.\\

Vemos assim que $T_p X$ é naturalmente um espaço vetorial complexo, onde definirmos a multiplicação por um escalar pela fórmula $(a+\smo b)v = av +bJ_p v$. Dizemos que $J_p$ é a \textit{estrutura complexa induzida} em $T_p X$ e denotamos por $J:TX \to TX$ o endomorfismo de $TX$ definido por $J(p,v)=J_p(v)$.\\

Em coordenadas, se $\varphi$ é dada por $\{z_i = x_i + \smo y_i\}_{i=1,\ldots,n}$, a base $\left \lbrace \frac{\del}{\del x_1},\ldots,\frac{\del}{\del x_n}, \frac{\del}{\del y_1}\ldots, \frac{\del}{\del y_n} \right \rbrace$ de $\R^{2n} \simeq \C^n$ induz pela diferencial $d \varphi_p : T_p X \to \C^n$ uma base
\begin{equation} \label{eq:baseTpX}
\left \lbrace \frac{\del}{\del x_1} \bigg|_p, \ldots, \frac{\del}{\del x_n} \bigg|_p,\frac{\del}{\del y_1} \bigg|_p, \ldots, \frac{\del}{\del y_n} \bigg|_p \right \rbrace \subset T_p X,
\end{equation}
e a aplicação $J_p$ é dada por
\begin{equation} \label{eq:J-coord}
J_p : \frac{\del}{\del x_i} \bigg|_p \longmapsto \frac{\del}{\del y_i} \bigg|_p, \;\;\; \frac{\del}{\del y_i} \bigg|_p \longmapsto -\frac{\del}{\del x_i} \bigg|_p 
\end{equation}

Da própria definição da estrutura complexa induzida, vemos que uma aplicação entre variedades complexas $f:X \to Y$ será holomorfa se e somente se $df_x J^X_x = J^Y_{f(x)} df_x$ para todo $x \in X$, onde $J^X$ e $J^Y$ denotam as estruturas complexas induzidas em $TX$ e $TY$ respectivamente.

\subsection{Exemplos}

Apresentaremos a seguir alguns dos exemplos básicos de variedades complexas.

\begin{example}
O exemplo trivial de variedade complexa é o espaço complexo $n$-dimensional $\C^n$. Mais geralmente, qualquer $\C$-espaço vetorial $V$ de dimensão finita é uma variedade complexa, pois podemos tomar como carta holomorfa global um isomorfismo $\C$-linear $\phi:V \to \C^n$, onde $n = \dim_\C V$.
\end{example}

\begin{example} \label{ex:projective-space} \index{espaço projetivo complexo} \textbf{O espaço projetivo complexo}.
O espaço projetivo complexo de dimensão $n$, denotado por $\pr^n$, é o conjunto de todas as retas complexas em $\C^{n+1}$ que passam pela origem. Podemos defini-lo como sendo o quociente
\begin{equation*}
\pr^n = \frac{\C^{n+1}\setminus \{0\}}{\sim},
\end{equation*}
onde $u \sim v$ se e somente se $u = \lambda v$ para algum $\lambda \in \C^*$.

Consideramos em $\pr^n$ a topologia quociente dada pela projeção canônica $\pi:\C^{n+1}\setminus \{0\} \to \pr^n$ que associa a cada $0 \neq v \in \C^{n+1}$ a sua classe de equivalência $[v] \in \pr^n$, isto é, um conjunto $U \subset \pr^n$ será aberto se e somente se sua pré-imagem $\pi^{-1}(U)$ for aberta em  $\C^{n+1}\setminus \{0\}$.

Com essa topologia a aplicação $\pi$ é aberta, pois se $V \subset \C^{n+1} \setminus \{0\}$ é um aberto então $\pi^{-1}(\pi(V)) = \{\lambda z : z \in V, \lambda \in \C^* \}$ é aberto em $\C^{n+1} \setminus \{0\}$  e portanto $\pi(V)$ é aberto em $\pr^n$. Isso mostra que $\pr^n$ tem base enumerável, pois a imagem de uma base de abertos de $\C^{n+1} \setminus \{0\}$ será uma base de abertos de $\pr^n$. Além disso, essa topologia é Hausdorff. De fato, dados $[v],[w] \in \pr^n$ distintos, os conjuntos $\{\lambda v : \lambda \in \C^*\}$ e $\{\mu w : \mu \in \C^*\}$ são fechados disjuntos em $\C^{n+1} \setminus \{0\}$ e portanto podemos separá-los por abertos disjuntos $V$ e $W$, e  assim as projeções $\pi(V)$ e $\pi(W)$ são abertos disjuntos separando $[v]$ e $[w]$.\\

Se $v=(z_0,\ldots,z_n) \in \C^{n+1}\setminus \{0\}$ denotamos por $[z_0:\cdots:z_n] \in \pr^n$ sua classe de equivalência e os números $z_0,\ldots,z_n$ são chamados \emph{coordenadas homogêneas} de $[v]$. Note que $[z_0:\cdots:z_n] = [\lambda z_0:\cdots: \lambda z_n]$ para todo $\lambda \in \C^*$.

Para definir coordenadas holomorfas em $\pr^n$ considere os subconjuntos
\begin{equation*}
U_i = \{[z_0:\cdots:z_n] \in \pr^n: z_i \neq 0 \} \subset \pr^n,~~~i=0,\ldots,n.
\end{equation*}

É claro que os $U_i$ são abertos em $\pr^n$ e que $\pr^n = \bigcup_{i=0}^n U_i$.

Defina $\varphi_i: U_i \to \C^n$ por
\begin{equation*}
\varphi_i([z_0:\cdots:z_n]) = \left(\frac{z_0}{z_i},\cdots,\frac{z_{i-1}}{z_i},\frac{z_{i+1}}{z_i}\cdots, \frac{z_n}{z_i}\right).
\end{equation*}

É fácil ver que $\varphi_i$ é um homemorfismo e que sua imagem é todo o $\C^n$. A sua inversa é dada por $\varphi_i^{-1}(w_1,\ldots,w_n) = [w_1:\cdots:w_{i-1}:1:w_{i+1}:\cdots:w_n]$.

A definição de $\varphi_i$ possue uma motivação geométrica. Se denotarmos por $H_i$ o hiperplano afim de $\C^n$ dado pela equação $z_i = 1$, uma reta complexa passando pela origem $\ell = [z_0:\cdots:z_n]$ com $z_i \neq 0$ intersectará $H_i$ em um único ponto $p = H_i \cap \ell = (z_0/z_i,\ldots,z_{i-1},1,z_{i+1},\ldots,z_n/z_i)$. Obtemos as coordenadas de $\ell$ projetando $p$ em $\C^n$ via $(z_0,\ldots,z_n) \mapsto (z_0,\ldots,z_{i-1},z_{i+1},\ldots,z_n)$.

Para calcular a expressão da mudança de coordenadas suponha, sem perda de generalidade, que $i < j$. Temos então que $\varphi_j(U_i \cap U_j) = \{(w_1,\ldots,w_n) \in \C^n: w_i \neq 0\}$ e usando a expressão acima para a inversa de $\varphi_j$ vemos que a mudança de coordenadas é dada por
\begin{equation*}
\varphi_{ij}(w_1,\ldots,w_n) = \left( \frac{w_1}{w_i},\cdots, \frac{w_{i-1}}{w_i},\frac{w_{i+1}}{w_i},\cdots,\frac{w_{j-1}}{w_i},\frac{1}{w_i},\frac{w_{j+1}}{w_i}, \cdots \frac{w_n}{w_i}\right),
\end{equation*}
que é claramente holomorfa em $\C^n \setminus \{w_i = 0\}$.

Vemos portanto que o espaço projetivo complexo é uma variedade complexa. É fácil ver que a projeção natural $\pi: \C^{n+1} \setminus \{0\} \to \pr^n$ é holomorfa.\\

Uma outra descrição útil de $\pr^n$ é como um quociente da esfera $S^{2n+1}$ por uma ação do círculo $S^1$. Considere a esfera unitária $S^{2n+1} \subset \C^{n+1}$. Note que dois pontos $u,v \in S^{2n+1}$ definem a mesma reta complexa se e somente se $u = \lambda v$ com $|\lambda| = 1$. Assim, $\pr^n$ pode ser visto como o quociente de $S^{2n+1}$ pela relação de equivalência $u \sim v \Leftrightarrow u = \lambda v$ para algum $\lambda \in S^1$, ou equivalentemente, como o quociente
\begin{equation*}
\pr^n = \frac{S^{2n+1}}{S^1},
\end{equation*}
onde $S^1$ age em $S^{2n+1}$ via $\lambda \cdot (z_0,\ldots,z_n) = (\lambda z_0, \ldots,\lambda z_n)$.

Denotando por $\rho: S^{2n+1} \to \pr^n$ a projeção canônica não é difícl ver que a topologia quociente induzida por $\rho$ é a mesma topologia induzida por $\pi$. Em particular isso mostra que $\pr^n$ é uma variedade compacta, pois é a imagem de $S^{2n+1}$ por uma aplicação contínua.\\

Podemos pensar em $\pr^n$ como sendo uma compactificação de $\C^n$. Considere o aberto $U_0=\{z_0 \neq 0\}$ definido acima. Note que a carta $\varphi_0:U_0 \to \C^n$ definida acima estabelece um isomorfismo holomorfo $U_0 \simeq \C^n$. O complementar de $U_0$ é o conjunto $H_0 = \{z_0 = 0\}$ que é biholomorfo ao espaço projetivo $\pr^{n-1}$ segundo a aplicação $\phi:H_0 \to \pr^{n-1}$, $[0:z_1:\cdots:z_n] \mapsto [z_1:\cdots:z_n]$.

Vemos então que $\pr^n$ pode ser escrito como
\begin{equation*} 
\pr^n = U_0 \cup H_0 \simeq \C^n \cup \pr^{n-1},
\end{equation*}
isto é, o espaço projetivo $\pr^n$ é obtido de $\C^n$ através da junção de uma cópia de $\pr^{n-1}$. O conjunto $H_0$ é chamado de hiperplano no infinito.

Repetindo o argumento acima para o fator $\pr^{n-1}$ obtemos, por indução, uma decomposição
\begin{equation} \label{eq:celullar-decomp-pn}
\pr^n \simeq \C^n \cup \C^{n-1} \cup \cdots \cup \C \cup \{p\},
\end{equation}
onde $p$ é um ponto correspondente a $p_0 = [0:\cdots:0:1] \in \pr^n$.

No caso $n=1$ vemos que $\pr^1$ é difeomorfo à compactificação de $\C$ por um ponto, e portanto $\pr^1 \simeq S^2$ como variedades diferenciáveis.

Lembre que um CW-complexo é um espaço topológico obtido de cópias homeomorfas de discos fechados identificados pela fronteira. Para a definição precisa e os resultados básicos sobre CW-complexos consulte \cite{hatcher}.

Na decomposição (\ref{eq:celullar-decomp-pn}), a cada passo colamos um disco fechado de dimensão $2k$ em $\pr^{k-1}$. O interior desse disco é homeomorfo a $\C^{2k}$ e sua fronteira é colada em $\pr^{k-1}$ segundo a projeção canônica $S^{2k-1} \to \pr^{k-1}$. Isso mostra que $\pr^n$ é um CW-complexo com uma célula em cada dimensão $0,2,\ldots,2n$, e portanto os grupos de homologia com coeficientes inteiros de $\pr^n$ são dados por
\begin{equation*}
H_i(\pr^n,\Z) = \left \lbrace \begin{split} \Z ~ &\text{ se } i=0,2,\ldots,2n \\ 0 ~ &\text{ caso contrário} \end{split} \right.
\end{equation*}

\end{example}

\begin{example} \label{ex:complex-tori} \index{toro complexo} \textbf{Toros complexos.}
Um reticulado em $\C^n$ é um subgrupo de $(\C^n,+)$ da forma
\begin{equation*}
L = \left \lbrace \alpha = \sum_i n_i \alpha_i : n_i \in \Z , i=1,\cdots,2n \right \rbrace ,
\end{equation*}
onde $\alpha_1,\cdots,\alpha_{2n} \in \C^n$ são linearmente independentes sobre $\R$.

Dado um reticulado $L \subset \C^n$ podemos considerar o grupo quociente
\begin{equation*}
X = \frac{\C^n}{L},
\end{equation*}
no qual dois elementos $z,w \in \C^n$ serão identificados se a $z-w \in L$. O conjunto $X$ é chamado \textbf{toro complexo} e veremos a seguir que $X$ é uma variedade complexa compacta de dimensão $n$.

Consideramos em $X$ a topologia quociente induzida pela projeção canônica $\pi: \C^n \to X$, que leva cada ponto $z \in \C^n$ na sua classe $z + L \in X$. Da própria definição de topologia quociente a aplicação $\pi$ é contínua. Além disso $\pi$ é aberta, pois dado um aberto $V \subset \C^n$ temos que
\begin{equation*}
\pi^{-1}(\pi(V)) = \bigcup_{\alpha \in L} (V + \alpha),
\end{equation*}
que é aberto em $\C^n$ e portanto $\pi(V)$ é aberto em $X$. Com isso é fácil mostrar que $X$ é de Hausdorff e tem base enumerável.

Note que todo ponto de $\C^n$ é equivalente a um ponto no paralelogramo fundamental $P = \{ \alpha = \sum_i t_i \alpha_i : t_i \in [0,1] , i=1,\cdots,2n \}$. Em particular temos que $X = \pi(P)$ e portanto, como $\pi$ é contínua, vemos que $X$ é compacto.

Como $L$ é discreto, dado um ponto $z \in \C^n$ podemos escolher uma vizinhança aberta $U$ de $z$ suficientemente pequena de modo que $U$ não contenha nenhum outro ponto de $z + L$. Desta forma $\pi|_U:U \to \pi(U)$ é bijetora e portanto um homeomorfismo. Vemos então que $\pi:\C^n \to X$ é um recobrimento.

Com isso podemos definir cartas holomorfas em $X$. Cobrimos $\C^n$ por abertos $V_i$ de modo que $\pi|_{V_i}:V_i \to \pi(V_i)=U_i$ são homeomorfismos e tomamos como cartas em $X$ as inversas $\varphi_i = (\pi|_{V_i})^{-1}:U_i \to V_i$.

Se $U_i \cap U_j \neq \emptyset$ é fácil ver que $V_j$ intersecta $V_i+\alpha$ para um único $\alpha \in L$ e então $\varphi_j(U_i \cap U_j) = (V_i + \alpha) \cap V_j$ e $\varphi_i(U_i \cap U_j) = V_i \cap (V_j -\alpha)$. A mudança de coordenadas nessa caso é dada por
\begin{equation*}
\begin{split}
\varphi_{ij}:(V_i + \alpha) \cap V_j &\longrightarrow V_i \cap (V_j - \alpha) \\
z &\longmapsto z - \alpha,
\end{split}
\end{equation*}
que é claramente holomorfa e portanto $(U_i,\varphi_i)$ é um atlas holomorfo em $X$.

Note que a aplicação de recobrimento $\pi:\C^n \to X$ é holomorfa, pois sua composição com uma carta $\varphi_i:U_i \to V_i$ é a identidade. Mais ainda, a estrutura complexa acima definida é a única que torna $\pi$ holomorfa.\\

Todo toro complexo $X = \C^n/L$ é difeomorfo a um produto de $2n$-cópias do círculo $S^1$. O difeomorfismo é obtido pela passagem ao quociente da aplicação $\psi:\R \times \cdots \times \R \to \C^n$ dada por $\psi(t_1,\ldots,t_{2n}) = \sum_i t_i \alpha_i$, lembrando que $S^1 \simeq \R / \Z$. Em particular, quaisquer dois toros complexos de mesma dimensão são difeomorfos. A estrutura complexa no entanto é mais rígida, isto é, podem existir toros complexos de mesma dimensão que não são isomorfos como variedades complexas (isto é, não existe um biholomorfismo entre eles). Esse fenômeno já ocorre em dimensão $1$, como veremos as seguir.\\

Sejam $X = \C/L$ e $Y = \C/M$ dois toros complexos unidimensionais e denote por $\pi_L:\C \to X$ e $\pi_M: \C \to Y$ as projeções canônicas. Seja $f:X \to Y$ uma aplicação holomorfa. Compondo com uma translação em $Y$ que leva $f(0)$ em $0$ podemos supor que $f(0) = 0$.

Como $\C$ é simplesmente conexo, a aplicação $f \circ \pi_L:\C \to Y$ se levanta a uma aplicação $F: \C \to \C$ satisfazendo $\pi_M \circ F = f \circ \pi_L$ e podemos escolhê-la de modo que $F(0) = 0$. Note que $F$ é holomorfa pois, localmente, $F = \pi_M^{-1} \circ f \circ \pi_L$.

Como $f(0)=0$ temos que $F(L) \subseteq M$. Mais ainda, temos que $F(z + \alpha) = F(z) \mod M$ para todo $\alpha \in L$ e portanto dados $z \in \C$ e $\alpha \in L$ existe $\omega(z,\alpha) \in M$ de modo que $F(z + \alpha) - F(z) = \omega(z,\alpha)$. Note que o membro esquerdo dessa igualdade é contínuo em $z$ e portanto, como $M$ é discreto, temos que $\omega(z,\alpha) = \omega (\alpha)$ independe de $z$.

Derivando a equação $F(z + \alpha) - F(z) = \omega(\alpha)$ em relação a $z$ temos que $F'(z+ \alpha) = F'(z)$ para todo $\alpha \in M$. Vemos então que $F'$ é duplamente periódica e portanto, pelo Teorema de Liouville, é constante. Logo $F$ deve ser da forma $F(z) = \gamma z$.

Vemos então que toda aplicação holomorfa $f:X \to Y$ é induzida por uma aplicação $F:\C \to \C$ da forma $F(z) = \gamma z + a $ onde $a,\gamma \in \C$ e $\gamma L \subseteq M$.

Esse fato nos permite construir exemplos de toros não isomorfos. Considere por exemplo os reticulados $L = \Z + \xi \Z$ onde $\xi = \exp(\frac{\smo \pi}{4})$ e $M = \Z + \smo \Z$ e sejam $X = \C \slash L$ e $Y = \C \slash M$ os toros associados. Se existisse um biholomorfismo $f:X \to Y$ existiria um número complexo $\gamma$ tal que $\gamma L \subseteq M$. Em particular teríamos que $\gamma \in M = \Z[\smo]$ e portanto $|\gamma|^2$ seria um número inteiro. Por outro lado teríamos também que $\gamma (1 + \xi) \in M$ e portanto $|\gamma (1+\xi)|^2 = |\gamma|^2|1+\xi|^2$ também seria inteiro, o que não pode ocorrer pois $|1+\xi|^2 = 2 + \sqrt 2$.

\end{example}

\begin{example} \label{ex:coverings}
No exemplo acima, para definir a estrutura holomorfa no toro complexo $X$ utilizamos apenas o fato de que $\pi:\C^n \to X$ é um recobrimento e que as transformações de recobrimento são holomorfas.

Podemos generalizar este exemplo para a seguinte situação. Suponha que um espaço topológico $X$ seja recoberto por uma variedade complexa $\widetilde X$ e denote por $\pi:\widetilde X \to X$ a aplicação de recobrimento. Se as transformações de recobrimento (que são homeomorfismos entre abertos de $\widetilde X$) forem holomorfas então existe uma única estrutura complexa em $X$ que torna $\pi$ holomorfa. As cartas holomorfas são obtidas pela composição das cartas em $\widetilde X$ com inversas locais de $\pi$.

Em particular, se $G$ é um grupo agindo em uma variedade complexa $X$ por biholomorfismos e a ação é livre e propriamente descontínua\footnote{Uma ação $G \times X \to X$, $(g,x)\mapsto g\cdot x$ é livre se $g \cdot x \neq x$ para todo $x$ em $X$ e todo $g \neq e$ em $G$ e é propriamente descontínua se dados $x,y \in X$ existem vizinhanças $U$ de $x$ e $V$ de $y$ de modo que o conjunto $\{g \in G: (g\cdot U) \cap V \neq \emptyset \}$ é finito.} então a aplicação quociente $\pi: X \to X \slash G$ é um recobrimento e portanto o epaço de órbitas $X \slash G$ admite uma estrutura complexa.
\end{example}

\subsection{Funções Meromorfas}

Como vimos, uma variedade complexa não possue muitas funções holomorfas. No caso compacto e conexo por exemplo, só existem as constantes (Corolário \ref{cor:compact-constant}). Sendo assim, somos naturalmente levados a considerar uma classe mais ampla de objetos, que são  as chamadas funções meromorfas.\\

Se $X$ uma variedade complexa, dado um ponto $x \in X$ e duas funções holomorfas $f$ e $g$ definidas em vizinhanças $U$ e $V$ de $x$ dizemos que $f$ e $g$ definem o mesmo germe em $x$, e denotamos $f \sim_x g$, se existe uma viznhança $W$ de $x$ com $W \subset U \cap V$ tal que $f|_W = g|_W$.

O \textbf{caule de funções holomorfas} em $x$, denotado por $\mathcal O_{X,x}$, é o conjunto das funções holomorfas definidas em torno de $x$ segundo a relação de equivalência $\sim_x$. A classe de uma função $f$ é denotada por $f_x$ e é chamado germe de $f$ em $x$. Esta definição é um caso particular do que se chama caule de um feixe, do qual vamos falar com mais detalhes no próximo capítulo (veja o Exemplo \ref{ex:stalk2}).

É fácil ver que $\mathcal O_{X,x}$ é domínio de integridade. Denotamos por $Q(\mathcal O_{X,x})$ o seu corpo de frações.

\begin{definition} \label{def:meromorphic-function}
Seja $X$ uma variedade complexa. Uma \textbf{função meromorfa} em $X$ é uma aplicação
\begin{equation*}
f: X \longrightarrow \bigcup_{x \in X} Q(\mathcal O_{X,x})
\end{equation*}
tal que para todo $x_0$ existe uma vizinhança $U$ de $x_0$ e funções holomorfas $g,h:U \to \C$ tal que $f_x = \frac{g_x}{h_x}$.
\end{definition}

Em outras palavras, uma função meromorfa é dada localmente como razão de duas funções holomorfas. Note que, apesar do nome, uma função meromorfa não é de fato uma função, pois seu valor não está definido nos pontos em que o denominador se anula. No entanto, se $f_x = \frac{g_x}{h_x}$ e $h_x \neq 0$ faz sentido calcularmos $f(x) \doteq \frac{g(x)}{h(x)}$ e o valor não depende dos representantes escolhidos.\\

O conjunto das funções meromorfas em $X$ é denotado por $K(X)$. Não é difícil ver que quando $X$ é conexa $K(X)$ é um corpo, chamado \textit{corpo de funções} de $X$. A chamada \emph{Geometria Birracional} se ocupa do estudo das variedades por meio do seu corpo de funções. Um resultado central da Geometria Birracioal Complexa é o chamado Teorema de Siegel, que dá informação sobre a extensão de corpos $\C \subset K(X)$.

Lembre que, dada uma extensão de corpos $K \subset L$ dizemos que $\alpha_1,\ldots,\alpha_r$ são algebricamente dependentes sobre $K$ se existe $f \in K[x_1,\cdots,x_r]$ tal que $f(\alpha_1,\ldots,\alpha_r)=0$ e são algebricamente independentes caso contrário. O grau de transcendência da extensão $K \subset L$ é o número máximo de elementos algebricamente independentes.
 
Note que as funções constantes podem ser vistas como funções meromorfas e portanto obtemos uma extensão de corpos $\C \subset K(X)$. O Teorema de Siegel dá uma limitação para o grau de transcendência dessa extensão. Para uma demonstração consulte \cite{huybrechts}, cap. 2.

\begin{proposition} \emph{(Teorema de Siegel)}
Se $X$ é uma variedade complexa de dimensão $n$, o grau de transcendência de $K(X)$ sobre $\C$ é no máximo $n$.
\end{proposition}

\chapter{Feixes e Cohomologia}
É muito comum, tanto na Geometria quanto na Análise, encontrarmos problemas que podem ser resolvidos localmente, mas que nem sempre possuem uma solução global. Como exemplo deste fenômeno podemos citar o problema decidir quando uma forma diferencial fechada é exata ou ainda o problema relacionado de existência de soluções de equações diferenciais parciais.

O conceito de feixe fornece uma linguagem natural para tratarmos esse tipo de problema e a sua cohomologia vem como uma ferramenta para entender as possíveis obstruções na passagem do local para o global.

\section{Definição}
Considere um espaço topológico $X$. Para cada aberto $U \subset X$ denote por $\mathcal{C}(U)$ o espaço das funções contínuas em $U$ a valores complexos. A continuidade é uma propriedade local, isto é, se $f \in \mathcal{C}(U)$ e $V \subset U$ é um outro aberto de $X$, a restrição $f|_V$ define uma função contínua em $V$, isto é, $f|_V \in \mathcal{C}(V)$.

Se além disso $X$ for uma variedade diferenciável, podemos definir o espaço $\mathcal{C}^{\infty}(U)$ das funções diferenciáveis sobre $U$ a valores complexos. Novamente, como diferenciabilidade é uma propriedade local, se $f \in \mathcal{C}^{\infty}(U)$ e $V \subset U$ então $f|_V \in \mathcal{C}^{\infty}(V)$.

Esses são exemplos de feixes. A ideia do conceito de feixe sobre um espaço topológico é concentrar em um só, objetos que são definidos por propriedades locais.

\begin{definition} \label{def:sheaf}
Um \textbf{pré-feixe} $\mathcal{F}$ de grupos abelianos em $X$ é uma coleção de grupos abelianos $\mathcal{F}(U)$, um para cada aberto $U \subset X$, com $\mathcal{F}(\emptyset) = 0$ e homomorfismos $r_{V,U}: \mathcal{F}(U) \to \mathcal{F}(V)$ sempre que $V \subset U$, satisfazendo
\begin{enumerate}
\item[S1.] $r_{U,U} = \text{id}_{\mathcal{F}(U)}$
\item[S2.] $r_{W,V} \circ r_{V,U} = r_{W,U}$ sempre que $W \subset V \subset U$
\end{enumerate}
O grupo $\mathcal{F}(U)$ é chamado de grupo de seções \index{seção!de um feixe} de $\mathcal F$ sobre $U$ e as aplicações $r_{V,U}$ são chamadas restrições. Com frequência, quando o aberto $U$ estiver subentendido, denotaremos $r_{V,U}(s) = s|_V$ para $s \in \mathcal{F}(U)$.\\

Dizemos que um pré-feixe $\mathcal{F}$ é um \textbf{feixe} se satisfizer as duas condições adicionais: para todo aberto $U \subset X$ e toda cobertura por abertos $U = \bigcup_{i \in I} U_i$ temos
\begin{enumerate}
\item[S3.] Se $s,t \in \mathcal{F}(U)$ e $r_{U_i,U} (s) = r_{U_i,U} (t) $ para todo $i \in I$, então $s = t$,
\item[S4.] Se $s_i \in \mathcal{F}(U_i)$ são tais que $r_{U_i \cap U_j,U_i} (s_i) = r_{U_i \cap U_j,U_j} (s_j)$ para todo $i,j \in I$ então existe $s \in \mathcal{F}(U)$ tal que $r_{U_i,U} (s) = s_i$.
\end{enumerate}
\end{definition}

A condição S3 diz que uma seção é determinada pelas suas restrições e a condição S4 diz que se temos seções dadas em uma cobertura de $U$ e elas concordam nas intersecções então elas se combinam em uma (única) seção definida em todo $U$.\\

\begin{remark}
De maneira análoga podemos definir feixes de anéis (ou de espaços vetoriais, de álgebras, etc.). Nesse caso cada $\mathcal{F}(U)$ será um anel (um espaço vetorial, uma álgebra, etc.) e cada restrição $r_{U,V}$ um homomorfismo de anéis (uma transformação linear, um homomorfismo de álgebras, etc.).
\end{remark}

\subsection*{Exemplos}
Nos exemplos a seguir $X$ é uma variedade complexa. As condições S1-S4 da definição \ref{def:sheaf} são imediatamente verificadas.
\begin{example} \label{ex:sheaf-diff-fcts}
O feixe das funções diferenciáveis em $X$, definido por
\begin{equation*}
\mathcal{C}^{\infty}_X (U) = \{ f:U \to \C : f \text{ é diferenciável} \}
\end{equation*}
com as aplicações de restrição $\mathcal{C}^{\infty}_X (U) \ni f  \mapsto f|_V \in \mathcal{C}^{\infty}_X (V) $ é um feixe, denotado por $\mathcal{C}^{\infty}_X$.

Observe que $\mathcal{C}^{\infty}_X$ é um feixe de anéis e também um feixe de $\C$-álgebras.
\end{example}

\begin{example} \label{ex:sheaf-hol-fcts}
O feixe das funções holomorfas em $X$, definido por
\begin{equation*} \index{Ox@$\mathcal O_X$}
\mathcal{O}_X (U) = \{ f:U \to \C : f \text{ é holomorfa} \}
\end{equation*}
com as aplicações de restrição $\mathcal{O}_X (U) \ni f  \mapsto f|_V \in \mathcal{O}_X (V) $ é um feixe, denotado por $\mathcal{O}_X$.

Nesse exemplo também temos que $\mathcal{O}_X$ é um feixe de anéis e um feixe de $\C$-álgebras.
\end{example}

\begin{example}
O feixe das funções holomorfas que não se anulam, definido por 
\begin{equation*} \index{Oes@$\mathcal O^*_X$}
\mathcal{O}^*_X (U) = \{ f:U \to \C : f \text{ é holomorfa e } f \neq 0 \text{ em } U\}
\end{equation*}
com a operação de multiplicação, é um feixe, denotado por $\mathcal{O}^*_X$. Note que $\mathcal{O}^*_X$ é apenas um feixe de grupos.
\end{example}

Os exemplos \ref{ex:sheaf-diff-fcts} e \ref{ex:sheaf-hol-fcts} acima podem ser generalizados da seguinte maneira.

\begin{example} \label{ex:sheaf-diff-sec}
Dado um fibrado vetorial complexo\footnote{Para a definição consulte a seção \ref{sec:cxvb}.} $\pi: E \to X$ definimos o feixe de seções (diferenciáveis) de $E$ por
\begin{equation*}
\mathcal{E} (U) = \{ s:U \to \pi^{-1}(U) : s \text{ é diferenciável e } \pi \circ s = \text{id}_U \}
\end{equation*}
com as aplicações de restrição naturais. Obtemos assim um feixe, denotado por $\mathcal{E}$. Note que quando $E = X \times \C$ é o fibrado trivial de posto 1, temos que $\mathcal{E} = \mathcal{C}^{\infty}_X$.
\end{example}

\begin{example} \label{ex:sheaf-hol-sec}
Agora se $\pi: E \to X$ é um fibrado holomorfo\footnote{Para a definição consulte a seção \ref{sec:hol-vb}.}, definimos o feixe de seções holomorfas de $E$ por
\begin{equation*}
\mathcal{O}(E) (U) = \{ s:U \to \pi^{-1}(U) : s \text{ é holomorfa e } \pi \circ s = \text{id}_U \}
\end{equation*}
com as aplicações de restrição naturais. Obtemos assim um feixe, denotado por $\mathcal{O}(E)$. Quando $E = X \times \C$ é o fibrado trivial de posto 1 temos que $\mathcal{O}(E) = \mathcal{O}_X$.
\end{example}

Os dois exemplos anteriores contém ainda mais estrutura: dada uma seção $s \in \mathcal{E}(U)$ e uma função $f \in \mathcal{C}^{\infty}(U)$ podemos multiplicá-los, e obtemos uma nova seção $fs \in \mathcal{E}(U)$, ou seja, $\mathcal{E}(U)$ é um módulo sobre o anel $\mathcal{C}^{\infty}(U)$. Analogamente, $\mathcal{O}(E)(U)$ é um módulo sobre $\mathcal{O}_X (U)$.

\begin{definition}
Seja $\mathcal{A}$ um feixe de anéis sobre $X$. Dizemos que $\mathcal{F}$ \textbf{é um feixe de $\mathcal{A}$-módulos} sobre $X$ se cada $\mathcal{F}(U)$ é um $\mathcal{A}(U)$-módulo e as restrições $\mathcal{F}(U) \to \mathcal{F}(V)$ são homomorfismos de $\mathcal{A}(U)$-módulos, onde $\mathcal{F}(V)$ é visto como módulo sobre $\mathcal{A}(U)$ via
\begin{equation*}
fs = r_{V,U}(f) \cdot s, \;\;\; f \in \mathcal{A}(U),\: s \in \mathcal{F}(V).
\end{equation*}

\begin{example} \label{ex:const-sheaf} \index{feixe!constante}
Se $G$ é um grupo abeliano poderíamos tentar definir um feixe constante, fazendo $\mathcal{F}(U) = G $ para cada aberto $\emptyset \neq U \subset X$, $\mathcal{F}(\emptyset) = 0$ e tomando $r_{U,V} = \text{id}_G$ se $V \neq \emptyset$ e $r_{U,\emptyset} = 0$. No entanto, desta maneira não obtemos um feixe: as condições S1-S3 da definção \ref{def:sheaf} são satisfeitas mas a condição S4 não.

De fato, sejam $U,V$ abertos disjuntos não vazios e $W = U \cup V$. Tome $g \in \mathcal{F}(U) = G $ e $h \in \mathcal{F}(V)=G$ elementos distintos de $G$. Temos que $r_{U\cap V,U}(g) = 0 = r_{U\cap V,V}(h)$, pois $U \cap V = \emptyset$, mas não podemos colar a nenhuma seção $f$ sobre $W$, pois teríamos $f = r_{U,W}(f) = g \neq h = r_{V,W}(f) = f$.

No entanto podemos remediar a situação da seguinte maneira: definimos um feixe em $X$ associando a cada aberto $U \subset X$ o grupo das funções localmente constantes em $U$ a valores em $G$. É fácil ver que agora as condições S1-S4 são satisfeitas. Obtemos assim um feixe, chamado \textbf{feixe localmente constante associado a $G$}, denotado por $\underline{G}$ ou simplesmente $G$. Os exemplos que aparecerão com mais frequencia são os feixes $\underline{\Z},\underline{\R}$ e $\underline{\C}$.
\end{example}
\end{definition}

\section{Limites Diretos}
Nesta seção definiremos o limite direto de um sistema de grupos abelianos. Este conceito será usado mais adiante na definição do caule de um feixe e na definição dos grupos de cohomologia de \v{C}ech.

\begin{definition}
Um conjunto dirigido $I$ é um conjunto com uma ordem parcial $\leq$ tal que para todo $i,j \in I$ existe $k \in I$ tal que $k \leq i$ e $k \leq j$.
\end{definition}

\begin{example} \label{ex:open-covering}
Seja $X$ um espaço topológico. Dadas duas coberturas abertas $\mathcal{U} = \{U_i\}_{i \in I}$ e $\mathcal{V} = \{V_j\}_{j \in J}$ de $X$, dizemos que $\mathcal{V}$ é um refinamento de $\mathcal{U}$, e escrevemos $\mathcal{V} \preceq \mathcal{U}$, se para todo $j \in J$ existe $i \in I$ tal que $V_j \subset U_i$.

É fácil ver que $\preceq$ é uma ordem parcial. Além disso, se $\mathcal{U} = \{U_i\}_{i \in I}$ e $\mathcal{V} = \{V_j\}_{j \in J}$, a cobertura $\mathcal{W} = \{U_i \cap V_j \}_{(i,j) \in I \times J}$ é um refinamento comum de $\mathcal{U}$ e $\mathcal{V}$, isto é, $\mathcal{W} \preceq \mathcal{U}$ e $\mathcal{W} \preceq \mathcal{V}$.

Desta maneira, o conjunto de todas as coberturas abertas de $X$ com  a relação de ordem $\preceq$ é um conjunto dirigido.
\end{example}

\begin{example} \label{ex:neighbourhoods}
Seja $X$ um espaço topológico e $x$ um ponto de $X$. Dadas duas vizinhanças abertas $U$ e $V$ de $x$, temos que $U \cap V$ é uma vizinhança de $x$ contida em $U$ e em $V$. Assim, o conjunto das vizinhanças abertas de $x$ com a ordem parcial dada pela inclusão $\subseteq$ é um conjunto dirigido.
\end{example}

\begin{definition}
Um \textit{sistema dirigido} de grupos abelianos é uma coleção $\{A_i\}_{i \in I}$ de grupos abelianos, indexada por um conjunto dirigido com homomorfismos $f_{ij}:A_j \to A_i$ sempre que $i \leq j$ satisfazendo
\begin{enumerate}
\item $f_{ii} = \text{id}_{A_i}$
\item $f_{ij} \circ f_{jk}= f_{ik}$ sempre que $i \leq j \leq k$
\end{enumerate}
\end{definition}

\begin{example} \label{ex:stalk1}
Se $\mathcal{F}$ é um feixe sobre $X$, então $\{\mathcal{F}(U): x \in U \}$, com a ordem definida no exemplo \ref{ex:neighbourhoods} e as restrições $r_{V,U}: \mathcal{F}(U) \to \mathcal{F}(V)$, é um sistema dirigido de grupos abelianos
\end{example}

\begin{definition} \label{def:direct-limit}
Um \textbf{limite direto} de um sistema dirigido de grupos abelianos $(\{A_i\}_{i \in I}, f_{ij})$ é um grupo $L$ definido pela seguinte propriedade universal
\begin{enumerate}
\item Existem homomorfismos $f_i:A_i \to L$ tal que $f_j = f_i \circ f_{ij}$ sempre que $i \leq j$.
\item Se $B$ é um grupo abeliano com homomorfismos $g_i:A_i \to B$ satisfazendo $g_j = g_i \circ f_{ij}$ sempre que $i \leq j$ então existe um homomorfismo $g: L \to B$ tal que $g \circ f_i = g_i$.
\end{enumerate}
\end{definition}

É facil ver que, se existir, o limite direto $L$ é único a menos de isomorfismo e será denotado por $\varinjlim A_i$. Para ver a existência, defina
\begin{equation} \label{eq:dirlim}
L = \bigsqcup_i A_i \big/\sim
\end{equation}
onde $x_i \in A_i$ é equivalente a $x_j \in A_j$ se $f_{ki} (x_i) = f_{kj} (x_j)$ para algum $k \leq i,j$, e defina $f_i: A_i \to L$ o homomorfismo obtido compondo a inclusão de $A_i$ na união disjunta com a projeção na classe de equivalência.\\

Essa descrição nos dá uma maneira intuitiva de pensar no lmite direto $\varinjlim A_i$: ele é formado pelos elementos dos $A_i$'s, e identificamos dois deles se ficam iguais a partir de um certo instante. Essa descrição deve ficar mais clara no próximo exemplo.

\begin{example} \label{ex:stalk2} \textbf{O caule de um feixe}

Seja $\mathcal{F}$ um feixe sobre $X$ e $x \in X$. Como vimos no exemplo \ref{ex:stalk1}, o conjunto das seções de $\mathcal{F}$ sobre abertos contendo $x$ formam um sistema dirigido de grupos abelianos. O limite direto
\begin{equation} \label{eq:stalk}
\mathcal{F}_x = \varinjlim_{x \in U} \mathcal{F}(U)
\end{equation}
é chamado \textbf{caule de $\mathcal{F}$ em $x$}.

Tendo em vista a descrição (\ref{eq:dirlim}), um elemento de $\mathcal{F}_x$ é representado por uma seção $s \in \mathcal{F}(U)$ sobre algum aberto contendo $x$, e identificamos duas seções $s \in \mathcal{F}(U)$ e $t \in \mathcal{F}(V)$ se existe algum aberto $x \in W \subset U \cap V$ tal que $s|_W = t|_W$.

Observe que $\mathcal{F}_x$ tem o mesmo tipo de estrutura que $\mathcal{F}$, isto é, se $\mathcal{F}$ for um feixe de anéis então $\mathcal{F}_x$ será um anel e assim por diante.
\end{example}

\section{Morfismos de feixes}

\begin{definition}
Um morfismo de (pré-)feixes $\phi: \mathcal{F} \to \mathcal{G}$ é uma coleção de homomorfismos $\phi_U: \mathcal{F}(U) \to \mathcal{G}(U)$, um para cada aberto $U$, que comutam com as restrições, isto é
\begin{equation*}
\phi_U (s)|_V = \phi_V (s|_V) \;\;\; s \in \mathcal{F}(U), \: V \subset U.   
\end{equation*}

Note que se $\phi: \mathcal{F} \to \mathcal{G}$ é um homomorfismo de feixes então temos homomorfismos induzidos nos caules
\begin{equation*}
\phi_x : \mathcal{F}_x \to \mathcal{G}_x
\end{equation*}
para todo $x \in X$.
\end{definition}

\begin{example} \label{ex:exp-hom}
A exponencial de funções $f \mapsto \exp (2\pi \smo f)$ define um morfismo $\exp:\mathcal{O}_X \to \mathcal{O}^*_X$.
\end{example}

\begin{proposition} \label{prop:kernel-sheaf}
Se $\phi: \mathcal{F} \to \mathcal{G}$ é um homomorfismo de feixes então a associação
\begin{equation*}
U \longmapsto \ker (\phi_U: \mathcal{F}(U) \to \mathcal{G}(U) )
\end{equation*}
define um feixe sobre $X$, denotado por $\ker \phi$ e chamado \textbf{feixe núcleo}.
\end{proposition}
\begin{proof}
As restrições de $\ker \phi$ são induzidas pelas restrições de $\mathcal{F}$, e portanto $\ker \phi$ é um pré-feixe.

Seja $U \subset X$ um aberto e $\{U_i\}_{i \in I}$ uma cobertura de $U$. A propriedade S3 vale em $\ker \phi$ pois vale em $\mathcal{F}$.

Agora, se $s_i \in \ker \phi_{U_i}$ satisfazem $s_i|_{U_i \cap U_j} = s_j|_{U_i \cap U_j}$ então, como $\mathcal{F}$ é um feixe, temos, pela propriedade S4, que existe $s \in \mathcal{F}(U)$ tal que $s|_{U_i} = s_i$.

A seção $\phi_U (s) \in \mathcal{G}(U)$ satisfaz $\phi_U(s)|_{U_i} = \phi_{U_i}(s|_{U_i}) = \phi_{U_i}(s_i) = 0$ e portanto, pela propriedade 3, temos que $\phi_U (s) = 0$ e portanto $s \in ker \phi_U$, o que mostra que a propriedade S4 vale para $\ker \phi$.

Logo $\ker \phi$ é um feixe.
\end{proof}

\begin{example}
Considere o homomorfismo $\exp: \mathcal{O}_X\to \mathcal{O}_X^*$ definido no exemplo \ref{ex:exp-hom}. Se $U \subset X$ é conexo então o núcleo de $\exp:\mathcal{O}_X (U)\to \mathcal{O}_X^*(U)$ é formado pelas funções constantes em $U$ com valores inteiros e portanto o núcleo do homomorfismo $\exp: \mathcal{O}_X \to \mathcal{O}^*_X$ é o feixe localmente constante $\underline{\Z}$.
\end{example}

\begin{definition}
Dizemos que $\phi: \mathcal{F} \to \mathcal{G}$ é injetor se $\ker \phi$ é o feixe nulo.
\end{definition}

Um análogo da proposição \ref{prop:kernel-sheaf} não vale para a imagem de um morfismo, isto é, a associação
\begin{equation*}
U \longmapsto \im (\phi_U: \mathcal{F}(U) \to \mathcal{G}(U) )
\end{equation*}
não define, em geral, um feixe, como mostra o seguinte exemplo.

\begin{example} \label{ex:im-exp} Tome $X = \C^*$ e considere o homomorfismo exponencial $\exp: \mathcal{O}_X \to \mathcal{O}_X^*$. Seja $f$ a função identidade em $\C^*$. Note que $f \in \mathcal{O}_X^*(\C^*)$.

Se $U = \C^* - \{z:\re z \leq 0 \}$, podemos definir um logaritmo em $U$, isto é, uma função $\log_U \in \mathcal{O}_X(U)$ tal que $\exp (\log_U) = \text{id}_U = f|_U$. Tomando agora $V = \C^* - \{z:\re z \geq 0 \}$ obtemos $\log_V \in \mathcal{O}_X(V)$ tal que $\exp (\log_U) = f|_V$.

Vemos portanto que $f|_U \in \im \exp_U$ e $f|_V \in \im \exp_V$. No entanto, não existe $\log \in \mathcal{O}_X(U \cup V) = \mathcal{O}_X(\C^*)$ com $\exp \circ \log = f$, o que mostra que a condição S4 falha para $\im \exp$. Logo  $\im \exp$ não é um feixe.
\end{example}

Embora a primeira tentativa de definição de $\im \phi$ forneça apenas um pré-feixe, existe uma maneira canônica de construir um feixe a partir de um pré-feixe.

\begin{definition} \label{def:assoc-sheaf}
O \textbf{feixe associado a um pré-feixe $\mathcal{F}$}, denotado por $\mathcal{F}^+$, é o feixe definido da seguinte maneira: as seções  de $\mathcal{F}^+$ sobre um aberto $U$ são aplicações $s: U \to \bigsqcup _{x \in U} \mathcal{F}_x$ tal que para todo $x \in U$ existe um aberto $V \subset U$ contendo $x$ e uma seção $t \in \mathcal{F}(V)$ tal que $s(y) = t(y)$ para todo $y \in V$.
\end{definition}

Em outras palavras, as seções de $\mathcal{F}^+$ são elementos que são dados localmente por seções de $\mathcal{F}$.\\

No exemplo \ref{ex:im-exp}, embora não exista uma seção do pré-feixe $W \mapsto \im \exp_W$ sobre $\C^*$ que se restrinja a identidade em $U$ e $V$, podemos definir uma seção $\iota$ sobre $\C^*$ do feixe associado, fazendo $\iota|_U = \text{id}_U \in \im \exp_U$ e $\iota|_V = \text{id}_V \in \im \exp_V$.

\begin{definition} \label{def:im-sheaf}
O \textbf{feixe imagem} de um morfismo $\phi:\mathcal{F} \to \mathcal{G}$, denotado por $\im \phi$, é o feixe associado ao pré-feixe
\begin{equation*}
U \longmapsto \im (\phi_U: \mathcal{F}(U) \to \mathcal{G}(U) ).
\end{equation*}

Dizemos que $\phi$ é sobrejetor se $\im \phi = \mathcal{G}$.
\end{definition}

\begin{example}
A exponencial $\exp: \mathcal{O}_X \to \mathcal{O}^*_X$ é um homomorfismo sobrejetor, pois todo aberto $U$ pode ser escrito como uma união $U = \bigcup U_i$ com $\exp_{U_i}: \mathcal{O}_X(U_i) \to \mathcal{O}^*_X (U_i)$ sobrejetora.
\end{example}

\begin{remark} Note que, como mostra o exemplo acima, o fato de $\phi:\mathcal{F} \to \mathcal{G}$ ser sobrejetor não implica que $\phi_U:\mathcal{F}(U) \to \mathcal{G}(U)$ é sobrejetor para todo $U$. No entanto, esse exemplo ilustra também que para um morfismo ser sobrejetor, basta que $\phi_U$ o seja para um quantidade suficiente de abertos de $X$. O seguinte lema torna isso mais preciso
\end{remark}

\begin{lemma} \label{lemma:stalk-hom}
Seja $\phi:\mathcal{F} \to \mathcal{G}$ um morfismo de feixes e sejam $\phi_x:\mathcal{F}_x \to \mathcal{G}_x$ os homomorfismos induzidos. Então
\begin{itemize}
\item[a.] $\phi$ é injetor se e somente se $\phi_x$ é injetor para todo $x \in X$
\item[b.] $\phi$ é sobrejetor se e somente se $\phi_x$ é sobrejetor para todo $x \in X$
\end{itemize}
\end{lemma}
\begin{proof}
a. É claro da definição do feixe núcleo que $(\ker \phi)_x = \ker (\phi_x)$. Logo, se $\phi$ é injetora temos que $\ker (\phi_x) = 0$ e portanto $\phi_x$ é injetora.

Para a recíproca vamos utilizar a seguinte afirmação: se $\mathcal{E}$ é um feixe tal que todos os seus caules são triviais então $\mathcal{E}$ é o feixo nulo. De fato: se $\mathcal{E}_x = 0$ existe um aberto $U \ni x$ tal que $\mathcal{E}(U) = 0$. Se $\mathcal{E}_x = 0$ para todo $x$, podemos cobrir $X$ por tais abertos e vemos que $\mathcal{E}(U) = 0$ para todo $U$.

Logo, se $\ker (\phi_x) = 0$ para todo $x$ temos que $\ker \phi$ é o feixe nulo e portanto $\phi$ é injetora.\\

b. Suponha $\phi$ é sobrejetor. Seja $x \in X$ e $\sigma_x \in \mathcal{G}_x$. O germe $\sigma_x$ é representado por algum $\sigma \in \mathcal{G}(U)$, com $x \in U$. Da definição de feixe associado, as seções de $\mathcal{G} = \im \phi$ são dadas localmente por elementos de $\im \phi_U$. Assim, podemos supor $U$ suficientemente pequeno de modo que $\sigma = \phi_U(s)$ para alguma $s \in \mathcal{F}(U)$. Como $\phi$ comuta com as projeções vemos, tomando o limite direto, que $\phi_x(s_x) = \sigma_x$ e portanto $\phi_x$ é sobrejetora.

Suponha agora $\phi_x$ sobrejetora para todo $x$ e sejam $U \subset X$ e $\sigma \in \mathcal G(U)$. Se $x \in U$ vemos, da sobrejetividade de $\phi_x$, que existe um aberto $U_x \subset U$ contendo $x$ e $s_x \in \mathcal G(U_x)$ tal que $\phi_{U_x}(s_x) = \sigma|_{U_x}$ e como $U = \cup_{x \in U} U_x$ temos, da definição do feixe imagem, que $\sigma \in \im \phi(U)$, mostrando que $\phi$ é sobrejetora.
\end{proof}
Tendo definido os feixes núcleo e imagem de um morfismo, podemos definir sequências exatas de feixes,

\begin{definition} \index{sequência!exata de feixes}
Dada uma sequência de feixes e morfismos,
\begin{equation*}
\mathcal{F} \stackrel{\phi}{\longrightarrow} \mathcal{G} \stackrel{\psi}{\longrightarrow} \mathcal{H}
\end{equation*}
dizemos que essa sequência \textbf{é exata em $\mathcal{G}$} se $\im \phi = \ker \psi$.

Mais geralmente, um \textbf{complexo de feixes} \index{complexo!de feixes} é uma sequência da forma
\begin{equation*}
 \cdots \longrightarrow \mathcal{F}^k \stackrel{\phi^k}{\longrightarrow} \mathcal{F}^{k+1} \stackrel{\phi^{k+1}}{\longrightarrow} \mathcal{F}^{k+2} \longrightarrow \cdots
\end{equation*}
com  $\phi^{k+1} \circ \phi^k = 0$ para todo $k$, e dizemos que esse complexo é exato se a sequência $\mathcal{F}^{k-1} \to \mathcal{F}^k \to \mathcal{F}^{k+1}$ for exata em $\mathcal{F}^k$ para todo $k$.
\end{definition}

\begin{remark}
Uma sequência da forma $0 \to \mathcal{F} \to \mathcal{G} \to \mathcal{H} \to 0$ é chamada de sequência exata curta.
\end{remark}

\begin{example} \index{sequência!exponencial}
Como vimos, a exponencial de funções holmorfas é um morfismo sobrejetor e seu núcleo é o feixe localmente constante $\Z$. Obtemos assim uma sequência exata curta
\begin{equation*}
0 \longrightarrow \Z \longrightarrow \mathcal{O}_X \longrightarrow \mathcal{O}^*_X\longrightarrow 0
\end{equation*}
chamada sequência exponencial. Voltaremos a falar dela com mais detalhes.
\end{example}

\begin{example} \label{ex:derhamcx} \textbf{O complexo de de Rham} \index{complexo!de de Rham}

Denote por $\mathcal{A}^k_X$ \index{A@$\mathcal A^k_X$} o feixe das $k$-formas diferenciais em $X$. A diferencial exterior define naturalmente um morfismo de feixes $d:\mathcal{A}^k_X \to \mathcal{A}^{k+1}_X$. Como $d^2 = 0$, obtemos um complexo
\begin{equation*}
 0 \longrightarrow \mathcal{A}^0_X \stackrel{d}{\longrightarrow} \mathcal{A}^1_X \stackrel{d}{\longrightarrow} \cdots \longrightarrow \mathcal{A}^k_X \stackrel{d}{\longrightarrow} \mathcal{A}^{k+1}_X \stackrel{d}{\longrightarrow} \cdots
\end{equation*}
chamado complexo de de Rham de $X$.
\end{example}

De maneira análoga ao lema \ref{lemma:stalk-hom} demonstramos o seguinte
\begin{lemma}
Uma sequência curta de feixes
\begin{equation*}
0 \longrightarrow \mathcal{F} \longrightarrow \mathcal{G} \longrightarrow \mathcal{H} \longrightarrow 0
\end{equation*}
é exata se e somente se a sequência induzida nos caules
\begin{equation*}
0 \longrightarrow \mathcal{F}_x \longrightarrow \mathcal{G}_x \longrightarrow \mathcal{H}_x \longrightarrow 0
\end{equation*}
é exata para todo $x \in X$.
\end{lemma}

Como deve ter ficado claro nos exemplos dados até aqui, a exatidão de uma sequência curta de feixes $0 \to \mathcal{F} \to \mathcal{G} \to \mathcal{H} \to 0$ não se traduz na exatidão da sequência $0 \to \mathcal{F}(X) \to \mathcal{G}(X) \to \mathcal{H}(X) \to 0$. A cohomologia de feixes é uma maneira de medir esta falha.

\section{Cohomologia de Feixes}
Dado um feixe $\mathcal{F}$ sobre um espaço topológico $X$, podemos deifnir grupos abelianos $H^i(X,\mathcal{F})$, chamados grupos de cohomologia de $X$ com coeficientes em $\mathcal{F}$, que carregam informações sobre sequências exatas envolvendo $\mathcal{F}$. Além disso, como no caso de outras teorias de cohomologia (cohomologia singular, cohomologia de de Rham, cohomologia de Dolbeaut, etc.), o conhecimento dos grupos $H^i(X,\mathcal{F})$, nos dá informações sobre as diversas estruturas de $X$ (topologia, estrutura diferenciável, estrutura complexa, etc.).

Existem diversas maneiras de definirmos tais grupos, cada qual com suas vantagens e desvantagens. Nesta seção apresentaremos duas delas: a cohomologia de \v{C}ech e a cohomologia via resoluções. A primeira tem uma natureza combinatória e de certo modo mais construtiva, enquanto a segunda possui boas propriedades formais, o que simplifica drasticamente alguns cálculos. A boa notícia é que poderemos usar esses dois pontos de vista, pois quando $X$ é um espaço paracompacto, por exemplo uma variedade, essas definições coincidem.

O objetivo aqui é apresentar a cohomologa de feixes mais como uma ferramenta do que como uma teoria em si. Tudo será feito tendo como objetivo principal as aplicações na geometria complexa. Desta maneira, alguns resultados técnicos serão enunciados sem demonstração. No entanto, resultados mais elementares ou aqueles cujas demonstrações são instrutivas para a compreensão da aplicação da teoria serão devidamente demonstrados.

\subsection{Cohomologia de \v{C}ech} \label{sec:cech-cohom} \index{cohomologia!de \v{C}ech}
Durante toda esta seção $X$ será um espaço topológico e $\mathcal{F}$ um feixe sobre $X$.

Seja $\mathcal{U} = \{U_i\}_{i \in I}$ uma cobertura de $X$ por abertos. Vamos supor daqui em diante que $I$ é um conjunto ordenado. É conveniente introduzir a notação
\begin{equation*}
U_{i_0 i_1 \cdots i_n} = U_{i_0} \cap U_{i_1} \cap \cdots \cap U_{i_n}, \;\; \{i_0, \cdots , i_n\} \subset I.
\end{equation*}

Note que $U_{i_0\cdots \widehat{i_k} \cdots i_n} \supseteq U_{i_0 \cdots i_n}$.

Uma \textbf{n-cocadeia de \v{C}ech} com relação a $\mathcal{U}$ é uma coleção $f=(f_{i_0,\cdots,i_n})$ de seções de $\mathcal{F}$, uma para cada $U_{i_0 \cdots i_n}$ com $i_0<\cdots<i_n$. Note que o espaço das n-cocadeias de \v{C}ech, denotado por $\check{C}^n (\mathcal{U},\mathcal{F})$, é o produto
\begin{equation} \label{eq:cech-cochains}
\check{C}^n (\mathcal{U},\mathcal{F}) = \prod_{i_0<\cdots<i_n} \mathcal{F}(U_{i_0 \cdots i_n})
\end{equation}

Existe uma aplicação de cobordo $\delta: \check{C}^n (\mathcal{F},\mathcal{U}) \to \check{C}^{n+1} (\mathcal{F},\mathcal{U})$, definida por
\begin{equation} \label{eq:cech-coboundary}
(\delta f)_{i_0 \cdots i_{n+1}} = \sum_{k=0}^{n+1} (-1)^k \: f_{i_0\cdots \widehat{i_k} \cdots i_{n+1}}|_{U_{i_0 \cdots i_{n+1}}}
\end{equation}

Definimos os cociclos de \v{C}ech \index{cociclos!de \v{C}ech} e os cobordos por
\begin{equation*}
\check{Z}^n (\mathcal{U},\mathcal{F}) = \ker \{ \delta: \check{C}^n (\mathcal{U},\mathcal{F}) \longrightarrow \check{C}^{n+1} (\mathcal{U},\mathcal{F})\} \;\; \text{ e } \;\; \check{B}^n (\mathcal{U},\mathcal{F}) = \im \{\delta: \check{C}^{n-1} (\mathcal{U},\mathcal{F}) \longrightarrow \check{C}^n (\mathcal{U},\mathcal{F}) \}
\end{equation*}
respectivamente.

Um cálculo direto mostra que $\delta \circ \delta = 0$, isto é, $\check{B}^n (\mathcal{U},\mathcal{F}) \subseteq \check{Z}^n (\mathcal{U},\mathcal{F})$.

\begin{definition} \label{def:cech-cohom-rel}
O \textbf{n-ésimo grupo de cohomologia de \v{C}ech} com coeficientes em $\mathcal{F}$ com relação a cobertura $\mathcal{U}$ é o grupo
\begin{equation} \label{eq:cech-cohom-rel}
\check{H}^n (\mathcal{U},\mathcal{F}) = \frac{\check{Z}^n (\mathcal{U},\mathcal{F})}{\check{B}^n (\mathcal{U},\mathcal{F})}.
\end{equation}
\end{definition}

Quando $n=0$, uma cocadeia é simplesmente uma coleção de seções $f_i \in \mathcal{F}(U_i)$. O espaço dos cobordos é trivialmente 0, de modo que $\check{H}^0 (\mathcal{U},\mathcal{F}) = \check{Z}^0 (\mathcal{U},\mathcal{F})$. A aplicação de cobordo nesse caso é dada por
\begin{equation*}
(\delta f)_{ij} = f_i - f_j
\end{equation*}
onde omitimos as restrições por conveniência.

Vemos então que $\delta f = 0$ se e somente se $f_i = f_j$ em $U_i \cap U_j$, para todo $i,j \in I$. Pelos axiomas S3 e S4 da definição de feixe isso ocorre se e somente se $f_i = f|_{U_i}$ para alguma seção global $f \in \mathcal{F}(X)$.

Assim vemos que
\begin{equation*} 
\check{H}^0 (\mathcal{U},\mathcal{F}) = \mathcal{F}(X),
\end{equation*}
isto é, no nível 0, o grupo de cohomologia é justamente o espaço das seções globais de $\mathcal{F}$.\\

Observe que dependemos da escolha de uma cobertura $\mathcal{U}$ de $X$ para definir os grupos de cohomologia e em geral não há uma maneira canônica de fazer essa escolha. Gostaríamos portanto de definir grupos de cohomologia que independem da cobertura $\mathcal{U}$. Isso será feito tomando coberturas mais e mais finas e tomando o limite direto dos grupos $\check{H}^n (\mathcal{U},\mathcal{F})$.\\

Conforme o Exemplo \ref{ex:open-covering}, o conjunto das coberturas abertas de $X$ é um conjunto dirigido: dizemos que $\mathcal{V} \preceq \mathcal{U}$ ($\mathcal{V}$ é mais fina que $\mathcal{U}$), se para todo $j \in J$ existe $i \in I$ tal que $V_j \subset U_i$. Isso pode ser codificado, utilizando o Axioma da Escolha, na existência de uma função $r:J \to I$ tal que $V_j \subset U_{r(j)}$.

Obtemos assim aplicações de refinamento $\widetilde{r}:\check{C}^n(\mathcal{U},\mathcal{F}) \to \check{C}^n(\mathcal{V},\mathcal{F})$, definidas por
\begin{equation*}
(\widetilde{r}f)_{j_0 \cdots j_n} = (f_{r(j_0) \cdots r(j_n)}|_{V_{j_0 \cdots j_n}})
\end{equation*}

Note que $\widetilde{r} \circ \delta = \delta \circ \widetilde{r}$ e portanto $\widetilde{r}$ manda cociclos em cociclos e cobordos em cobordos. Obtemos assim uma aplicação induzida nas cohomologias
\begin{equation*}
H(r): \check{H}^n(\mathcal{U},\mathcal{F}) \longrightarrow \check{H}^n(\mathcal{V},\mathcal{F})
\end{equation*}

Não é difícil mostrar também que a aplicação induzida $H(r)$ não depende da escolha de $r:J \to I$. Mais precisamente, duas aplicações $H(r)$ e $H(r')$ são homotópicas no sentido de álgebra homológica (veja \cite{miranda}, cap. 9, lema 3.1). Obtemos assim aplicações $\check{H}^n(\mathcal{U},\mathcal{F}) \to \check{H}^n(\mathcal{V},\mathcal{F})$ que dependem apenas das coberturas $\mathcal{U}$ e $\mathcal{V}$. Denotaremos essa aplicação por $H^{\mathcal{U}}_{\mathcal{V}}$.

Como $H^{\mathcal{U}}_{\mathcal{U}}=\text{id}$ e $H^{\mathcal{V}}_{\mathcal{W}} \circ H^{\mathcal{U}}_{\mathcal{V}} = H^{\mathcal{U}}_{\mathcal{W}}$ quando $\mathcal{W} \preceq \mathcal{V} \preceq \mathcal{U}$, obtemos assim um sistema dirigido de grupos abelianos indexado pelo conjunto das coberturas de $X$.

\begin{definition}
O \textbf{n-ésimo grupo de cohomologia de \v{C}ech} de $X$ com coeficientes em $\mathcal{F}$ é o limite direto
\begin{equation} \label{eq:cech-cohom}
\check{H}^n(X,\mathcal{F}) = \varinjlim _{\mathcal{U}} \check{H}^n(\mathcal{U},\mathcal{F})
\end{equation}
\end{definition}

Em grau 0, como $\check{H}^n(\mathcal{U},\mathcal{F})= \mathcal{F}(X)$ para toda cobertura $\mathcal{U}$, temos que
\begin{equation} \label{eq:H0cech=global-sec}
\check{H}^0 (X,\mathcal{F}) = \mathcal{F}(X).
\end{equation}

Em geral, não há uma interpretação tão clara para os grupos de cohomologia de grau mais alto. No entanto, como veremos mais adiante, para certos feixes, a cohomologia de \v{C}ech coincide com outras cohomolgias conhecidas.\\

Um dos fatos mais importantes sobre a cohomologia de feixes é que a associação $\mathcal{F} \mapsto \check{H}^{\bullet}(X,\mathcal{F})$ transforma sequências exatas curtas de feixes em sequências exatas longas nos grupos de cohomologia.

Dada uma sequência exata curta
\begin{equation*}
0 \longrightarrow \mathcal{F} \stackrel{\phi}{\longrightarrow} \mathcal{G} \stackrel{\psi}{\longrightarrow} \mathcal{H} \longrightarrow 0
\end{equation*}
obtemos naturalmente aplicações nos grupos de $n$-cocadeias de \v{C}ech
\begin{equation*}
\check{C}^n (\mathcal{U},\mathcal{F}) \stackrel{\phi}{\longrightarrow} \check{C}^n (\mathcal{U},\mathcal{G}) \;\; \text{ e } \;\; \check{C}^n (\mathcal{U},\mathcal{G}) \stackrel{\psi}{\longrightarrow} \check{C}^n (\mathcal{U},\mathcal{H})
\end{equation*}
que comutam com $\delta$. Obtemos portanto aplicações induzidas nas cohomologias
\begin{equation*}
\check{H}^n (X,\mathcal{F}) \stackrel{\phi^*}{\longrightarrow} \check{H}^n (X,\mathcal{G}) \;\; \text{ e } \;\; \check{H}^n (X,\mathcal{G}) \stackrel{\psi^*}{\longrightarrow} \check{H}^n (X,\mathcal{H}),
\end{equation*}
simplesmente considerando as aplicações induzidas nos $\check{H}^n(\mathcal{U},\cdot)$ e tomando o limite.\\

Vamos construir agora aplicações $\check{H}^n(X,\mathcal{H}) \to \check{H}^{n+1}(X,\mathcal{F})$.

Dado um elemento $[\sigma] \in \check{H}^n(X,\mathcal{H})$ ele é representado por algum $\sigma \in \check{C}^n(\mathcal{U},\mathcal{H})$ com $\delta \sigma = 0$. Como $\psi$ é sobrejetora, existe um refinamento $\mathcal{U}' \preceq \mathcal{U}$ e $\tau \in \check{C}^n(\mathcal{U}',\mathcal{G})$ tal que $\psi \tau = r (\sigma)$, onde $r$ denota a restrição de $\mathcal{U}$ para $\mathcal{U}'$. Agora temos que
\begin{equation*}
\psi \delta \tau = \delta \psi \tau = \delta r \sigma = r \delta \sigma = 0,
\end{equation*}
e portanto $\delta \tau|_U \in \ker \psi_U$ para todo $U \in \mathcal{U}'$.

Da exatidão da sequência em $\mathcal{G}$ existe um refinamento $\mathcal{U}'' \preceq \mathcal{U}'$ e $\mu \in \check{C}^{n+1}(\mathcal{U}'',\mathcal{F})$ tal que $\phi \mu = \delta \tau$ e portanto temos que
\begin{equation*}
\phi \delta \mu = \delta \phi \mu = \delta^2 \tau = 0
\end{equation*}
e como $\phi$ é injetora segue que $\delta \mu = 0$ e portanto $\mu$ define uma classe em $\check{H}^{n+1}(\mathcal{U}'',\mathcal{F})$ e consequentemente um elemento $[\mu] \in \check{H}^{n+1}(X,\mathcal{F})$. É facil ver que esta classe de cohomologia independe das escolhas feitas.

Definimos então $\Delta: \check{H}^n(X,\mathcal{H}) \to \check{H}^{n+1}(X,\mathcal{F})$ por $\Delta [\sigma] = [\mu]$.

\begin{proposition} \label{prop:cech-long-exact} \textbf{Sequência exata longa na cohomologia de \v{C}ech}\\
A sequência
\begin{equation*}
\begin{split}
0 &\longrightarrow \check{H}^0(X,\mathcal{F}) \stackrel{\phi^*}{\longrightarrow} \check{H}^0(X,\mathcal{G}) \stackrel{\psi^*}{\longrightarrow} \check{H}^0(X,\mathcal{H}) \stackrel{\Delta}{\longrightarrow} \\
&\longrightarrow \check{H}^1(X,\mathcal{F}) \stackrel{\phi^*}{\longrightarrow} \check{H}^1(X,\mathcal{G}) \stackrel{\psi^*}{\longrightarrow} \check{H}^1(X,\mathcal{H}) \stackrel{\Delta}{\longrightarrow} \cdots \\
 & \vdots \\
 & \longrightarrow \check{H}^n(X,\mathcal{F}) \stackrel{\phi^*}{\longrightarrow} \check{H}^n(X,\mathcal{G}) \stackrel{\psi^*}{\longrightarrow} \check{H}^n(X,\mathcal{H}) \stackrel{\Delta}{\longrightarrow} \cdots
\end{split}
\end{equation*}
é exata.
\end{proposition}

\subsection{Cohomologia via resoluções} \label{sec:derived-cohomology} \index{cohomologia!via resoluções}
A cohomologia de feixes usando resoluções tem um caráter mais abstrato que a cohomologia de \v{C}ech, mas tem a vantagem de ser computacionalmente mais eficiente. A linguagem mais adequada para definir tal cohomologia é a de funtores derivados em categorias abelianas. No entanto isso exigiria introduzir toda uma teoria que só seria usada na definição e em alguns resultados básicos. Como o objetivo aqui é usar a cohomolgia de feixes como uma ferramenta apenas, usaremos uma abordagem mais \textit{ad hoc}, via resoluções flasque ou acíclicas. Para o tratamento usando funtores derivados consulte \cite{voisin}.\\

\begin{definition} \index{resolução!de um feixe}
Uma resolução de um feixe $\mathcal{F}$ é um complexo de feixes $0 \to \mathcal{F}^0 \to \mathcal{F}^1 \to \cdots$ junto com um morfismo $\mathcal{F} \to \mathcal{F}^0$ de modo que o complexo
\begin{equation*}
0 \longrightarrow \mathcal{F} \longrightarrow \mathcal{F}^0 \longrightarrow \mathcal{F}^1 \longrightarrow \mathcal{F}^2 \longrightarrow \cdots
\end{equation*}
é exato.
\end{definition}

\begin{example} \label{ex:derham-res}
Seja $(\mathcal{A}_X^{\bullet},d)$ o complexo de de Rham de $X$ (cf. exemplo \ref{ex:derhamcx}). O núcleo de $d: \mathcal{A}^0_X \to \mathcal{A}^1_X$ é formado pelas funções localmente constantes, ou seja, $\ker \{d: \mathcal{A}^0_X \to \mathcal{A}^1_X \} = \underline{\R}$ é o feixe localmente constante com caule $\R$.

Agora, pelo lema de Poincaré, para $U \subset X$ um aberto contrátil suficientemente pequeno, toda $k$-forma fechada $\alpha \in \mathcal{A}^k_X (U)$ é exata, e portanto obtemos uma sequência exata de feixes
\begin{equation*}
0 \longrightarrow \underline{\R} \longrightarrow \mathcal{A}^0_X \longrightarrow \mathcal{A}^1_X \longrightarrow \mathcal{A}^2_X \longrightarrow \cdots.
\end{equation*}

Vemos então que o complexo de de Rham é uma resolução do feixe localmente constante $\underline{\R}$.
\end{example}

A cohomologia de um feixe será definida por meio de resoluções por feixes flasque.
\begin{definition} \index{feixe!flasque}
Um feixe $\mathcal{F}$ sobre $X$ é \textit{flasque} se para todo aberto $U \subset X$ a restrição $r_{U,X}:\mathcal{F}(X) \to \mathcal{F}(U)$ é sobrejetora.
\end{definition}

Em outras palavras, um feixe é flasque se as seções definidas em um aberto de $X$ podem ser estendidas para todo $X$.

\begin{definition}
Seja $\mathcal{F}$ um feixe sobre $X$ e $0 \to \mathcal{F}^0 \stackrel{\phi^0}{\to} \mathcal{F}^1 \stackrel{\phi^1}{\to} \cdots$ uma resolução tal que cada $\mathcal{F}^k, k \geq 0$ é flasque.

O \textbf{$i$-ésimo grupo de cohomologia de $X$ com coeficientes em $\mathcal{F}$}, denotado por $H^i(X,\mathcal{F})$, é o $i$-ésimo grupo de cohomologia do complexo
\begin{equation*}
\mathcal{F}^0 (X) \stackrel{\phi^0}{\longrightarrow} \mathcal{F}^1 (X) \stackrel{\phi^1}{\longrightarrow} \mathcal{F}^2 (X) \stackrel{\phi^2}{\longrightarrow} \cdots, 
\end{equation*}
ou seja,
\begin{equation} \label{eq:cohom}
H^i(X,\mathcal{F}) = \frac{ \ker \{\phi^i: \mathcal{F}^i(X) \to \mathcal{F}^{i+1}(X) \} }{ \im \{\phi^{i-1}: \mathcal{F}^{i-1}(X) \to \mathcal{F}^i(X) \} }.
\end{equation}
\ \ 
\end{definition}

Todo feixe admite uma \textbf{resolução canônica} \index{resolução!canônica} por feixes \textit{flasque}, também chamada de resolução de Godement, construída da seguinte maneira. Dado um feixe $\mathcal F$ considere o feixe $\mathcal F_{\text{God}}$ definido por
\begin{equation*}
\mathcal F_{\text{God}}(U) = \prod_{x \in U} \mathcal F_x,
\end{equation*}
com as aplicações de restrição óbvias.

É claro que $\mathcal F_{\text{God}}$ é \textit{flasque} e há uma inclusão natural $\mathcal F \to \mathcal F_{\text{God}}$, que associa a cada seção $\sigma \in \mathcal F(U)$ o elemento induzido $\sigma_x \in \mathcal F_x$. A resolução canônica é definida indutivamente: fazendo $\mathcal F^0 = F_{\text{God}}$ e\footnote{Dados feixes $\mathcal{F}$ e $\mathcal{G}$ de modo que $\mathcal{G}(U)$ é um subgrupo de $\mathcal{F}(U)$ para todo $U \subset X$ definimos o feixe quociente $\mathcal{F}/\mathcal{G}$ como sendo o feixe associado ao pré-feixe $U \mapsto \mathcal{F}(U)/\mathcal{G}(U)$.} $\mathcal C^1 = F_{\text{God}} / \mathcal F$ obtemos uma sequência exata de feixes
\begin{equation} \label{eq:canonical-res1}
0 \longrightarrow \mathcal F \longrightarrow \mathcal F^0 \longrightarrow \mathcal C^1 \longrightarrow 0.
\end{equation}

Definimos então $\mathcal F^n = \mathcal C^n_{\text{God}}$ e $\mathcal C^{n+1} = \mathcal C^n_{\text{God}} / \mathcal C^n$, obtendo assim sequências exatas de feixes
\begin{equation} \label{eq:canonical-res2}
0 \longrightarrow \mathcal C^n \longrightarrow \mathcal F^n \longrightarrow \mathcal C^{n+1} \longrightarrow 0, ~~~ (n \geq 1).
\end{equation}

Compondo as aplicações $\mathcal F^n \to \mathcal C^{n+1}$ e $\mathcal C^{n+1} \to \mathcal F^{n+1}$ obtemos $\mathcal F^n \to \mathcal F^{n+1}$, produzindo assim uma resolução $0 \to \mathcal F \to \mathcal{F}^0 \to \mathcal{F}^1 \to \cdots$ de $\mathcal F$ por feixes \textit{flasque}.\\

É um fato que duas resoluções \textit{flasque} diferentes produzem grupos de cohomologia naturalmente isomorfos (isso seguirá por exemplo da equação (\ref{eq:flasque-cohomology}) e da proposição \ref{prop:acyclic-res}), de modo que a definição acima faz sentido.\\

Assim como no caso da cohomologia de \v{C}ech temos uma identificação
\begin{equation} \label{eq:H0res=global-sec}
H^0(X,\mathcal{F}) = \mathcal{F}(X).
\end{equation}

Mais precisamente, se $0 \to \mathcal{F} \stackrel{j}{\to} \mathcal{F}^{\bullet}$ é uma resolução então $H^0(X,\mathcal{F}) = j (\mathcal{F}(X))$.
De fato: no nível das seções globais temos um complexo $0 \to \mathcal{F}(X) \stackrel{j}{\to} \mathcal{F}^0(X) \stackrel{\phi^0}{\to} \mathcal{F}^1(X) \to \cdots$ e portanto, se $\sigma \in \mathcal{F}(X)$, temos $\phi^0(j \sigma) = 0$, ou seja, $j \sigma \in \ker \phi^0 = H^0(X,\mathcal{F})$. Reciprocamente, dada $\widetilde{\sigma} \in H^0(X,\mathcal{F}) = \ker \phi^0$ temos que $\widetilde{\sigma}|_{U_i} \in \ker \phi^0_{U_i}$ e, passando a um refinamento se necessário, temos $\widetilde{\sigma}|_{U_i} = j_{U_i} \sigma_i$ para $\sigma_i \in \mathcal{F}(U_i)$. Como $j$ é injetora as $\sigma_i$'s concordam nas intersecções e portanto $\sigma|_{U_i} = \sigma_i$ define um elemento $\sigma \in \mathcal{F}(X)$ tal que $j \sigma = \widetilde{\sigma}$.\\

Se $0 \to \mathcal A \to \mathcal B \to \mathcal C \to 0$ é uma sequência exata de feixes e $\mathcal A$ é \textit{flasque} não é difícil ver que complexo induzido nas seções globais $0 \to \mathcal A(X) \to \mathcal B(X) \to \mathcal C(X) \to 0$ é exato e além disso $\mathcal B$ é \textit{flasque} se e somente se $\mathcal C$ é \textit{flasque}. Assim sendo, se $\mathcal F$ é um feixe \textit{flasque} e $0 \to \mathcal F \to \mathcal{F}^0 \to \mathcal{F}^1 \to \cdots$ é sua a resolução canônica, temos que as sequências exatas (\ref{eq:canonical-res1}) e (\ref{eq:canonical-res2}) envolvem apenas feixes \textit{flasque} e portanto induzem sequências exatas nas seções globais. Desta forma o complexo $0 \to \mathcal F(X) \to \mathcal{F}^0(X) \to \mathcal{F}^1(X) \to \cdots$ é exato, de onde concluimos que
\begin{equation} \label{eq:flasque-cohomology}
H^i(X,\mathcal{F}) = 0 \text { para todo } i \geq 1 \text{ se } \mathcal{F} \text{ é flasque.}
\end{equation}

Feixes com cohomologia trivial recebem um nome especial:
\begin{definition} \index{feixe!acíclico}
Um feixe $\mathcal{F}$ é acíclico se $H^i(X,\mathcal{F}) = 0$ para todo $i \geq 1$.
\end{definition}

Devido a existência de partições diferenciáveis da unidade, um argumento simples (veja por exemplo \cite{g-h} p. 42) mostra que 

\begin{example} \label{ex:acyclic}
Os feixes, $\mathcal{C}^{\infty}_X$, $\mathcal{A}^k_X$ ou ainda qualquer feixe de seções (diferenciáveis) de um fibrado $E$ são exemplos de feixes acíclicos.
\end{example}

Os feixes acíclicos aparecem com muito mais frequência e de maneira bem mais natural que os feixes flasque. O seguinte resultado mostra que esses podem igualmente ser usados para calcular os grupos de cohomologia. Para uma demonstração consulte \cite{voisin}, prop. 4.32.

\begin{proposition} \label{prop:acyclic-res}
Seja $\mathcal{F} \to \mathcal{F}^{\bullet}$ uma resolução de $\mathcal{F}$ e suponha que cada $\mathcal{F}^k$ é um feixe acíclico. Então $H^i(X,\mathcal{F})$, é o i-ésimo grupo de cohomologia do complexo
\begin{equation*}
\mathcal{F}^0 (X) \stackrel{\phi^0}{\longrightarrow} \mathcal{F}^1 (X) \stackrel{\phi^1}{\longrightarrow} \mathcal{F}^2 (X) \stackrel{\phi^2}{\longrightarrow} \cdots, 
\end{equation*}
ou seja,
\begin{equation*}
H^i(X,\mathcal{F}) = \frac{ \ker \{\phi^i: \mathcal{F}^i(X) \to \mathcal{F}^{i+1}(X) \} }{ \im \{\phi^{i-1}: \mathcal{F}^{i-1}(X) \to \mathcal{F}^i(X) \} }.
\end{equation*}
\end{proposition}

Essa proposição nos permite identificar, para certos feixes, os grupos $H^i(X,\mathcal{F})$ com os outros grupos de cohomologia conhecidos.
\begin{example} \textbf{Cohomologia de de Rham}\\
Vimos no exemplo \ref{ex:derham-res} que o complexo de de Rham $(\mathcal{A}^{\bullet}_X,d)$ é uma resolução do feixe localmente constante $\underline{\R}$. Como os feixes $\mathcal{A}^k_X$ são acíclicos, a proposição acima nos diz que
\begin{equation*}
H^i(X,\underline{\R}) = \frac{ \ker \{d: \mathcal{A}^i(X) \to \mathcal{A}^{i+1}(X) \} }{ \im \{d: \mathcal{A}^{i-1}(X) \to \mathcal{A}^i(X) \} },
\end{equation*}
ou seja, o $i$-ésimo grupo de cohomologia de $X$ com coeficientes no feixe constante $\underline{\R}$ é o i-ésimo grupo de cohomolgia de de Rham:
\begin{equation} \label{eq:derhamcohom}
H^i(X, \underline{\R}) = H^i_{dR}(X).
\end{equation}
\\
\end{example}

\begin{example} \label{ex:sing-cohom} \textbf{Cohomologia singular}\\
Se $X$ é uma variedade podemos definir as cadeias singulares com coeficientes em $\Z$ como sendo somas formais $\eta = \sum n_i \phi_i$, onde cada $\phi_i: \Delta^p \to X$ é uma aplicação contínua definida sobre um simplexo padrão $\Delta^p \subset \R^{p+1}$. O espaço das cadeias singulares é denotado por $C_p^{sing}(X,\Z)$.

O espaço das cocadeias singulares é definido por $C^p_{sing}(X,\Z)  = \text{Hom}(C_p^{sing}(X,\Z),\Z)$. A aplicação de bordo $C_{p+1}^{sing}(X,\Z) \to C_p^{sing}(X,\Z)$ induz uma aplicação de cobordo $\delta: C^p_{sing}(X,\Z)  \to C^{p+1}_{sing}(X,\Z)$.

Podemos definir o feixe das $p$-cocadeias singulares $C^p_{sing}$ por
\begin{equation*}
U \longmapsto C^p_{sing}(U,\Z).
\end{equation*}
e obtemos assim um complexo de feixes
\begin{equation*}
C^0_{sing} \to C^1_{sing} \to \cdots \to C^p_{sing} \to C^{p+1}_{sing} \to \cdots.
\end{equation*}

Agora, se $U \subset X$ é contrátil, os grupos de cohomologia singular $H^i_{sing}(U,\Z)$ se anulam para $i > 0$, o que quer dizer que a sequência $C^0_{sing}(U) \to C^1_{sing}(U) \to \cdots$ é exata. Isso mostra que a sequência acima é uma sequência exata de feixes. O núcleo $\ker \delta: C^0_{sing}(U) \to C^1_{sing}(U)$ é formado pelas funções localmente constantes em $U$ a valores inteiros.

Obtemos assim uma sequência exata de feixes
\begin{equation*}
0 \to \underline{\Z} \to C^0_{sing} \to C^1_{sing} \to \cdots \to C^p_{sing} \to C^{p+1}_{sing} \to \cdots.
\end{equation*}
que é uma resolução do feixe localmente constante $\underline{\Z}$.

Pode-se mostrar que essa resolução é acíclica (ver \cite{voisin}, teorema 4.47) e portanto, pela proposição \ref{prop:acyclic-res} vemos que
\begin{equation} \label{eq:singcohom}
H^i(X, \underline{\Z}) = H^i_{sing}(X,\Z),
\end{equation}
isto é, a cohomologia de $X$ com coeficientes no feixe constante $\underline{\Z}$ é a cohomologia singular de $X$ com coeficientes em $\Z$.
\end{example}

Dados dois feixes $\mathcal F$ e  $\mathcal G$ sobre um espaço topológico $X$, um morfismo $\phi:\mathcal F \to \mathcal G$ induz naturalmente homomorfismos $h^k: H^k(X,\mathcal F) \to H^k(X,\mathcal G)$, construido da seguinte maneira.

Considere $\mathcal F \to \mathcal F^ \bullet$ e $\mathcal G \to \mathcal G^ \bullet$ resoluções \textit{flasque}. É um fato que podemos escolher essas resoluções de modo que o morfismo se estenda a morfismos $\phi^k:\mathcal F^k \to \mathcal G^k$ que comutam com as aplicações $\mathcal F^k \to \mathcal F^{k+1}$ e $\mathcal G^k \to \mathcal G^{k+1}$ (veja por exemplo \cite{voisin}, cap. 4) e portanto obtemos um diagrama comutativo da forma.

\begin{equation*}
\begin{matrix}
 0 & \longrightarrow  & \mathcal{F} & \longrightarrow & \mathcal{F}^0 & \longrightarrow  & \mathcal{F}^1 & \longrightarrow & \mathcal{F}^2 & \longrightarrow  \cdots \\
   & &~~ \downarrow \phi& & ~~ \downarrow \phi^0 & & ~~ \downarrow \phi^1  & & ~~ \downarrow \phi^2 &  \\
 0 & \longrightarrow  & \mathcal{G} & \longrightarrow & \mathcal{G}^0 & \longrightarrow  & \mathcal{G}^1 & \longrightarrow & \mathcal{G}^2 & \longrightarrow  \cdots
 \end{matrix}
\end{equation*}

Definimos então $h^k[\sigma] = [\phi^k \sigma] \in H^k(X,\mathcal G)$, onde $[\sigma] \in H^k(X,\mathcal F)$. Pode-se mostrar ainda que $h^k$ está bem definida e que não depende das extensões $\phi^k$ nem das resoluções escolhidas.

Para a cohomologia via resoluções também existem sequências exatas longas associadas a sequências exatas curtas de feixes. Para uma demonstração consulte \cite{voisin}.

\begin{proposition} \textbf{Sequência exata longa na cohomologia via resoluções}\\
Se
\begin{equation*}
0 \longrightarrow \mathcal{F} \longrightarrow \mathcal{G} \longrightarrow \mathcal{H} \longrightarrow 0
\end{equation*}
é uma sequência exata curta de feixes então existe uma sequência exata longa
\begin{equation*}
\begin{split}
0 &\longrightarrow H^0(X,\mathcal{F}) \longrightarrow H^0(X,\mathcal{G}) \longrightarrow H^0(X,\mathcal{H}) \longrightarrow \\
&\longrightarrow H^1(X,\mathcal{F}) \longrightarrow H^1(X,\mathcal{G}) \longrightarrow H^1(X,\mathcal{H}) \longrightarrow \cdots \\
 & \vdots \\
 & \longrightarrow H^n(X,\mathcal{F}) \longrightarrow H^n(X,\mathcal{G}) \longrightarrow H^n(X,\mathcal{H}) \longrightarrow \cdots.
\end{split}
\end{equation*}
\end{proposition}

\subsection{Comparando $H^i(X,\mathcal{F})$ e $\check{H}^i(X,\mathcal{F})$}

Nessa seção vamos comparar a cohomologia de \v{C}ech e a cohomologia definida via resoluções. Veremos que existe um isomorfismo $H^i(X,\mathcal{F}) \simeq \check{H}^i(X,\mathcal{F})$ quando $X$ é um espaço paracompacto. Em particular, o isomorfismo existe quando $X$ é uma variedade diferenciável, que é o caso que nos interessa.

Note que para grau $0$, já sabemos da existência de um isomorfismo $H^0(X,\mathcal{F}) \simeq \check{H}^0(X,\mathcal{F})$, pois os dois grupos são naturalmente isomorfos ao espaço das seções globais $\mathcal{X}$. Mais adiante usaremos esse fato e  indução em $i$ para obtermos os demais isomorfismos $H^i(X,\mathcal{F}) \simeq \check{H}^i(X,\mathcal{F})$.\\

Dado um feixe $\mathcal{F}$, nosso objetivo é relacionar a cohomologia de \v{C}ech e a cohomologia definida via resoluções. Para calcular a segunda, precisamos de uma resolução acíclica de $\mathcal{F}$. Uma resolução natural é a chamada resolução de \v{C}ech.

Dada uma cobertura $\mathcal{U}$ de $X$, definimos a partir de $\mathcal{F}$ uma sequência de feixes $\check{C}^0, \check{C}^1,\cdots$ por
\begin{equation*}
\check{C}^n (V) = \prod_{i_0,\cdots,i_n \in I} \mathcal{F}(U_{i_0 \cdots i_n} \cap V).
\end{equation*}

Note que o espaço das seções globais de $\check{C}^n$ é o espaço $\check{C}^n(\mathcal{U},\mathcal{F})$ das $n$-cocadeias de \v{C}ech  (veja a equação (\ref{eq:cech-cochains})).

De maneira semelhante à equação (\ref{eq:cech-coboundary}) definimos uma aplicação de cobordo $\delta: \check{C}^n (V) \to \check{C}^{n+1} (V)$ e obtemos assim um complexo de feixes
\begin{equation} \label{eq:cech-complex}
\check{C}^0 \stackrel{\delta}{\longrightarrow} \check{C}^1 \stackrel{\delta}{\longrightarrow} \check{C}^2 \stackrel{\delta}{\longrightarrow} \cdots
\end{equation}

Assim como vimos na seção \ref{sec:cech-cohom} o núcleo de $\delta: \check{C}^0 (V) \to \check{C}^1 (V)$ é o espaço das seções de $\mathcal{F}$ sobre $V$. Assim, obtemos um complexo de feixes
\begin{equation} \label{eq:cech-resolution}
0 \longrightarrow \mathcal{F} \longrightarrow \check{C}^0 \stackrel{\delta}{\longrightarrow} \check{C}^1 \stackrel{\delta}{\longrightarrow} \check{C}^2 \stackrel{\delta}{\longrightarrow} \cdots
\end{equation}

Note que a cohomologia do complexo $\check{C}^0 (X) \to \check{C}^1 (X) \to \cdots$ é a cohomologia de \v{C}ech $\check{H}^{\bullet}(\mathcal{U},\mathcal{F})$ de $\mathcal{F}$ com respeito a $\mathcal{U}$.

Para uma demonstração do resultado a seguir consulte \cite{voisin}, proposição 4.17.
\begin{proposition}
A sequência (\ref{eq:cech-resolution}) é uma sequência exata de feixes, e portanto o complexo de \v{C}ech $(\check{C}^{\bullet}, \delta)$ é uma resolução de $\mathcal{F}$.
\end{proposition}

Temos até agora uma resolução de $\mathcal{F}$ cuja cohomologia no nível das seções globais é a cohomologia de \v{C}ech $\check{H}^{\bullet}(\mathcal{U},\mathcal{F})$. Para garantir que essa cohomologia é de fato a cohomologia definida via resoluções, precisamos garantir que os feixes $\check{C}^k$ são acíclicos, pelo menos para uma quantidade suficiente de coberturas, para podermos tomar o limite direto.

\begin{definition}
Uma cobertura $\mathcal{U} = \{U_i\}_{i \in I}$ é acíclica para o feixe $\mathcal{F}$ se $H^k(U_{i_0\cdots i_n},\mathcal{F}|_{U_{i_0\cdots i_n}}) = 0$ para todo $k > 0$ e todos $i_0< \ldots < i_n \in I$.
\end{definition}

Para esse tipo de cobertura temos o isomorfismo desejado.

\begin{proposition}
Se $\mathcal{U}$ é uma cobertura localmente finita que é acíclica para $\mathcal{F}$ então $\check{H}^i(\mathcal{U},\mathcal{F}) \simeq H^i(X,\mathcal{F})$
\end{proposition}

A ideia da demonstração é a seguinte: da resolução de \v{C}ech obtemos sequências exatas curtas de feixes
\begin{equation*}
0 \longrightarrow \check{Z}^p \longrightarrow \check{C}^p \longrightarrow \check{Z}^{p+1} \longrightarrow 0,
\end{equation*}
e portanto temos uma sequência exata longa associada. Uma parte desta sequência é
\begin{equation*}
H^{q-1}(X,\check{C}^p) \longrightarrow H^{q-1}(X,\check{Z}^{p+1}) \longrightarrow H^q(X,\check{Z}^p) \longrightarrow H^q(X,\check{C}^p).
\end{equation*}

Do fato de $\mathcal{U}$ ser localmente finita pode-se mostrar que\footnote{Esse isomorfismo segue por exemplo do Teorema 4.4.4. em \cite{godement}.} $H^k(X,\check{C}^p) \simeq \prod_{i_0<\cdots<i_n \in I} H^k(U_{i_0 \cdots i_n},\mathcal{F})$ e como $\mathcal{U}$ é acíclica para $\mathcal F$ concluímos que $H^k(\mathcal{U},\check{C}^p) = 0$ para $k>0$ e portanto os termos das pontas da sequência acima se anulam para $q >1$. Temos assim que $H^q(X,\check{Z}^p) \simeq H^{q-1}(X,\check{Z}^{p+1})$ para $q>1$ e, como $\check{Z}^0 = \mathcal{F}$ (pelos axiomas de feixe), temos, para $p>1$,
\begin{equation*}
H^p(X,\mathcal{F}) = H^p(X,\check{Z}^0) \simeq  H^{p-1}(X,\check{Z}^1) \simeq \cdots \simeq H^1(X,\check{Z}^{p-1})  \simeq \frac{H^0(X,\check{Z}^p)}{\delta \check{C}^{p-1}(X)} = \check{H}^p(\mathcal{U},\mathcal{F}).
\end{equation*}

\begin{corollary}
Se $\mathcal{F}$ é um feixe flasque sobre um espaço paracompacto $X$ então $\check{H}^i(X,\mathcal{F}) = 0$ para todo $i>0$.
\end{corollary}
\begin{proof}
Como $\mathcal{F}$ é flasque temos, da equação (\ref{eq:flasque-cohomology}), que  $H^i(X,\mathcal{F}) = 0$ para $i> 0$. Em particular toda cobertura é acíclica para $\mathcal F$ e portanto, da proposição acima, segue que $\check{H}^i (\mathcal{U},\mathcal{F})  = 0$ para todo $i > 0$ e toda cobertura localmente finita de $X$. Como $X$ é paracompacto, toda cobertura admite um refinamento localmente finito e portanto, tomando o limite direto, obtemos $\check{H}^i(X,\mathcal{F}) = 0$ para $i>0$.
\end{proof}

Já temos agora os ingredientes necessários para mostrar a existência dos isomorfismos $\check{H}^i(X,\mathcal{F})$ $\simeq H^i(X,\mathcal{F})$. Precisaremos apenas dos seguintes fatos: 1)~Em grau zero $\check{H}^0(X,\mathcal{F}) \simeq H^0(X,\mathcal{F})$, 2)~Para feixes flasque vale $\check{H}^i(X,\mathcal{F}) = 0$ para todo $i>0$ e 3)~As duas cohomologias transformam sequências exatas curtas de feixes em sequências exatas longas nos grupos de cohomologia.

\begin{proposition} \label{prop:iso-cech-res}
Se $X$ é um espaço paracompacto então para todo feixe $\mathcal{F}$ existem isomorfismos naturais $\check{H}^i(X,\mathcal{F}) \simeq H^i(X,\mathcal{F})$ para todo $i \geq 0$.
\end{proposition}

\begin{proof}
Vamos provar o resultado usando indução em $i$. Se $i=0$ já vimos que, dado um feixe $\mathcal{S}$, tanto $H^0(X,\mathcal{S})$ quanto $\check{H}^0(X,\mathcal{S})$ são naturalmente identificados com o espaço das seções globais de $\mathcal{S}$ (veja equações (\ref{eq:H0cech=global-sec}) e (\ref{eq:H0res=global-sec})). Vemos assim que $\check{H}^0(X,\mathcal{S}) \simeq H^0(X,\mathcal{S})$ para todo feixe $\mathcal{S}$.

Seja agora $\mathcal{F}$ um feixe e considere $0 \to \mathcal{F} \to \mathcal{F}^0 \to \cdots$ uma resolução de $\mathcal{F}$ por feixes flasque. Denotando  $\mathcal{F}' = \mathcal{F}^0 / \mathcal{F}$ temos uma sequência exata curta de feixes $0 \to \mathcal{F} \to \mathcal{F}^0 \to \mathcal{F}' \to 0$. Obtemos assim duas sequências exatas longas na cohomologia, uma para a cohomologia via resoluções e outra para a cohomologia de \v{C}ech:
\begin{equation*}
\begin{matrix}
 0 & \longrightarrow  & H^0(X,\mathcal{F}) & \longrightarrow & H^0(X,\mathcal{F}^0) & \longrightarrow  & H^0(X,\mathcal{F}') & \longrightarrow & H^1(X,\mathcal{F}) & \longrightarrow  0 \\
   & & \downarrow~\simeq & & \downarrow~\simeq & & \downarrow~\simeq & & &  \\
 0 & \longrightarrow  & \check{H}^0(X,\mathcal{F}) & \longrightarrow & \check{H}^0(X,\mathcal{F}^0) &  \longrightarrow  & \check{H}^0(X,\mathcal{F}') & \longrightarrow & \check{H}^1(X,\mathcal{F}) & \longrightarrow  0
 \end{matrix}
\end{equation*}
onde usamos que $H^1(X,\mathcal{F}^0) = 0 =  \check{H}^1(X,\mathcal{F}^0)$ pois $\mathcal{F}^0$ é flasque. Pelo que vimos acima, as flechas verticais são isomorfismos naturais, obtidos da identificação de $H^0$ e $\check H ^0$ com o espaço das seções globais. Vemos portanto que o diagrama acima é comutativo, de onde obtemos o isomorfismo  $\check{H}^1(X,\mathcal{F}) \simeq H^1(X,\mathcal{F})$.\\

Suponha agora que $\check{H}^i(X,\mathcal{S}) \simeq H^i(X,\mathcal{S})$ para algum $i \geq 1$ e para todo feixe $\mathcal{S}$. Dado um feixe  $\mathcal{F}$ tome, como acima, feixes $\mathcal{F}^0$ e $\mathcal{F}'$ com $\mathcal{F}^0$ flasque e de modo que exista a sequência exata $0 \to \mathcal{F} \to \mathcal{F}^0 \to \mathcal{F}' \to 0$. Usando o fato que $H^k(X,\mathcal{F}^0) = \check{H}^k(X,\mathcal{F}^0) = 0$ para $k=i,i+1$ obtemos, das duas sequências longas associadas, as sequências
\begin{equation*}
\begin{matrix}
 0 & \longrightarrow  & H^i(X,\mathcal{F}') & \longrightarrow & H^{i+1}(X,\mathcal{F}) &  \longrightarrow  & 0 \\
   & & & &  &  &  \\
0 & \longrightarrow  & \check{H}^i(X,\mathcal{F}') & \longrightarrow & \check{H}^{i+1}(X,\mathcal{F}) &  \longrightarrow  & 0
 \end{matrix}
\end{equation*}
Da hipótese de indução temos que $\check{H}^i(X,\mathcal{F}') \simeq H^i(X,\mathcal{F}')$ e portanto concluímos que $\check{H}^{i+1}(X,\mathcal{F})$ $\simeq H^{i+1}(X,\mathcal{F})$, terminando a demonstração.
\end{proof}

\chapter{Fibrados de Linha e Divisores} \label{ch:div-lb}

Neste capítulo introduziremos os fibrados de linha e os divisores, dois conceitos centrais na Geometria Complexa. Os dois conceitos surgem naturalmente quando estudamos hipersuperfícies complexas (i.e., subvariedades complexas de codimensão $1$). Um divisor por exemplo é uma soma formal de hipersuperfícies com coeficientes inteiros. Já os fibrados de linha (que são fibrados vetoriais complexos de posto $1$) aparecem, nesse contexto, como sendo o fibrado normal da hipersuperfície.

\section{Fibrados Vetoriais Complexos} \label{sec:cxvb}
Um \emph{fibrado vetorial complexo} sobre uma variedade diferenciável $X$ é uma variedade $E$ com uma aplicação diferenciável $\pi: E \to X$ satisfazendo as seguintes condições:

\begin{enumerate}
	\item Para cada $x \in X$ a fibra $E_x = \pi^{-1}(x)$ é um espaço vetorial complexo.
	\item Existe uma cobertura aberta $\{U_i\}_i$ de $X$ e difeomorfismos
   	  \begin{equation}
	     \varphi_i: \pi^{-1}(U_i) \to U_i \times \C^r
	  \label{eq:triv-vb}
    \end{equation}
    tal que  $\text{pr}_1 \circ \varphi_i = \pi|_{\pi^{-1}(U_i)}$, onde $\text{pr}_1:U_i \times \C^r \to U_i$ é a projeção na primeira coordenada, e de modo que $ \varphi_i|_{E_x}:E_x \to \{x\} \times \C^r$ é um isomorfismo $\C$-linear.
\end{enumerate}

Os difeomorfismos $ \varphi_i$ são chamados \textit{trivializações locais} com respeito a cobertura $\{U_i\}_i$ e o número $r=\dim_{\C}E_x$ (que, quando $X$ é conexa, não depende de $x$ pela condição (\ref{eq:triv-vb})) é chamado \textit{posto} do fibrado $E$. Quando $r=1$ dizemos que $E$ é um \textit{fibrado de linha}.\\

Pela condição 2 temos que as funções
\begin{equation*}
 \varphi_{ij} (x) = ( \varphi_i \circ  \varphi_j^{-1})|_{\{ x\} \times \C^r}
\end{equation*}
são transformações lineares inversíveis, e portanto definem aplicações diferenciáveis
\begin{equation}
 \varphi_{ij}: U_i \cap U_j \to \GL(r,\C)
\label{eq:trans-fcts}
\end{equation}
chamadas \textit{funções de transição}, ou \textit{cociclos} \index{cociclos!de um fibrado vetorial} com respeito a $\{U_i\}_i$.\\

Note que os cociclos satisfazem as seguintes condições
\begin{equation}
 \begin{split}
   \varphi_{ij} \cdot \varphi_{ji} &= I \;\; \text{ em } U_i \cap U_j \\
   \varphi_{ij} \cdot \varphi_{jk} \cdot \varphi_{ki} &= I \;\; \text{ em } U_i \cap U_j \cap U_k
 \end{split}
 \label{eq:cocycle-cond}
\end{equation}

\begin{example}
O fibrado complexo $E = X \times \C^r$ com a projeção na primeira coordenada $\pi: X \times \C^r \to X$ é chamado fibrado trivial. Neste caso os cociclos são constantes iguais à matriz identidade $I \in \GL(r,\C)$.
\end{example}

Dados dois fibrados $\pi_E: E\to X$ e $\pi_F: F \to X$, um \textit{morfismo de $E$ em $F$} é uma aplicação diferenciável $\varphi:E \to F$ tal que $\pi_E = \pi_F \circ \varphi$, a restrição $\varphi_x = \varphi|_{E_x}:E_x \to F_x$ é linear e seu posto independe de $x$. Um \textit{isomorfismo} é um morfismo bijetor. Denotamos por $\Hom(E,F)$ o conjunto dos morfismos de $E$ em $F$ e por $\End(E) = \Hom(E,E)$ o conjunto dos endomorfismos de $E$.

Dado um morfismo $\varphi:E \to F$ definimos $\ker \varphi = \{ v \in E : \varphi_x(v) = 0_x \text{ se } v \in E_x\}$. É fácil ver que $\ker \varphi$ é um subfibrado de $E$. Podemos mostrar também que a imagem $\im \varphi \subset F$ é um subfibrado de $F$.

Dizemos que uma sequência de feixes e morfismos $E \stackrel{\varphi}{\to} F \stackrel{\psi}{\to} G$ é exata em $F$ se $\im \varphi = \ker \psi$. Mais geralmente uma sequência $\cdots \to E_{i-1} \to E_i \to E_{i+1}$ é exata se é exata em cada $E_i$. Note que $0 \to E \to F$ é exata em $E$ se e somente se $E \to F$ é injetora e $E \to F \to 0$ é exata se e somente se $E \to F$ é sobrejetora.

 Da condição (\ref{eq:triv-vb}) vemos que dado um fibrado $\pi: E \to X$, uma trivialização $\pi^{-1}(U) \to U \times \C^r$ é um isomorfismo entre $\pi^{-1}(U) \to U$ e o fibrado trivial sobre $U$. Por essa razão dizemos que todo fibrado vetorial complexo é localmente trivial.
 
  A proposição a seguir, cuja demonstração é simples, dá uma condição em termos de seus cociclos, para que dois fibrados sejam isomorfos.
\begin{proposition} \label{prop:iso-diffvb}
Dois fibrados $E \to X$ e $F \to X$ de posto $r$ são isomorfos se e somente se existe uma cobertura $\mathcal{U} = \{U_i\}_i$ de $X$ e funções diferenciáveis $\lambda_i:U_i \to \GL (r,\C)$ tais que $\phi_{ij} = \lambda_i \psi_{ij} \lambda_j^{-1}$ em $U_i \cap U_j$, onde $\phi_{ij}$ e $\psi_{ij}$ são os cociclos de $E$ e $F$ com relação à $\mathcal{U}$, respectivamente.
\end{proposition}

\begin{remark} \label{rmk:cocycles} Os cociclos determinam o fibrado no seguinte sentido: dada uma cobertura $\{U_i\}_i$ de $X$ e funções diferenciáveis $\varphi_{ij} : U_i \cap U_j \to \GL(r,\C)$ satisfazendo as condições (\ref{eq:cocycle-cond}) então
\begin{equation*}
E \simeq \frac{\coprod_i U_i \times \C^r}{(x,v) \in U_i \times \C^r \sim (x,\varphi_{ij}(x)\cdot v) \in U_j \times \C^r } 
\end{equation*}

Isso nos dá uma maneira bastante prática de descrever os fibrados complexos: precisamos apenas especificar funções em abertos de $X$ com valores em $\GL(r,\C)$ satisfazendo certas condições de compatibilidade. Esta descrição também é útil quando queremos construir novos fibrados a partir de fibrados conhecidos, como veremos mais adiante.

 A notação $E \sim \{(U_i,\varphi_{ij})\}$ (ou simplesmente $E \sim \{\varphi_{ij}\}$ quando os abertos estiverem subentendidos) significa que $E$ é o fibrado dado pelos cociclos $\varphi_{ij}$ com relação a cobertura $\{U_i\}_i$
\end{remark}

\begin{example} \label{ex:tang-bundle} \textbf{O Fibrado Tangente de uma Variedade Complexa}

O fibrado tangente de uma variedade complexa admite naturalmente uma estrutura de fibrado vetorial complexo, como veremos a seguir.

Seja $X$ uma variedade complexa e $TX$ o seu fibrado tangente. Vimos na seção \ref{sec:complex-manifolds} que existe um endomorfismo $J:TX\to TX$ satisazendo $J^2 = -\id$ e assim cada $T_pX$ pode ser visto como um espaço vetorial complexo se definirmos a multiplicação por escalar via $(a + \smo b) \cdot v = av +bJ_p v$.

 Se $X = \bigcup U_i$ e $\varphi_i : U_i \to \varphi_i(U_i) \subset \C^n$ são cartas holomorfas, obtemos trivializações
\begin{equation*}
 \begin{split}
 \phi_i: \pi^{-1}(U_i) &\longrightarrow U_i \times \C^n \\ 
 (p,v) &\longmapsto (p, (d \varphi_i)_p \cdot v)
 \end{split}
\end{equation*}

Em coordenadas $\phi_i$, leva a base $\big \{ \frac{\del}{\del x_1}\big|_p,\cdots,\frac{\del}{\del x_n}\big|_p, \frac{\del}{\del y_1}\big|_p \cdots, \frac{\del}{\del y_n}\big|_p \big \} \subset T_p X$ na base $\big \{ \frac{\del}{\del x_1},\cdots, \frac{\del}{\del x_n},$ $\frac{\del}{\del y_1}\cdots, \frac{\del}{\del y_n} \big \}$ de $\C^n$.\\

 É claro que $\phi_i$ é um difeomorfismo e, pela própria definição de $J_p$, temos que $\phi_i|_{T_p X}$ é $\C$-linear. Logo o fibrado tangente $TX \to X$ é naturalmente um fibrado vetorial complexo de posto $n = \dim _{\C}X$.\\

Os cociclos com relação a essa trivialização são dados por
\begin{equation} \label{eq:cocycles-tb}
 \begin{split}
   \phi_{ij} : U_i \cap U_j & \longrightarrow \GL (n,\C)\\
   x & \longmapsto d(\varphi_i \circ \varphi_j^{-1})_{\varphi_j (x)},
 \end{split}
\end{equation}
ou seja, pelo jacobiano da mudança de coordenadas. Aqui identificamos $\GL(n,\C)$ como subgrupo de $\GL(2n,\R)$ via $A + \smo B \mapsto \left ( \begin{smallmatrix} A & B \\ -B & A \end{smallmatrix} \right )$.

Note que, em relação à base $\left \lbrace \frac{\del}{\del x_1},\cdots,\frac{\del}{\del x_n}, \frac{\del}{\del y_1}\cdots, \frac{\del}{\del y_n} \right \rbrace$ de $\C^n$, o jacobiano de $\varphi = \varphi_{ij}$ é representado pela matriz
\begin{equation*}
 J_{\R}\varphi_{ij} = \begin{pmatrix} \left ( \frac{\del u^k_{ij}}{\del x_l} \right )_{k,l} & \left ( \frac{\del u^k_{ij}}{\del y_l} \right )_{k,l} \\
 \left ( \frac{\del v^k_{ij}}{\del x_l} \right )_{k,l} & \left ( \frac{\del v^k_{ij}}{\del y_l} \right )_{k,l}  \end{pmatrix}, \text{ onde } \varphi_{ij} = (u^1_{ij} + \smo v^1_{ij},\cdots,u^n_{ij} + \smo v^n_{ij}).
\end{equation*}
 
Existe uma outra maneira de fazer do fibrado tangente um fibrado vetorial complexo, que é através da complexificação de cada fibra. Vamos relacionar estas duas construções.\\

Denote por $T_{\C}X = TX \otimes \C$ a complexificação do fibrado tangente, ou seja, a fibra sobre $p \in X$ é $T_p X \otimes \C$. Obtemos assim um fibrado vetorial complexo de posto $2n$.

Uma maneira de obter trivializações de $T_{\C}X$ é simplesmente a complexificar as trivializações de $TX$, i.e., se $\varphi(p,v) = (p,\phi \cdot v)$ é uma trivialização de $TX$ então $\varphi_c(p,v \otimes \lambda) = (p,(\varphi \cdot v)\otimes \lambda)$ é uma trivialização de $T_{\C}X$. Mas há uma outra maneira também natural, que veremos a seguir.\\

Da base exibida em (\ref{eq:baseTpX}) obtemos uma base
\begin{equation*}
\left \lbrace \frac{\del}{\del x_1} \bigg|_p \otimes 1, \cdots, \frac{\del}{\del x_n} \bigg|_p \otimes 1 ,\frac{\del}{\del y_1} \bigg|_p \otimes 1 , \cdots, \frac{\del}{\del y_n} \bigg|_p \otimes 1 \right \rbrace \subset T_p X \otimes \C
\end{equation*}
Existe uma outra base natural de $T_p X \otimes \C$, dada por

\begin{equation} \label{eq:baseTpXC}
   \left \lbrace \frac{\del}{\del z_1} \bigg|_p , \cdots, \frac{\del}{\del z_n} \bigg|_p, \frac{\del}{\del \zbar_1} \bigg|_p , \cdots, \frac{\del}{\del \zbar_n} \bigg|_p \right \rbrace,
\end{equation}
onde
\begin{equation} \label{eq:def-partialz}
\frac{\del}{\del z_i} = \frac{1}{2} \left ( \frac{\del}{\del x_i} - \smo \frac{\del}{\del y_i} \right) \text{ e } \frac{\del}{\del \zbar_i} = \frac{1}{2} \left ( \frac{\del}{\del x_i} + \smo \frac{\del}{\del y_i} \right ).
\end{equation}

A aplicação $J$ se estende naturalmente a um endomorfismo $\C$-linear $J:T_{\C}X \to T_{\C}X$ fazendo $J(v \otimes \lambda) = J(v) \otimes \lambda$.

 Como $J$ satisfaz $J^2 = -\text{id}$ temos que cada $J_p$ é diagonalizável, com autovalores $\pm \smo$. Defina
\begin{equation} \label{eq:def-T1,0}
T^{1,0}X = \{ v \in T_{\C}X : Jv = \smo v \} \;\; \text{ e } \;\; T^{0,1}X = \{ v \in T_{\C}X : Jv = -\smo v \}
\end{equation}
Temos então que $T^{1,0}X$ e $T^{1,0}X$ são subfibrados de $T_{\C}X$ tais que 
\begin{equation} \label{eq:dec-Tc}
T_{\C}X = T^{1,0}X \oplus T^{0,1}X
\end{equation}
 e a conjugação complexa $v \otimes \lambda \mapsto v \otimes \overline{\lambda}$ permuta os dois fatores.\\

Note que $\{ \del / \del z_i\}_{i=1,\ldots,n}$ é uma base de $T^{1,0}X$ e $\{ \del / \del \zbar_i\}_{i=1,\ldots,n}$ é uma base de $T^{0,1}X$. Podemos assim obter trivializações de $T^{1,0}X$ e $T^{0,1}X$.\\

As duas construções estão relacionadas: o fibrado tangente $TX$ com a estrutura complexa induzida é isomorfo ao fibrado $T^{1,0}X$. O isomorfismo é dado por
\begin{equation} \label{eq:iso-tang-bundle}
TX \ni v \longmapsto \frac{1}{2}(v - \smo Jv) \in T^{1,0}X.
\end{equation}
segundo o qual a base $\{ \del / \del z_i\}_{i=1,\ldots,n}$ se transforma como na equação (\ref{eq:def-partialz}).\\

Os cociclos de $T_{\C}X$ são as complexificações dos cociclos  $\phi_{ij}$ de $TX$ (cf. \ref{eq:cocycles-tb}), ou seja, são as matrizes $J_{\R}\varphi_{ij}$ vistas como matrizes em $\GL(2n,\C)$.

 Agora, com relação à base $\{\del/ \del z_1, \cdots,\del / \del z_n, \del/ \del \zbar_1, \cdots \del/ \del \zbar_n \}$ de $\R^{2n} \otimes \C = \C^{2n}$ a matriz de $d(\varphi_i \circ \varphi_j^{-1})_{\C}$ fica

\begin{equation*}
\begin{pmatrix} \left ( \frac{\del \varphi^k_{ij}}{\del z_l} \right )_{k,l} & \left ( \frac{\del \varphi^k_{ij}}{\del \zbar_l} \right )_{k,l} \\
 \left ( \frac{\del \overline{\varphi}^k_{ij}}{\del z_l} \right )_{k,l} & \left ( \frac{\del \overline{\varphi}^k_{ij}}{\del \zbar_l} \right )_{k,l}  \end{pmatrix} = \begin{pmatrix} \left ( \frac{\del \varphi^k_{ij}}{\del z_l} \right )_{k,l} & 0 \\
 0 & \left ( \frac{\del \overline{\varphi}^k_{ij}}{\del \zbar_l} \right )_{k,l}  \end{pmatrix}, \text{ onde } \varphi_{ij} = (\varphi^1_{ij},\cdots,\varphi^n_{ij}).
\end{equation*}

Note que isso mostra que os cociclos de $T^{1,0}X$ são dados pelo jacobiano complexo da mudança de coordenadas
\begin{equation*}
J_{\C} \varphi_{ij} = \left ( \frac{\del \varphi^k_{ij}}{\del z_l} \right )_{k,l}
\end{equation*}
e os cociclos de $T^{0,1}X$ são obtidos por conjugação dos cociclos de $T^{1,0}X$.

\end{example}

\subsection{Seções}
Uma seção \index{seção!de um fibrado vetorial} diferenciável de um fibrado complexo $\pi: E \to X$ é uma aplicação que associa a cada ponto $x \in X$ um vetor $s(x)$ na fibra $E_x=\pi^{-1}(x)$ de maneira diferenciável. Mais precisamente, \textit{uma seção diferenciável de~$E$} é uma aplicação diferenciável $s:X \to E$ tal que $\pi \circ s = \text{id}_X$. Com frequência iremos omitir o adjetivo diferenciável.

Todo fibrado complexo admite uma seção canônica, a seção nula, definida por $s(x) = 0_x \in E_x$ para todo~$x$. As trivializações definem isomorfismos locais $\psi_i|_{s(U_i)}:s(U_i) \to U_i \times \{0\}$ que colam a um isomorfismo global $s(X) \simeq X \times \{0\}$, mostrando que a imagem da seção nula é uma subvariedade de $E$ isomorfa a $X$.

Usando partições da unidade é fácil ver que o conjunto de seções de um fibrado é um espaço vetorial de dimensão infinita.\\

Seja $s:X \to E$ uma seção. Restringindo $s$ ao aberto de uma trivialização obtemos $s|_{U_i}:U_i \to \pi^{-1}(U_i)$ e portanto podemos compor com uma trivialização $\varphi_i$ definida em $U_i$. Obtemos assim funções $\widetilde{s}_i = \varphi_i \circ (s|_{U_i}):U_i \to U_i \times \C^r$, que têm a forma
\begin{equation*}
\widetilde{s}_i (x) = (x,s_i(x))\,, \;\;\;\;\; s_i : U_i \to \C^r.
\end{equation*}
Se $x \in U_i \cap U_j$ temos que
\begin{equation*}
\varphi_i^{-1} \circ \widetilde{s}_i (x) = s(x) = \varphi_j^{-1} \circ \widetilde{s}_j (x)
\end{equation*}
e portanto as funções $s_i$ devem satisfazer $s_i = \varphi_{ij} \cdot s_j$ em $U_i \cap U_j$.\\

Reciprocamente, dadas funções $s_i : U_i \to \C^r$ satisfazendo $s_i = \varphi_{ij} \cdot s_j$ em $U_i \cap U_j$ podemos definir uma seção $s:X \to E$ por $s(x) = \varphi_i^{-1}(x,s_i(x))$ se $x \in U_i$.\\

Obtemos assim uma correspondência biunívoca
\begin{equation}
\left \lbrace {\begin{array}{c}
 \text{Seções diferenciáveis}  \\
 \text{ de }  E \sim \{(U_i,\varphi_{ij})\} \\
 \end{array} } \right \rbrace
 \longleftrightarrow
\left \lbrace {\begin{array}{c}
 \text{Funções diferenciáveis}   \\
 s_i: U_i \to \C^r \text{ satisfazendo }  \\
 s_i = \varphi_{ij} \cdot s_j \text{ em } U_i \cap U_j \\
 \end{array} } \right \rbrace
\end{equation}

Essa correspondência nos dá uma outra maneira de pensarmos nas seções de um fibrado complexo: uma seção é uma coleção de funções vetoriais localmente definidas em $X$ satisfazendo certas condições de compatibilidade (que dependem do fibrado).
  
\subsection{Construindo novos fibrados} \label{sec:newbundles}

Usando a observação \ref{rmk:cocycles} podemos estender aos fibrados vetoriais complexos algumas construções feitas com espaços vetoriais.

 Sejam $E \to X$ e $F \to X$ dois fibrados vetoriais complexos sobre $X$ de posto $r$ e $s$ respectivamente. Suponha que $E$ e $F$ sejam trivializados por uma mesma cobertura\footnote{Sempre podemos supor isto pois duas coberturas sempre possuem um refinamento comum.} $\{U_i\}_i$ e sejam $\{\varphi_{ij}\}$ e $\{\psi_{ij}\}$ os cociclos correspondentes.
\begin{enumerate}
	\item A \textit{soma direta} $E \oplus F$ é o fibrado dado pelos cociclos
	\begin{equation*}
	 \phi_{ij}(x) = \begin{pmatrix} \varphi_{ij}(x) & 0 \\ 0 & \psi_{ij} (x)  \end{pmatrix} \in \GL(r+s,\C).
	\end{equation*}
	
	\item O \textit{produto tensorial} $E \otimes F$ é o fibrado dado pelos cociclos
	\begin{equation*}
	\phi_{ij}(x) = \varphi_{ij}(x) \otimes \psi_{ij}(x) \in \GL(rs,\C).
	\end{equation*}
	
	\item O \textit{dual} $E^*$ é o fibrado dado pelos cociclos
	\begin{equation*}
	\phi_{ij}(x) = (\varphi_{ij}(x)^t)^{-1} \in \GL(r,\C).
	\end{equation*}
	
	Utilizando este item e o anterior podemos dar uma estrutura de fibrado vetorial a $\Hom(E,F) = E^*\otimes F$ e $\End(E) = E^* \otimes E$.
	
	\item A \textit{$k$-ésima potência exterior} $E^{\wedge k}$ é o fibrado dado pelos cociclos
	\begin{equation*}
	\phi_{ij}(x) = \wedge^k \varphi_{ij}(x) \in \GL(m,\C), \;\; m = \binom{r}{k}
	\end{equation*}
	No caso em particular em que $k=r$, $\bigwedge^r E = \det E$ é um fibrado de linha, chamado \textit{fibrado determinante de $E$}.

 \item Se $f:Y \to X$ é uma aplicação diferenciável, definimos o \textit{pullback} de $E$ por $f$ como sendo o fibrado $f^*E$ sobre $Y$ dado pelos cociclos $\phi(x) = \varphi_{ij} \circ f (x) \in \GL(r,\C)$. Se $Y \subset X$ e $\iota:Y \to X$ denota a inclusão dizemos que $i^*E$ é a restrição de $X$ a $Y$, e a denotamos por $E|_Y$.
\end{enumerate}

Há ainda uma outra construção importante que podemos aplicar aos fibrados vetoriais complexos, que é a projetivização. Dado um fibrado complexo $E$ de posto $r$ consideremos $s:X \to E$ sua seção nula. A multiplicação fibra a fibra define uma ação livre e própria de $\C^*$ em $E \setminus s(X)$. Definimos a \textit{projetivização de $E$} como sendo o quociente
\begin{equation*}
\pr (E) = \frac{E \setminus s(X)}{\C^*}.
\end{equation*}
A projeção $\pi: E \to X$ define uma aplicação $\Pi:\pr(E) \to X$ cuja fibra sobre $x$ é $\Pi^{-1}(x) = \pr(E_x) \simeq \pr^{r-1}$. Vemos assim que $\pr(E)$ é um fibrado sobre $X$ com fibra típica $\pr^{r-1}$.

Analogamente ao caso de fibrados vetoriais, a projetivização de um fibrado complexo de posto $r$ pode ser dada por meio de cociclos com valores em $\text{PGL}(r,\C)$, \footnote{O grupo $\text{PGL}(r,\C)$ é o quociente de $\GL(r,\C)$ pelo seu centro, formado pelos múltiplos da identidade.} isto é, aplicações $\varphi_{ij}:U_i \cap U_j \to \text{PGL}(r,\C)$ satisfazendo $\varphi_{ij} \cdot \varphi_{ji} = I$ em $U_i \cap U_j$ e $\varphi_{ij} \cdot \varphi_{jk} \cdot \varphi_{ki} = I$ em $U_i \cap U_j \cap U_k$.

\section{Fibrados Holomorfos} \label{sec:hol-vb}
 Suponha agora que $X$ seja uma variedade complexa. Um \textit{fibrado vetorial holomorfo} sobre $X$ é um fibrado vetorial complexo $\pi: E \to X$ tal que $E$ é uma variedade complexa e as trivializações
 \begin{equation*}
   \varphi_i: \pi^{-1}(U_i) \to U_i \times \C^r
 \end{equation*}
são biholomorfismos.

 Equivalentemente, podemos pedir que os cociclos
\begin{equation*}
 \varphi_{ij}: U_i \cap U_j \to \GL(r,\C)
\end{equation*}
sejam aplicações holomorfas.\\

Assim como na categoria diferenciável, podemos recuperar o fibrado holomorfo $E$ a partir de aplicações holomorfas $\varphi_{ij}: U_i \cap U_j \to \GL(r,\C)$ satisfazendo as condições (\ref{eq:cocycle-cond}). Além disso, todas as construções discutidas na seção \ref{sec:cxvb} também funcionam na categoria holomorfa, isto é, a soma direta, o produto tensorial, o dual e a potência exterior de fibrados holomorfos ainda são fibrados holomorfos.\\

\begin{example} \label{ex:hol-tangent-bundle} \textbf{O fibrado tangente holomorfo, o fibrado cotangente e o fibrado canônico}\\
\index{fibrado!tangente holomorfo} \index{fibrado!canônico} Seja $X$ uma variedade complexa e $\varphi_i:U_i \to \C^n$ cartas holomorfas cobrindo $X$. O fibrado tangente holomorfo $\mathcal{T}_X$ é o fibrado holomorfo de posto $n$ definido pelos cociclos
\begin{equation*}
\left ( \frac{\del \varphi^k_{ij}}{\del z_l}\bigg|_{\varphi_j(x)} \right )_{k,l},\; x \in U_i \cap U_j,
\end{equation*}
onde $\varphi_{ij} = (\varphi^1_{ij},\ldots,\varphi^1_{ij})$ é a mudança de coordenadas $\varphi_{ij}=\varphi_i \circ \varphi_j^{-1}:\varphi_j(U_i \cap U_j) \to \varphi_i(U_i \cap U_j)$.

 O fibrado dual $\Omega_X = \mathcal{T}_X^*$ é chamado \textbf{fibrado cotangente holomorfo} e suas potências $\Omega_X^p = \bigwedge^p \Omega_X$ \index{Omega@$\Omega_X^p$} são os \textbf{fibrados de $p$-formas holomorfas}.
 
  Note que $K_X = \bigwedge^n \Omega_X = \det (\Omega_X)$ é um fibrado de linha, chamado \textbf{fibrado canônico de $X$}. Este nome vem do fato de que $K_X$ é um fibrado de linha naturalmente associado a qualquer variedade complexa, e sua definição não requer nehuma informação adicional sobre $X$, como por exemplo a presença de uma métrica riemanniana ou existência de algum mergulho $X \to \pr^N$.\\

Da definição do fibrado tangente holomorfo, vemos que $\mathcal{T}_X$ tem os mesmos cociclos que o fibrado $T^{1,0}X$ (veja o exemplo \ref{ex:tang-bundle}), o que mostra que este fibrado é holomorfo e $\mathcal{T}_X \simeq T^{1,0}X$.

 Da decomposição (\ref{eq:dec-Tc}) obtemos uma decomposição no dual $T_{\C}X^* \doteq TX^* \otimes \C = (T^{1,0}X)^* \oplus (T^{0,1}X)^*$. O espaço $(T^{1,0}X)^*$ é formado pelas 1-formas complexas que se anulam em $T^{0,1}X$, e analogamente para o outro somando, e do isomorfismo $T^{1,0}X \simeq \mathcal{T}_X$ obtemos um isomorfismo $(T^{1,0}X)^* \simeq \Omega_X$.
 
  A decomposição (\ref{eq:dec-Tc}) fornece ainda uma decomposição das potências exteriores
\begin{equation} \label{eq:decomp-ext-power}
\textstyle \bigwedge^k_{\C}X \doteq \bigwedge^k(T_{\C}X)^* = \displaystyle \bigoplus_{p+q=k} \textstyle \bigwedge^{p,q} X, \text{ onde } \textstyle \bigwedge^{p,q} X = \textstyle \bigwedge^p (T^{1,0}X)^* \otimes \bigwedge^q (T^{0,1}X)^*,
\end{equation}
e o isomorfismo $(T^{1,0}X)^* \simeq \Omega_X$ induz isomorfismos $\bigwedge^{p,0}X \simeq \Omega^p_X$.

 Note que temos projeções naturais $\Pi^{p,q}: \bigwedge^k_{\C}X \to \bigwedge^{p,q}X$.
 
  Um caso particular que nos será útil mais adiante é o caso $k=2$. Temos
\begin{equation}
\textstyle \bigwedge^2_{\C}X = \bigwedge^{2,0} X \oplus \bigwedge^{1,1} X \oplus \bigwedge^{0,2} X
\end{equation}
\indent O espaço $\bigwedge^{2,0} X$ (resp. $\bigwedge^{0,2} X$) é formado pelas 2-formas $\alpha(\cdot,\cdot)$ que se anulam se um dos argumentos pertence a $T^{0,1}X$ (resp.  $T^{1,0}X$). Já $\bigwedge^{1,1} X$ é formado pelas 2-formas que se anulam se ambos os argumentos pertencem a $T^{0,1}X$ ou ambos a $T^{0,1}X$.

Os feixes de seções de $\bigwedge^k_{\C} X$ e $\bigwedge^{p,q} X$ serão denotados, respectivamente, por $\mathcal{A}^k_{X,\C}$ e \index{Apq@$\mathcal A_X^{p,q}$} $\mathcal{A}^{p,q}_X$. Também temos projeções $\Pi^{p,q}: \mathcal{A}^k_{X,\C}\to \mathcal{A}^{p,q}_X$.

Em coordenadas, a base $\{\del/\del z_1,\ldots,\del/\del z_n,\del/\del \bar{z}_1,\ldots,\del/\del \bar{z}_n \}$ induz uma base $\{dz_1 \ldots,dz_n,$ $d\bar{z}_1,\cdots,d\bar{z}_n\}$ de $T_{\C}X^*$. Um elemento de $\mathcal{A}^{p,q}_X$ será então uma soma de elementos da forma
\begin{equation*}
\alpha =  f\: dz_{i_1} \wedge \cdots \wedge d z_{i_p} \wedge d\bar{z}_{j_1} \wedge d\bar{z}_{j_q},
\end{equation*}
com $f$ uma função suave.

Os elementos de $\mathcal{A}^{p,q}_X$ são chamados de \textbf{$(p,q)$-formas} \index{formas diferenciais!de tipo $(p,q)$} ou formas de tipo $(p,q)$. Da decomposição (\ref{eq:decomp-ext-power}) vemos que uma $k$-forma complexa se escreve de modo único com uma soma se $(p,q)$-formas com $p+q=k$.
\end{example}

\subsubsection{Os operadores $\del$ e $\overline{\del}$} \index{operador! $\del$ e $\delbar$}
A diferencial exterior $d:\mathcal{A}^k_X \to \mathcal{A}^{k+1}_X$ se estende de modo $\C$-linear a $d:\mathcal{A}^k_{X,\C} \to \mathcal{A}^{k+1}_{X,\C}$. Note que $d$ aumenta o grau total da forma em 1, mas não sabemos em geral como ele afeta o bigrau $(p,q)$. Os operadores $\del$ e $\overline{\del}$ são definidos restringindo $d$ a $\mathcal{A}^{p,q}$ e tomando as partes $(p+1,q)$ e $(p,q+1)$, respectivamente. Mais precisamente:
\begin{equation*}
\del: \mathcal{A}^{p,q}_X \longrightarrow \mathcal{A}^{p+1,q}_X \;\text{ e } \; \overline{\del}: \mathcal{A}^{p,q}_X \longrightarrow \mathcal{A}^{p,q+1}_X
\end{equation*}
são definidos por 
\begin{equation} \label{eq:def-d-dbar}
\del = \Pi^{p+1,q} \circ d \;\; \text{ e } \;\; \overline{\del} = \Pi^{p,q+1} \circ d.
\end{equation}

\indent Em coordenadas, se $\alpha =  f\: dz_{i_1} \wedge \cdots \wedge d z_{i_p} \wedge d\bar{z}_{j_1} \wedge d\bar{z}_{j_q}$, então
\begin{equation*}
\del \alpha  = \sum^n_{k=1} \frac{\del f}{\del z_k} dz_k \wedge dz_{i_1} \wedge \cdots \wedge d z_{i_p} \wedge  d\bar{z}_{j_1} \wedge d\bar{z}_{j_q} \;\; \text{ e }\;\; \overline{\del} \alpha  = \sum^n_{k=1} \frac{\del f}{\del \bar{z}_k} d\bar{z}_k \wedge dz_{i_1} \wedge \cdots \wedge d z_{i_p} \wedge  d\bar{z}_{j_1} \wedge d\bar{z}_{j_q}
\end{equation*}
e para uma $(p,q)$-forma qualquer usamos a linearidade de $\del$ e $\overline{\del}$.\\
\indent Note que, da descrição em coordenadas acima, temos que $d = \del + \delbar$, de onde vemos que $0=d^2=\del^2 + \del \delbar + \delbar\del + \delbar^2$. Como a soma \ref{eq:decomp-ext-power} é direta concluimos vemos que
\begin{equation*}
\del^2 = 0, \;\;\; \delbar^2 = 0 \;\;\; \text{ e } \;\;\; \del \delbar = -\delbar\del.
\end{equation*}

\subsubsection{Cohomologia de Dolbeaut} \label{sec:dolbeaut}

Devido a relação $\delbar^2 = 0$, temos, para cada $p\geq 0$, um complexo de feixes
\begin{equation} \label{eq:dolbeaut-complex}
\mathcal{A}_X^{p,0} \stackrel{\delbar}{\longrightarrow} \mathcal{A}_X^{p,1} \stackrel{\delbar}{\longrightarrow} \mathcal{A}_X^{p,2} \stackrel{\delbar}{\longrightarrow} \cdots
\end{equation}
chamado \textit{complexo de Dolbeaut} \index{complexo!de Dolbeaut}, e os espaços de  cohomologia no nível de seções globais
\begin{equation} \label{eq:dolbeaut-cohomology}
H^{p,q}_{\delbar}(X) = \frac{\ker\{\delbar: \mathcal{A}^{p,q}(X) \to \mathcal{A}^{p,q+1}(X)\}}{\im \{\delbar: \mathcal{A}^{p,q-1}(X) \to \mathcal{A}^{p,q}(X)\}}
\end{equation}
são chamados espaços de cohomologia de Dolbeaut. \index{cohomologia!de Dolbeaut} De certa forma, eles são um análogo complexo aos espaços de cohomologia de de Rham. Assim como no caso da cohomologia de de Rham, temos um análogo ao Lema de Poincaré, o que mostra que $H^{p,q}_{\delbar}(B)=0$ para polidiscos $B \subset \C^n$.

\begin{proposition} \label{prop:poincare-lemma} \index{delbarrapoincare@$\delbar$-Lema de Poincaré} \textbf{$\delbar$-Lema de Poincaré.} Seja $B \subset \C^n$ um polidisco. Se $\alpha \in \mathcal{A}^{p,q}(B)$ satisfaz $\delbar \alpha = 0$ então existe $\beta \in \mathcal{A}^{p,q-1}(B)$ tal que $\alpha = \delbar \beta$.
\end{proposition}

A demonstração se reduz essencialmente a resolver equações do tipo $\frac{\del f}{\del \bar z_k} = g$ para $f$, onde $g:B \to \C$ é uma função suave dada, e esse tipo de equação sempre pode ser resolvido em $B$.

No caso unidimensional por exemplo, se $g:B(R) \subset \C \to \C$ é uma função suave em $B(R)$, uma solução para a equação $\frac{\del f}{\del \bar z} = g$ na bola $B(r)$, $r<R$ é dada pela fórmula
\begin{equation*}
f_r(z) = \frac{1}{2\pi \smo}\int_{B(r)} \frac{g(w)}{w-z} dw \wedge d \bar w,
\end{equation*}

Considerando uma sequência $r_n \to R$ e funções $\rho_n:B(R)\to \R$ que valem $1$ em $B(r_n)$ e $0$ fora de $B(r_{n+1})$, as funções $\rho_n f_{r_n}$ convergem para uma solução da equação em $B(R)$.

Para os detalhes e a demonstração completa do lema veja por exemplo \cite{g-h}, cap 0.

\subsubsection{Seções holomorfas}
\indent Podemos também definir seções holomorfas de um fibrado holomorfo $\pi:E \to X$: são aplicações holomorfas $s:X \to E$ satisfazendo $\pi \circ s = \text{id}_X$. Assim como no caso diferenciável temos a correspondência
\begin{equation} \label{eq:hol-sec}
\left \lbrace {\begin{array}{c}
 \text{Seções holomorfas}  \\
 \text{ de }  E \sim \{(U_i,\varphi_{ij})\} \\
 \end{array} } \right \rbrace
 \longleftrightarrow
\left \lbrace {\begin{array}{c}
 \text{Funções holomorfas}   \\
 s_i: U_i \to \C^r \text{ satisfazendo }  \\
 s_i = \varphi_{ij} \cdot s_j \text{ em } U_i \cap U_j \\
 \end{array} } \right \rbrace
\end{equation}

Denotamos por $\mathcal O(E)$ o feixe de seções holomorfas de $E$.

\begin{remark}
O feixe $\mathcal O(E)$ é de um tipo especial. Dizemos que um feixe de $\mathcal O_X$-módulos $\mathcal F$ sobre $X$ é \emph{localmente livre} se para todo ponto $x \in X$ existe uma vizinhança $U$ de $x$ e um inteiro $r$ tal que $\mathcal F|_U \simeq \mathcal O_X^{\oplus r}|_U$. Se o inteiro $r$ não depende de $x$ dizemos que $\mathcal F$ é um feixe localmente livre de posto $r$.

Se $E$ é um fibrado holomorfo de posto $r$, uma trivialização $\psi:E|_U \to U \times \C^r$ fornece um isomorfismo $E|_U \simeq U \times \C^r$ e portanto $\mathcal O(E) \simeq \mathcal O_X^{\oplus r}|_U$, o que mostra que $\mathcal O(E)$ é um feixe localmente livre de posto $r$.

Reciprocamente, dado um feixe $\mathcal F$ sobre $X$ que é localmente livre de posto $r$, podemos cobrir $X$ por abertos $U_i$ de modo que  existam isomorfismos $\psi_i: \mathcal F|_{U_i} \simeq \mathcal O_X^{\oplus r}|_{U_i}$. As composições $\psi_{ij} = \psi_i \circ \psi_j^{-1}: \mathcal O_X^{\oplus r}|_{U_i \cap U_j} \to \mathcal O_X^{\oplus r}|_{U_i \cap U_j}$ são dadas por matrizes de funções holomorfas e portanto são cociclos de um fibrado holomorfo $E$.

É fácil ver que as duas construções acima são mutuamente inversas e portanto obtemos uma correspondência biunívoca
\begin{equation*}
\left \lbrace {\begin{array}{c}
 \text{Fibrados vetoriais holomorfos}  \\
 \text{de posto } r \text{ sobre } X 
 \end{array} } \right \rbrace
 \longleftrightarrow
\left \lbrace {\begin{array}{c}
 \text{Feixes de } \mathcal O_X-\text{módulos}   \\
  \text{localmente livres de posto }r
 \end{array} } \right \rbrace\end{equation*}

Por esse motivo denotaremos também por $E$ o feixe de seções holomorfas de $E$. Note que, com essa convenção, o espaço das seções holomorfas globais de $E$ é $H^0(X,E)$.
\end{remark}

Diferentemente do caso diferenciável, em que o conjunto das seções é um espaço vetorial de dimensão infinita, um fibrado holomorfo não admite muitas seções holomorfas. Por exemplo, se $E=X \times \C$ então suas seções holomorfas são simplesmente funções holomorfas globais em $X$, que devem ser constantes se $X$ é compacta. Logo, nesse caso o espaço das seções é unidimensional.\\
\indent Pode ocorrer ainda que um fibrado holomorfo não admita nehnuma seção holomorfa não trivial. Veja por exemplo o Corolário \ref{cor:dual-lb} e a equação (\ref{eq:sec-Ok<0}).\\

\indent Um \textit{morfismo entre fibrados holomorfos} $\pi_E: E\to X$ e $\pi_F: F \to X$ é uma aplicação holomorfa $\varphi:E \to F$ tal que $\pi_E = \pi_F \circ \varphi$ e induz uma aplicação linear $\varphi_x = \varphi|_{E_x}:E_x \to F_x$  cujo posto independe de $x$, e um \textit{isomorfismo} é um morfismo bijetor. Assim como no caso diferenciável temos

\begin{proposition} \label{prop:iso-holvb}
Dois fibrados holomorfos de posto $r$, $E \to X$ e $F:X \to X$  são isomorfos se e somente se existe uma cobertura $\mathcal{U} = \{U_i\}_i$ de $X$ e funções holomorfas $\lambda_i:U_i \to \GL (r,\C)$ tais que $\phi_{ij} = \lambda_i \psi_{ij} \lambda_j^{-1}$ em $U_i \cap U_j$, onde $\phi_{ij}$ e $\psi_{ij}$ são os cociclos de $E$ e $F$ com relação à $\mathcal{U}$, respectivamente.
\end{proposition}

\indent Da descrição local do operador $\overline{\del}$ (eq.~(\ref{eq:def-d-dbar})), e lembrando que temos um isomorfismo $\bigwedge^{p,0}X \simeq \Omega^p_X$ obtemos o seguinte resultado
\begin{proposition} \label{prop:hol-forms}
O espaço das seções holomorfas globais de $\Omega^p_X$ (isto é, das $p$-formas holomorfas) é dado por $H^0(X,\Omega_X^p) = \{\alpha \in \mathcal{A}^{p,0}(X): \overline{\del}\alpha =0 \}$.
\end{proposition}

\indent Da proposição acima, vemos que o núcleo de $\delbar: \mathcal{A}^{p,0}_X \to \mathcal{A}^{p,1}_X$ é o feixe $\Omega_X^p$ das $p$-formas holomorfas, de onde obtemos um complexo de feixes
\begin{equation*}
0 \longrightarrow \Omega^p_X \longrightarrow \mathcal{A}^{p,0}_X \longrightarrow \mathcal{A}^{p,1}_X \longrightarrow \cdots .
\end{equation*}

\indent Do $\delbar$-lema de Poincaré, temos que o complexo acima é exato, e portanto o complexo $(\mathcal{A}^{p,\bullet}_X,\delbar)$ é uma resolução do feixe $\Omega^p_X$. Como os feixes $\mathcal{A}^{p,q}_X$ são acíclicos, obtemos (veja a proposição \ref{prop:acyclic-res}) o chamado \textbf{isomorfismo de Dolbeaut}
\begin{equation} \label{eq:dolbeaut-iso} \index{isomorfismo de Dolbeaut}
H^{p,q}_{\delbar}(X) \simeq H^q(X,\Omega^p_X).
\end{equation}

\section{Fibrados de linha e o grupo de Picard} \label{sec:lb}

Nesta seção estudaremos fibrados de linha holomorfos. Esses fibrados são importantes por diversas razões. A primeira delas é que estes são os fibrados holomorfos mais simples, e portanto mais fáceis de serem estudados. Uma outra vantagem é que o produto tensorial de fibrados de linha sobre uma variedade $X$ é novamente um fibrado de linha, o que nos permite definir um grupo associado a $X$: o grupo de Picard $\Pic(X)$.\\
\indent Veremos adiante também que fibrados de linha estão intimamente relacionados com subvariedades complexas de codimensão 1 (i.e. hipersuperfícies) e mais geralmente com divisores. Além disso, seções desses fibrados permitem definir aplicações $X \to \pr^N$ e portanto são centrais na interface entre a geometria analítica e a geometria algébrica complexa.\\

\indent Vamos começar estudando fibrados de linha sobre o espaço projetivo $\pr^n$.

\begin{example} \textbf{O fibrado tautológico}\\
\indent Considere o subconjunto de $\pr^n \times \C^{n+1}$ definido por
\begin{equation*}
\mathcal{O}(-1) = \{(\ell,z) \in \pr^n \times \C^{n+1} : z \in \ell\}
\end{equation*}
e seja $\pi: \mathcal{O}(-1) \to \pr^n$ a restrição a $\mathcal{O}(-1)$ da projeção na primeira coordenada.\\
\indent Note que $\pi^{-1}(\ell) = \{ \ell\} \times \ell$, onde vemos o segundo fator como $\ell \subset \C^{n+1}$. Assim a fibra sobre uma reta $\ell$ é naturalmente isomorfa a $\ell$. 

\begin{proposition} \label{prop:taut-lb} O conjunto $\mathcal{O}(-1)$ é uma subvariedade complexa de $\pr^n \times \C^{n+1}$ e  $\pi: \mathcal{O}(-1) \to \pr^n$ é um fibrado de linha holomorfo sobre $\pr^n$.
\end{proposition}

\begin{proof}
Vamos provar as duas afirmações simultaneamente.\\
\indent Considere a cobertura $\pr^n = \bigcup_{i=0}^n U_i$, onde $U_i =\{z_i \neq 0\}$ e sejam
\begin{equation}
 \begin{split}
  \psi_i: \pi^{-1}(U_i) &\longrightarrow U_i \times \C \\
  (\ell,z) &\longmapsto (\ell,z_i)
 \end{split}
\end{equation}

É claro que $\psi_i$ é contínua. A inversa é dada por $\psi_i^{-1}:(\ell,w) \mapsto (\ell,(w \frac{z_0}{z_i},\ldots,w,\ldots,w \frac{z_n}{z_i}))$, onde $\ell=[z_0:\cdots:z_n]$, e é claramente contínua. Logo as $\psi_i$ são homeomorfismos, que dão ao mesmo tempo trivializações locais e coordenadas em $\mathcal{O}(-1)$.

Para ver que o fibrado e as coordenadas são holomorfos basta notar que os cociclos são dados por
\begin{equation}
 \begin{split}
 \psi_{ij}: U_i \cap U_j &\longrightarrow \GL(1,\C) \simeq \C^* \\
  [ z_0:\cdots:z_n ] &\longmapsto \frac{z_i}{z_j}
 \end{split}
\label{eq:cocyle-taut}
\end{equation}
que são claramente holomorfos.
\label{eq:triv-taut}
\end{proof}

O fibrado $\mathcal{O}(-1)$ é chamado \textit{fibrado tautológico sobre $\pr^n$} \index{fibrado!tautológico}. Essa nomenclatura vem do fato de que a fibra sobre uma reta $\ell \in \pr^n$ é naturalmente identificada com a própria reta $\ell$. O dual de $\mathcal{O}(-1)$ é denotado por $\mathcal{O}(1)$ e, mais geralmente, denotamos
\begin{equation} \index{Ok@$\mathcal O(k)$}
\mathcal{O}(k) = \left \lbrace
\begin{split}
&\mathcal{O}(1)^{\otimes k} = \mathcal{O}(1)\otimes \cdots \otimes \mathcal{O}(1) & \text{ se } k>0 \\
&\mathcal{O}_{\pr^n} = \pr^n \times \C & \text{ se } k=0 \\
&\mathcal{O}(-1)^{\otimes k} = \mathcal{O}(-1)\otimes \cdots \otimes \mathcal{O}(-1) & \text{ se } k<0
\end{split}
\right .
\label{eq:def-Ok}
\end{equation}

\end{example}

\begin{proposition} \label{prop:dual-lb}
Seja $L\to X$ um fibrado de linha holomorfo e $L^*$ o seu dual. Então $L \otimes L^* \simeq \mathcal{O}_X$, onde $ \mathcal{O}_X$ é o fibrado de linha trivial sobre $X$.
\end{proposition}

\begin{proof}
Sejam $\{ \varphi_{ij} \}$ os cociclos de $L$ associados a uma cobertura $\{U_i\}$. Então, em relação a mesma cobertura, o fibrado $L^*$ é dado pelos cociclos $\{ \varphi_{ij}^{-1} \}$. O produto tensorial então é dado pelos cociclos $\{\varphi_{ij}\cdot \varphi_{ij}^{-1} \equiv 1\}$, que são os cociclos do fibrado trivial.
\end{proof}

\begin{corollary} \label{cor:dual-lb}
Seja $L$ um fibrado de linha sobre uma variedade compacta $X$. Se $L$ e $L^*$ admitem seções não triviais então $L$ é trivial.
\end{corollary}
\begin{proof}
Sejam $s$ e $t$ seções de $L$ e $L^*$ respectivamente. Então $f=s\otimes t$ é uma seção de $L\otimes L^* \simeq \mathcal{O}_X$, isto é, uma função holomorfa global. Como $X$ é compacta segue que $f$ é constante.\\
\indent Suponha $L$ não trivial. Então $s(x)=0$ para algum $x$ e portanto $f \equiv 0$. Mas então $s \otimes t = 0$ e pelo princípio da identidade devemos ter $s=0$ ou $t=0$, ou seja uma das seções deve ser trivial.
\end{proof}
\indent A proposição \ref{prop:dual-lb} nos permite definir um grupo abeliano associado a $X$

\begin{definition}
O \textbf{grupo de Picard de $X$}, \index{grupo de Picard} denotado por $\Pic (X)$, é o grupo formado pelas classes de isomorfismo de fibrados de linha holomorfos cuja operação de grupo é o produto tensorial $(L,L') \mapsto L \otimes L'$, o elemento neutro é o fibrado trivial $\mathcal{O}_X$ e a inversão é a dualização $L \mapsto L^*$.
\end{definition}

\indent Em termos de trivializações locais, se $L \sim \{\varphi_{ij}\}$ e $L' \sim \{\varphi'_{ij}\}$, temos
\begin{equation*}
L \otimes L' \sim \{\varphi_{ij} \cdot \varphi'_{ij}\} \;\; \text{ e } \;\; L^* = \{ \varphi_{ij}^{-1}\}
\end{equation*}

\begin{remark} \label{rmk:cocyclesOk}
Da demonstração da proposição \ref{prop:taut-lb}, vemos que o fibrado tautológico $\mathcal{O}(-1)$ é dado pelos cociclos $\{z_i/z_j\}$, com relação à cobertura padrão de $\pr^n$. Logo, os fibrados $\mathcal{O}(k)$ são dados pelos cociclos $\{z_j^k/z_i^k\}$. Note que $\mathcal{O}(m) \otimes \mathcal{O}(n) = \mathcal{O}(m+n)$.\\
\end{remark}

\begin{example} \label{ex:sec-Ok} \index{seções de $\mathcal O(k)$} \textbf{Seções holomorfas de $\mathcal{O}(k)$}\\
\indent Seja $s \in \C[z_0,\cdots,z_n]$ um polinômio homogêneo de grau $k>0$. A partir de $s$ podemos definir uma seção holomorfa global de $\mathcal{O}(k)$ sobre $\pr^n$ da seguinte maneira.\\
\indent Se $U_i=\{z_i \neq 0\} \subset \pr^n$, defina $s_i: U_i \to \C$ por $s_i([z_0:\cdots:z_n]) = s\left(\frac{z_0}{z_i},\cdots,1,\cdots,\frac{z_n}{z_i} \right)$. É claro que $s_i$ está bem definida e é holomorfa em $U_i$.\\
\indent Da homogeneidade de $s$ vemos que
\begin{equation*}
s_i = s\left(\frac{z_0}{z_i},\cdots,1,\cdots,\frac{z_n}{z_i} \right) = \frac{1}{z_i^k}s(z_0,\cdots,z_n) = \frac{z_j^k}{z_i^k}s\left(\frac{z_0}{z_j},\cdots,1,\cdots,\frac{z_n}{z_j} \right) = \frac{z_j^k}{z_i^k}s_j
\end{equation*}
e como $z_j^k/z_i^k$ são os cociclos de $\mathcal{O}(k)$ com relação a cobertura $\{U_i,i=0\cdots,n\}$, obtemos, pela correspondência (\ref{eq:hol-sec}) uma seção $s \in H^0(\pr^n,\mathcal{O}(k))$.\\

\indent Reciprocamente, seja $t \in H^0(\pr^n,\mathcal{O}(k))$ com $k > 0$. Fixe $s_0 \in H^0(\pr^n,\mathcal{O}(k))$ dada por um polinômio como acima. Note que a razão $F = t/s_0$ define uma função meromorfa em $\pr^n$. Compondo com a projeção $\pi:\C^{n+1}\setminus \{0\} \to \pr^n$ obtemos uma função meromorfa em $\C^{n+1}\setminus\{0\}$, $\widetilde{F}=F \circ \pi$. A função $G=s_0 \widetilde{F}$ é holomorfa em $\C^{n+1}\setminus\{0\}$ e portanto, pelo teorema de Hartogs (Teorema \ref{thm:hartogs}), se estende a $\C^{n+1}$.\\
\indent Note que $\widetilde{F}$ é invariante pela ação de $\C^*$ e portanto, da homogeneidade de $s_0$, vemos que $G(\lambda z) = \lambda^k G(z)$ para todo $\lambda \in \C$, $z \in \C^{n+1}$. Olhando para a expansão de $G$ em série de potências em torno de $z=0$ vemos que $G$ deve ser um polinômio homogêneo de grau $k$, que induz a seção $t$.\\

\indent Obtemos assim uma identificação
\begin{equation} \label{eq:sec-Ok>0}
H^0(\pr^n,\mathcal{O}(k)) \simeq \left\lbrace \begin{split} \text{Polinômios } &\text{homogêneos } \\ \text{ de grau } &k \text{ em } \C^{n+1} \end{split} \right\rbrace
\end{equation}
para $k > 0$.\\

\indent Observe que de um lado temos um objeto analítico, isto é, seções holomorfas de um fibrado de linha, enquanto de outro um objeto algébrico (polinômios). Esse tipo de correspondência é bastante importante e muito útil na geometria complexa, sendo responsável pela transformação da geometria complexa analítica em $\pr^n$ em geometria algébrica projetiva. Esses pensamentos podem ser levados mais adiante, levando a resultados como o teorema de Chow (toda subvariedade complexa fechada de $\pr^n$ é algébrica) e, em um nível mais sofisticado a correspondência GAGA (\textit{Géometrie Algébrique et Géométrie Analytique}) de Serre.\\

\indent Do corolário \ref{cor:dual-lb} vemos também que
\begin{equation} \label{eq:sec-Ok<0}
H^0(\pr^n,\mathcal{O}(k))=0 \text{ para } k<0.
\end{equation}
\end{example}

\indent Observe que, das equações (\ref{eq:sec-Ok>0}) e (\ref{eq:sec-Ok<0}), vemos que nenhum dos fibrados $\mathcal{O}(k), k \neq 0$ é trivial, pois $\dim H^0(\pr^n,\mathcal{O}_{\pr^n}) = 1$.\\

\indent A seguinte proposição mostra que o grupo de Picard de $X$ também pode ser visto como um grupo de cohomologia.

\begin{proposition} \label{prop:isopic}
Existe um isomorfismo natural
\begin{equation*}
\Pic (X) \simeq H^1(X,\mathcal{O}^*_X)
\end{equation*}
\end{proposition}
\begin{proof}
Dada uma cobertura $\mathcal{U}$ que trivializa um fibrado de linha, os cociclos associados são funções holomorfas $\varphi_{ij}:U_i \cap U_j \to \C^*$, ou seja, $\varphi_{ij} \in \mathcal{O}^*_X(U_i\cap U_j)$. Das condições de cociclo (\ref{eq:cocycle-cond}) vemos que $\delta (\varphi_{ij}) = 0$, e portanto $(\varphi_{ij})$ define uma classe em $\check{H}^1(\mathcal{U},\mathcal{O}^*_X)$.\\
\indent Construimos assim um homomorfismo
\begin{equation*}
\Pic (X) \to \check{H}^1(\mathcal{U},\mathcal{O}^*_X).
\end{equation*}
\indent Pela proposição \ref{prop:iso-holvb} vemos que dois fibrados $L$ e $L'$ são isomorfos se e somente se seus cociclos em relação a alguma cobertura satisfazem $\varphi_{ij} = \lambda_i \psi_{ij} \lambda_j^{-1}$ em $U_i \cap U_j$ com $\lambda_i \in \mathcal{O}^*(U_i)$ e $\lambda_j \in \mathcal{O}^*(U_j)$, ou seja, se somente se $\varphi_{ij}$ e $\psi_{ij}$ diferem pelo cobordo $\lambda_i \lambda_j^{-1} \in \mathcal{O}^*(U_i \cap U_j)$. Isso mostra que o homomorfismo acima é injetor se tomarmos uma cobertura suficientemente fina.\\
\indent O homomorfismo acima também é sobrejetor: pela observação \ref{rmk:cocycles}, dados cociclos $\{\varphi_{ij}\}$ com relação a $\mathcal{U}$ existe um fibrado de linha $L$ que tem $\{\varphi_{ij}\}$ como funções de transição, e pela proposição \ref{prop:iso-holvb}, a classe de isomorfismo de $L$ independe do representante de $\{\varphi_{ij}\}$ em $\check{H}^1(\mathcal{U},\mathcal{O}^*_X)$.\\
\indent Temos então uma coleção de homomorfismos $\Pic (X) \to \check{H}^1(\mathcal{U},\mathcal{O}^*_X)$ que comutam com as aplicações  $\check{H}^1(\mathcal{U},\mathcal{O}^*_X) \to \check{H}^1(\mathcal{V},\mathcal{O}^*_X)$ induzidas por um refinamento $\mathcal{V} \prec \mathcal{U}$. E além disso $\Pic (X) \simeq \check{H}^1(\mathcal{U},\mathcal{O}^*_X)$ tomando $\mathcal{U}$ suficientemente fina. Assim, tomando o limite direto obtemos 
\begin{equation*}
\Pic (X) \simeq \lim_{\longrightarrow} \check{H}^1(\mathcal{U},\mathcal{O}^*_X) = \check{H}^1(X,\mathcal{O}^*_X).
\end{equation*}

\indent Como $\check{H}^q(X,\mathcal{F}) \simeq H^q(X,\mathcal{F})$ (cf. Proposição \ref{prop:iso-cech-res}), o resultado segue.
\end{proof}

\subsection{A sequência exponencial} \label{subsec:expseq} \index{sequência!exponencial}

A exponencial de funções holomorfas em um aberto $U \subset X$ define um morfismo de feixes $\mathcal{O}_X \to \mathcal{O}^*_X$, dado por
\begin{equation*}
\mathcal{O}(U) \ni f \longmapsto \exp(2 \pi \smo f) \in \mathcal{O}^*(U),
\end{equation*}
lembrando que vemos $\mathcal{O}_X$ como feixe aditivo e $\mathcal{O}^*_X$ como feixe multiplicativo.\\
\indent Se $U$ é conexo então o núcleo de $\exp:\mathcal{O}_X (U)\to \mathcal{O}_X^*(U)$ é formado pelas funções constantes em $U$ com valores inteiros e portanto o núcleo do homomorfismo $\exp: \mathcal{O}_X \to \mathcal{O}^*_X$ é o feixe localmente constante $\Z$. A existência do logaritmo assegura que $\exp:\mathcal{O} (U)\to \mathcal{O}^*(U)$ é sobrejetora para todo aberto simplesmente conexo suficientemente pequeno.\\

\indent Obtemos assim uma sequência exata de feixes
\begin{equation} \label{eq:expseq}
0 \longrightarrow \Z \longrightarrow \mathcal{O}_X \longrightarrow \mathcal{O}^*_X\longrightarrow 0
\end{equation}
chamada \textbf{sequência exponencial} em $X$.\\

\indent A sequência exata longa associada à sequência exponencial
\begin{equation*}
\begin{split}
0 &\longrightarrow H^0(X,\Z) \longrightarrow H^0(X,\mathcal{O}_X) \longrightarrow H^0(X,\mathcal{O}^*_X) \longrightarrow \\
&\longrightarrow H^1(X,\Z) \longrightarrow H^1(X,\mathcal{O}_X) \longrightarrow H^1(X,\mathcal{O}^*_X) \longrightarrow \\
&\longrightarrow H^2(X,\Z) \longrightarrow \cdots
\end{split}
\end{equation*}
contém algumas informações interessantes sobre $X$.\\

\indent Destaquemos a seguinte parte da sequência
\begin{equation} \label{eq:explongseq}
H^1(X,\Z) \longrightarrow H^1(X,\mathcal{O}_X) \longrightarrow H^1(X,\mathcal{O}^*_X) \longrightarrow H^2(X,\Z)
\end{equation}

\indent Primeiramente note que, quando $X$ é compacta, a aplicação $H^1(X,\Z) \to H^1(X,\mathcal{O}_X)$ é injetora. De fato: para $X$ compacta temos que $H^0(X,\mathcal{O}_X) = \C$ e $H^0(X,\mathcal{O}^*_X) = \C^*$, e $\exp:H^0(X,\mathcal{O}_X) \to H^0(X,\mathcal{O}^*_X)$ é a exponencial complexa usual, que é sobrejetora. Portanto $H^0(X,\mathcal{O}^*_X)$ $\to H^1(X,\Z)$ é a aplicação nula, de onde segue a injetividade de $H^1(X,\Z) \to H^1(X,\mathcal{O}_X)$.\\

\indent A sequência (\ref{eq:explongseq}) nos dá uma maneira de calcular o grupo de Picard, $\Pic (X) = H^1(X,\mathcal{O}^*_X)$ conhecendo os grupos $H^1(X,\mathcal{O}_X)$, $H^i(X,\Z)$, $i=1,2$, e as aplicações induzidas entre eles. Note que $H^i(X,\Z)$ é o $i$-ésimo grupo de cohomologia singular de $X$ (ver exemplo \ref{ex:sing-cohom}), que é usualmente calculado por métodos topológicos (como por exemplo pela sequência de Mayer-Vietoris).\\

\indent A aplicação de cobordo $H^1(X,\mathcal{O}^*_X) \to H^2(X,\Z)$ nos dá um invariante importante associado aos fibrados de linha:

\begin{definition} \label{def:chernclass}
A \textbf{primeira classe de Chern} \index{primeira classe de Chern!de um fibrado de linha} de um fibrado $L \in \Pic(X)$, denotada por $c_1(L)$, é a imagem de $L$ pela aplicação de cobordo
\begin{equation*}
\Pic(X) = H^1(X,\mathcal{O}^*_X) \ni L \longmapsto c_1(L) \in H^2(X,\Z).
\end{equation*}
\end{definition}

\indent Existem algumas definições equivalentes da primeira classe de Chern, como por exemplo como sendo a classe de cohomologia da forma de curvatura de $L$ com relação a alguma métrica hermitiana. Falaremos disso  no capítulo \ref{ch:vb}.\\

É possível dar uma descrição explícita da primeira classe de Chern de um fibrado de linha holomorfo em termos de seus cociclos. Suponha que $L$ seja trivializado sobre uma cobertura $\mathcal U = (U_i)_i$ e sejam $\varphi_{ij} \in \mathcal O_X^*(U_i \cap U_j)$ os cociclos de \v{C}ech associados. Da sobrejetividade da sequência exponencial podemos supor, tomando $\mathcal U$ suficientemente fina, que $\varphi_{ij} = \exp(2\pi \smo f_{ij})$. Da condição de cociclo $\varphi_{ij} \varphi_{jk} \varphi_{ki} = 1$ temos que $f_{ij} + f_{jk} + f_{ki} \in \Z$ e pela definição da aplicação de cobordo $\check H^1(X,\mathcal O_X^*) \to \check H^2(X,\Z)$ (veja a Proposição \ref{prop:cech-long-exact}), o elemento $(\varphi_{ij})$ corresponde a classe de $\delta (f_{ij}) = (z_{ijk})$ em $\check{H}^2(X,\Z)$, onde $z_{ijk} = f_{jk} - f_{ik} + f_{ij}$, e portanto vemos que
\begin{equation*}
c_1(L) =  \left[f_{jk} - f_{ik} + f_{ij}\right] = \left[\frac{1}{2 \pi \smo} \left( \log \varphi_{jk} - \log \varphi_{ik} + \log \varphi_{ij} \right)\right] \in \check{H}^2(X,\Z),
\end{equation*}
onde $\log$ denota um ramo do logaritmo.\\

\indent Vemos da sequência (\ref{eq:explongseq}) que para determinar os fibrados de linha holomorfos sobre $X$, que são parametrizados por $H^1(X,\mathcal{O}^*_X)$, devemos entender a parte desse grupo vinda de  $H^1(X,\mathcal{O}_X)$, e sua imagem pela aplicação $c_1:H^1(X,\mathcal{O}^*_X) \to H^2(X,\Z)$. De certa maneira, podemos ver $\Pic (X)$ como sendo constituido de uma parte contínua, associada ao feixe $\mathcal{O}_X$, e uma discreta, associada ao feixe $\Z$. A parte contínua, isto é, a imagem de $H^1(X,\mathcal{O}_X) \to H^1(X,\mathcal{O}^*_X)$ é chama variedade jacobiana. Algumas de suas propriedades são estudadas na seção \ref{sec:jacobi-albanese}\\

\begin{remark} \label{rmk:irregularity}
Note que quando $H^1(X,\mathcal{O}_X)=0$ temos que $c_1:H^1(X,\mathcal{O}^*_X) \to H^2(X,\Z)$ é injetora, o que quer dizer que a primeira classe de Chern determina (a menos de isomorfismo) o fibrado de linha holomofo.\\
\indent Classicamente, no caso em que $\dim X = 2$, o número $q= \dim H^1(X,\mathcal{O}_X)$ é chamado \textit{irregularidade} da superfície complexa $X$. 
\end{remark}

\subsection{Construindo aplicações $X\to \pr^N$}

 Vimos no capítulo \ref{ch:cap0} que uma variedade complexa compacta $X$ não pode ser mergulhada holomorficamente no espaço euclideano $\C^n$ (Corolário \ref{cor:compact-submanifold}).
 
  Em vez disso, podemos então tentar obter mergulhos de $X$ em algum espaço projetivo $\pr^N$. Uma maneira de construir tais mergulhos é por meio de seções de fibrados de linha sobre $X$. Este será o assunto desta seção.\\

Seja $L$ um fibrado de linha sobre $X$.
\begin{definition}
Dizemos que $x \in X$ é um ponto base de $L$ se $s(x)=0$ para toda seção $s \in H^0(X,L)$. O conjunto dos pontos base é denotado por $\text{Bs}(L)$.
\end{definition}

Suponha que $H^0(X,L)$ tenha dimensão finita e seja $s_0,\ldots,s_N \in H^0(X,L)$ uma base. Note que $\text{Bs}(L) = Z(s_0)\cap \cdots \cap Z(s_N)$ e é uma subvariedade analítica de $X$.

 Se $U \subset X \setminus \text{Bs}(L)$ é um aberto e $\psi$ uma trivialização sobre $U$, a aplicação
\begin{equation*}
U \ni x \longmapsto [\psi(s_0(x)):\cdots:\psi(s_n(x))] \in \pr^N
\end{equation*}
está bem definida, pois as coordenadas homogêneas não são simultaneamente nulas.

 Agora, outra trivialização sobre $U$ será da forma $\lambda \psi$, com $\lambda \in \mathcal{O}^*_X$ e nesta trivialização a aplicação acima fica
\begin{equation*}
 x \longmapsto [\lambda(x) \psi(s_0(x)):\cdots:\lambda(x) \psi(s_n(x))] = [\psi(s_0(x)):\cdots:\psi(s_n(x))]
\end{equation*}

Assim, podemos definir a aplicação globalmente, e obtemos
\begin{equation*}
\varphi_L:X\setminus \text{Bs}(L) \longrightarrow \pr^N
\end{equation*}

 Observe que ainda dependemos da escolha de uma base de $H^0(X,L)$, mas uma outra escolha produziria uma aplicação que difere da original por um automorfismo linear de $\pr^N$.\\
 
 Mais geralmente poderiamos considerar aplicações induzidas por seções pertencentes a um subespaço $Q \subset H^0(X,L)$. Dizemos que $Q$ é um \textit{sistema linear de $L$} e $Q = H^0(X,L)$ é chamado \textit{sistema linear completo de $L$}.

\begin{definition}
Dizemos que um fibrado de linha $L$ é \textbf{amplo} se para algum $k>0$ e algum sistema linear de $L^k$ a aplicação induzida $X \to \pr^N$ é um mergulho e é \textbf{muito amplo} se $k=1$.
\end{definition}

\begin{example}\label{ex:veronese} \textbf{O mergulho de Veronese}

 Se $X=\pr^n$ e $L = \mathcal{O}(d)$, vimos que os polinômios homogêneos de grau $d$ podem ser vistos como seções globais de $L$. Nesse caso $\text{Bs}(L)=\emptyset$, pois dado $p \in \pr^n$, alguma coordenada homogênea $z_k$ de $p$ é não nula e daí $z_k^d$ é uma seção que não se anula em $p$.
 
  A aplicação induzida é
\begin{equation*}
\begin{split}
\varphi_{\mathcal{O}(d)}: \pr^n &\longrightarrow \pr^N\\
[z_0:\cdots:z_n] &\longmapsto [\cdots:z_0^{i_0} \cdots z_n^{i_n}:\cdots]\
\end{split}
\end{equation*}
onde $(i_0,\ldots,i_n)$ percorrem todos os multi-índices com $\sum^n_{k=0} i_k = d$. Vamos denotar por $z_{i_0 \cdots i_n}$ as coordenadas homogêneas em $\pr^N$. Note que o número dessas coordenadas é o número de monômios de grau $d$ em $n$ variáveis e portanto $N = \binom{n+d}{n}-1$.

 Para ver que $\varphi_{\mathcal{O}(d)}$ é um mergulho, primeiro reordenamos as coordenadas em $\pr^N$ de modo que $\varphi_{\mathcal{O}(d)}$ é dada por $[z_0:\cdots:z_n] \mapsto [z_0^d:z_0^{d-1}z_1: z_0^{d-1}z_2\cdots:z_0^{d-1}z_n:\cdots]$. Assim, em $U_0 = \{z_0 \neq 0\} \subset \pr^n$, $[1:z_1:\cdots:z_n] \mapsto [1:z_1:z_2\cdots:z_n:\cdots]$ e portanto, em coordenadas a aplicação de Veronese é dada por $(w_1, \ldots, w_n) \mapsto (w_1, \ldots, w_n,\ldots)$, o que mostra que $\varphi_{\mathcal{O}(d)}$ é uma imersão em $U_0$. De modo análogo mostramos que $\varphi_{\mathcal{O}(d)}$ é uma imersão nos demais abertos coordenados de $\pr^n$. É fácil ver também que  $\varphi_{\mathcal{O}(d)}$ é injetora e como $\pr^n$ é compacto concluímos que a aplicação de Veronese é um mergulho. Em particular isso mostra que os fibrados $\mathcal{O}(d)$ são muito amplos.
 
 A aplicação $\varphi_{\mathcal{O}(d)}$ é chamada \index{mergulho de Veronese} \textbf{mergulho de Veronese}.
 
  No caso $n=1$, $d=2$ por exemplo temos que $\pr^1$ é mapeado na curva plana $\{x_1^2 = x_0 x_2\} \subset \pr^2$, via $[z_0:z_1] \mapsto [z_0^2:z_0z_1:z_2^2]$.
 
 Uma propriedade interessante da aplicação de Veronese $\varphi_{\mathcal{O}(d)}$ é que ela transforma hipersuperfícies projetivas de grau $d$ em $\pr^n$ em seções de hiperplanos da imagem $P = \varphi_{\mathcal{O}(d)}(\pr^n) \subset \pr^N$. De fato, uma hipersuperfície de grau $d$ em $\pr^n$ é dada pelos zeros de um polinômio homogêneo de grau $d$, que é da forma $\sum a_{i_0 \cdots i_n} z_0^{i_0} \cdots z_n^{i_n}$ e portanto sua imagem será dada pela intersecção de $P$ com o conjunto de zeros do polinômio $\sum a_{i_0 \cdots i_n} z_{i_0 \cdots i_n}$, que é linear nas coordenadas $z_{i_0 \cdots i_n}$ de $\pr^N$.

\end{example}

\section{Divisores} \label{sec:div}

\begin{definition} \label{def:divisors} \index{divisor!de Weil}
Um \textbf{divisor de Weil} $D$ em uma variedade complexa $X$ é uma soma formal localmente finita
\begin{equation*}
D = \sum a_i Y_i,
\end{equation*}
 onde $Y_i \subset X$ são hipersuperfícies irredutíveis e $a_i \in \Z$. Com localmente finita queremos dizer que para todo $x \in X$ existe uma vizinhança aberta $U$ tal que $Y_i \cap U \neq  \emptyset$ apenas para um número finito de índixes $i$.\\
\indent Um divisor é \textit{efetivo} se $a_i \geq 0$ para todo $i$.\\
\indent O conjunto de todos os divisores de Weil, denotado por, $\Div (X)$ é o grupo abeliano livre gerado pelas hipersuperfícies de $X$, chamado \textit{grupo de divisores de $X$}.

Podemos definir uma ordem parcial em $\Div (X)$: se $D = \sum a_i Y_i $ e $E = \sum b_i Y_i$, dizemos que $D\leq E$ se $a_i \leq b_i$ para todo $i$.
\end{definition}

\begin{example}
Toda hipersuperfície $Y \subset X$ pode ser vista como um divisor de Weil efetivo.
\end{example}

\begin{example} \textbf{Divisores principais}\\
\indent Seja $X$ uma variedade complexa e $Y \subset X$ uma hipersuperfície irredutível. Dada uma função holomorfa $f$ em uma vizinhança de $x \in Y$, ela define um germe $f \in \mathcal{O}_{X,x}$. Se $g$ é um elemento irredutível de $\mathcal{O}_{X,x}$ induzido pela equação que define $Y$, a ordem de $f$ em $x$ é o maior inteiro $m=\Ord_{Y,x}$ tal que $g^m$ divide $f$ em $\mathcal{O}_{X,x}$ ou seja, é o maior inteiro tal que a equação
\begin{equation*}
f = g^m \cdot h,\;\; h \in \mathcal{O}_{X,x},
\end{equation*} 
vale em $\mathcal{O}_{X,x}$.

A partir da estrutura algébrica de $\mathcal{O}_{X,x}$ (é um anel local, noetheriano e de fatorização única) podemos ver que a definição acima independe da escolha da equação $g$ que define $Y$. Além disso, elementos relativamente primos em  $\mathcal{O}_{X,x}$ continuam relativamente primos em  $\mathcal{O}_{X,y}$ para $y$ suficientemente próximos de $x$.\footnote{Para uma demonstração de todos esses fatos consulte o capítulo 0 de \cite{g-h}.} Isso permite definir a \textbf{ordem de $f$ ao longo de $Y$} por
\begin{equation*}
\Ord_Y(f) = \Ord_{Y,x} (f)\; \text{ para algum } x \in Y.
\end{equation*}

Mais geralmente, se $f$ é meromorfa, podemos escrever $f = g/h$ localmente, e definimos $\Ord_{Y,x} (f) = \Ord_{Y,x} (g) - \Ord_{Y,x} (h)$ e em seguida $\Ord_Y(f) = \Ord_{Y,x} (f)$ para algum $x \in Y.$

Note que a ordem é aditiva, isto é, $\Ord_Y(f_1 f_2) = \Ord_Y(f_1) + \Ord_Y(f_2) $ e que $\Ord_Y(f)=0$ se e só se $f$ é holomorfa não se anula identicamente em $Y$.

Dada uma função meromorfa global $f \in K(X)$, o \textit{divisor associado a} $f$ é o divisor de Weil
\begin{equation*}
(f) = \sum \Ord (f) _Y \cdot Y,
\end{equation*}
onde a soma é tomada sobre todas as hipersuperfícies irredutíveis de $X$.

Note que $(f)$ pode ser escrito como a diferença de dois divisores efetivos:
\begin{equation*}
(f) = Z(f) - P(f)\;\;\; \text{ onde } \;\;\; Z(f) = \sum_{\Ord_Y >0} \Ord _Y (f) \cdot Y \;\; \text{ e } \;\; P(f) = \sum_{\Ord_Y <0} \Ord _Y (f) \cdot Y.
\end{equation*}

Os divisores $Z(f)$ e $P(f)$ são chamados, respectivamente, de \textit{divisor de zeros e de pólos de} $f$. Note que $Z(f) \geq 0$ e $P(f) \geq 0$.
\end{example}

\begin{definition} \label{def:princ-div} \index{divisor!principal}
Um divisor de Weil $D$ é dito um \textbf{divisor principal} se $D = (f)$ para alguma $f \in K(X)$. O conjunto dos divisores principais em $X$ é denotado por $\PDiv (X)$.
\end{definition}
\begin{remark}
Note que $\PDiv (X)$ é um subgrupo de $\Div (X)$, pois se $D = (f)$ e $D'= (g)$ então $D + D' = (fg)$ e $-D = (f^{-1})$.\\
\end{remark}

Seja $Y \subset X$ uma hipersuperfície irredutível. Localmente, $Y$ é dada como zeros de funções holomorfas, isto é, existe uma cobertura $X = \bigcup U_i$ e funções holomorfas $f_i \in \mathcal{O}(U_i)$ tais que $Y \cap U_i = Z(f_i)$. Se $U_i \cap U_j \neq  \emptyset$ temos que $f_i$ e $f_j$ definem a mesma hipersuperfície irredutível e portanto $f_i$ e $f_j$ diferem por uma função holomorfa que não se anula, isto é, $f_i f_j^{-1} \in \mathcal{O}^*(U_i \cap U_j)$. Isso nos permite ver $Y$ como uma seção global do feixe quociente $\mathcal{O}_X / \mathcal{O}^*_X$.

 De maneira análoga, qualquer divisor de Weil efetivo pode ser visto como uma seção global de $\mathcal{O}_X / \mathcal{O}^*_X$. No caso de divisores de Weil em geral, não basta considerar feixe das funções holomorfas, devido a presença de coeficientes negativos, o que motiva a seguinte definição.

\begin{definition} \index{divisor!de Cartier}
Um \textbf{divisor de Cartier} em $X$ é uma seção global do feixe $\mathcal{K}^*_X/\mathcal{O}^*_X$, isto é, um elemento de $H^0(\mathcal{K}^*_X/\mathcal{O}^*_X)$. Aqui $\mathcal{K}^*_X$ é o feixe das funções meromorfas não identicamente nulas em $X$.
\end{definition}

Ocorre que, no caso das variedades complexas, as definições de divisor de Cartier e de Weil coincidem:

\begin{proposition} \label{prop:weil-cartier}
Existe um isomorfismo natural
\begin{equation*}
H^0(X,\mathcal{K}^*_X/\mathcal{O}^*_X) = \Div (X)
\end{equation*}
\end{proposition}
\begin{proof}
A ideia é simples: associar a uma função meromorfa seu divisor. Como a multiplicação por uma função holomorfa que não se anula não altera o divisor, essa associação fica definida em $K^*_X$ módulo $\mathcal{O}^*_X$.

Mais precisamente, um divisor de Cartier $f = H^0(X,\mathcal{K}^*_X/\mathcal{O}^*_X)$ é dado em uma cobertura por funções meromorfas $f_i \in \mathcal{K}^*_X(U_i)$ com $f_i f_j^{-1} \in \mathcal{O}^*_X(U_i \cap U_j)$. Assim, para toda hipersuperfície $Y \subset X$, temos que $\Ord_Y(f_i) = \Ord_Y(f_j)$ e portanto podemos definir $\Ord_Y (f)$ como sendo $\Ord_Y(f_i)$ para algum $i$. Associamos a $f$ o divisor de Weil $D = \sum \Ord_Y (f) \cdot Y$.

Reciprocamente, seja $D = \sum a_i Y_i$ um divisor de Weil. Existe uma cobertura $\mathcal{U}=\{U_j\}_j$ tal que $Y_i \cap U_j = Z(g_{ij})$, onde $g_{ij} \in \mathcal{O}_X(U_j)$ e tais funções são únicas a menos de elementos de $\mathcal{O}^*_X(U_j)$. Defina $f_j = \prod_i g_{ij}^{a_i} \in \mathcal{K}^*_X(U_j)$. Como $Z(g_{ij})\cap U_k = Y_i \cap U_j \cap U_k = Z(g_{ik})\cap U_j$ e $Y_i$ é irredutível segue que $g_{ij}$ e $g_{ik}$ diferem por uma função em $\mathcal{O}^*_X(U_j\cap U_k)$. Consequentemente $f_j$ e $f_k$ diferem, em $U_j\cap U_k$, por um elemento de $\mathcal{O}^*_X(U_j\cap U_k)$ e portanto definem uma seção $f = H^0(X,\mathcal{K}^*_X/\mathcal{O}^*_X)$.

É fácil ver, pela aditividade da ordem, que tais construções preservam a estrutura de grupo e são mutuamente inversas, dando o isomorfimso desejado.
\end{proof}

Devido a proposição acima, não faremos mais distinção entre divisores de Weil e divisores de Cartier e com frequência os chamaremos apenas de divisores. Assim, um divisor será tanto uma soma $D = \sum a_i Y_i$ como uma seção global de $\mathcal{K}^*_X/\mathcal{O}^*_X$.\\

\subsection{A relação entre divisores e fibrados de linha} \label{sec:rel-div-lb}

Nesta seção veremos que os conceitos de divisores em $X$ e de fibrados de linha sobre $X$ estão intimamente ligados. A ideia básica é associar a cada divisor um fibrado de linha. Veremos também como definir divisores a partir de fibrados de linha que admitem seções globais.

Seja $D \in \Div (X)$. Pela Proposição \ref{prop:weil-cartier}, $D$ corresponde a uma seção global do feixe $\mathcal{K}^*_X/\mathcal{O}^*_X$, ou seja, a uma coleção de funções meromorfas $f_i$ definidas em abertos $U_i$ de uma cobertura tais que $f_i f_j^{-1} \in \mathcal{O}_X^*(U_i \cap U_j)$. Denotamos por $\mathcal{O}(D)$ o fibrado de linha associado aos cociclos $(f_{ij} = f_i f_j^{-1})$.

Se tomarmos outro divisor $D'$, associado a $f'_i$, a soma $D+D'$ corresponde ao produto $f_i f_i'$ e portanto o fibrado associado a $D + D'$ é dado pelos cociclos
\begin{equation*}
g_{ij} = \frac{f_i f'_i}{f_j f'_j} = \frac{f_i}{f_j} \cdot \frac{f'_i}{f'_j},
\end{equation*}
que são os cociclos do fibrado $\mathcal{O}(D) \otimes \mathcal{O}(D)$.

Obtemos assim um homomorfismo
\begin{equation} \label{eq:div-pic}
 \begin{split}
 \Div(X) &\longrightarrow \Pic(X) \\
 D &\longmapsto \mathcal{O}(D).
 \end{split}
\end{equation}

A proposição seguinte mostre que o núcleo desse homomorfismo é formado justamente pelos divisores principais:

\begin{proposition} \label{prop:kerOD}
O fibrado $\mathcal{O}(D)$ é trivial se e somente se $D$ é um divisor principal.
\end{proposition}
\begin{proof}
Suponha primeiro que $D = (f)$ é um divisor principal. Então $D$ corresponde à imagem da seção global $f \in H^0(X,\mathcal{K}^*_X)$ pela aplicação induzida pela projeção natural $\mathcal{K}^*_X \to \mathcal{K}^*_X/\mathcal{O}^*_X$. Assim, o cociclo que define $\mathcal{O}(D)$ é $f/f \equiv 1$, o que mostra que $\mathcal{O}(D)\simeq \mathcal{O}_X$.

Reciprocamente, suponha que $\mathcal{O}(D)\simeq \mathcal{O}_X$. Pela proposição \ref{prop:iso-holvb} existe uma cobertura $\mathcal{U} = \{U_i\}$ e funções $g_i \in \mathcal{O}^*_X (U_i)$ tais que os cociclos $(\varphi_{ij})$ de $\mathcal{O}(D)$ com relação a $\mathcal{U}$ satisfazem $\varphi_{ij} = g_i g_j^{-1}$ em $U_i \cap U_j$. Se $D$ é dado por $\{f_i \in \mathcal{K}^*_X (U_i) \}$, então $\varphi_{ij} = f_i f_j^{-1}$, e portanto temos que $f_i f_j^{-1} =  g_i g_j^{-1}$ de onde segue que 
\begin{equation*}
f_i g_i^{-1} =  f_j g_j^{-1} \;\;\;\ \text{ em } \;\; U_i \cap U_j.
\end{equation*}

 Assim, as seções locais $f_i g_i^{-1} \in \mathcal{K}^*_X (U_i)$ se colam a uma seção global $f \in K(X)$, definida por $f|_{U_i} = f_i g_i^{-1}$.
 
Como as $g_i$'s são holomorfas e não se anulam temos que $\Ord_Y (f) = \Ord_Y (f_i)$ se $Y \cap U_i \neq  \emptyset$, de onde vemos que $(f) = D$, e portanto $D$ é principal.
\end{proof}

A proposição acima mostra que há uma injeção
\begin{equation} \label{eq:inj-pic}
\frac{\Div(X)}{\PDiv(X)} \hookrightarrow \Pic(X).
\end{equation}

\begin{definition} \label{def:lin-equiv}
Dizemos que dois divisores $D,D' \in \Div(X)$ são linearmente equivalentes se eles representam a mesma classe em $\Div(X)/ \PDiv(X)$, isto é, se $D-D'$ é um divisor principal.
\end{definition}

Vemos então que $\mathcal{O}(D) \simeq \mathcal{O}(D')$ se e somente se $D$ e $D'$ são linearmente equivalentes.

\begin{example} \textbf{Divisores de hiperplano em $\pr ^n$}

Seja $H_0 \subset \pr^n$ o hiperplano dado pela equação $z_0 = 0$. Temos que $H_0$ é uma hipersuperfície de $\pr ^n$ e portanto define um divisor de Weil $H_0 \in \Div (\pr^n)$. Note que $H_0 \cap U_0 =  \emptyset$ e $H_0 \cap U_i \neq  \emptyset$ para $i=1,\cdots,n$ e que
\begin{equation*}
H_0 \cap U_0 = Z(1) \;\;\ \text{ e } \;\; H_0 \cap U_i = Z(z_0/z_i),\; i=1,\cdots,n.
\end{equation*}
onde vemos aqui $1$ como a função constante em $U_0$ e $z_0/z_i \in \mathcal{O}_{\pr^n}(U_i)$.

 Assim, o divisor de Cartier associado a $H_0$ é $\{1 \in \mathcal{O}^*_{\pr^n}(U_0), z_0/z_i \in \mathcal{O}^*_{\pr^n}(U_i)\}$, e portanto o fibrado $\mathcal{O}(H_0)$ é dado pelos cociclos $z_j/z_i \in \mathcal{O}_{\pr^n}^* (U_i \cap U_j)$, que são os cociclos do fibrado $\mathcal{O}(1)$ (ver observação \ref{rmk:cocyclesOk}), ou seja
\begin{equation*}
\mathcal{O}(H_0) \simeq \mathcal{O}(1).
\end{equation*}

 O mesmo isomorfismo é obtido se tomarmos outro hiperplano $H \subset \pr^n$. Se $H$ é dado pelos zeros de uma equação linear $P \in \C[z_0,\cdots,z_n]$ então $f = P/z_0$ define uma função meromorfa global em $\pr^n$ e $(f) = H - H_0$. Assim $H$ e $H_0$ são linearmente equivalentes e portanto $\mathcal{O}(H_0) \simeq \mathcal{O}(H)$. Mostramos então que
\begin{equation*}
\mathcal{O}(H) \simeq \mathcal{O}(1) \;\; \text{ para todo hiperplano } \;\; H \subset \pr^n. 
\end{equation*}

 Por essa razão, o fibrado $\mathcal{O}(1)$ também é chamado de fibrado de hiperplanos sobre $\pr ^n$.\\
\end{example}

\begin{example} \textbf{Funções meromorfas com pólos limitados por um divisor}

 Seja $D = \sum a_i Y_i$ um divisor em $X$ e considere o espaço.
\begin{equation*}
L(D) = \{f \in K(X) : \Ord_{Y_i} (f) \geq -a_i \} = \{f \in K(X) :(f) + D \geq 0\}.
\end{equation*}

 Note que se $D' \leq D$ então $L(D') \subset L(D)$ e que $L(0) = H^0(X,\mathcal{O}_X)$ é o espaço das funções holomorfas globais em $X$. \\

 Da demonstração da proposição \ref{prop:weil-cartier} vemos que $D$ corresponde a um elemento $f = \{f_i \in K^*_X(U_i) \}\in H^0(X,\mathcal{K}^*_X/\mathcal{O}^*_X)$ tal que $D = \sum \Ord_Y (f) \cdot Y$.\\

Seja $g$ uma função em $L(D)$. Da desigualdade $\Ord_{Y_i} (g) \geq -a_i$ temos que $g_i = gf_i \in \mathcal{O}_X (U_i)$ e essas funções satisfazem $g_i = f_i f_j^{-1} \cdot g_j$. Como $f_i f_j^{-1}$ são os cociclos de $\mathcal{O}(D)$, as $g_i$'s definem uma seção holomorfa $gf \in H^0(X,\mathcal{O}(D))$.

 Reciprocamente dada $s \in H^0(X,\mathcal{O}(D))$, obtemos $s_i \in \mathcal{O}(U_i)$ satisfazendo $s_i = f_i f_j^{-1} \cdot s_j$. Podemos então definir $g \in K(X)$ por $g|_{U_i} = s_i/f_i$ e como as $f_i$'s definem $D$ temos que $g$ é holomorfa em $X \setminus D$ e $\Ord_{Y_i}(g) = \Ord_{Y_i}(s_i) - \Ord_{Y_i}(f_i) \geq - \Ord_{Y_i}(f_i) = -a_i$, ou seja, $g \in L(D)$.\\

Obtemos assim um isomorfismo
\begin{equation*}
L(D) \cong H^0(X, \mathcal{O}(D)),
\end{equation*}
 dado pela multiplicação pela seção de $\mathcal{K}^*_X/\mathcal{O}^*_X$ que define $D$.\\
\end{example}

\begin{example} \textbf{Hipersuperfícies e funções meromorfas em $\C^n$ e o Problema de Cousin.} \index{Problema de Cousin}

A sequência exponencial em $\C^n$ induz a sequência exata
\begin{equation*}
H^1(\C^n,\Z) \longrightarrow H^1(\C^n,\mathcal{O}_{\C^n}) \longrightarrow \Pic(\C^n) \longrightarrow H^2(X,\Z).
\end{equation*}
Como $\C^n$ é contrátil temos que $H^2(X,\Z) = 0$. Além disso, do $\delbar$-Lema de Poincaré (proposição \ref{prop:poincare-lemma}) e do isomorfismo de Dolbeaut temos que $H^1(X,\mathcal{O}_{\C^n}) \simeq H^{0,1}_{\delbar}(\C^n) = 0$ de onde concluímos que $\Pic(\C^n) = 0$, ou seja, \textbf{todo fibrado de linha sobre $\C^n$ é trivial}.

Em particular, se $D$ é um divisor em $\C^n$ então $\mathcal{O}(D)$ é trivial e portanto, pela proposição \ref{prop:kerOD}, concluímos que $D$ é principal. Provamos assim que todo divisor em $\C^n$ é principal, o que quer dizer que \emph{podemos prescrever zeros e pólos de uma função meromorfa}, isto é, dada uma família localmente finita de hipersuperfícies $Y_i \subset \C^n$ e inteiros $a_i$ existe uma função meromorfa global $f \in K(\C^n)$ com $\Ord _{Y_i}(f) = a_i$. Este é um problema clássico na teoria de funções analíticas e o caso $n=1$ foi resolvido por Weierstrass.
Considerando o caso particular em que $D=Y$ é uma hipersuperfície, concluímos que toda hipersuperfíce em $\C^n$ é dada pelos zeros de uma única função holomorfa global.\\

Uma outra consequência interessante é a de que \emph{toda função meromorfa em $\C^n$ pode ser escrita, globalmente, como razão de duas funções holomorfas}. De fato, seja $f \in K(X)$ não nula e $(f) = Z(f) - P(f)$ seu divisor. Como todo divisor é principal segue que $P(f) = (h)$ para alguma $h \in K(\C^n)$, mas como $P(f) \geq 0$ vemos que na verdade $h \in \mathcal{O}(\C^n)$. Definindo $g = fh$ temos que $(g) = (f) + (h) = Z(f) \geq 0$ de onde vemos que $g \in \mathcal{O}(\C^n)$ e portanto $f = \frac{g}{h}$ é razão de duas funções holomorfas globais.

Esse resultado foi provado para $\C^2$ por Poincaré em 1883, em um trabalho intitulado `` Sur les fonctions de deux variables''.\\

O problema de decidir se todos divisores em uma variedade complexa são principais é chamado \emph{Problema de Cousin} e o problem de decidir se as funções meromorfas podem ser escritas globalmente como razão de funções holomorfas é chamado \emph{Problema de Poincaré}. Vimos que o primeiro implica o segundo e ambos podem ser resolvidos em $\C^n$. Note que para tirar essa conclusão só usamos o fato de que $\Pic(\C^n) = 0$. Obtemos portanto a seguinte generalização.

\begin{proposition}
Seja $X$ uma variedade complexa tal que $\Pic(X)=0$. Então
\begin{itemize}
\item[1.] O problema de Cousin pode ser resolvido em $X$, isto é, todo divisor em $X$ é principal. Em particular toda hipersuperfície é dada pelos zeros de uma função holomorfa global.
\item[2.] O problema de Poincaré pode ser resolvido em $X$, isto é, toda função meromorfa global em $X$ pode ser escrita como razão de duas funções holomorfas globais.
\end{itemize}
\end{proposition}
\end{example}

 Grande parte da discussão desta seção, como a definição da aplicação $\Div(X) \to \Pic (X)$ e a proposição \ref{prop:kerOD}, pode ser resumida na existência da sequência exata de feixes
\begin{equation} \label{eq:seq-mero}
0 \longrightarrow \mathcal{O}^*_X \longrightarrow \mathcal{K}^*_X \longrightarrow \mathcal{K}^*_X / \mathcal{O}^*_X \longrightarrow 0
\end{equation}

 Olhando para a sequência longa associada temos

\begin{equation}
 \begin{matrix}
 H^0(X,\mathcal{K}^*_X) & \longrightarrow & H^0(X,\mathcal{K}^*_X/\mathcal{O}^*_X) & \longrightarrow & H^1(X,\mathcal{O}^*_X) \\
 \parallel &  & \parallel & & \parallel \\
 K(X)^* &  & \Div(X) & & \Pic(X)
 \end{matrix}
\end{equation}

 A sequência acima nos diz que após as identificações $\Div(X) = H^0(X,\mathcal{K}^*_X/\mathcal{O}^*_X)$ e $\Pic(X) = H^1(X,\mathcal{O}^*_X)$, a  aplicação $D \to \mathcal{O}(D)$ nada mais é que a aplicação de cobordo $H^0(X,\mathcal{K}^*_X/\mathcal{O}^*_X) \to H^1(X,\mathcal{O}^*_X)$, e que seu núcleo é a imagem da aplicação $K(X)^* \to \Div(X)$ que associa a cada função meromorfa global seu divisor, ou seja, são os divisores principais.
\chapter{Geometria de Kähler} \label{ch:kahler}

Nesse capítulo estudaremos os aspectos métricos das variedades complexas. A primeira ideia é analisar as relações entre as métricas riemannianas e a estrutura complexa da variedade. Uma primeira condição de compatibiliadade entre essas duas estruturas (a saber, que a estrutura complexa induzida $J_x:T_x X \to T_x X$ é uma isometria) leva a noção de métrica hermitiana. Essa condição não é muito restritiva, pois é muito simples obter uma métrica hermitiana a partir de uma métrica riemanniana. Uma outra condição de compatibilidade é a chamada condição de Kähler, que pode ser formulada de diversas maneiras equivalentes. A mais geométrica delas é, possivelmente, exigir que o tensor $J$ seja paralelo com respeito a conexão de Levi-Civita. A condição de Kähler é bem mais forte, impondo diversas restrições sobre $X$, quer sobre a estrutura complexa ou até mesmo sobre a sua topologia, mas não é tão forte a ponto de limitar muito a quantidade de exemplos. Toda variedade projetiva, por exemplo, admite uma métrica de Kähler.

\section{Métricas hermitianas}
Se $X$ é uma variedade complexa, vimos no capítulo \ref{ch:cap0} que a estrutura complexa de $X$ induz uma estrutura complexa em cada espaço tangente $T_x X$, que na sua totalidade dão um automorfismo $J:TX \to TX$ tal que $J^2 = - \text{id}$.

Se $g$ é uma métrica riemanniana em $X$, dizemos que $g$ é uma métrica hermitiana se $J$ é uma isometria, isto é, se
\begin{equation*} \index{métrica!hermitiana}
g_x(J_x u,J_x v) = g_x (u,v) \; \text{ para todo } u,v \in T_x X.\\
\end{equation*}

As métricas hermitianas são o análogo natural das métricas riemannianas para variedades complexas.\\

\begin{example} O produto interno usual em $\C^n \simeq \R^{2n}$, dado por $\langle z,w \rangle = \re \sum_{j=1}^n z_j \bar{w}_j$ é uma métrica hermitiana, pois é claramente invariante pela multiplicação por $\smo$.
\end{example}

\begin{example}
Se $f:X \to Y$ é uma imersão holomorfa e $g$ é uma métrica hermitiana em $Y$ então $f^*g$ é uma métrica hermitiana em $X$.
\end{example}

A partir de uma métrica riemanniana $g$ sempre podemos definir uma métrica hermitiana $\widetilde{g}$, fazendo
\begin{equation*}
\widetilde{g}(X,Y) = \frac{1}{2}\lbrace g(X,Y) + g(JX,JY) \rbrace.
\end{equation*}

Como toda variedade admite uma métrica riemanniana, segue que toda variedade complexa admite também uma métrica hermitiana.\\

\begin{remark} Note que se $Y \in TX$ então $Y$ e $JY$ são ortogonais com respeito a qualquer métrica hermitiana, pois
\begin{equation*}
g(Y,JY) = g(JY,J^2 Y) = - g(JY,Y) = -g(Y,JY).
\end{equation*}
\end{remark}

\begin{remark} \label{rmk:herm-riem}
Se $g$ é uma métrica hermitiana podemos definir um produto hermitiano em cada espaço tangente, fazendo
\begin{equation*}
h = g + \smo g(\cdot,J\cdot),
\end{equation*}
isto é, $h$ é uma forma sesquilinear em $(TX,J)$ positiva definida. Reciprocamente, se $h$ é um produto hermitiano em $TX$, então $g = \re h$ é uma métrica hermitiana em $X$.

Por essa razão, daqui em diante, o termo métrica hermitiana designará tanto um produto hermitiano em $(TX,J)$ quanto uma métrica em $TX$ compatível com $J$.\\
\end{remark}

\begin{definition}
A \textbf{forma fundamental} \index{forma fundamental} ou forma de Kähler de uma métrica hermitiana $g$ é a 2-forma
\begin{equation} \label{eq:fundamental-form}
\omega = g(J\cdot,\cdot)
\end{equation}
\end{definition}

Note que $\omega$ é de fato antisimétrica:
\begin{equation*}
\omega(X,Y) = g(JX,Y) = -g(X,JY) = -g(JY,X) = - \omega(Y,X).
\end{equation*}

Além disso $\omega$ é não degenerada, pois se $\omega(X,Y)=0$ para todo $Y$, então $0 = \omega(JX,X) = -g(X,X)$ e portanto $X=0$, ou seja, $\omega$ é uma forma quase-simplética (uma 2-forma não degenerada).

\begin{lemma}
Dois dos tensores $J,g$ e $\omega$ determinam o terceiro. 
\end{lemma}
\begin{proof}
Se $J$ é conhecida, então $g$ determina $\omega$ por (\ref{eq:fundamental-form}) e $\omega$ determina $g$, pois $g = g(J \: \cdot, J \: \cdot) = \omega(\:\cdot,J \: \cdot)$.

Suponha agora que temos uma métrica $g$ e uma forma quase-simplética $\omega$. Podemos, em cada espaço tangente, representar $\omega$ por uma matriz anti-simétrica não degenerada $S$, isto é,
\begin{equation*}
\omega(X,Y) = g(X,SY)
\end{equation*}

Agora $-S^2$ é simétrica e positiva definida e portanto existe $R$ tal que $R^2 = -S^2$ e $RS=SR$. Defina $J = SR^{-1}$.
Então
\begin{equation*}
J^2 = S^2 (R^{-1})^2 = S^2 (R^2)^{-1} = - S^2 (S^2)^{-1} = -I
\end{equation*}
e portanto $J$ é uma estrutura complexa, segundo a qual $g$ é uma métrica hermitiana.
\end{proof}

O lema acima pode ser resumido na igualdade de grupos de Lie
\begin{equation*}
\text U(n) = O(2n) \cap \GL(n,\C) \cap \text{Sp}(n,\R)
\end{equation*}
em que a interseção entre dois grupos quaisquer no membro direito já é igual a $U(n)$.

Aqui vemos as matrizes de $\text U(n)$ e $GL(n,\C)$ como matrizes reais de tamanho $2n$ segundo a identificação $A + \smo B \mapsto \begin{pmatrix} A & B \\ -B & A\end{pmatrix}$.

Assim, o grupo $GL(n,\C)$ é formado pelas matrizes que comutam com a estrura complexa padrão $J_0$ em $\R^{2n}$ e $U(n)$ é o grupo das matrizes que comutam com $J_0$  preservam a forma hermitiana padrão. O grupo $\text{Sp}(n,\R)$ é, por definição, o grupo das matrizes que preservam a forma simplética padrão de $\R^{2n}$, $\omega_0 = \sum_{i=0}^n dx_i \wedge dy_i$.

\subsection{Expressão em coordenadas}

Tanto a forma fundamental $\omega$ quanto a forma hermitiana $h = g - \smo \omega$ se estendem de forma $\C$-bilinear a $T_{\C}X$. Denotaremos também por $\omega$ e $h$ essas extensões. Note que $h$ é $J$-sesquilinear, isto é,
\begin{equation*}
h(Ju,v) = \smo h(u,v) \;\; \text{ e } \;\; h(u,Jv) = -\smo h(u,v),
\end{equation*}
mas é bilinear com relação a multiplicação por $\smo$, isto é
\begin{equation*}
h(\smo u,v) = \smo h(u,v) = h(u,\smo v).
\end{equation*}

\begin{lemma}
Em um sistema de coordenadas $(x_1,\cdots,x_n,y_1,\cdots,y_n)$ valem as seguintes relações
\begin{enumerate}
\item $h \left( \frac{\partial}{\partial x_i},\frac{\partial}{\partial x_j} \right) = h \left(\frac{\partial}{\partial y_i},\frac{\partial}{\partial y_j} \right)$ e  $h \left( \frac{\partial}{\partial x_i},\frac{\partial}{\partial y_j} \right) = - h \left( \frac{\partial}{\partial y_i},\frac{\partial}{\partial x_j} \right) = -\smo\: h \left( \frac{\partial}{\partial x_i},\frac{\partial}{\partial x_j} \right)$;
\item $h \left( \frac{\partial}{\partial z_i},\frac{\partial}{\partial z_j} \right) = h \left( \frac{\partial}{\partial \bar{z}_i},\frac{\partial}{\partial \bar{z}_j} \right) = 0$ e $h \left( \frac{\partial}{\partial z_i},\frac{\partial}{\partial \bar{z}_j} \right) = h \left( \frac{\partial}{\partial x_i},\frac{\partial}{\partial x_j} \right)$.
\end{enumerate}
\end{lemma}

\begin{proof}
1. A primeira relação segue do fato de $h$ ser $J$-invariante e de $J(\partial \slash \partial x_i) = \partial \slash \partial y_i$ (cf. eq. \ref{eq:J-coord}).

Para a segunda relação também usamos a $J$-invariância de $h$:
\begin{equation*}
h \left( \frac{\partial}{\partial x_i},\frac{\partial}{\partial y_j} \right) = h \left( \frac{\partial}{\partial x_i},J \frac{\partial}{\partial x_j} \right) = h \left( J \frac{\partial}{\partial x_i}, J^2 \frac{\partial}{\partial x_j} \right) = - h \left( \frac{\partial}{\partial y_i},\frac{\partial}{\partial x_j} \right)
\end{equation*}
e o fato de $h$ ser $J$-sesquilinear.

2. Lembrando da definição de $\partial \slash \partial z_i$ (veja eq. \ref{eq:def-partialz}) e usando a parte 1 temos
\begin{equation*}
 \begin{split}
  h \left( \frac{\partial}{\partial z_i},\frac{\partial}{\partial z_j} \right) &= \frac{1}{4} h \left( \frac{\partial}{\partial x_i} - \smo \frac{\partial}{\partial y_i}, \frac{\partial}{\partial x_j} - \smo \frac{\partial}{\partial y_j} \right)\\
  &=\frac{1}{4} \left \lbrace h \left( \frac{\partial}{\partial x_i},\frac{\partial}{\partial x_j} \right) - \smo\: h \left( \frac{\partial}{\partial x_i},\frac{\partial}{\partial y_j} \right) - \smo \: h \left( \frac{\partial}{\partial y_i},\frac{\partial}{\partial x_j} \right)  - h \left( \frac{\partial}{\partial y_i},\frac{\partial}{\partial y_j} \right) \right \rbrace \\
  &=0,
 \end{split}
\end{equation*}
e analogamente $h \left( \frac{\partial}{\partial \bar{z}_i},\frac{\partial}{\partial \bar{z}_j} \right) = 0$.

Para a segunda relação temos
\begin{equation*}
 \begin{split}
  h \left( \frac{\partial}{\partial z_i},\frac{\partial}{\partial \bar{z}_j} \right) &= \frac{1}{4} h \left( \frac{\partial}{\partial x_i} - \smo \frac{\partial}{\partial y_i}, \frac{\partial}{\partial x_j} + \smo \frac{\partial}{\partial y_j} \right)\\
  &=\frac{1}{4} \left \lbrace h \left( \frac{\partial}{\partial x_i},\frac{\partial}{\partial x_j} \right) + \smo \: h \left( \frac{\partial}{\partial x_i},\frac{\partial}{\partial y_j} \right) - \smo \: h \left( \frac{\partial}{\partial y_i},\frac{\partial}{\partial x_j} \right) + h \left( \frac{\partial}{\partial y_i},\frac{\partial}{\partial y_j} \right) \right \rbrace \\
  &= \frac{1}{2} \left \lbrace h \left( \frac{\partial}{\partial x_i},\frac{\partial}{\partial x_j} \right) + \smo \: h\left( \frac{\partial}{\partial x_i},\frac{\partial}{\partial y_j} \right) \right \rbrace \\
  & =  h \left( \frac{\partial}{\partial x_i},\frac{\partial}{\partial x_j} \right).
 \end{split}
\end{equation*}
\end{proof}

Do lema acima vemos que na expressão do 2-tensor $h$ na base $\{dz_i \otimes d z_j, dz_i \otimes d \bar{z}_j, d\bar{z}_i \otimes d z_j,d\bar{z}_i \otimes d \bar{z}_j\}_{i,j=1}^n$ os termos $dz_i \otimes d z_j$ e $d\bar{z}_i \otimes d \bar{z}_j$ não aparecem. Portanto, uma métrica hermitiana é dada em coordenadas por
\begin{equation} \label{eq:metric-coord}
h = \sum_{ij}  h_{ij} \; dz_i \otimes d\bar{z}_j, 
\end{equation}
onde $h_{ij} = h(\frac{\partial}{\partial z_i}, \frac{\partial}{\partial \bar{z}_j})$ é uma matriz hermitiana positiva definida.

Nessas mesmas coordenadas, a forma fundamental é dada por
\begin{equation} \label{eq:fund-form-coord}
\omega = \frac{\smo}{2} \sum_{ij}  h_{ij} \; dz_i \wedge d\bar{z}_j.
\end{equation}

De fato, como $\partial \slash \partial x_i$ e $\partial \slash \partial y_i$ são vetores reais temos que
\begin{equation*}
\omega \left( \frac{\partial}{\partial x_i},\frac{\partial}{\partial x_j} \right) = -\im h \left( \frac{\partial}{\partial x_i},\frac{\partial}{\partial x_j} \right) = -\im h_{ij},
\end{equation*}
e
\begin{equation*}
\omega \left( \frac{\partial}{\partial x_i},\frac{\partial}{\partial y_j} \right) = g \left( \frac{\partial}{\partial y_i},\frac{\partial}{\partial y_j} \right) = \re h \left( \frac{\partial}{\partial y_i},\frac{\partial}{\partial y_j} \right) = \re h_{ij}
\end{equation*}
de onde vemos que
\begin{equation*}
\begin{split}
 \omega \left( \frac{\partial}{\partial z_i},\frac{\partial}{\partial \bar{z}_j} \right) &= \frac{1}{4} \left \lbrace \omega \left( \frac{\partial}{\partial x_i},\frac{\partial}{\partial x_j} \right) + \smo \: \omega \left( \frac{\partial}{\partial x_i},\frac{\partial}{\partial y_j} \right) - \smo \: \omega \left( \frac{\partial}{\partial y_i},\frac{\partial}{\partial x_j} \right) + \omega \left( \frac{\partial}{\partial y_i},\frac{\partial}{\partial y_j} \right) \right \rbrace \\
 &=\frac{1}{2} \left \lbrace - \im h_{ij} + \smo \: \re h_{ij} \right \rbrace = \frac{\smo}{2} \left \lbrace  \re h_{ij} + \smo \: \im h_{ij} \right \rbrace \\
 &=\frac{\smo}{2} h_{ij},
\end{split}
\end{equation*}
demonstrando a representação (\ref{eq:fund-form-coord}).

Este cálculo mostra também que $\omega$ é uma $(1,1)$-forma e do fato de $(h_{ij})$ ser hermitiana, vemos que $\omega$ é real, isto é, $\overline{\omega}=\omega$.

\begin{example} \label{ex:std-metric} No caso em que $X = \C^n$ e $h=h_0$ é a métrica hermitiana padrão temos que $h_{ij} = \delta_{ij}$, ou seja,
\begin{equation*}
h_0 = \sum_{i=1}^n dz_i \otimes d \bar{z}_i = \sum_{i=1}^n dx_i \otimes dx_i + \sum_{i=1}^n dy_i \otimes dy_i
\end{equation*}
e a forma fundamental
\begin{equation*}
\omega_0 = \frac{\smo}{2}\sum_{i=1}^n dz_i \wedge d \bar{z}_i = \sum_{i=1}^n dx_i \wedge dy_i
\end{equation*}
é a forma simplética padrão em $\C^n \simeq \R^{2n}$, onde usamos as substituições $dz_i = dx_i + \smo dy_i$ e $d\bar{z}_i = dx_i - \smo dy_i$, de modo que $dz_i \wedge d\bar{z}_i = -2\smo dx_i \wedge dy_i$.
\end{example}

Essa descrição de $\omega$ também nos permite obter uma expressão simples para a forma volume de uma métrica hermitiana.

\begin{proposition} \label{prop:volume-form}
Seja $g$ uma métrica hermitiana em $X$ e $\omega$ a forma fundamental associada. Então a forma volume de $g$ é dada por
\begin{equation} \label{eq:volume-form}
\text{vol}_X = \frac{\omega^n}{n!}.
\end{equation}
\end{proposition}

\begin{proof}
A igualdade (\ref{eq:volume-form}) pode ser verificada pontualmente. Podemos supor então que $X=\C^n$ e $h= h_0 =\sum dz_i \otimes d\bar{z}_i$ é a metrica padrão. Temos nesse caso que
\begin{equation*}
\omega_0^n = \left( \sum_{i=1}^n dx_i \wedge dy_i \right)^n = n! \: dx_1 \wedge \cdots \wedge dx_m \wedge dy_1 \wedge \cdots \wedge dy^n = n! \: \text{vol}_{\C^n},
\end{equation*}
de onde o resultado segue.
\end{proof}

Note que como consequência temos que a forma volume $\text{vol}_X$ é uma $(n,n)$-forma.

Se  $Y \subset X$ é uma subvariedade complexa de dimensão $d$ e $\omega$ é a forma fundamental de $X$, a forma fundamental de $Y$ é simplesmente a restrição de $\omega$ a $Y$, de modo que, pela proposição \ref{prop:volume-form}, a sua forma volume é $\text{vol}_Y = (\omega^d)|_Y / d!$, de onde obtemos a seguinte consequência interessante.

\begin{proposition} \textbf{(Teorema de Wirtinger)} Seja $Y \subset X$ uma subvariedade complexa de dimensão $d$ de uma variedade hermitiana $X$, ambas compactas. O volume de $Y$, com a métrica induzida é
\begin{equation*}
\text{vol}(Y) = \int_Y \frac{\omega^d}{d!}.
\end{equation*}
\end{proposition}

Vemos portanto que o volume de $Y$ é dado pela integral de uma forma globalmente definida em $X$. Esse fenômeno não ocorre em geral no caso de variedades riemannianas.

\subsection{Decomposição de Lefschetz} \label{subsec:lefschetz}
Na seção \ref{sec:hol-vb} vimos que se $X$ é uma variedade hermitiana então existe uma decomposição do fibrado de $k$-formas
\begin{equation*}
\textstyle \bigwedge^k_{\C}X = \displaystyle \bigoplus_{p+q=k} \textstyle \bigwedge^{p,q} X.
\end{equation*}

Existe ainda uma outra decomposição de $\bigwedge^k_{\C}X$, chamada decomposição de Lefschetz, que é obtida construindo uma representação dá álgebra de Lie $\mathfrak{sl}(2,\C)$ na álgebra exterior $\bigwedge^* (TX)^*$.\\

Todos os cálculos serão pontuais, e portanto vamos supor que $X=V$ é um espaço vetorial de dimensão real $d=2n$ com estrutura complexa $I$ e $g= \langle \cdot,\cdot \rangle$ um produto escalar compatível.

Definimos o \textit{operador de Lefschetz} por
\begin{equation*} \index{operador!de Lefschetz}
 \begin{split}
 L: \textstyle \bigwedge^* V^* &\longrightarrow \textstyle \bigwedge^* V^*\\
 \alpha &\longmapsto \omega \wedge \alpha.
 \end{split}
\end{equation*}

Note que $L$ tem grau $2$, isto é, $L(\bigwedge^k V^*) \subset \bigwedge^{k+2} V^*$.\\

O produto escalar $\langle \cdot, \cdot \rangle$ em $V$ define um isomorfismo $V \ni v \mapsto \langle\: \cdot, v \rangle \in V^*$ e com isso define naturalmente um produto em $V^*$. Temos assim um produto escalar induzido em todas as potências exteriores $\bigwedge^k V^*$ definido da seguinte maneira: se $\{e^1,\ldots,e^d \}$ é uma base ortonormal de $V^*$ declaramos $\{e^{i_1} \wedge \cdots \wedge e^{i_k}: i_1<\cdots<i_k\}$ como base ortonormal de $\bigwedge^k V^*$. Finalmente, temos um produto em $\bigwedge^* V^* = \bigoplus_k \bigwedge^k V^*$, pedindo que os fatores sejam mutuamente ortogonais. Todos esses produtos serão denotados por $\langle \cdot, \cdot \rangle$.\\

O \textit{operador dual de Lefschetz} $\Lambda$ é definido como sendo o adjunto de $L: \bigwedge^* V^* \to \bigwedge^* V^*$ com respeito ao produto $\langle \cdot, \cdot \rangle$, ou seja, é o operador $\Lambda: \bigwedge^* V^* \to \bigwedge^* V^*$ definido pela condição
\begin{equation*}
\langle \Lambda \alpha, \beta \rangle = \langle \alpha, L \beta \rangle \; \text{ para todos } \alpha,\beta \in \textstyle \bigwedge^* V^*.
\end{equation*}

Denotaremos também por $\Lambda$ a sua extensão $\C$-linear a $\bigwedge^* V_{\C}$.

Note que, como $L$ tem grau $2$, $\Lambda$ tem grau $-2$, isto é, $\Lambda (\bigwedge^k V^*) \subset \bigwedge^{k-2} V^*$.\\

Como a estrutura complexa induz uma orientação natural em $V$, podemos definir o operador de Hodge $*:\bigwedge^* V^* \to \bigwedge^* V^*$, pela fórmula
\begin{equation*}
\alpha \wedge * \beta = \langle \alpha,\beta\rangle \text{vol}
\end{equation*}

Note que, como os $\bigwedge^k V^*$ são mutualmente ortogonais, o operador $*$ mapeia $\bigwedge^k V^*$ em $\bigwedge^{d-k} V^*$.

O operador dual de Lefschetz tem uma descrição simples em termos de $L$ e $*$:
\begin{lemma}
O dual do operador de Lefschetz é dado, em $\bigwedge^* V^*$, por $\Lambda = *^{-1} \circ L \circ *$.
\end{lemma}
\begin{proof}
A demonstração segue diretamente do seguinte cálculo
\begin{equation*}
\langle \alpha, L \beta \rangle \text{vol} = \langle L\beta,\alpha \rangle \text{vol} = L\beta \wedge * \alpha = \omega \wedge \beta \wedge * \alpha = \beta \wedge \omega \wedge * \alpha = \beta \wedge L(* \alpha) = \langle \beta, *^{-1}(L  (* \alpha) ) \rangle \text{vol}.
\end{equation*}
\end{proof}

Finalmente definimos um último operador $H: \bigwedge^* V^* \to \bigwedge^* V^*$, dado por
\begin{equation*}
H(\alpha) = (k-n)\alpha\; \text{ se }\: \alpha \in \textstyle \bigwedge^k V^*,
\end{equation*}
ou de modo mais sucinto, $H = \sum_{k=0}^{2n} (k-n)\Pi^k$.\\

Temos então três operadores, $L$,$\Lambda$ e $H$ agindo em $\bigwedge^* V^*$. Note que se $\alpha \in \bigwedge^k V^*$ então
\begin{equation*}
[H,L] \alpha = (H \circ L - L \circ H) \alpha = (k+2 - n)(\omega \wedge \alpha) - \omega \wedge ((k-n)\alpha) = 2\: \omega \wedge \alpha = 2 L \alpha,
\end{equation*}
e também
\begin{equation*}
[H,\Lambda]\alpha = (H \circ \Lambda - \Lambda \circ H) \alpha = (k-2- n)(\Lambda \alpha) - \Lambda ((k-n)\alpha) = -2 \Lambda \alpha.
\end{equation*}

Para o cálculo do comutador $[L,\Lambda]$ podemos usar indução sobre a dimensão de $V$, e concluimos que $[L,\Lambda] = H$. Para uma demonstração desse fato veja \cite{huybrechts}, prop. 1.2.26. Resumindo temos que
\begin{equation*}
[H,L] = 2L,\;\; [H,\Lambda] = -2\Lambda \; \text{ e } \; [L,\Lambda] = H.
\end{equation*}

Lembre que a álgebra de Lie $\mathfrak{sl}(2,\C)$ é formada pelas matrizes complexas de traço $0$, e tem como base as matrizes
\begin{equation*}
X = \begin{pmatrix} 0&1\\0&0 \end{pmatrix}, \;\; Y = \begin{pmatrix} 0&0\\1&0 \end{pmatrix}\; \text{ e }\; B=\begin{pmatrix} 1&0\\0&-1 \end{pmatrix},
\end{equation*}
com as relações
\begin{equation*}
[B,X] = 2L,\;\; [B,Y] = -2Y \; \text{ e } \; [X,Y] = B.
\end{equation*}

Assim, a associação
\begin{equation*}
X \mapsto L,\;\; Y \mapsto \Lambda \; \text{ e } \; H \mapsto B
\end{equation*}
define uma respresentação de $\mathfrak{sl}(2,\C)$ em $\bigwedge^* V^*$.\\

Como $\mathfrak{sl}(2,\C)$ é semisimples temos, pelo Teorema de Weyl, que toda representação de $\mathfrak{sl}(2,\C)$ é a soma direta de representações irredutíveis de dimensão finita. Além disso, da teoria de representações de $\mathfrak{sl}(2,\C)$ (veja por exemplo  \cite{serre} cap. IV) toda representação irredutível é gerada por um conjunto de vetores da forma $\{\alpha,L \alpha, L^2 \alpha, \ldots \}$, onde $\alpha \in \bigwedge^* V^*$ é um elemento primitivo, isto é, um autovetor de $H$ tal que\footnote{Note que os autoespaços de $H$ são justamente os $\bigwedge^k V^* \subset \bigwedge^* V^*$.} $\Lambda \alpha=0$.

Obtemos assim a decomposição de Lefschetz linear.

\begin{equation} \label{eq:lefschetz-linear}
\textstyle \bigwedge^k V^* = \displaystyle \bigoplus_{j \geq 0} L^j(P^{k-2j}),
\end{equation}
onde $P^k = \{ \alpha \in \bigwedge^k V^* : \Lambda \alpha = 0 \}$ é o espaço das $k$-formas primitivas.\\

Podemos dizer ainda um pouco mais
\begin{proposition}1. Se $k>n$ então $P^k=0$, isto é, não existem formas primitivas de grau maior que $n$.

2. A aplicação $L^{n-k}:P^k \to \bigwedge^{2n-k}V^*$ é injetora para $k\leq n$.

3. A aplicação $L^{n-k}: \bigwedge^k V^* \to \bigwedge^{2n-k}V^*$ é bijetora para $k\leq n$
\end{proposition}
\begin{proof}
As três afirmações seguirão da seguinte fórmula
\begin{equation*}
[L^i,\Lambda]\alpha = i(k-n+i-1)L^{i-1}\alpha, \: \text{ para } \alpha \in \textstyle \bigwedge^k V^*,
\end{equation*}
cuja demonstração segue facilmente usando $[L,\Lambda] = H$ e aplicando indução sobre $i$.

1. Seja $\alpha \in P^k$, $k>n$ e tome o menor $i\geq 0$ tal que $L^i \alpha = 0$. Então, da fórmula acima, temos que $0 = [L^i,\Lambda]\alpha = i(k-n+i-1)L^{i-1}\alpha$, de onde vemos que $i=0$ e portanto $\alpha = 0$.

2. Para mostrar a injetividade vamos mostrar que $L^{n-k}\alpha \neq 0$ se $\alpha \in P^k\setminus\{0\}$. Tomando novamente o menor $i\geq0$ tal que $L^i \alpha = 0$, obtemos, repetindo o cálculo acima que $k-n+i-1 = 0$ (pois como $\alpha$ é não nulo, $i\geq1$). Logo temos que $L^{n-k}\alpha = L^{i-1}\alpha \neq 0$.

3. Usando a parte 2, temos a injetividade de $L^{n-k}$ em $L^j(P^{k-2j}) \subset \bigwedge^k V^*$ e como a soma (\ref{eq:lefschetz-linear}) é direta, segue a injetividade em $\bigwedge^k V^*$. Como $\bigwedge^k V^*$ e $\bigwedge^{2n-k} V^*$ têm a mesma dimensão, segue que $L^{n-k}: \bigwedge^k V^* \to \bigwedge^{2n-k}V^*$ é bijetora.
\end{proof}

Toda a discussão acima pode ser feita fibra a fibra no fibrado de formas diferenciais de uma variedade hermitiana. Obtemos assim
\begin{theorem} \label{thm:lefschetz-forms} \index{decomposição de Lefschetz} \textbf{Decomposição de Lefschetz para formas diferenciais.} Seja $X$ uma variedade hermitiana de dimensão $n$ e defina os seguintes morfismos de fibrados vetoriais:
\begin{enumerate}
\item O operador de Lefschetz
\begin{equation*}
 \begin{split}
 L: \textstyle \bigwedge^k X &\longrightarrow \textstyle \bigwedge^{k+2} X\\
 \alpha &\longmapsto \omega \wedge \alpha.
 \end{split}
\end{equation*}
\item O operador $*$ de Hodge
\begin{equation*}
\begin{split}
*: \textstyle \bigwedge^k X &\longrightarrow \textstyle \bigwedge^{2n-k} X\\
\alpha \wedge * \beta &= \langle \alpha, \beta \rangle \text{vol}_X.
\end{split}
\end{equation*}
\item O operador dual de Lefschetz
\begin{equation*}
\Lambda = *^{-1} \circ L \circ * : \textstyle \bigwedge^k X \longrightarrow \textstyle \bigwedge^{k+2} X.
\end{equation*}
\end{enumerate}

Então existe uma decomposição em soma direta de fibrados vetoriais
\begin{equation*}
\textstyle \bigwedge^k X = \displaystyle \bigoplus_{j \geq 0} L^j(P^{k-2j}X),
\end{equation*}
onde $P^k = \ker \big(\Lambda:\bigwedge^k X \to \bigwedge^{k-2}X \big)$ é o fibrado de $k$-formas primitivas.

Além disso temos que $L^{n-k}:\textstyle \bigwedge^k X \to \textstyle \bigwedge^{2n-k} X$ é um isomorfismo para $k \leq n$.

\end{theorem} 

\begin{remark} \label{rmk:lefschetz-forms}
A decomposição de Lefschetz para o fibrado $\bigwedge^k X$ induz naturalmente uma decomposição nas seções globais
\begin{equation} \label{eq:lefschetz-forms}
\mathcal{A}^k (X) = \displaystyle \bigoplus_{j \geq 0} L^j(\mathcal{P}^{k-2j}(X)),
\end{equation}
onde $\mathcal{P}^k(X)$ é o espaço das $k$-formas que são primitivas em cada ponto (podendo se anular em alguns desses pontos).

Temos ainda que $L^{n-k}:\mathcal{A}^k(X) \to \mathcal{A}^{2n-k}(X)$ é um isomorfismo para todo $k\leq n$.
\end{remark}

\section{Métricas de Kähler}

Uma métrica de Kähler em uma variedade complexa é uma métrica hermitiana que satisfaz uma condição extra de compatibilade com a estrutura complexa. Existem diversas maneiras equivalentes de enunciarmos essa condição. 

\begin{proposition} \label{prop:kahler-equivalence}
Sejam $M$ uma variedade complexa, $J:TM \to TM$ a estrutura complexa induzida, $g$ uma métrica hermitiana em $M$ e $\omega$ a forma fundamental associada. São equivalentes
\begin{enumerate}
\item $J$ é paralela com relação a conexão de Levi-Civita, isto é, $\nabla J = 0$.
\item $\omega$ é fechada, isto é, $d \omega = 0$.
\item Para todo $x \in M$ existem coordendas holomorfas tais que, nessas coordenadas, $h_{ij}(x) = \delta_{ij}$ e $\frac{\partial h_{ij}}{\partial z_k} (x) = \frac{\del h_{ij}}{\del \bar{z}_k} (x)=0$, ou seja, a métrica $g$ é, a menos de termos de ordem maior ou igual a $2$, a métrica padrão de $\C^n$. Tais coordenadas são ditas normais em $x$.
\end{enumerate}
\end{proposition}

\begin{proof}
($ 1\Rightarrow 2$). Da fórmula intrínseca para a diferencial exterior temos
\begin{equation} \label{eq:d-omega}
 d \omega(X,Y,Z) = X \omega(Y,Z) - Y \omega(X,Z) + Z \omega(X,Y) - \omega([X,Y],Z) + \omega([X,Z],Y) - \omega([Y,Z],X).
\end{equation}

Do fato de $\nabla$ ser compatível com a métrica temos que
\begin{equation*}
X \omega(Y,Z) = X g(JY,Z) = g(\nabla_X JY,Z) + g(JY,\nabla_X Z),
\end{equation*}
e analogamente para o segundo e terceiro termos.

Como $J$ é paralela temos que $\nabla_X JY = J \nabla_X Y$. Usando o fato de $\nabla$ ser livre de torção ($\nabla_X Y - \nabla_Y X = [X,Y]$) e o fato de $g$ ser hermitiana, os três primeiros termos de (\ref{eq:d-omega}) ficam
\begin{equation*}
 \begin{split}
 &g(J \nabla_X Y,Z) + g(JY, \nabla_X Z) - g(J \nabla_Y X,Z) - g(JX, \nabla_Y Z) + g(J \nabla_Z X,Y) + g(J X, \nabla_Z Y) \\
 &= g(J[X,Y],Z) - g(J[X,Z],Y) + g(J[Y,Z],X)\\
 &= \omega([X,Y],Z) - \omega([X,Z],Y) + \omega([Y,Z],X),
 \end{split}
\end{equation*}
de onde vemos que $d \omega(X,Y,Z)=0$.\\

(2$\Rightarrow$ 1) Primeiramente note que para todo $X \in TM$ os endomorfismos  $J$ e $\nabla_X J$ anitcomutam. De fato,
\begin{equation*}
J(\nabla_XJ)Y +(\nabla_XJ)JY  = J[\nabla_X(JY)-J\nabla_X Y] + \nabla_X (JJY) - J\nabla_X(JY) = 0.
\end{equation*}
Além disso o operador $\nabla_XJ$ é anti-simétrico, o que pode ser visto derivando na direção $X$ a igualdade $g(JY,Z) + g(Y,JZ) = 0$  e usando a compatibilidade de $\nabla$ com a métrica.

Vamos precisar ainda do seguinte resultado
\begin{lemma}
Para todos $X,Y \in TM$ vale que $(\nabla_{JX}J)Y = J(\nabla_X J)Y$.
\end{lemma}
\begin{proof}
Considere o tensor $A(X,Y,Z)=g(J(\nabla_X J)Y - (\nabla_{JX}J)Y,Z)$. Nosso objetivo é mostrar que $A=0$.

O tensor de Nijenhuis associado a $J$ é definido por
\begin{equation*}
N^J(X,Y) = [X,Y] + J[JX,Y] + J[X,JY] - [JX,JY],~~X,Y \in TM.
\end{equation*}
Note que esse tensor é trivial, isto é, $N^J(X,Y) = 0$ para todos $X,Y \in TM$. Uma maneira de ver isso é calcular $N^J$ em uma base $\{\del \slash \del x_1,\ldots,\del \slash \del x_n,\del \slash \del y_1,\ldots,\del \slash \del y_n\}$ induzida por um sistema de coordenadas holomorfas e usar a expressão de $J$ em coordenadas (eq. \ref{eq:J-coord}).

Usando o fato que $\nabla$ é livre de torção obtemos uma expressão alternativa para $N^J$:
\begin{equation} \label{eq:nijenhuis}
N^J(X,Y) = [J(\nabla_X J)Y - (\nabla_{JX}J)Y] - [J(\nabla_Y J)X -(\nabla_{JY}J)X]. 
\end{equation}

A equação acima mostra que $A(X,Y,Z) = A(Y,X,Z)$ e do fato de $\nabla_XJ$ e $J$ serem operadores anti-siméticos que anticomutam temos que $A(X,Y,Z) = -A(X,Z,Y)$. Juntando esses dois fatos vemos que
\begin{equation*}
A(X,Y,Z) = -A(X,Z,Y) = -A(Z,X,Y) = A(Z,Y,X) = A(Y,Z,X) = -A(Y,X,Z) = -A(X,Y,Z)
\end{equation*}
e portanto $A=0$.
\end{proof}

Para mostrar que $\nabla J=0$ considere o tensor $B(X,Y,Z) = g((\nabla_XJ)Y,Z)$. Temos que mostrar que $B=0$.

Do fato de $J$ e $\nabla_X J$ anitcomutarem temos que $B(X,Y,JZ) = B(X,JY,Z)$ e do lema acima temos que $B(X,Y,JZ) + B(JX,Y,Z) = 0$. Juntando essas duas equações vemos também que $B(X,JY,Z) + B(JX,Y,Z) = 0$.

A fórmula (\ref{eq:d-omega}) pode ser reescrita como $d \omega (X,Y,Z)= B(X,Y,Z) + B(Y,Z,X) + B(Z,X,Y)$ e como estamos supondo $d \omega =0$ temos que
\begin{equation*}
\begin{split}
B(X,Y,JZ) + B(Y,JZ,X) + B(JZ,X,Y) &= 0\\
B(X,JY,Z) + B(JY,Z,X) + B(Z,X,JY) &=0.
\end{split}
\end{equation*}

Somando as duas equações e usando as relações obtidas acima vemos que $2B(X,Y,JZ) = 0$ para todos $X,Y,Z \in TM$ e portanto $B=0$.

(3 $\Rightarrow$ 2) Se $x \in M$ e $\{z_1,\ldots,z_n\}$ são coordenadas normais em $x$, temos usando a expressão (\ref{eq:fund-form-coord}) para $\omega$, que
\begin{equation*}
d \omega (x) = \frac{\smo}{2} \sum_{ijk} \left( \frac{\del h_{ij}}{\del z_k}(x) d z_k + \frac{\del h_{ij}}{\del \bar{z}_k}(x) d \bar{z}_k \right)\wedge dz_i \wedge d\bar{z}_j = 0.
\end{equation*}

Como $x \in M$ é arbitrário segue que $d\omega=0$.\\

(2 $\Rightarrow$ 3) Seja $x \in M$. Fazendo uma mudança linear de coordenadas podemos supor que $h_{ij}(x)=\delta_{ij}$, e portanto $\omega$ é dada por
\begin{equation*}
\omega = \frac{\smo}{2}\sum_{ijk} \left( \delta_{ij} + a_{ijk}z_k + a_{ij\bar{k}}\bar{z}_k) + O(2)\right) dz_i \wedge d\bar{z}_j,
\end{equation*}
onde $O(2)$ denota uma função tal que ela e suas derivadas primeiras se anulam em $x$.

Note que, como $h_{ij} = \overline{h_{ji}}$ temos que $a_{ij\bar{k}} = \overline{a_{ijk}}$ e a condição $d \omega (x) = 0$ implica que $a_{ijk}=a_{kji}$.

Vamos procurar uma mudança de coordenadas holmorfas da forma 
\begin{equation*}
z_k = w_k + \frac{1}{2} \sum_{lm} b_{klm}w_l w_m
\end{equation*}
com $b_{klm}=b_{kml}$ constantes.

Como
\begin{equation*}
dz_k = d w_k + \sum_{lm} b_{klm} w_l d w_m,
\end{equation*}
temos que, a menos de termos de ordem maior ou igual a $2$,
\begin{equation*}
\begin{split}
\frac{2}{\smo} \omega &= \sum  \big( dw_i + \sum b_{ilm} w_l dw_m  \big)  \wedge \sum \big( d\bar{w}_i + \sum \overline{b_{ipq}} \bar{w}_p d\bar{w}_q\big) + \sum \big( a_{ijk} w_k + a_{ij\bar{k}} \bar{w}_k\big)dw_i \wedge d\bar{w}_j\\
& = \sum \left( \delta_{ij} + \sum_{k}\left( a_{ijk}w_k + a_{ij\bar{k}}\bar{w}_k + b_{jki} w_k + \overline{b_{ikj}}\bar{w}_k \right) \right) dw_i \wedge d\bar{w}_j.
\end{split}
\end{equation*}

Defina então $b_{jki}=-a_{ijk}$. Note que essa definição é compatível com a escolha $b_{jki}=b_{jik}$, pois
\begin{equation*}
b_{jki} = -a_{ijk} = -a_{kji} = b_{jik},
\end{equation*}
e além disso temos
\begin{equation*}
\overline{b_{ikj}}=-\overline{a_{jik}} = -a_{ij\bar{k}},
\end{equation*}
de onde vemos que, a menos de termos de ordem maior ou igual a $2$, $h_{ij}(x)=\delta_{ij}$, provando $(3)$.
\end{proof}

\begin{remark} \label{rmk:nijenhuis}
O tensor de Nijenhuis $N^J$ definido acima (eq. \ref{eq:nijenhuis}) mede a integrabilidade do endomorfismo $J:TM \to TM$. Uma \textit{estrutura quase complexa} em uma variedade diferenciável $M$ é um endomorfismo $J:TM\to TM$ satisfazendo $J^2=-\id$. Dizemos que $J$ é integrável se é induzida por uma estrutura complexa em $M$.

Na demonstração acima vimos que se $J$ é integrável então $N^J = 0$. O celebrado Teorema de Newlander-Niremberg (veja \cite{newlander-nierenberg}) diz que vale a recíproca: se $J$ é uma estrutura quase complexa em $M$ e $N^J = 0$ então existe uma única estrutura complexa em $M$ que induz $J$. Não é difícil ver que a condição $N^J=0$ é equivalente a $[T^{1,0}X,T^{1,0}X] \subset T^{1,0}X$ e portanto $J$ é integrável se e somente se $T^{1,0}X$ é uma distribuição involutiva.\\ 
\end{remark}

\begin{definition} \index{métrica!de Kähler} \index{variedade!de Kähler}
Dizemos que uma métrica hermitiana é uma \textbf{métrica de Kähler} se uma das condições da proposição \ref{prop:kahler-equivalence} é satisfeita. Uma variedade $X$ é uma \textbf{variedade de Kähler} se $X$ admite uma métrica de Kähler.
\end{definition}

\begin{example}
A métrica padrão em $\C^n$ é uma métrica de Kähler, pois a forma fudamental $\omega_0 = \frac{\smo}{2} \sum dz_i \wedge d\bar{z}_i$ é claramente fechada. Mais ainda, $\omega_0$ é exata, pois $\omega_0 = \frac{\smo}{2} d(\sum_i z_i d\bar{z}_i)$
\end{example}

\begin{example} \textbf{Superfícies de Riemann.} Uma superfície de Riemann é uma variedade complexa de dimensão 1. Se $X$ é uma superfície de Riemann, então $X$ é uma variedade diferenciável de dimensão $2$. Qualquer que seja a métrica hermitiana em $X$, $d \omega$ é uma 3-forma e portanto $d \omega = 0$. Logo toda métrica hermitiana em $X$ é de Kähler e portanto toda superfície de Riemann é uma variedade de Kähler.
\end{example}

Já em dimensão 2 veremos que existem exemplos de variedades complexas que não são de Kähler (veja o exemplo \ref{ex:hopf}).

\begin{example} \textbf{Toros complexos.} No exemplo \ref{ex:complex-tori} definimos um toro complexo como sendo um quociente $X= \C^n/L$ de $\C^n$ por um subgrupo da forma $L = \left \lbrace \alpha = \sum_i n_i \alpha_i : n_i \in \Z , i=1,\cdots,2n \right \rbrace$ onde $\alpha_1,\cdots,\alpha_{2n} \in \C^n$ são linearmente independentes sobre $\R$. Equivalentemente podemos ver $X$ como sendo o quociente de $\C^n$ pelo grupo gerado pelas $2n$ translações $z \mapsto z + \alpha_j$.

Vimos também que a projeção natural $\pi: \C^n \to X$ é um recobrimento e que $X$ admite uma única estrutura complexa que torna $\pi$ holomorfa.

Além disso, as translações são isometrias de $(\C^n,h_0)$ e portanto induzem uma métrica $h$ em $X$ e como $h_0$ é de Kähler, $h$ também é.

Mais precisamente: como $\pi$ é um biholomorfismo local, podemos definir uma métrica em $X$ localmente. Em um aberto $U = \pi(V)$ tal que $\pi: V \to U$ é biholomorfa defina
\begin{equation*}
h|_U = \varphi^* h_0, \text{ onde } \varphi = (\pi|_V)^{-1}:U \to V.
\end{equation*}

Em outro aberto $U'=\pi(V')$ com $U \cap U' \neq \emptyset$ temos que $h|_{U'} = \psi^* h_0$, onde $\psi = (\pi|_{V'})^{-1}$. Agora, para $z \in V' \cap \pi^{-1}(U)$ temos $\varphi \circ \psi^{-1} (z) = z + \alpha$ para algum $\alpha \in L$, que é uma isometria, e portanto $(\varphi \circ \psi^{-1})^*h_0 = h_0$ em $V' \cap \pi^{-1}(U)$.

Logo, em $U \cap U'$ temos
\begin{equation*}
 \varphi^* h_0 = \psi^* (\varphi \circ \psi^{-1})^*h_0 =\psi^* h_0,
\end{equation*}
e portanto podemos definir $h$ globalmente em $X$.

A métrica $h$ é hermitiana, pois 
\begin{equation*}
\begin{split}
h_{\pi(p)}(J_{\pi(p)} d\pi_p u,J_{\pi(p)} d \pi_p v) &= h_{\pi(p)}(d\pi_p J_0 u,d\pi_p J_0 v) \\
 &= (h_0)_p (J_0 u, J_0 v)\\
 &= (h_0)_p (u,v) \\
 &= h_{\pi(p)}( d\pi_p u, d \pi_p v),
\end{split}
\end{equation*}
onde usamos que $\pi$ é holomorfa e uma isometria local.

A forma fundamental de $h$ é
\begin{equation*}
\omega_{\pi(p)}(d\pi_p u, d \pi_p v) = h_{\pi(p)}(J_{\pi(p)} d\pi_p u, d \pi_p v) = h_{\pi(p)}(d\pi_p J_0 u,d\pi_p v) = (h_0)_p (J_0 u, v) = \omega_0 (u,v),
\end{equation*}
ou seja, $\pi^* \omega = \omega_0$ é a forma simplética padrão. Logo temos que $\pi^*(d \omega) = d (\pi^* \omega) = d \omega_0 = 0$ e como $\pi^*$ é injetora segue que $d \omega = 0$, ou seja, $h$ é uma métrica de Kähler em $X$.\\

Mostramos portanto que todo toro complexo $X = \C^n/L$ é uma variedade de Kähler.\\
\end{example}

Decidir se uma variedade $X$ é ou não de Kähler em geral não é uma tarefa fácil. No entanto, a existência de uma métrica de Kähler em $X$ impõe fortes condições sobre sua topologia e sua estrutura complexa, o que acarretam em algumas condições necessárias para a existencia de tais métricas. Deduziremos algumas delas ao longo do trabalho.

Como consequência da condição $d \omega = 0$ obtemos uma primeira obstrução.
\begin{proposition} \label{prop:cohom-kahler}
Seja $X$ uma variedade complexa compacta de dimensão $n$. Se $X$ admite uma métrica de Kähler então $H^{2k}(X,\R) \neq 0$ para $k=0,\cdots,n$.
\end{proposition}

\begin{proof}
Como $d \omega = 0$ a 2-forma $\omega$ define uma classe $[\omega] = H^2(X,\R)$. Essa classe é não nula, pois se $\omega$ fosse exata, $\omega=d \eta$, sua n-ésima potência também seria exata, $\omega^n = d(\eta \wedge \omega^{n-1})$. Mas isso não é possível pois, pelo teorema de Stokes teríamos,
\begin{equation*}
\int_X \omega^n  = \int_X d(\eta \wedge \omega^{n-1}) = 0,
\end{equation*}
enquanto que, pela proposição \ref{prop:volume-form}, temos
\begin{equation*}
\int_X \omega^n = n! \int_X \text{vol}_X = n! \text{ vol}(X).
\end{equation*}

Isso mostra que $H^2(X,\R) \neq 0$.

Agora, as formas $\omega^k,k=1,\cdots,n$ também são fechadas e nenhuma é exata, pois $\omega^k \wedge \omega^{n-k} = \omega^n = n! \text{ vol}_X $, e portanto definem elementos não nulos $[\omega^k] \in H^{2k}(X,\R)$.
\end{proof}

\begin{remark}
Do resultado acima vemos que a única esfera que pode admitir uma estrutura kähleriana é $S^2 \simeq \pr^1$ e  veremos mais adiante que essa estrutura de fato existe. Na verdade, como mencionado no capítulo \ref{ch:cap0}, as esferas $S^{2n}$ não admitem sequer um estrutura complexa para $n>3$ e $n=2$.
\end{remark}

A proposição \ref{prop:cohom-kahler} nos permite construir exemplos de variedade complexas que não são de Kähler.
\begin{example} \label{ex:hopf} \textbf{Variedades de Hopf} \index{variedade!de Hopf}

Seja $\alpha$ um número complexo com $0<|\alpha|<1$. Temos uma ação de $\Z$ em $\C^n \setminus \{0\}$ definida por
\begin{equation*}
 \begin{split}
 \Z \times \C^n \setminus \{0\} &\longrightarrow \C^n \setminus \{0\} \\
 (k,z) &\longmapsto \alpha^k z
 \end{split}
\end{equation*}
que é livre e propriamente descontínua. Obtemos assim uma variedade
\begin{equation*}
X^n_{\alpha} = \frac{\C^n \setminus \{0\}}{\Z},
\end{equation*}
chamada variedade de Hopf.

Note que, como $\Z$ age em $\C^n \setminus \{0\}$ por biholomorfismos, existe uma única estrutura complexa em $X^n_{\alpha}$ que torna a projeção $\pi: \C^n \setminus \{0\} \to X^n_{\alpha}$ holomorfa (veja o exemplo \ref{ex:coverings}).

Todo elemento de $\C^n \setminus \{0\}$ se escreve de modo único como $z = r \xi$, com $r > 0$ e $\xi \in S^{2n-1}$. Isso dá um isomorfismo $\C^n \setminus \{0\} \simeq \R^{>0} \times S^{2n-1}$. A ação de $\Z$ segundo esse isomorfismoo se traduz em $
k \cdot (r,\xi) = (\lambda^k r,\xi)$ e portanto é trivial no segundo fator. É facil ver que no primeiro fator o quociente $\R^{>0}/\Z$ é difeomorfo a $S^1$, de onde conluimos que $X^n_{\alpha}$ é difeomorfa a $S^1 \times S^{2n-1}$. 

Em particular, para um dado $n$, as $X^n_{\alpha}$ são todas difeomorfas. No entanto, a estrutura complexa de $X^n_{\alpha}$ depende do parâmetro $\alpha$.

Do teorema de Künneth temos que $H^*(X^n_{\alpha},\R) = H^*(S^1,\R) \otimes H^*(S^{2n-1},\R)$ e em particular
\begin{equation*}
 H^k(X^n_{\alpha},\R) = \bigoplus_{p+q = k} H^p(S^1,\R) \otimes H^q(S^{2n-1},\R).
\end{equation*}

Para $k=2$ obtemos
\begin{equation*}
\begin{split}
 H^2(X^n_{\alpha},\R) &= [H^0(S^1,\R) \otimes H^2(S^{2n-1},\R)] \oplus [H^1(S^1,\R) \otimes H^1(S^{2n-1},\R)] \\
 &\simeq  H^2(S^{2n-1},\R) \oplus H^1(S^{2n-1},\R),
\end{split}
\end{equation*}
de onde concluimos que $H^2(X^n_{\alpha},\R) = 0$ para $n\geq 2$.

Assim, pela proposição \ref{prop:cohom-kahler}, concluimos que nenhuma variedade de Hopf de dimensão maior ou igual a 2 é uma variedade de Kähler.

\begin{remark}
Embora não admitam métricas de Kähler, as variedades de Hopf podem ser munidas de uma classe importante de métricas hermitianas, as chamadas métricas localmente conformes de Kähler (LCK-metrics). Tais métricas podem ser caracterizadas pela equação $d \omega = \eta \wedge \omega$, onde $\eta$ é uma 1-forma fechada, chamada forma de Lee. Veja \cite{dragomir-ornea} para um estudo detalhado de tais métricas.
\end{remark}
\end{example}

Note que, dada uma métrica hermitiana $g$, podemos recuperá-la a partir da sua forma fundamental, pois $g(u,v) = g(Ju,Jv) = \omega(u,Jv)$. Podemos então partir de uma $(1,1)$-forma real e tentar definir uma métrica, fazendo $g = \omega(\cdot,J\: \cdot)$. Dessa maneira obtemos uma forma bilinear compatível com $J$, mas não sabemos se é positiva definida. Isso leva a seguinte definição
\begin{definition} \label{def:positive-form}
Uma $(1,1)$-forma real $\alpha$ é positiva se $\alpha(v,Jv)>0$ para todo $v \in TX$ não nulo, ou equivalentemente, se $\alpha$ é da forma
\begin{equation*}
\alpha = \frac{\smo}{2} \sum_{ij}  h_{ij} \; dz_i \wedge d\bar{z}_j,
\end{equation*}
onde $(h_{ij})$ é uma matriz hermitiana positiva definida.
\end{definition}

\begin{remark} \label{rmk:positive-form}
Note que, usando o isomorfismo $TX \simeq T^{1,0}X$, $v \mapsto \frac{1}{2}(v - \smo Jv)$, a condição $\alpha(v,Jv)>0$ para todo $v \in TX$ não nulo é equivalente a condição $-\smo \alpha(u,\bar u) > 0$ para todo $u \in T^{1,0}X$ não nulo, pois $J$ age em $T^{1,0}X$ pela multiplicação por $\smo$.
\end{remark}

Temos o seguinte resultado, cuja demonstração é imediata.
\begin{proposition} \label{prop:positive-form}
Dada uma $(1,1)$-forma real e positiva $\omega$ em uma variedade complexa $X$, $g = \omega(\cdot,J\: \cdot)$ define uma métrica hermitiana em $X$ cuja forma fundamental associada é $\omega$. Além disso, se $\omega$ for fechada, então a métrica $g$ é de Kähler.
\end{proposition}

\begin{example} \label{ex:fubini-study} \index{métrica!de Fubini-Study} \textbf{A métrica de Fubini-Study em $\pr ^n$}\\
Seja $\pi:\C^{n+1}\setminus \{0\} \to \pr ^n$ a projeção natural e $Z:U \to \C^{n+1} \setminus \{0\}$ uma aplicação holomorfa definida em um aberto $U \subset \pr^n$ tal que $\pi \circ Z = \text{id}_U$ (ou seja, uma seção sobre $U$ do fibrado tautológico $\mathcal{O}(-1)$).

Defina uma $(1,1)$-forma em $U$ por
\begin{equation*}
\omega = \frac{\smo}{2 \pi} \del \delbar \log ||Z||^2
\end{equation*}

Se $W:U' \to  \C^{n+1} \setminus \{0\}$ é outra seção então $W = \lambda Z$ para uma função holomorfa $\lambda:U\cap U' \to \C^*$ (pois $Z(p)$ e $W(p)$ pertencem a mesma reta complexa). E portanto
\begin{equation*}
\del \delbar \log ||W||^2 = \del \delbar \log ||\lambda Z||^2 =  \del \delbar \log \lambda +  \del \delbar \log \bar{\lambda} + \del \delbar \log ||Z||^2 = \del \delbar \log ||Z||^2,
\end{equation*}
onde, para fazer o cálculo acima usamos algum ramo do logaritmo e o fato de $\log \lambda$ ser holomorfa e $\log \bar{\lambda}$ anti-holomorfa.\\

Como tais seções sempre existem localmente, obtemos uma $(1,1)$-forma globalmente definida em $\pr ^n$, chamada \textbf{forma de Fubini-Study}, denotada por $\omega_{FS}$.

Note que se $U_i = \{z_i \neq 0 \} \subset \pr ^n$ e $Z_i = \left(\frac{z_0}{z_i},\cdots,1,\cdots\frac{z_n}{z_i} \right)$ então,
\begin{equation} \label{eq:omegaFS-local}
\omega_{FS} = \frac{\smo}{2\pi} \del \delbar \log \bigg( 1 + \sum_{j \neq i} \left| \frac{z_j}{z_i} \right|^2 \bigg),\; \text{ em } U_i.
\end{equation}

Note que $\omega_{FS}$ é uma $(1,1)$-forma fechada. Além disso, como $\del \delbar = - \delbar\del$, segue que $\overline{\omega_{FS}} = \omega_{FS}$ e portanto $\omega_{FS}$ é real. Logo, para ver que $\omega_{FS}$ define uma métrica de Kähler em $\pr ^n$, só resta verificar a positividade.

Para isso, note primeiramente que $\pi^*\omega_{FS} = \widetilde{\omega}$, onde $\widetilde{\omega}$ é a $(1,1)$-forma em $\C^{n+1}\setminus \{0\}$ dada por
\begin{equation} \label{eq:pullback-fubini-study}
\widetilde{\omega} = \frac{\smo}{2\pi} \del \delbar \log ||z||^2
\end{equation}

Seja agora $A \in U(n+1)$. Como $A:\C^{n+1}\to \C^{n+1}$ é holomorfa, $A^*$ comuta com $\del$ e $\delbar$ e portanto temos que
\begin{equation*}
A^* \widetilde{\omega} = \frac{\smo}{2\pi} \del \delbar A^*(\log ||z||^2) = \frac{\smo}{2\pi} \del \delbar \log ||Az||^2 = \frac{\smo}{2\pi} \del \delbar \log ||z||^2 = \widetilde{\omega},
\end{equation*}
ou seja, $\widetilde{\omega}$ é $U(n+1)$-invariante.

Seja agora $F_A:\pr^n \to \pr^n$ a aplicação induzida por $A$, isto é, $F_A [z] = [Az]$. Temos então que $\pi \circ A = F_A \circ \pi$ e portanto $\pi^* F_A^* \omega_{FS} = A^* \pi^* \omega_{FS} = A^* \widetilde{\omega} = \widetilde{\omega} = \pi^* \omega_{FS}$ e como $\pi^*$ é injetora segue que $F_A^* \omega_{FS} = \omega_{FS}$, ou seja, $\omega_{FS}$ é $U(n+1)$-invariante.

Como $U(n+1)$ age transitivamente em $\pr^n$, para verificarmos que $\omega_{FS}$ é positiva, basta fazê-lo num ponto. Da expressão (\ref{eq:omegaFS-local}) temos que, nas coordenadas $w_i = z_i/z_0$ em $U_0$ temos
\begin{equation*}
\begin{split}
\omega_{FS} &= \frac{\smo}{2\pi} \del \delbar \log \bigg( 1 + \sum_{i=1}^n w_i \bar{w}_i \bigg) =  \frac{\smo}{2\pi} \del \left( \frac{\sum_i w_i d\bar{w}_i}{1 + \sum_i w_i \bar{w}_i} \right)\\
&= \frac{\smo}{2\pi} \left[ \frac{\sum_i dw_i \wedge d\bar{w}_i}{1 + \sum_i w_i \bar{w}_i} + \frac{\left( \sum_i w_i d\bar{w}_i\right) \wedge \left( \sum_i \bar{w}_i dw_i\right)}{\left( 1 + \sum_i w_i \bar{w}_i\right)^2}\right].
\end{split}
\end{equation*}

Calculando no ponto $p=[1:0:\cdots:0]$, cujas coordenadas são $w_i=0$, vemos que
\begin{equation*}
(\omega_{FS})_p = \frac{\smo}{2\pi} \sum_i  \; dw_i \wedge d \bar{w}_i,
\end{equation*}
que é claramente positiva (é um múltiplo da forma simplética padrão nas coordenadas $w_i,\bar{w}_i$).

Logo $\omega_{FS}$ é positiva em toda parte e portanto $g_{FS} = \omega_{FS}(\cdot,J\cdot)$ define uma métrica em $\pr^n$, chamada \textbf{métrica de Fubini-Study}. O produto hermitiano associado será denotado por $h_{FS}$.

Há também uma maneira geométrica de definir a métrica $g_{FS}$, exigindo que a restrição de $\pi:\C^{n+1}\setminus \{0\} \to \pr^n$ a $S^{2n+1}$ seja uma submersão riemanniana. Os detalhes dessa construção são dados na seção \ref{sec:proj-var}.
\end{example}

Mais exemplos de variedade de Kähler podem ser obtidos de subvariedades complexas de variedades de Kähler. Esse fato, enunciado na proposição abaixo, juntamente com o exemplo acima oferecerá uma classe importante de exemplos de variedades de Kähler, que são as chamadas variedades projetivas (veja a seção \ref{sec:proj-var}). 
\begin{proposition} \label{prop:kahler-submanifold}
Se $Y \subset X$ é uma variedade complexa de uma variedade de Kähler $X$ então $Y$ também é de Kähler.
\end{proposition}
\begin{proof}
Denote por $\iota:Y \to X$ a inclusão. Seja $g$ uma métrica de Kähler em $X$ e $\omega$ sua forma fundamental.

A restrição $\iota^* g$ é uma métrica riemanniana em $X$. Usando o fato que $\iota$ é holomorfa é fácil ver que $\iota^* g$ é hermitiana e que sua forma fundamental é $\iota^* \omega$. Como $g$ é de Kähler temos que $d (\iota^* \omega) = \iota^*(d \omega) = 0$ e portanto $\iota^*g$ é uma métrica de Kähler em $Y$.
\end{proof}

\subsection{Identidades de Kähler} \label{sec:kahler-identities} \index{identidades de Kähler}
Em uma variedade Riemanniana $(X,g)$  orientada de dimensão $m$, considere os espaços $\mathcal{A}_c^k(X)$ de formas diferenciais com suporte compacto. Definimos o operador
\begin{equation*}
d^* = (-1)^{m(k+1)+1} * \circ d \circ *: \mathcal{A}_c^k(X) \longrightarrow \mathcal{A}_c^{k-1}(X),
\end{equation*}
que é o adjunto formal de $d:\mathcal{A}_c^k(X) \longrightarrow \mathcal{A}_c^{k-1}(X)$ com a métrica $L^2$ por $g$ (mais detalhes na seção \ref{sec:hodge-decomp}. Note que no caso em que $X$ é hermitiana, $m=2n$ e portanto $d^* = - * \circ d \circ *$.

O \textit{operador de Laplace} é então definido por
\begin{equation*}
\Delta = d d^* + d^* d:  \mathcal{A}_c^k(X) \longrightarrow \mathcal{A}_c^k(X)
\end{equation*}

Esses operadores se estendem de forma $\C$-linear a $\mathcal{A}^k_{\C,c}(X)$.
\begin{definition}
Se $(X,g)$ é uma variedade hermitiana definimos os operadores
\begin{equation*}
\del^* = - * \circ \delbar \circ * : \mathcal{A}_c^{p,q}(X) \longrightarrow \mathcal{A}_c^{p-1,q}(X)
\end{equation*}
e
\begin{equation*}
\delbar^* = - * \circ \del \circ * : \mathcal{A}_c^{p,q}(X) \longrightarrow \mathcal{A}_c^{p,q-1}(X),
\end{equation*}
e os Laplacianos associados $\Delta_{\del}, \Delta_{\delbar}: \mathcal{A}_c^{p,q}(X) \to \mathcal{A}_c^{p,q}(X)$, dados por
\begin{equation*}
\Delta_{\del} = \del \del^* + \del^* \del \;\; \text{ e } \;\;\ \Delta_{\delbar} = \delbar \delbar^* + \delbar^* \delbar.
\end{equation*}
\end{definition}

\begin{remark}
Nessa definição consideramos a extensão $\C$-linear do operador $*$ para formas complexas. Essa extensão satisfaz $*\big( \bigwedge^{p,q} X \big ) \subset \bigwedge^{n-q,n-p} X$, de modo que de fato $\del^*$ leva $\mathcal{A}^{p,q}(X)$ em  $\mathcal{A}^{p-1,q}(X)$ e $\delbar^*$ leva $\mathcal{A}^{p,q}(X)$ em $\mathcal{A}^{p,q-1}(X)$.
\end{remark}

Os operadores $\del^*$ e $\delbar^*$ são adjuntos formais de $\del$ e $\delbar$ com relação ao produto hermitiano $L^2$ definido por
\begin{equation} \label{eq:L^2-compact-support}
(\alpha,\beta) = \int_X ( \alpha,\beta ) \text{vol} ~,\;\;\; \alpha,\beta \in \mathcal{A}_{\C,c}^k(X).
\end{equation}
Para mais detalhes e uma demonstração desse fato consulte a seção \ref{sec:hodge-decomp}.\\

Em geral, não há relação entre  os três Laplacianos definidos acima. Porém, se a métrica $g$ é de Kähler, eles coincidem a menos de uma constante multiplicativa. Esse fato é uma consequência das chamadas identidades de Kähler\footnote{Alguns autores as designam por identidades de Hodge.}, relacionando os operadores $\del, \delbar,\del^*, \delbar^*$ com os operadores de Lefschetz $L$ e $\Lambda$ (veja a seção \ref{subsec:lefschetz} para a definição desses operadores).\\
\begin{theorem} \textbf{Identidades de Kähler.} \label{thm:kahler-identities} Se $X$ é uma variedade complexa com uma métrica de Kähler então valem as seguintes relações de comutação
\begin{itemize}
\item[1.] $[L,\del] = [L,\delbar] = 0$ ~e~~ $[\Lambda,\del^*] = [\Lambda,\delbar^*] = 0$
\item[2.] $[\Lambda,\delbar] = - \smo \del^*$ ~e~~ $[\Lambda,\del] = \smo \delbar^*$,
\item[3.] $[\delbar^*,L] = \smo \del$ ~e~~ $[\del^*,L] = -\smo \delbar$.

\end{itemize}
\end{theorem}

\begin{proof} 1. Como a métrica é de Kähler temos que $d \omega= 0$ e portanto $\del \omega = \delbar \omega = 0$. Assim, para $\alpha \in \mathcal{A}^*_{\C}(X)$ temos que
\begin{equation*}
\del (L (\alpha)) = \del (\omega \wedge \alpha) = \omega \wedge \del \alpha = L (\del (\alpha))~\text{ e }~ \delbar (L (\alpha)) = \delbar (\omega \wedge \alpha) = \omega \wedge \delbar \alpha = L (\delbar (\alpha))
\end{equation*}
o que mostra que $[L,\del] = [L,\delbar] = 0$. Tomando o adjunto obtemos $0 = [L,\del]^* = - [\Lambda,\del^*]$  e do mesmo modo mostramos que $[\Lambda,\delbar^*]=0$.\\

2. Vamos demonstrar primeiro o resultado para $X=\C^n$ com a métrica hermitiana padrão.

Defina os operadores
\begin{equation*}
\begin{split}
e_k:\mathcal{A}^{p,q}_c (\C^n) &\longrightarrow \mathcal{A}^{p+1,q}_c (\C^n)\\
\alpha &\longmapsto dz_k \wedge \alpha
\end{split}
~~~~~e~~~~~
\begin{split}
\bar{e}_k:\mathcal{A}^{p,q}_c (\C^n) &\longrightarrow \mathcal{A}^{p,q+1}_c (\C^n)\\
\alpha &\longmapsto d\bar{z}_k \wedge \alpha
\end{split}
\end{equation*}
para $k=1,\ldots,n$ e sejam $i_k$ e $\bar{i}_k$ os respectivos adjuntos pontuais.

\begin{lemma} \label{lemma:ki1}
Os operadores definidos acima satisfazem
\begin{itemize}
\item[a.] $e_k i_l + i_l e_k = 0$ se $k \neq l$,
\item[b.] $e_k i_k + i_k e_k = 2$ para todo $k$
\item[c.] $e_k \bar i_l + \bar i_l e_k =0$ para todo $k,l$.
\end{itemize}
\end{lemma}
\begin{proof}
Note que todos os operadores são $\mathcal{C}^{\infty}(\C^n)$-lineares, e portanto basta mostrarmos as identidades acima calculando os operadores nas formas $dz_I \wedge d \bar z _J$, o que segue de cálculos elementares (veja por exemplo \cite{g-h} p.112).
\end{proof}

Defina agora
\begin{equation*}
\begin{split}
\del_k:\mathcal{A}^{p,q}_c (\C^n) &\longrightarrow \mathcal{A}^{p,q}_c (\C^n)\\
\del_k \left (\sum f_{IJ} dz_I \wedge d \bar{z}_J \right) &= \sum \frac{\del f_{IJ}}{\del z_k} dz_I \wedge d \bar{z}_J
\end{split}
~~~~~e~~~~~
\begin{split}
\delbar_k:\mathcal{A}^{p,q}_c (\C^n) &\longrightarrow \mathcal{A}^{p,q}_c (\C^n)\\
\delbar_k \left (\sum f_{IJ} dz_I \wedge d \bar{z}_J \right) &= \sum \frac{\del f_{IJ}}{\del \bar{z}_k} dz_I \wedge d \bar{z}_J
\end{split}\end{equation*}
Note que $\del_k$ e $\delbar_k$ comutam e ambos comutam com $e_l$ e $\bar e_l$. Usando integração por partes é fácil ver que.

\begin{lemma}
O adjunto de $\del_k$ com respeito à métrica $L^2$ (eq.~\ref{eq:L^2-compact-support}) é $-\delbar_k$ e o adjunto de $\delbar_k$ é $-\del_k$.
\end{lemma}

Como consequência do lema acima temos que $[\del_k,i_l] = -[\delbar_k^*,e_l^*] = [\delbar_k,e_l]^* = 0$ e analogamente $[\delbar_k,i_l] = [\del_k,\bar i_l] = [\delbar_k,\bar i_l] =0$.

Como $\del = \sum_k \del_k e_k$ e $\delbar = \sum_k \delbar_k e_k$  temos, tomando o adjunto, que
\begin{equation*}
\del^* = \sum_k (\del_k e_k)^* =  \sum_k (e_k \del_k)^* = \sum_k \del_k^* e_k^* = - \sum_k \delbar_k i_k ~~\text{ e }~~ \delbar^* = - \sum_k \del_k \bar{i}_k.
\end{equation*}

A forma de Kähler em $\C^n$ é dada por $\omega = \frac{\smo}{2} \sum_k dz_k \wedge d\bar{z}_k$ de modo que o operador de Lefschetz se escrever como $L = \frac{\smo}{2} \sum_k e_k \bar{e}_k$ e portanto, tomando o adjunto obtemos uma expressão para o operador dual
\begin{equation*}
\Lambda = -\frac{\smo}{2} \sum_k \bar{i}_k i_k.
\end{equation*}

Temos portanto que
\begin{equation*}
\Lambda \del = - \frac{\smo}{2} \sum_{k,l} \bar{i}_k i_k \del_l e_l = -\frac{\smo}{2} \sum_{k,l} \del_l \bar{i}_k i_k e_l = -\frac{\smo}{2} \left( \sum_k \del_k \bar{i}_k i_k e_k +  \sum_{k\neq l} \del_l \bar{i}_k i_k e_l \right).
\end{equation*}

Usando o lema \ref{lemma:ki1} temos que o primeiro termo da soma acima é
\begin{equation*}
-\frac{\smo}{2} \sum_k \del_k \bar{i}_k i_k e_k = -\frac{\smo}{2} \sum_k \del_k \bar{i}_k (- e_k i_k + 2) = -\frac{\smo}{2} \sum_k \del_k e_k \bar{i}_k i_k - \smo \sum_k \del_k \bar i_k,
\end{equation*}
e o segundo fica
\begin{equation*}
-\frac{\smo}{2} \sum_{k\neq l} \del_l \bar{i}_k i_k e_l = \frac{\smo}{2} \sum_{k\neq l} \del_l \bar{i}_k e_l i_k = -\frac{\smo}{2} \sum_{k \neq l} \del_l e_l \bar{i}_k  i_k\
\end{equation*}
de onde concluimos que
\begin{equation*}
\Lambda \del = -\frac{\smo}{2} \sum_{k,l} \del_l e_l \bar i_k i_k - \smo \sum_k \del_k \bar i_k = \del \Lambda + \smo \delbar^*,
\end{equation*}
ou seja, $[\Lambda,\del] = \Lambda \del - \del \Lambda = \smo \delbar^*$.\\

Seja agora $X$ uma variedade com uma métrica de Kähler $h$. Fixe $x \in X$ e considere coordenadas normais holomorfas centradas em $x$ definidas em um aberto $U$ (cf. item 3. Proposição da \ref{prop:kahler-equivalence}). Sejam $h_{kl}$ os coeficientes da métrica nesse sistema de coodenadas.

Como acima podemos definir os operadores $e_k:\mathcal{A}^{p,q}_c(U) \to \mathcal{A}^{p+1,q}_c (U)$, $\bar e_k:\mathcal{A}^{p,q}_c(U) \to \mathcal{A}^{p,q+1}_c (U)$, seus adjuntos $i_k$, $\bar i_k$ e os operadores $\del_k,\delbar_k:\mathcal{A}^{p,q}_c(U) \to \mathcal{A}^{p,q}_c (U)$. Assim, o operador de Lefschetz é dado por $L = \frac{\smo}{2} \sum_{k,l} h_{kl} e_k \bar e_l$ e o seu dual é dado por $\Lambda = - \frac{\smo}{2} \sum_{k,l} h_{kl} i_k \bar i_l$.

Note que os operadores $e_k,\bar e_k, i_k$ e $\bar i_k$ são algébricos, isto é, não envolvem nenhuma derivada parcial dos coeficientes das formas diferenciais e os operadores $\del_k$ e $\delbar_k$ são de primeira ordem, isto é, envolvem derivadas parciais de ordem $1$. Sendo assim, como $h_{kl}(x) = \delta_{kl}$ e $\frac{\partial h_{kl}}{\partial z_i} (x) = \frac{\del h_{kl}}{\del \bar{z}_j} (x)=0$, as fórmulas obtidas acima para $\Lambda$,$\del$ e $\delbar^*$ são válidas em $x$ e portanto temos que $[\Lambda,\del] (x)= \smo \delbar^* (x)$.

Como coordenadas normais holomorfas existem em torno de cada ponto de $X$ concluímos que $[\Lambda,\del] = \smo \delbar^*$.\\

Conjugando a relação acima e lembrando que $\Lambda$ é um operador real obtemos
\begin{equation*}
-\smo \del^* = \overline{\smo \delbar^*} = \overline{[\Lambda,\del]} = [\overline{\Lambda},\delbar] = [\Lambda,\delbar],
\end{equation*}
terminando a demonstração de 2.\\

3. Tomando o adjunto da relação $[\Lambda,\delbar] = -\smo \del^*$ obtemos
\begin{equation*}
\smo \del = (-\smo \del^*)^* = [\Lambda,\delbar]^* =-[\Lambda^*,\delbar^*] = -[L,\delbar^*] = [\delbar^*,L],
\end{equation*}
o que demonstra a primeira relação. A relação $[\del^*,L] = -\smo \delbar$ é obtida por conjugação e usando que $L$ é um operador real.

\end{proof}

\begin{corollary} \label{cor:laplacian}
Em uma variedade de Kähler os diferentes Laplacianos são relacionados por
\begin{equation*}
\Delta_{\del} = \Delta_{\delbar} = \frac{1}{2} \Delta.
\end{equation*}
\end{corollary}

\begin{proof}
Primeiramente note que $\del \delbar^* + \delbar^* \del = 0$, pois como $[\Lambda,\del] = \smo\delbar^*$ temos que
\begin{equation*}
\smo (\del \delbar^* + \delbar^* \del) = \del [\Lambda,\del] + [\Lambda,\del] \del = \del \Lambda \del - \del \Lambda \del =0,
\end{equation*}
e conjugando obtemos $\delbar \del^* + \del^* \delbar=0$.

Temos portanto que
\begin{equation*}
\begin{split}
\Delta &= (\del + \delbar)(\del^* + \delbar^*) +  (\del^* + \delbar^*)(\del + \delbar) \\
 &= (\del \del^* + \del^* \del) + (\delbar \delbar^* + \delbar^* \delbar) + (\del \delbar^* + \delbar \del^* + \del^* \delbar + \delbar^* \del) \\
 &= (\del \del^* + \del^* \del) + (\delbar \delbar^* + \delbar^* \delbar)\\
 &= \Delta_{\del} + \Delta_{\delbar}.
\end{split}
\end{equation*}

Portanto, só resta mostrar que $\Delta_{\del} = \Delta_{\delbar}$.

Usando a identidade $[\Lambda,\delbar] = - \smo \del^*$, temos
\begin{equation*}
-\smo \Delta_{\del} = \del(\Lambda \delbar- \delbar\Lambda) + (\Lambda \delbar- \delbar\Lambda) \del = \del \Lambda \delbar - \del \delbar\Lambda + \Lambda \delbar \del - \delbar \Lambda \del.
\end{equation*}
de onde segue, usando a identidade $[\Lambda, \del] =\smo \delbar^*$ e lembrando que $\del \delbar =- \delbar \del$, que
\begin{equation*}
\begin{split}
\smo \Delta_{\delbar} &= \delbar(\Lambda \del - \del \Lambda) + (\Lambda \del - \del \Lambda)\delbar\\
&= \delbar \Lambda \del - \delbar \del\Lambda + \Lambda \del \delbar - \del \Lambda \delbar\\
&= \smo \Delta_{\del}
\end{split}
\end{equation*}
terminando a demonstração.
\end{proof}

O corolário acima pode servir também como uma outra motivação, mais analítica, para a definição de uma métrica de Kähler. Por um cálculo simples sabemos, que para funções em $\C^n$, o Laplaciano $\Delta_{\delbar}$ é igual ao Laplaciano usual a menos de uma constante. Essencialmente o mesmo cálculo mostra que isso também vale em uma variedade plana $(X,g)$. No entanto, como os operadores de Laplace envolvem apenas derivadas de primeira ordem da métrica, ainda teremos que  $\Delta = 2\Delta_{\delbar}$ se $g$ for, a menos de termos de ordem maior ou igual a $2$, a métrica padrão de $\C^n$, o que leva a uma das possíveis definções de mértica de Kähler (item 3. proposição \ref{prop:kahler-equivalence}).\\

\section{Variedades projetivas} \label{sec:proj-var}

Nesta seção estudaremos com mais detalhes alguma propriedades do espaço projetivo e de suas subvariedades.

Uma subvariedade complexa do espaço projetivo $\pr^n$ é chamada \textbf{variedade projetiva}.\index{variedade!projetiva} 

Como o espaço projetivo é uma variedade de Kähler (exemplo \ref{ex:fubini-study}) temos, como consequência da proposição \ref{prop:kahler-submanifold}:
\begin{corollary}
Toda variedade projetiva é de Kähler.
\end{corollary}

\indent Como consequência interessante vemos, da proposição \ref{prop:cohom-kahler}, que se $X$ é uma variedade projetiva de dimensão $d$ então os grupos de cohomologia $H^2(X,\R),\ldots,H^{2d}(X,\R)$ são não triviais, o que não pode ser tão facilmente obtido por métodos puramente topológicos.

Os exemplos clássicos de variedades projetivas provém da Geometria Algébrica\footnote{Na verdade, por um teorema de W.L. Chow toda subvariedade complexa fechada de $\pr^n$ é algébrica e portanto as variedades algébricas esgotam todos os exemplos de variedades projetivas fechadas.}. Se $f$ é um polinômio homogêneo de grau $d$ em $n+1$ variáveis e $p = [z_0:\cdots:z_n]$ é um ponto do espaço projetivo, não faz sentido calcularmos $f(p)$, pois  temos que $f(\lambda z_0,\ldots,\lambda z_n) = \lambda^d f(z_0,\ldots,z_n)$ para todo $\lambda \in \C^*$ enquanto $[z_0:\cdots:z_n] = [\lambda z_0:\cdots:\lambda z_n]$. No entanto podemos dizer quando $f(p) = 0$ independentemente do representante de $p$.

Em outras palavras o conjunto
\begin{equation} \label{eq:zero-set-pol}
Z(f) = \{p \in \pr^n : f(p) = 0\} \subset \pr^n
\end{equation}
está bem definido.

Uma variedade algébrica será um conjunto dado pelos zeros simultâneos de polinômios homogêneos.
\begin{definition}
Um subconjunto $V \subset \pr^n$ é uma \textbf{variedade algébrica (projetiva)} se existem polinômios homogêneos $f_1,\ldots,f_k \in \C[z_0,\cdots,z_n]$ tal que $V = Z(f_1,\cdots,f_k) = Z(f_1) \cap \cdots \cap Z(f_k)$.
\end{definition}

Nem toda variedade algébrica é uma variedade projetiva\footnote{Inelizmente, essa nomenclatura é usual, embora a semelhança entre os termos possa ser fonte de confusão. No inglês por exemplo o nome \textit{variety} é usado para designar as variedades algébricas, enquanto o nome \textit{manifold} é usado para designar as variedades complexas. Ocorre que os dois termos são traduzidos como \textit{variedade} para o português.}, pois pode haver singularidades.

\begin{example}
Denote por $[x:y:z]$ as coordenadas homogêneas em $\pr^2$. A variedade algébrica em $\pr^2$ dada pelo polinômio homogêneo $f = x^2z - y^3$  tem uma singularidade no ponto $p=[0:0:1]$. De fato, tomando as coordenadas $u=x/z$ e $v=y/z$ definidas em $U=\{z \neq 0\}$ e centradas em $p$ vemos que $Z(f) \cap U \simeq \{(u,v):u^2 = v^3\}$ que é singular em $u=0,v=0$.
\end{example}

Quando uma variedade algébrica é de fato uma subvariedade complexa de $\pr^n$ (e portanto uma variedade projetiva no sentido acima) dizemos que $V$ é \textit{suave} ou \textit{não singular}. O seguinte resultado é uma consequência imediata do teorema da função implícita.

\begin{proposition}
Se $f \in \C[z_0,\ldots,z_n]$ é um polinômio homogêneo e $0$ é um valor regular de $f:\C^{n+1} \setminus \{0\} \to \C$ então $X = Z(f) \subset \pr ^n$ é suave. Neste caso dizemos que $V$ é uma \textbf{hipersuperfície}.

Mais geralmente se $f_1,\ldots,f_k \in \C[z_0,\ldots,z_n]$ são polinômios homogêneos e $0$ é um valor regular de $(f_1,\ldots,f_k):\C^{n+1} \setminus \{0\} \to \C^k$ então  $V = Z(f_1,\ldots,f_k)$ é suave. Nesse caso dizemos que $V$ é uma \textbf{intersecção completa}.
\end{proposition}

\begin{example} \textbf{Hipersuperfíces de Fermat.}
Considere o polinômio homogêneo $F_d = z_0^d + z_1^d + \cdots - z_n^d$. Note que $\partial{F_d}/\partial{z_i} = \pm d\cdot z_i^{d-1}$ não se anula em $\C^{n+1}\setminus \{0\}$ e portanto $F_d:\C^{n+1}\setminus \{0\} \to \C$ não tem pontos críticos. Portanto, da proposição acima, temos que $X_d = Z(F_d) \subset \pr^n$ é uma hipersuperfície projetiva, chamada \emph{hipersuperfíce de Fermat}.

Quando $n=2$ obtemos uma família de curvas $X_d$ em $\pr^2$, chamadas de curvas de Fermat, dadas pela equação $x^d + y^d - z^d =0$. A existência de pontos racionais em $X_d$ (i.e., pontos em que todas as coordenadas homogêneas são números racionais) é equivalente a existência de soluções inteiras da equação $x^d + y^d  = z^d $. Pelo famoso Último Teorema de Fermat, demonstrado por Andrew Wiles em 1995, as curvas $X_d$ com $d\geq 3$ só contém os dois pontos racionais $[0:1:1]$ e $[1:0:1]$ se $d$ é par e além deles o ponto $[1:-1:0]$ se $d$ é ímpar.
\end{example}

\subsubsection{O fibrado tangente e o fibrado canônico de $\pr^n$}
Nesta seção iremos estudar com um pouco mais de detalhe o fibrado tangente do espaço projetivo complexo. Veremos que ele se insere em uma sequência exata curta de fibrados holomorfos envolvendo outros fibrados já conhecidos, o que permitirá descrever explicitamente o fibrado canônico $K_{\pr^n}$.\\

Considere a projeção natural $\pi: \C^{n+1} \setminus \{0\} \to \pr^n$. Dado $x = \pi(z) \in \pr^n$, a diferencial (real)
\begin{equation*}
d\pi_z:T_z (\C^{n+1} \setminus \{0\}) \simeq \R^{2(n+1)} \to T_x \pr^n
\end{equation*}
é sobrejetora e seu núcleo é exatamente o espaço tangente a fibra $\pi^{-1}(x) = x \subset \C^{n+1}$, que é identificado com a própria reta $x$.

Complexificando os espaços tangentes e a diferencial obtemos $d\pi_z: \R^{2(n+1)} \otimes \C \to T_x \pr^n \otimes \C$. Como $\pi$ é holomorfa, $d\pi_z$ leva $(\R^{2(n+1)} \otimes \C)^{1,0} \simeq \C^{n+1}$ em $(T_x \pr^n )^{1,0}$. Temos portanto uma aplicação $\C$-linear sobrejetora
\begin{equation*}
d\pi_z: \C^{n+1} \to (T_x \pr^n )^{1,0},
\end{equation*}
cujo núcleo é o subespaço complexo $x \subset \C^{n+1}$. Com respeito a base $\{\del / \del z_0, \ldots \del / \del z_n\}$ de $\C^{n+1}$ a aplicação acima é dada por $\del / \del z_i \mapsto d \pi_z (\del / \del z_i)$ e o núcleo é gerado pelo vetor radial $Z = \sum z_i ~\del / \del z_i$.

Poderíamos tentar realizar essa construção fibra a fibra a fim de obter um morfismo de fibrados complexos $\pr^n \times \C^{n+1} \to T^{1,0} \pr^n$, definido por $(x,v) \mapsto d \pi_z \cdot v$, mas essa tentativa não produz uma aplicação bem definida, pois embora $x = \pi(z) = \pi(\lambda z)$ temos que $d \pi_z \neq d \pi_{\lambda z}$. Derivando a equação $\pi(z) = \pi(\lambda z)$ vemos na verdade que $\lambda d \pi_{\lambda z} \cdot v =d \pi_z \cdot v$ e portanto
\begin{equation} \label{eq:hom-dpi}
d\pi_{\lambda z} \cdot v = \lambda^{-1} d \pi_z \cdot v~\text{ para } \lambda \neq 0.
\end{equation}
Vemos portanto que $d\pi_z$ tem um comportamento homogêneo de grau $-1$ em $z$. Dessa observação podemos tornar bem definida a aplicação acima, tomando o produto tensorial do fibrado trival com o fibrado de linha $\mathcal{O}(1)$. Obtemos assim uma sequência exata envolvendo o fibrado tangente de $\pr^n$.

\begin{proposition} \label{prop:euler-seq} \textbf{A sequência de Euler.} \index{sequência!de Euler}
Existe uma sequência exata de fibrados holomorfos sobre $\pr^n$
\begin{equation} \label{eq:euler-seq}
0 \longrightarrow \mathcal{O}_{\pr^n} \longrightarrow \mathcal{O}_{\pr^n}^{\oplus n+1} \otimes \mathcal{O}(1) \longrightarrow \mathcal{T}_{\pr^n} \longrightarrow 0.
\end{equation}
\end{proposition}

\begin{proof}
O primeiro passo é definir a sequência de Euler fibra a fibra. Como a fibra de $\mathcal{O}(1) = \mathcal{O}(-1)^*$ sobre um ponto $x \in \pr^n$ é $x^*= \text{Hom}(x,\C)$ temos que a fibra de  $\mathcal{O}_{\pr^n}^{\oplus n+1} \otimes \mathcal{O}(1)$ sobre $x$ é $\C^{n+1}\otimes \text{Hom}(x,\C)$. Definimos então
\begin{equation*}
\begin{split}
P_x: \C^{n+1}\otimes \text{Hom}(x,\C) &\longrightarrow (T_x \pr^n)^{1,0} \\
P_x(v \otimes \sigma) &= d \pi_z(\sigma(z) \cdot v)
\end{split}
\end{equation*}
onde $z$ é qualquer elemento de $x$. Da linearidade de $\sigma$ e da equação (\ref{eq:hom-dpi}) vemos que $P_x$ independe da escolha de $z$.

Da discussão que precede o enunciado vemos que $P_x$ é sobrejetora e $P_x(v \otimes \sigma)=\sigma(z) d \pi_z(v) = 0 $ se e somente se $v \in x$ ou $\sigma = 0$, isto é, se e somente se $v \otimes \sigma \in x \otimes x^*$. Em outras palavras o núcleo de $P_x$ é $x \otimes x^*$, que é a fibra sobre $x$ do fibrado $\mathcal{O}(-1) \otimes \mathcal{O}(1) \simeq \mathcal{O}_{\pr^n}$.

Fazendo essa construção fibra a fibra obtemos um morfismo sobrejetor  $P: \C^{n+1} \otimes \mathcal{O}(1) \to T^{1,0} \pr^n$ com $\ker P = \mathcal{O}(-1) \otimes \mathcal{O}(1) \simeq \mathcal{O}_{\pr^n}$. É claro da definção que $P_x$ depende holomorficamente de $x$. Sendo assim $P$ é um morfismo de fibrados holomorfos e como  $\mathcal{T}_{\pr^n}$ é simplesmente $T^{1,0} \pr^n$ com a estrutura holomorfa obtemos a sequência (\ref{eq:euler-seq}).
\end{proof}

\begin{corollary} \label{cor:canonical-bundle-Pn} \index{fibrado!canônico!de $\pr^n$}
O fibrado canônico de $\pr^n$ é isomorfo a $\mathcal{O}(-(n+1))$.
\end{corollary}

\begin{proof}
O resultado segue aplicando o seguinte lema à sequência de Euler.
\begin{lemma} \label{lemma:det-bundle}
Se $0 \to E \to F \to G \to 0$ é exata então $\det (F) \simeq \det(E) \otimes \det(G)$.
\end{lemma}
\begin{proof}
Como $E \to F$ é injetora, podemos ver $E$ como subfibrado holomorfo de $F$. Podemos então escolher trivializações de $F$ de modo que seus cociclos sejam da forma $\varphi_{ij} = \begin{pmatrix} \phi_{ij} & * \\ 0 & \psi_{ij} \end{pmatrix}$, onde $\{\phi_{ij}\}$ são os cociclos de $E$ e $\{\psi_{ij}\}$ são os cociclos de $F/E \simeq G$. Os cociclos de $\det(F)$ são portanto $\det (\varphi_{ij}) = \det (\phi_{ij})\cdot \det(\psi_{ij})$, que são os cociclos de $\det (E) \otimes \det(G)$.
\end{proof}
Note que, em particular, se $F = E\oplus G$ então temos que $\det (F) \simeq \det(E) \otimes \det(G)$ e por indução vemos que, se $F = L_1 \oplus \cdots \oplus L_k$ é uma soma direta de fibrados de linha temos que $\det(F) \simeq L_1 \otimes \cdots \otimes L_k$.

Usando este fato e aplicando o lema à sequência (\ref{eq:euler-seq}) obtemos
\begin{equation*}
K_{\pr^n}^* = \det(\mathcal{T}_{\pr^n}) \simeq \det(\mathcal{T}_{\pr^n}) \otimes \det(\mathcal{O}_{\pr^n}) \simeq \det(\mathcal{O}_{\pr^n}^{\oplus n+1} \otimes \mathcal{O}(1)) \simeq \det(\mathcal{O}(1)^{\oplus n+1}) \simeq \mathcal{O}(n+1),
\end{equation*}
e portanto, dualizando, obtemos $K_{\pr^n} \simeq \mathcal{O}(-(n+1))$.
\end{proof}

\subsubsection{O fibrado canônico de variedades projetivas}
Podemos usar o corolário \ref{cor:canonical-bundle-Pn} para determinar também o fibrado canônico de algumas variedades projetivas. Para fazê-lo precisamos entender como o fibrado canônico de uma subvariedade está relacionado com o fibrado canônico da variedade ambiente.\\

Seja $X$ uma variedade complexa $Y \subset X$ uma subvariedade complexa de $X$. Existe então uma inclusão natural $\mathcal{T}_Y \subset \mathcal{T}_X$ entre os fibrados tangentes holomorfos de $X$ e $Y$. O \textbf{fibrado normal} de $Y$ em $X$, denotado por $\mathcal{N}_{Y|X}$, é um fibrado sobre $Y$ definido pela sequência exata
\begin{equation} \label{eq:normal-bundle}
0 \longrightarrow \mathcal{T}_Y \longrightarrow (\mathcal{T}_X)\big|_Y \longrightarrow \mathcal{N}_{Y|X} \longrightarrow 0
\end{equation}
ou equivalentemente é o fibrado quociente $\mathcal{N}_{Y|X} = (\mathcal{T}_X)|_Y / \mathcal{T}_Y$.

Note que $\mathcal{N}_{Y|X}$ é um fibrado holomorfo cujo posto é igual a codimensão de $Y$ em $X$. Em particular, se $Y$ é uma hipersuperfície, então $\mathcal{N}_{Y|X}$ é um fibrado de linha.

\begin{remark}
Na presença de uma métrica hermitiana em $X$ esta definição de fibrado normal coincide com a usual. De fato, se equiparmos $Y$ com a métrica induzida e denotarmos por $\mathcal{T}_Y^{\perp} \subset \mathcal{T}_X$ o complemento ortogonal de $\mathcal{T}_Y$ em $\mathcal{T}_X$ temos que $\mathcal{T}_X|_Y = \mathcal{T}_Y \oplus \mathcal{T}_Y^
{\perp}$ e portanto $\mathcal{N}_{Y|X} = (\mathcal{T}_X)|_Y / \mathcal{T}_Y \simeq \mathcal{T}_Y^{\perp}$.
\end{remark}

A partir do fibrado normal e do fibrado canônico do ambiente, podemos calcular o fibrado canônico da subvariedade.

\begin{proposition} \label{prop:adjunction} \textsl{(Fórmula de Adjunção)} \index{fórmula de adjunção}
Se $Y \subset X$ é uma subvariedade complexa então existe um isomorfismo $K_Y \simeq (K_X)|_Y \otimes \det \mathcal{N}_{Y|X}$.
\end{proposition}
\begin{proof}
Aplicando o Lema \ref{lemma:det-bundle} à sequência (\ref{eq:normal-bundle}) temos que $\det \mathcal (T_X)|_Y \simeq \det \mathcal T_Y \otimes \det \mathcal N_{Y|X}$ e portanto $\det \mathcal T_Y \simeq \det \mathcal (T_X)|_Y \otimes \det \mathcal N_{Y|X}^*$ de onde obtemos $K_Y = \det T_Y^* \simeq \det \mathcal (T_X^*)|_Y \otimes \det \mathcal N_{Y|X} = (K_X)|_Y \otimes \det \mathcal{N}_{Y|X}$.
\end{proof}

Uma hipersuperfície algébrica pode ser vista como um caso particular de uma hipersuperfície complexa $Y \subset X$ dada pelos zeros de uma seção holomorfa de um fibrado de linha sobre $X$. No capítulo \ref{ch:div-lb} vimos que se $L \to X$ é um fibrado de linha holomorfo trivializado sobre uma cobertura $\{U_i\}$, uma seção holomorfa global $s \in H^0(X,L)$ corresponde uma coleção de funções holomorfas $s_i:U_i \to \C$ satisfazendo $s_i = \psi_{ij}s_j$, onde $\psi_{ij} \in \mathcal O_X^*(U_i \cap U_j)$ são os cociclos de $L$. Em particular vemos que se $x \in U_i \cap U_j$ então $s_i(x) = 0$ se o só se $s_j(x) = 0$. Com isso podemos definir o \textbf{conjunto de zeros} de $s$, pela condição
\begin{equation*}
Z(s) \cap U_i = Z(s_i).
\end{equation*}
O conjunto $Z(s)$ é um subconjunto analítico de $X$. Em geral $Z(s)$ não será uma subvariedade complexa pois, assim como no caso algébrico, podem existir pontos singulares.

No caso particular em que $X=\pr^n$ e $L= \mathcal O(d)$ uma seção $f \in H^0(X,\mathcal O (d))$ corresponde a um polinômio homogêneo de grau $d$ (veja o exemplo \ref{ex:sec-Ok}) e a definição $Z(f)$ coincide com a definção do conjunto de zeros de um polinômio homogêneo em $\pr^n$ (eq \ref{eq:zero-set-pol}).

\begin{lemma}
Se $Y = Z(s) \subset X$ uma subvariedade dada pelos zeros de uma seção holomorfa $s \in H^0(X,L)$ então o fibrado normal $N_{Y|X}$ é isomorfo a $L|_Y$. Consequentemente o fibrado canônico de $Y$ é dado por $K_Y \simeq (K_X \otimes L)|_Y$.
\end{lemma}

\begin{proof}
Sejam $\psi_i:L|_{U_i} \to U_i \times \C$ trivializações de $L$ e seja $s_i:U_i \to \C$ a representação local de $s$ (isto é, $\psi_i s(x) = (x,s_i(x))$ para $x \in U_i$). Considere o morfismo de fibrados holomorfos dado por
\begin{equation*}
\Psi: \mathcal T_X\big|_Y \to L\big|_Y,~~ \mathcal T_X \big|_{U_i \cap Y} \ni v \mapsto \psi_i^{-1}ds_i(v) \in L\big|_{U_i \cap Y}.
\end{equation*}

Note que $\Psi$ está bem definida. De fato, em $U_i \cap U_j$ temos que $s_i = \psi_{ij} s_j$ e portanto $ds_i = d \psi_{ij}s_j + \psi_{ij}s_j$. Restringindo a $Y$ e usando o fato de que $Y = Z(s)$ temos que $ds_i = \psi_{ij} ds_j$ e portanto $\psi_i^{-1}ds_i = \psi_j^{-1}ds_j$ em $U_i \cap U_j \cap Y$.

Como $s$ é constante em $Y$ segue que $\Psi|_{\mathcal T_Y} = 0$, isto é, $\mathcal T_Y \subset \ker \Psi$. Do fato de  $Y$ ser suave, a diferencial $ds_i$ nunca se anula e portanto $\Psi$ é não nula, de onde vemos, por razões dimensionais, que $\ker \Psi= \mathcal T_Y$.

Obtemos assim uma sequência exata de fibrados holomorfos
\begin{equation*}
0 \longrightarrow \mathcal{T}_Y \longrightarrow (\mathcal{T}_X)\big|_Y \longrightarrow L\big|_Y \longrightarrow 0,
\end{equation*}
que é exatamente a sequência que define $\mathcal{N}_{Y|X}$, mostrando que  $\mathcal{N}_{Y|X} \simeq L|_Y$.

O isomorfismo $K_Y \simeq (K_X \otimes L)|_Y$ segue da fórmula de adjunção (Proposição \ref{prop:adjunction}).
\end{proof}

\begin{corollary}
Se $X \subset \pr^n$ é uma hipersuperfície suave de grau $d$ então seu fibrado canônico é  $K_X \simeq \mathcal O(d-n-1)|_X$. Mais geralmente, se $X = Z(f_1)\cap \cdots \cap Z(f_k)$ é uma intersecção completa com $\deg(f_k)=d_k$ então $K_X \simeq \mathcal O(d_1 + \ldots + d_k -n-1)|_X$.
\end{corollary}

\begin{proof}
Uma hipersuperfície de grau $d$ é dada pelos zeros de uma seção $f \in H^0(\pr^n, \mathcal O (d))$. Assim, pela proposição anterior, temos que $K_X \simeq (K_{\pr^n} \otimes \mathcal O (d))|_X$. Do corolário \ref{cor:canonical-bundle-Pn} temos que $K_{\pr^n} \simeq \mathcal O (-n-1)$ de onde vemos que $K_X \simeq (\mathcal O (-n-1) \otimes \mathcal O (d))|_X \simeq \mathcal O(d-n-1)|_X$.

Para demonstrar o resultado para as intersecções completas vamos usar indução sobre número $k$ de equações. Se $k=1$ temos que $X$ é uma hipersuperfície e o resultado foi provado acima.

Suponha portanto que $X = Z(f_1) \cap \ldots \cap Z(f_k) \subset \pr^n$. Como $X$ é uma intersecção completa, temos que $X$ é uma subvariedade de $\widetilde X = Z(f_1) \cap \ldots \cap Z(f_{k-1})$ e além disso $X$ é dada pelos zeros da seção $f_k \in H^0(X,\mathcal O (d_k)|_{\widetilde X})$.

Da hipótese de indução temos que $K_{\widetilde X} \simeq \mathcal O(d_1 + \ldots + d_{k-1} -n-1)|_{\widetilde X}$ e da proposição acima concluimos que
\begin{equation*}
K_X \simeq (K_{\widetilde X} \otimes \mathcal O (d_k))|_X \simeq (\mathcal O(d_1 + \ldots + d_{k-1} -n-1) \otimes \mathcal O (d_k))|_X \simeq \mathcal O(d_1 + \ldots + d_k -n-1)|_X.
\end{equation*}
\end{proof}

\begin{example}
Seja $X \subset \pr^3$ uma quártica suave, isto é, uma variedade projetiva dada pelos zeros de um polinômio de grau $4$ em $\pr^3$. Usando a proposição acima vemos que o fibrado canônico de $X$ é trivial, pois $K_X \simeq \mathcal O (4-3-1)|_X \simeq \mathcal O_X$. A variedade $X$ é simplesmente conexa\footnote{Esse fato segue, por exemplo, de uma versão para grupos de homotopia do Teorema de Hiperplanos de Lefschetz. Uma versão desse teorema para os grupos de homologia será apresentada no capítulo \ref{ch:stein-lesfchetz}} e é um exemplo particular das chamadas superfícies K3, que são superfícies complexas simplesmente conexas com $H^1(X,\mathcal O_X) = 0$ e fibrado canônico trivial.

Mais geralmente, qualquer hipersupefície suave de grau $n+1$ em $\pr^n$ terá fibrado canônico trivial.
\end{example}
\subsubsection{Propriedades geométricas de métrica de Fubini-Study} \index{métrica!de Fubini-Study}

Vamos terminar essa seção discutindo algumas propriedades geométricas da métrica de Fubini-Study.

Nesta seção denotaremos por $(\cdot,\cdot)$ o produto hermitiano canônico em $\C^{n+1}$ e por $\langle \cdot,\cdot \rangle = \re (\cdot,\cdot)$ a métrica riemanniana padrão.

Dado um ponto $z \in \C^{n+1}\setminus \{0\}$, seja $x \subset \C^{n+1}$ a reta gerada por $z$ e considere a decomposição ortogonal $T_z(\C^{n+1} \setminus \{0\}) = x \oplus x^\perp$ com respeito ao produto $(\cdot,\cdot)$. Denote por $X^\perp$ a projeção de $X \in T_z(\C^{n+1} \setminus \{0\})$ em $x^\perp$.

Defina uma forma bilinear $\widetilde h$ em $T(\C^{n+1}\setminus \{0\})$
\begin{equation} \label{eq:h-tilde}
\widetilde h_z (X,Y) = \frac{1}{||z||^2}(X^\perp,Y^\perp),~X,Y \in T_z(\C^{n+1} \setminus \{0\}).
\end{equation}
Note que $\widetilde h$ é positiva mas não é definida, pois se $X \in x$ então $X^\perp = 0$ e portanto $\widetilde h(X,X) = 0$.

\begin{lemma}
A $(1,1)$-forma $\eta = \widetilde h (J\cdot,\cdot)$ associada a $\widetilde h$ é dada por
\begin{equation} \label{eq:eta-fubini}
\eta = \frac{\smo}{2} \del \delbar \log (||z||^2).
\end{equation}
\end{lemma}

\begin{proof}
Note que ambos os membros da equação (\ref{eq:eta-fubini}) são $U(n+1)$-invariantes e como $U(n+1)$ age transitivamente nas esferas de $\C^{n+1}\setminus \{0\}$ basta provarmos que a equação é valida em um ponto da forma $p=(r,0,\ldots,0)$ com $r>0$.

Dados dois vetores $X = (a_0,\ldots,a_n) $ e $Y=(b_0,\ldots,b_n)$, as suas projeções em $p^\perp$ são dadas por $X^\perp = (0,a_1,\ldots,a_n)$ e $Y^\perp = (0,b_1,\ldots,b_n)$ e portanto
\begin{equation*}
\widetilde h_p(X,Y) = \frac{1}{r^2} ((0,a_1,\ldots,a_n),(0,b_1,\ldots,b_n)) = \frac{1}{r^2} \sum_{j=1}^n a_j \bar b_j,
\end{equation*}
ou seja, $\widetilde h_p$ é $1 \slash r^2$ vezes a métrica hermitiana padrão $h_0$ em $\C^n \subset \C^{n+1}$.

Sendo assim, sua forma fundamental é dada por
\begin{equation*}
\eta_p = \frac{1}{r^2}h_0(J\cdot,\cdot) = \frac{1}{r^2} \omega_0 = \frac{1}{r^2} \frac{\smo}{2}\sum_{j=1}^n dz_j \wedge d \bar z_j. 
\end{equation*}

Por outro lado, como $\del (1\slash||z||^2) = - (\sum_i \bar z_i dz_i)\slash |z|^4$, temos que 
\begin{equation*}
\del \delbar \log (||z||^2) = \del \bigg( \frac{1}{||z||^2} \sum_{j=0}^m z_j d\bar z_j\bigg) = - \frac{1}{|z|^4} \bigg( \sum_{i=0}^m \bar z_i d z_i\bigg) \wedge \bigg(\sum_{j=0}^m z_j d\bar z_j\bigg) + \frac{1}{||z||^2} \sum_{j=0}^n dz_j \wedge d \bar z_j
\end{equation*}
e portanto, calculando em $p$ obtemos
\begin{equation*}
\frac{\smo}{2}\del \delbar \log (||z||^2)_p =\frac{\smo}{2} \bigg(- \frac{1}{r^4} (r dz_0) \wedge (r d \bar z_0) + \frac{1}{r^2}\sum_{j=0}^n dz_j \wedge d \bar z_j \bigg) = \frac{\smo}{2} \frac{1}{r^2} \sum_{j=1}^n dz_j \wedge d \bar z_j = \eta_p.
\end{equation*}
\end{proof}

Note que, conforme a equação (\ref{eq:pullback-fubini-study}), o membro direito da equação (\ref{eq:eta-fubini}) é, a menos de um múltiplo,a expressão do pullback da forma de Fubini-Study pela projeção $\pi: \C^{n+1}\setminus \{0\}$, mais precisamente temos que $\pi^* \omega_{FS} = \frac{1}{\pi} \eta$. Como $\pi$ é holomorfa (e portanto preserva $J$) concluimos do lema acima que $\pi^*h_{FS} = \frac{1}{\pi}\widetilde h$.

Em outras palavras, a métrica de Fubini-Study em $\pr^n$ é, a menos de um múltiplo, obtida pela projeção por $\pi$ da forma bilinear (\ref{eq:h-tilde}).\\

A métrica de Fubini-Study admite ainda uma outra decrição, em termos da submersão $S^{2n+1} \to \pr^n$ dada pela restrição da projeção $\pi:C^{n+1} \setminus \{0\} \to \pr^n$ à esfera unitária de $\C^{n+1}$.

Dadas duas variedades riemannianas $M$ e $N$ e uma submersão $\rho:M \to N$, o espaço tangente de $M$ em um ponto $p \in \rho^{-1}(q)$ se decompõe como $T_pM = H_p \oplus V_p$, onde $V_p = T_p (\rho^{-1}(q)) = \ker (d \rho)_p$ e $H_p=(V_p)^\perp$. Os espaços $V_p$ e $H_p$ são chamados, respectivamente de espaço vertical e espaço horizontal em $p$.

Note que, como $(d\rho)_p|_{V_p} = 0$, a restrição da diferencial $(d\rho)_p: T_pM = H_p \oplus V_p \to T_q M$ ao espaço horizontal define um isomorfismo $(d\rho)_p:H_p \simeq T_q N$. Dizemos que $\rho$ é uma \emph{submersão riemanniana} se esse isomorfismo é uma isometria para todo $p \in M$.\\

Considere agora $\rho:S^{2n+1} \to \pr^n$ a restrição da projeção $\pi:\C^{n+1} \setminus \{0\} \to \pr^n$ à esfera unitária $S^{2n+1}~\subset~\C^{n+1}$.

Seja $z \in S^{2n+1}$ e denote por $x$ a reta complexa gerada por $z$. O espaço tangente à esfera no ponto $z$ é dado por $T_zS^{2n+1} = \{w \in \C^{n+1}:\langle z,w\rangle=0\}$, de onde vemos que $x^ \perp \subset T_zS^{2n+1}$. Além disso o subespaço real gerado por $\smo z$ está contido no espaço tangente, pois $\langle z,\smo z\rangle = \re (z,\smo z) = \re(-\smo (z,z)) = 0$. Concluimos então que o epsaço tangente da esfera se decompõe como uma soma direta de subespaços reais
\begin{equation*}
T_z S^{2n+1} = x^{\perp} \oplus \smo z.
\end{equation*}

Na discussão que precede a proposição \ref{prop:euler-seq}, vimos que o núcleo de $d \pi_z$ é a própria reta $x$, de onde concluímos que $\ker (d\rho)_z = x \cap T_z S^{2n+1} = \smo x$. Sendo assim, o espaço espaço vertical da submersão $\rho:S^{2n+1} \to \pr^n$ é dado por  $V_z = \smo x$ e o espaço horizontal é $H_z = x^\perp$, onde $S^{2n+1}$ está equipada com a métrica riemanniana canônica.

Vimos acima que o pullback da métrica $\pi \cdot h_{FS}$ em $\pr^n$ pela projeção canônica é a forma bilinear $\widetilde h$ dada pela equação (\ref{eq:h-tilde}), de modo que $\rho^* (\pi \cdot h_{FS}) = \widetilde h|_{S^{2n+1}}$. Como no espaço horizontal a projeção ortogonal $X \mapsto X^{\perp}$ é a identidade, vemos que $\widetilde h|_{H_z} = \widetilde h|_{x^\perp} = (\cdot,\cdot)$. Em outras palavras, restrita ao espaço horizontal, $\rho$ leva o produto $(\cdot,\cdot)$ em  $H_z$ no produto $\pi \cdot h_{FS}$ em $T_z \pr^n$. Tomando a parte real e lembrando que  $g_{FS} =\re h_{FS}$ concluímos que $d \rho_z$ é uma isometria entre $(H_z,\langle \cdot, \cdot \rangle)$ e $(T_z \pr^n, \pi \cdot g_{FS})$.

Provamos então o seguinte resultado
\begin{proposition} \label{prop:FS-submersion}
A aplicação $\rho:S^{2n+1} \to \pr^n$ obtida pela restrição da projeção canônica é uma submersão riemanniana, onde consideramos $S^{2n+1}$ com a métrica induzida de $\C^{n+1}$ e $\pr^n$ com a métrica $\pi \cdot g_{FS}$.
\end{proposition}

\section{A decomposição de Hodge} \label{sec:hodge-decomp}

Um resultado central em Geometria Riemanniana é o chamado Teorema de Hodge, que prova que em uma variedade Riemanniana compacta e orientada existe uma bijeção entre o espaço das $k$-formas harmônicas e o $k$-ésimo espaço de cohomologia de de Rham. Em outras palavras, cada classe de cohomologia $c \in H^k(X,\R)$ pode ser representada por uma única $k$-forma harmônica $\alpha \in c$.

Agora, se além disso, $X$ admite uma estrutura de Kähler, a igualdade dos laplacianos permite decompor ainda uma classe $c \in H^k(X,\C)$ em componentes de tipo $(p,q)$ com $p+q=k$, resultando na chamada decomposição de Hodge. Esta seção se ocupará da demonstração desses fatos.

\subsubsection{O caso riemanniano}

Comecemos considerando uma variedade Riemanniana $(X,g)$ compacta e orientada de dimensão $m$. A métrica $g$ induz uma métrica nos fibrados $\bigwedge^k TX^*$ da seguinte maneira: se $\{e_1,\cdots,e_m\}$ é uma referencial ortonormal local e $\{e^1,\ldots, e^m \}$ é o referencial dual, declaramos $\{e^{i_1} \wedge \cdots \wedge e^{i_k}:i_1<\cdots<i_k\}$ como referencial ortononormal de $\bigwedge^k TX^*$.\\
\indent Denotando essa métrica por $\langle \cdot,\cdot \rangle$ obtemos uma métrica $L^2$ no espaço das $k$-formas em $X$, fazendo
\begin{equation*}
(\alpha,\beta) = \int_X \langle \alpha,\beta \rangle \text{vol} = \int_X \alpha \wedge *\beta ~,\;\;\; \alpha,\beta \in \mathcal{A}^k(X).
\end{equation*}

\indent Considere o operador
\begin{equation*}
d^* = (-1)^{m(k+1)+1} * \circ d \circ *: \mathcal{A}^k(X) \longrightarrow \mathcal{A}^{k-1}(X),
\end{equation*}
e o laplaciano
\begin{equation*}
\Delta = dd^*+d^* d: \mathcal{A}^k(X) \longrightarrow \mathcal{A}^k(X)
\end{equation*}

Uma $k$-forma $\alpha$ é dita harmônica se $\Delta \alpha = 0$.
\begin{lemma}
O operador $d^*$ é o adjunto formal de $d:\mathcal{A}^k(X) \longrightarrow \mathcal{A}^{k+1}(X)$ com respeito à métrica $L^2$, isto é, $(\alpha,d^* \beta)=(d \alpha,\beta)$ para toda $\alpha,\beta \in \mathcal{A}^k(M)$. Consequentemente, o laplaciano $\Delta$ é auto-adjunto.
\end{lemma}
\begin{proof}
Sejam $\alpha \in \mathcal{A}^k(X)$ e $\beta \in \mathcal{A}^{k+1}(X)$. Da regra de Leibniz, temos que
\begin{equation*}
d(\alpha \wedge *\beta) = d\alpha \wedge * \beta + (-1)^k\alpha \wedge d*\beta
\end{equation*}

Integrando ambos os membros da equação acima e usando o teorema de Stokes vemos que
\begin{equation*}
(d\alpha,\beta) = \int_X d\alpha \wedge *\beta = (-1)^{k+1}\int_X \alpha \wedge d* \beta
\end{equation*}
Usando que $*^2 = (-1)^{l(m-l)} \text{id}$ em $\mathcal{A}^l(X)$ temos que $**d*\beta = (-1)^{(m-k)k} d*\beta$, pois $d* \beta \in \mathcal{A}^{m-k}(X)$ e, da definição de $d^*$, temos $d^* \beta = (-1)^{m(k+2)+1} *d* \beta = (-1)^{mk+1} *d* \beta$. Juntando essas observações obtemos
\begin{equation*}
\begin{split}
(\alpha,d^* \beta) &= \int_X \alpha \wedge *d^* \beta = (-1)^{mk+1} \int_X \alpha \wedge **d*\beta = (-1)^{mk+1} (-1)^{(m-k)k} \int_X \alpha \wedge d* \beta \\
 &= (-1)^{1-k^2} \int_X \alpha \wedge d* \beta = (-1)^{1+k} \int_X \alpha \wedge d* \beta = (d\alpha,\beta).
\end{split}
\end{equation*}
\end{proof}

\begin{corollary}
Uma forma $\alpha$ é harmônica se e somente se $d \alpha = 0$ e $d^* \alpha=0$.
\end{corollary}

\begin{proof}
É claro que se $d \alpha = 0$ e $d^* \alpha=0$ então $\Delta \alpha = (dd^* + d^*d)\alpha= 0$.\\
\indent Suponha agora que $\Delta \alpha = 0$. Então temos que
\begin{equation*}
0 = (\Delta \alpha,\alpha) = (dd^* \alpha,\alpha) + (d^*d \alpha,\alpha) = (d^*\alpha,d^*\alpha) + (d\alpha,d \alpha) = ||d^* \alpha||^2 + ||d\alpha||^2,
\end{equation*}
de onde segue que $d \alpha = 0$ e $d^* \alpha = 0$.
\end{proof}

\indent Denotando por $\mathcal{H}^k(X,g) = \{\alpha \in \mathcal{A}^k(X): \Delta \alpha = 0\}$ o espaço das $k$-formas harmônicas (ou simplesmente por $\mathcal{H}^k(X)$ quando a métrica estiver subbentendida), obtemos então uma aplicação $\mathcal{H}^k(X,g) \to H^k(X,\R)$, que associa a cada $\alpha$ harmônica sua classe de cohomologia.

É fácil ver que esta aplicação é injetora. De fato se $\alpha \in \mathcal{H}^k(X)$ e $[\alpha] = 0$ então existe $\beta \in \mathcal{A}^{k-1}(X)$ tal que $\alpha = d \beta$. Como $\alpha$ é harmônica temos que $0 = d^* \alpha = d^* d \beta$ e portanto $(\alpha,\alpha) = (d \beta,d \beta) = (\beta, d^* d \beta) = 0$ e assim $\alpha = 0$.

Além disso, há uma maneira de distinguir a representante harmônica em uma dada classe de cohomologia (se esta existir).

\begin{lemma} \label{lemma:harm-min}
Se $\alpha$ é harmônica então $||\beta|| \geq ||\alpha||$ para toda $\beta \in [\alpha]$. Reciprocamente, se $\alpha$ tem norma mínima em sua classe $[\alpha]$ então $\alpha$ é harmônica. Em outras palavras as formas harmônicas são as que tem a menor norma em sua classe de cohomologia.
\end{lemma}
\begin{proof}
Se $\alpha$ é harmônica então $d^* \alpha = 0$ e portanto, para $\beta = \alpha + d \eta \in [\alpha]$ temos
\begin{equation*}
\begin{split}
||\beta||^2 & = (\alpha + d \eta,\alpha + d \eta) = ||\alpha||^2 + ||d \eta||^2 + 2(\alpha,d \eta) \\
 & = || \alpha||^2 + ||d \eta||^2 + 2(d^*\alpha,\eta) = ||\alpha||^2 + ||d \eta||^2 \\
 & \geq ||\alpha||^2.
\end{split}
\end{equation*}

Reciprocamente, se $\alpha$ tem norma mínima em sua classe temos, para $\eta \in \mathcal{A}^{k-1}(X)$,
\begin{equation*}
0 = \frac{d}{dt} ||\alpha + t d\eta||^2 \bigg|_{t=0} = \frac{d}{dt} [ ||\alpha||^2 + t^2||d \eta||^2 + 2t(\alpha,d \eta) ] \bigg|_{t=0} = 2(\alpha,d \eta) = 2(d^* \alpha, \eta),
\end{equation*}
 e como $\eta$ é arbitrária segue que $d^* \alpha = 0$ e portanto $\alpha$ é harmônica.
\end{proof}

O lema acima nos dá uma possível maneira de encontrar a única forma harmônica $\alpha^h$ em uma dada classe $[\alpha] \in H^k(X,\R)$, basta encontrarmos um elemento de norma mínima no subespaço afim $[\alpha] = \alpha + d \mathcal{A}^{k-1}(X) \subset \mathcal{A}^k(X)$.

Infelizmente, o argumento acima não é tão simples, uma vez que o espaço $\mathcal{A}^k(X)$ não é um espaço de Hilbert com relação ao produto $(\cdot,\cdot)$ definido acima, e portanto não podemos garantir a existência de um elemento de norma mínima em $\alpha + d \mathcal{A}^{k-1}(X)$. No entanto podemos remediar essa situação considerando o completamento $\mathcal{L}^k(X)$ de $\mathcal{A}^k(X)$ com relação a norma $L^2$ e tomando um elemento de norma mínima no fecho $\overline{\alpha + d \mathcal{A}^{k-1}(X)}$. Note que ainda temos um problema, pois mesmo que $\alpha \in \mathcal{A}^k(X)$ só podemos garantir que este elemento está no fecho de $\alpha + d \mathcal{A}^{k-1}(X)$.

Essa é a grande dificuldade do teorema de Hodge, a saber, mostrar que o elemento obtido desta maneira é suave, isto é, pertence a $\alpha + d \mathcal{A}^{k-1}(X)$. O ingrediente principal é a chamada regularidade elíptica, que de maneira simplificada diz que se $\alpha \in \mathcal{L}^k(X)$ e $\Delta \alpha = 0$ (em um sentido fraco) então na verdade $\alpha \in \mathcal{A}^k(X)$.

O teorema completo é enunciado abaixo. Uma demonstração pode ser encontrada em \cite{demailly}

\begin{theorem} \textbf{Decomposição de Hodge para formas diferenciais} \label{thm:hodge-decomp}\\
Seja $X$ uma variedade riemanniana compacta e orientada. Então existe uma decomposição ortogonal
\begin{equation} \label{eq:hodge-dec-riemann}
\mathcal{A}^k(X) = \mathcal{H}^k(X) \oplus d \mathcal{A}^{k-1}(X) \oplus d^*\mathcal{A}^{k+1}(X)
\end{equation}
\end{theorem}

\begin{remark}
A hipótese de orientabilidade no teorema acima pode ser descartada. Se $X$ é não orientável não possuimos uma forma volume, e portanto não podemos definir o operador de Hodge. Mesmo assim ainda é possível definir o operador $d^*$, da seguinte maneira: a métrica riemanniana induz o que se chama uma \textit{densidade} em $X$, e com ela obtemos uma medida $\mu$ associada (veja por exemplo \cite{lang}, cap.XI). Definindo $(\alpha,\beta) = \int_X \langle \alpha, \beta \rangle d\mu$ obtemos os espaços $\mathcal L^k(X)$. O operador $d^*$ neste caso é o adjunto de $d$. Definido o laplaciano pela fórmula usual podemos dizer quem são as formas harmônicas e nesse contexto a decomposição (\ref{eq:hodge-dec-riemann}) continua válida.
\end{remark}

\begin{corollary} \emph{(Teorema de Hodge)} \\
Se $X$ é uma variedade riemanniana compacta e orientada então existe um isomorfismo
\begin{equation*}
H^k(X,\R) \simeq \mathcal{H}^k(X)
\end{equation*}
\end{corollary}

\begin{proof}
Primeiramente note que $\ker d \cap d^*\mathcal{A}^{k+1}(X) = \{0\}$ pois se $\alpha = d^* \beta$ e $0 = d \alpha = d d^* \beta$ então $0 = (d d^*\beta,\beta) = (d^* \beta,d^* \beta)$ e portanto $\alpha = d^* \beta = 0$.

Temos então, da decomposição de Hodge, que $\ker d = \mathcal{H}^k(X) \oplus d \mathcal{A}^{k-1}(X)$ de onde vemos que
\begin{equation*}
H^k(X,\R) = \frac{\ker d}{d \mathcal{A}^{k-1}(X)} = \frac{\mathcal{H}^k(X) \oplus d \mathcal{A}^{k-1}(X)}{d \mathcal{A}^{k-1}(X)} \simeq \mathcal{H}^k(X)
\end{equation*}
\end{proof}

\begin{corollary} \label{cor:poincare-duality} \emph{(Dualidade de Poincaré para a cohomologia com coeficientes reais)}\\ 
Seja $X$ uma variedade compacta e orientada. Então o pareamento
\begin{equation*}
H^k(X,\R) \times H^{m-k}(X,\R) \longrightarrow \R~,~~([\alpha],[\beta]) \longmapsto \int_X \alpha \wedge \beta
\end{equation*}
é não degenerado. Consequentemente existe um isomorfismo
\begin{equation*}
H^k(X,\R) \simeq H^{m-k}(X,\R)^*
\end{equation*}

\end{corollary}

\begin{proof}
Seja $[\alpha] \in H^k(X,\R)$, precisamos mostrar que existe $[\beta] \in H^{m-k}(X,\R)$ tal que $\int_X \alpha \wedge \beta \neq 0$.

Fixe uma métrica riemanniana em $X$ e seja $\alpha^h$ a representante harmônica de $\alpha$. Como $\Delta * = * \Delta$, $*\alpha^h$ também é harmônica, e portanto $\beta=*\alpha$ define uma classe $[\beta] \in H^{m-k}(X,\R)$. Da definição do operador $*$ temos que
\begin{equation*}
\int_X \alpha \wedge \beta = \int_X \alpha \wedge *\alpha = \int_X (\alpha,\alpha) \text{vol}_X = ||\alpha||^2 \text{Vol}(X) > 0
\end{equation*}
o que mostra que o pareamento é não degenerado.

O isomorfismo $H^k(X,\R) \simeq H^{m-k}(X,\R)^*$ é dado por $\alpha \mapsto \int_X \alpha \wedge (~\cdot~)$.
\end{proof}

\begin{remark} \label{rmk:complex-hodge}
Podemos estender de modo $\C$-linear os operadores $d,d^*$ e $\Delta$ aos espaços de $k$-formas com coeficientes complexos $\mathcal{A}^k_{\C}(X)$ e o produto interno $L^2$ em $\mathcal{A}^k(X)$ se estende a um produto hermitiano em $\mathcal{A}^k_{\C}(X)$ (mais detalhes na próxima seção). Podemos repetir, \textit{mutatis mutandis}, os argumentos desta seção, e obtemos uma decomposição de Hodge
\begin{equation*}
\mathcal{A}_{\C}^k(X) = \mathcal{H}_{\C}^k(X) \oplus d \mathcal{A}_{\C}^{k-1}(X) \oplus d^*\mathcal{A}_{\C}^{k+1}(X)
\end{equation*}
e um isomorfismo de Hodge
\begin{equation} \label{eq:complex-hodge}
H^k(X,\C) \simeq \mathcal{H}^k_{\C}(X)
\end{equation}
para formas complexas.
\end{remark}

\subsubsection{O caso hermitiano}

No caso em que $X$ é uma variedade complexa e compacta com uma métrica hermitiana podemos aplicar os pensamentos acima para os operadores $\del$ e $\delbar$, obtendo assim uma decomposição semelhante à do teorema \ref{thm:hodge-decomp} para os espaços $\mathcal{A}^{p,q}(X)$. Como consequência obtemos um isomorfismo de Hodge para a cohomologia de Dolbeaut e a dualidade de Serre, um análogo complexo da dualidade de Poincaré.\\

Se $g$ é uma métrica hermitiana em $X$, temos produtos hermitianos induzidos nos fibrados de formas complexas $\bigwedge_{\C}^k X$. O operador de Hodge $*:\bigwedge^* X \to \bigwedge^* X$ se estende de forma $\C$-linear a um operador $*:\bigwedge_{\C}^* X \to \bigwedge_{\C}^* X$, satisfazendo
\begin{equation} \label{eq:c-hodgestar}
\alpha \wedge \overline{* \beta}  = (\alpha,\beta) \text{vol}
\end{equation}
para toda $\alpha \in \bigwedge_{\C}^k X$ e $\beta \in \bigwedge_{\C}^{n-k} X$. Do fato de $\text{vol}$ ser uma forma de tipo $(n,n)$, a igualdade acima mostra que $*\big( \bigwedge^{p,q} X \big ) \subset \bigwedge^{n-q,n-p} X$.

Em analogia ao caso riemanniano, podemos definir um produto hermitiano $L^2$ em $\mathcal{A}^k_{\C}(X)$:
\begin{equation*}
(\alpha,\beta) = \int_X ( \alpha,\beta ) \text{vol} = \int_X \alpha \wedge \overline{*\beta} ~,\;\;\; \alpha,\beta \in \mathcal{A}_{\C}^k(X).
\end{equation*}

Com essa definição temos que os espaços $\mathcal{A}^{p,q}(X)$ são dois a dois ortogonais e portanto a soma direta $\mathcal{A}^*_{\C}(X) = \bigoplus_{k=0}^{2n} \mathcal{A}^k_{\C}(X) = \bigoplus_{k=0}^{2n}\bigoplus_{p+q=k} \mathcal{A}^{p,q}(X)$ é ortogonal.

Relembrando a definição dos operadores $\del^* = - *  \delbar  *$ e $\delbar^* = - *  \del  *$ dadas na seção \ref{sec:kahler-identities} temos o seguinte lema.

\begin{lemma}
Os operadores $\del^*$ e $\delbar^*$ são, respectivamente, os adjuntos formais de $\del^*$ e $\delbar^*$ com respeito ao produto $L^2$. Consequentemente os Laplacianos $\Delta_{\del}$ e $\Delta_{\delbar}$ são autoadjuntos. 
\end{lemma}

\begin{proof}
Vamos mostar que $\del^*$ é o adjunto de $\del$. A demonstração para $\delbar$ é análoga.

Se $\alpha \in \mathcal A^{p-1,q}(X)$ e $\beta \in \mathcal A^{p,q}(X)$ temos que $\del(\alpha \wedge *\bar \beta) = \del \alpha \wedge * \bar \beta + (-1)^{p+q-1} \alpha \wedge \del(*\bar \beta)$ e portanto
\begin{equation*}
(\del \alpha,\beta) = \int_X \del \alpha \wedge * \bar \beta = \int_X \del (\alpha \wedge * \bar \beta) - (-1)^{p+q-1} \int_X \alpha \wedge \del(*\bar \beta).
\end{equation*}
Note que $\alpha \wedge * \bar \beta$ é uma forma de tipo $(n-1,n)$ e portanto $\del (\alpha \wedge * \bar \beta) = d (\alpha \wedge * \bar \beta)$. Assim, pelo teorema de Stokes, a primeira integral acima se anula.

Usando que $*^2 = (-1)^{p+q}\id$ em $\mathcal A^{p,q}(X)$ e o fato que $\del *\bar \beta$ é de tipo $(n-p+1,n-q)$ temos que $**\del*\bar \beta = (-1)^{p+q+1} \del *\bar \beta$ e portanto a segunda integral acima fica
\begin{equation*}
\int_X \alpha \wedge \del(*\bar \beta) = (-1)^{p+q+1} \int_X \alpha \wedge **\del*\bar \beta = (-1)^{p+q+1} \int_X \alpha \wedge \overline{**\delbar * \beta} = -(-1)^{p+q+1} (\alpha,\del^* \beta),
\end{equation*}
onde usamos que $\del \eta = \overline{\delbar \eta}$ e $\overline * =*$.

Combinando com a primeira equação vemos finalmente que
\begin{equation*}
(\del \alpha,\beta) = - (-1)^{p+q-1} \cdot(-(-1)^{p+q+1} (\alpha,\del^* \beta)) = (\alpha,\del^* \beta).
\end{equation*}
\end{proof}

Definimos os espaços de formas $\delbar$-harmônicas por
\begin{equation*}
\mathcal{H}^k_{\delbar}(X) = \{\alpha \in \mathcal{A}^k(X) : \Delta_{\delbar} \alpha = 0\} ~~~ \text{ e } ~~~ \mathcal{H}^{p,q}_{\delbar}(X) = \{\alpha \in \mathcal{A}^{p,q}(X) : \Delta_{\delbar} \alpha = 0\}
\end{equation*}
e analogamente os espaços de formas $\del$-harmônicas $\mathcal{H}^k_{\del}(X)$ e $\mathcal{H}^{p,q}_{\del}(X)$.

\begin{proposition} \label{prop:*-harm}
Se $X$ é uma variedade hermitiana então
\begin{enumerate}
\item Os espaços de formas harmônicas se decompõem como
\begin{equation} \label{eq:harm-pq}
\mathcal{H}^k_{\delbar}(X) = \bigoplus_{p+q=k} \mathcal{H}^{p,q}_{\delbar}(X)~ \text{ e } ~~ \mathcal{H}^k_{\del}(X) = \bigoplus_{p+q=k} \mathcal{H}^{p,q}_{\del}(X)
\end{equation}
\item O operador de Hodge induz um isomorfismo
\begin{equation*}
*: \mathcal{H}^{p,q}_{\delbar}(X) \simeq  \mathcal{H}^{n-q,n-p}_{\del}(X).
\end{equation*}
\end{enumerate}
\end{proposition}

\begin{proof}
1. Vamos demonstrar a primeira decomposição, a demonstração da segunda é análoga.

Seja $\alpha \in \mathcal \mathcal{A}^k_{\C}(X)$ e seja $\alpha = \sum_{p,q} \alpha^{p,q}$ sua decomposição em tipos. Se $\alpha \in \mathcal{H}^k_{\delbar}(X)$ temos $0 = \Delta_{\delbar} (\alpha) = \sum_{p,q} \Delta_{\delbar} (\alpha^{p,q})$, e como $\Delta_{\delbar}$ leva o espaço $\mathcal{A}^{p,q}(X)$ nele mesmo cada $\Delta_{\delbar} (\alpha^{p,q})$ pertence a $\mathcal{A}^{p,q}(X)$. Como a soma $\mathcal{A}_{\C}^k (X)= \bigoplus_{p+q = k} \mathcal{A}^{p,q}(X)$ é direta segue que cada $\Delta_{\delbar} (\alpha^{p,q}) = 0$, isto é $\alpha^{p,q} \in \mathcal{H}^{p,q}_{\delbar}(X)$. Reciprocamente, se $\alpha^{p,q} \in \mathcal{H}^{p,q}_{\delbar}(X)$ então, por linearidade $\Delta(\alpha) = 0$ e portanto $\alpha \in \mathcal \mathcal{H}^k_{\delbar}(X)$.\\

2. O isomorfismo segue da igualdade $* \Delta_{\del} = \Delta_{\delbar} *$ e do fato de $*: \mathcal{A}^{p,q}(X) \to  \mathcal{A}^{n-q,n-p}(X)$ ser um isomorfismo. Para ver a igualdade acima note primeiro que se $\alpha \in \mathcal{A}^k(X)$ então
\begin{equation*}
\delbar * \del **(\alpha) = (-1)^k \delbar * \del (\alpha) = (-1)^{2n-k} \delbar * \del (\alpha) = ** \delbar * \del (\alpha)
\end{equation*}
pois $ \delbar * \del (\alpha) \in \mathcal{A}^{2n-k}(X)$ e $*^2=(-1)^l\text{id}$ em $\mathcal{A}^l(X)$. Logo $\delbar * \del ** = ** \delbar * \del$ e portanto
\begin{equation*}
* \Delta_{\del} = *(\del \del^* + \del^* \del) = - *\del * \delbar* - ** \delbar * \del = - *\del * \delbar* -  \delbar * \del ** = (\delbar^* \delbar + \delbar \delbar^*)* = \Delta_{\delbar}*
\end{equation*}
\end{proof}

Como no caso riemanniano temos que que $\alpha$ é $\delbar$-harmônica se e somente se $\delbar(\alpha) = 0$ e $\delbar^*(\alpha)=0$ e, em analogia ao lema \ref{lemma:harm-min}, podemos mostrar que as formas $\delbar$-harmônicas são as que tem a menor norma em sua classe de cohomologia de Dolbeaut (veja a equação (\ref{eq:dolbeaut-cohomology}) para a definição).
\begin{lemma} \label{lemma:del-harm-min}
Se $\alpha$ é uma $(p,q)$-forma que é $\delbar$-harmônica então $||\alpha + \delbar \eta|| \geq ||\alpha||$ para toda $\eta \in \mathcal{A}^{p,q-1}(X)$. Reciprocamente, se $\alpha$ tem norma mínima em sua classe de cohomologia de Dolbeaut então $\alpha$ é  $\delbar$-harmônica. 
\end{lemma}

A partir do lema acima obtemos, com as mesmas técnicas usadas para demonstrar o teorema \ref{thm:hodge-decomp}, uma decomposição de Hodge para as formas de tipo $(p,q)$. Uma demonstração pode ser encontrada em \cite{demailly}.

\begin{theorem} \textbf{Decomposição de Hodge para $(p,q)$-formas} \label{thm:herm-hodge-decomp}\\
Se $X$ é uma variedade hermitiana compacta, existem duas decomposições ortogonais
\begin{equation*}
\mathcal{A}^{p,q}(X) = \mathcal{H}_{\del}^{p,q}(X) \oplus \del \mathcal{A}^{p-1,q}(X) \oplus \del^*\mathcal{A}^{p+1,q}(X)
\end{equation*}
e
\begin{equation*}
\mathcal{A}^{p,q}(X) = \mathcal{H}_{\delbar}^{p,q}(X) \oplus \delbar \mathcal{A}^{p,q-1}(X) \oplus \delbar^*\mathcal{A}^{p,q+1}(X).
\end{equation*}
\end{theorem}

\begin{corollary} \label{cor:hodge-pq}
Se $X$ é uma variedade hermitiana compacta então existe um isomorfismo
\begin{equation} \label{eq:hodge-pq}
H^{p,q}_{\delbar}(X) \simeq \mathcal{H}_{\delbar}^{p,q}(X)
\end{equation}
\end{corollary}

\begin{corollary} \label{cor:serre-dual} \textsl{Dualidade de Serre para formas harmônicas}\\
Se $X$ é uma variedade hermitiana compacta então o pareamento \index{Dualidade de Serre!para formas harmônicas}
\begin{equation*}
\mathcal{H}_{\delbar}^{p,q}(X) \times \mathcal{H}_{\delbar}^{n-p,n-q}(X) \longrightarrow \C~,~~ (\alpha,\beta) \longmapsto \int_X \alpha \wedge \beta
\end{equation*}
é não degenerado. Consequentemente existe um isomorfismo $\mathcal{H}_{\delbar}^{p,q}(X) \simeq \mathcal{H}_{\delbar}^{n-p,n-q}(X)^*$
\end{corollary}

\begin{proof}
Se $\alpha \in \mathcal{H}_{\delbar}^{p,q}(X)$ é não nula então $\overline{*\,\alpha} \in \mathcal{H}_{\delbar}^{n-p,n-q}(X)$ (ver proposição \ref{prop:*-harm}) e $\int_X \alpha \wedge \overline{*\,\alpha} = ||\alpha||^2 >0$, mostrando que o pareamento é não degenerado.
\end{proof}

Seja agora $X$ uma variedade complexa compacta. Fixada uma métrica hermitiana em $X$ temos, combinando os dois últimos resultados, que $H_{\delbar}^{p,q}(X) \simeq \mathcal{H}_{\delbar}^{p,q}(X) \simeq \mathcal{H}_{\delbar}^{n-p,n-q}(X)^* \simeq H_{\delbar}^{n-p,n-q}(X)^*$. Do isomorfismo de Dolbeaut (eq. \ref{eq:dolbeaut-iso}) obtemos.

\begin{corollary} \label{cor:serre-dolbeaut-duality} \index{Dualidade de Serre}
Se $X$ é uma variedade complexa compacta de dimensão $n$ então $H^q(X,\Omega_X^p) \simeq H^{n-q}(X,\Omega^{n-p}_X)^*$.
\end{corollary}

O resultado acima também é conhecido como Dualidade de Serre, e possue uma versão um pouco mais geral para fibrados vetoriais holomorfos (veja a Observação \ref{rmk:serre-duality-vb}).

\subsubsection{O caso kähleriano}

Consideremos agora $X$ uma variedade compacta com uma métrica de Kähler $g$. Em particular $g$ é uma métrica hermitiana em $X$ e portanto temos as decomposições
\begin{equation*}
\begin{split}
\mathcal{A}_{\C}^k(X) &= \mathcal{H}_{\C}^k(X) \oplus d \mathcal{A}_{\C}^{k-1}(X) \oplus d^*\mathcal{A}_{\C}^{k+1}(X)\\
\mathcal{A}^{p,q}(X) &= \mathcal{H}_{\del}^{p,q}(X) \oplus \del \mathcal{A}^{p-1,q}(X) \oplus \del^*\mathcal{A}^{p+1,q}(X)\\
\mathcal{A}^{p,q}(X) &= \mathcal{H}_{\delbar}^{p,q}(X) \oplus \delbar \mathcal{A}^{p,q-1}(X) \oplus \delbar^*\mathcal{A}^{p,q+1}(X)
\end{split}
\end{equation*}
bem como os isomorfismos $H^k(X,\C) \simeq \mathcal{H}^k_{\C}(X)$ e $H^{p,q}_{\delbar} \simeq \mathcal{H}_{\delbar}^{p,q}(X)$.

Em princípio, para uma variedade hermitiana qualquer, as decomposições acima não estão relacionadas. No entanto, como $X$ é de Kähler, a igualdade entre os laplacianos $\Delta_{\del} = \Delta_{\delbar} = \frac{1}{2} \Delta$ implica na igualdade entre os espaços
\begin{equation*}
\mathcal{H}_{\C}^k(X) =  \mathcal{H}_{\del}^k(X) = \mathcal{H}_{\delbar}^k(X)
\end{equation*}
que aparecem nas diferentes decomposições. Esse fato permite relacionar o grupo de cohomologia de de Rham $H^k(X,\C)$ com os grupos de cohomologia de Dolbeaut $H_{\delbar}^{p,q}(X)$.

A ferramenta essencial para fazermos essa comparação é o chamado $\del \delbar$-Lema.

\begin{lemma} \index{deldelbarra@$\del \delbar$-Lema} \textbf{$\del \delbar$-Lema}. Se $X$ é uma variedade de Kähler compacta e $\alpha \in \mathcal{A}^{p,q}(X)$ é $d$-fechada são equivalentes
\begin{itemize}
\item[a.] $\alpha$ é $d$-exata, i.e., existe $\beta \in \mathcal{A}^{p+q-1}_{\C}(X)$ tal que $\alpha = d \beta$.
\item[b.] $\alpha$ é $\del$-exata, i.e., existe $\beta \in \mathcal{A}^{p-1,q}(X)$ tal que $\alpha = \del \beta$.
\item[c.] $\alpha$ é $\delbar$-exata, i.e., existe $\beta \in \mathcal{A}^{p,q-1}(X)$ tal que $\alpha = \delbar \beta$.
\item[d.] $\alpha$ é $\del \delbar$-exata, i.e., existe $\beta \in \mathcal{A}^{p-1,q-1}(X)$ tal que $\alpha = \del \delbar \beta$.
\end{itemize}
\end{lemma}

\begin{proof}
Fixe uma métrica de Kähler em $X$. Dada $\alpha \in \mathcal{A}^{p,q}(X)$ uma forma $d$-fechada considere a afirmação
\begin{center}
$(\star)$ $\alpha$ é ortogonal ao espaço das $(p,q)$-formas harmônicas \footnote{Como $X$ é de Kähler, não há necessidade  em especificar com relação a qual operador ($d$, $\del$ ou $\delbar$) a forma $\alpha$ é harmônica.}.
\end{center}
Vamos mostrar primeiro que $(\star) \Leftrightarrow (a)$. Se $\alpha \perp \mathcal{H}^{p,q}_d$, da decompoisção de Hodge para $d$ temos que $\alpha = d \beta + d^* \beta'$. Como $\alpha$ é fechada segue que $0 = d \alpha = d d^* \beta'$ e portanto $(d^*\beta',d^*\beta') = (\beta',dd^*\beta')=0$. Logo $d^* \beta'=0$ e $\alpha = d \beta$. A reciproca segue do fato da decomposição de Hodge ser ortogonal.

Usando a decomposição de Hodge para $\del$ e $\delbar$ obtemos de maneira análoga que $(\star) \Leftrightarrow (b)$ e $(\star) \Leftrightarrow (c)$.

As implicações $(d) \Rightarrow (b)$ e $(d) \Rightarrow (c)$ são triviais e para ver que $(d) \Rightarrow (a)$ basta notar que se $\alpha = \del \delbar \beta$ então $\alpha = (\del + \delbar) \delbar \beta = d (\delbar \beta)$. Portanto só resta mostrar que $(\star) \Rightarrow (d)$.

Suponha $\alpha$ $d$-fechada e ortogonal ao espaço das formas harmônicas. Usando decomposição de Hodge para $\del$ temos que $\alpha = \alpha_0 + \del \beta + \del^* \gamma$ com $\alpha_0$ harmônica. Como $\alpha$ é $d$-fechada, é em particular $\del$-fechada e portanto $0 = \del \del^* \gamma$ e assim $\del^* \gamma = 0$. Como $\alpha$ é ortogonal às formas harmônicas temos que $0 = (\alpha,\alpha_0) = ||\alpha_0||^2$. Logo temos que $\alpha = \del \beta$.

Aplicando agora a decomposição de Hodge para $\delbar$ a $\beta$ temos que $\beta = \beta_0 + \delbar \beta' + \delbar^* \beta''$ com $\beta_0$ harmônica. Temos então que $\alpha = \del \delbar \beta' + \del \delbar^* \beta''$. Como $\alpha$ é $\delbar$-fechada concluimos que $0 = \delbar \del \delbar^* \beta'' = -\delbar \delbar^* \del \beta ''$, onde usamos que $\del \delbar = -\delbar \del$ e $\del \delbar^* = - \delbar^* \del$ (veja a demonstração do Corolário \ref{cor:laplacian}).

Concluimos então que $(\del \delbar^* \beta'',\del \delbar^* \beta'') = (\delbar^* \del \beta'', \delbar^* \del \beta'') = (\del \beta'',  \delbar \delbar^* \del \beta'') = 0$ e portanto $\alpha = \del \delbar \beta'$.
\end{proof}

\begin{example} \label{ex:del-delbar}
Seja $X$ uma variedade de Kähler compacta e $\omega$ e $\omega'$ duas formas de Kähler em uma mesma classe $c \in H^2(X,\R)$. Temos então que $\omega - \omega'$ é $d$-exata e portanto, pelo $\del \delbar$-Lema, existe uma função $\varphi \in C^{\infty}(X,\C)$ tal que $\omega = \omega' + \del \delbar \varphi $. Agora, como $\omega$ e $\omega'$ são reais temos que $\omega = \overline{\omega} = \overline{\omega'} + \overline{\del \delbar \varphi} = \omega' + \delbar \del \overline{\varphi} = \omega' - \del \delbar \overline{\varphi}$. Vemos então que $\overline{\varphi} = - \varphi$, ou seja, $\varphi$ toma valores imaginários puros.

 Escrevendo $\varphi = \smo f$ para $f \in C^{\infty}(X;\R)$ concluimos que duas formas de Kähler $\omega$ e $\omega'$ são cohomólogas se e somente se $\omega = \omega' +\smo \del \delbar f$ para alguma função real $f$.
\end{example}

Denote por $H^{p,q}(X) \subset H^k(X,\C)$ o conjunto de classes em $H^k(X,\C)$ que são representadas por formas de tipo $(p,q)$.
Se $c = [\alpha] \in H^{p,q}(X)$ com $\alpha \in \mathcal{A}^{p,q}(X)$ temos que $d \alpha = 0$ e consequentemente $\delbar \alpha = 0$. Portanto $\alpha$ define uma classe em $H^{p,q}_{\delbar}(X)$. Obtemos assim uma aplicação $H^{p,q}(X) \to H^{p,q}_{\delbar}(X)$, que associa a cada classe de de Rham a sua classe de Dolbeaut.
\begin{lemma}
A aplicação $H^{p,q}(X) \to H^{p,q}_{\delbar}(X)$ definida acima é um isomorfismo.
\end{lemma}
\begin{proof}

Seja $c \in H^{p,q}(X)$ e $\alpha \in c$ uma $(p,q)$-forma. Se a classe de Dolbeaut de $\alpha$ em $H^{p,q}_{\delbar}(X)$ é zero então $\alpha$ é $\delbar$-exata e portanto, pelo $\del \delbar$-lema, $\alpha$ é $d$-exata. Logo $c=0$ e portanto $H^{p,q}(X) \to H^{p,q}_{\delbar}(X)$ é injetora.

Seja agora $a \in H^{p,q}_{\delbar}(X)$ uma classe de Dolbeaut e seja $\eta \in a$ a única representante $\delbar$-harmônica. Como $X$ é de Kähler temos que $\eta$ também é $d$-harmônica. Em particular $\eta$ é fechada e assim $\eta$ define uma classe $[\eta] \in H^{p,q}(X)$ que é levada em $a$ pela aplicação $H^{p,q}(X) \to H^{p,q}_{\delbar}(X)$ acima.
\end{proof}

\begin{proposition} (\textsl{Dualidade de Serre em cohomologia}) \label{prop:serre-duality} \index{Dualidade de Serre!em variedades de Kähler}
Se $X$ é uma variedade de Kähler compacta então existe um isomorfismo $H^{p,q}(X) \simeq H^{n-p,n-q}(X)^*$.
\end{proposition}

\begin{proof}
Usando o lema acima, o isomorfismo (\ref{eq:hodge-pq}) e a dualidade de Serre para formas harmônicas (Corolário \ref{cor:serre-dual}) temos
\begin{equation*}
H^{p,q}(X) \simeq H^{p,q}_{\delbar}(X) \simeq \mathcal{H}_{\delbar}^{p,q}(X) \simeq \mathcal{H}_{\delbar}^{n-p,n-q}(X)^* \simeq H_{\delbar}^{n-p,n-q}(X)^* \simeq H^{n-p,n-q}(X)^*.
\end{equation*}
\end{proof}

\begin{theorem} \index{Teorema!da decomposição de Hodge} \label{thm:hodge-decomposition} \textbf{Decomposição de Hodge para a cohomologia de variedades de Kähler}.

Se $X$ é uma variedade de Kähler compacta então existe uma decomposição
\begin{equation} \label{eq:kahler-hodge}
H^k(X,\C) = \bigoplus_{p+q=k} H^{p,q}(X) ~\text{ com }~H^{p,q}(X) = \overline{H^{q,p}(X)}.
\end{equation}
Além disso, essa decomposição independe da métrica de Kähler escolhida.
\end{theorem}

\begin{proof}
Seja $g$ uma métrica de Kähler em $X$. Compondo os isomorfismos (\ref{eq:complex-hodge}) e (\ref{eq:hodge-pq}) e usando a soma direta (\ref{eq:harm-pq}) obtemos um isomorfismo
\begin{equation*}
H^k(X,\C) \simeq \mathcal{H}^k_{\C}(X) = \bigoplus_{p+q=k} \mathcal{H}^{p,q}(X) = \bigoplus_{p+q=k} \mathcal{H}_{\delbar}^{p,q}(X) \simeq  \bigoplus_{p+q=k} H_{\delbar}^{p,q}(X) \simeq \bigoplus_{p+q=k} H^{p,q}(X)
\end{equation*}
Os isomorfismos acima são de dois tipos, um que associa a uma classe sua representante harmônica e outro que associa a uma forma harmônica sua classe. É facil ver então que a composição dos três isomorfismos acima é na verdade uma igualdade, resultando na decomposição desejada.

Como a própria definição dos subespaços $H^{p,q}(X)$ (classes de de Rham representáveis por formas de tipo $(p,q)$) independe da métrica, a decomposição (\ref{eq:kahler-hodge}) também idependende da escolha de $g$.

Para ver que $H^{p,q}(X) = \overline{H^{q,p}(X)}$ basta notar que se $c \in H^{q,p}(X)$ é representada por $\alpha \in \mathcal{A}^{q,p}(X)$ então $\overline{\alpha} \in \mathcal{A}^{p,q}(X)$ é uma representante de $\overline{c} \in H^{p,q}(X)$.
\end{proof}

\begin{definition}
Os números $h^{p,q}(X) = \dim H^{p,q}(X)$ são chamados de \textbf{números de Hodge} de $X$.
\end{definition}
Do isomorfismo de Dolbeaut \ref{eq:dolbeaut-iso} temos que $h^{p,q}(X) = \dim H^q(X,\Omega_X^p)$.\footnote{Essa identidade pode ser usada para definir os números $h^{p,q}(X)$ quando $X$ não é de Kähler.} Em particular o número $h^{p,0}(X)$ é a dimensão do espaço de $p$-formas holomorfas em $X$.

\begin{corollary}
Seja $X$ uma variedade de Kähler compacta e $b_k(X) = \dim H^k(X,\C),~k=1,2,\ldots$ os números de Betti de $X$. Então
\begin{itemize}
\item[a.] $h^{p,q}(X) = h^{q,p}(X)$
\item[b.] $b_k(X) = \sum_{p+q=k} h^{p,q}(X)$
\item[c.] $b_k(X)$ é par se $k$ é ímpar
\item[d.] $h^{1,0}(X) = \dim H^0(X,\Omega_X) = \frac{1}{2} b_1(X)$ é um invariante topológico
\item[e.] $h^{n-p,n-q}(X)=h^{p,q}(X)$
\item[f.] $h^{p,p}(X) \geq 1$ para $p=1,\ldots,n$.
\end{itemize}
\end{corollary}

\begin{proof}
a. Temos a igualdade $H^{q,p}(X) = \overline{H^{p,q}(X)}$ e a conjugação complexa estabelece um isomorfismo $\C$-antilinear $\overline{H^{p,q}(X)} \simeq H^{p,q}(X)$.

b. Segue imediatamente da existência da soma direta (\ref{eq:kahler-hodge}).

c. Usando os itens a. e b., vemos que se $k=2r+1$ então $b_k(X) = 2 \displaystyle \sum_{p=0}^r h^{p,2r+1-p}(X)$.

d. Dos a. e b. temos que $b_1(X) = h^{1,0}(X) + h^{0,1}(X) = 2h^{1,0}(X)$ e a igualdade $h^{1,0}(X) = \dim H^0(X,\Omega_X)$ segue do isomorfismo de Dolbeaut.

e. Da dualidade de Serre temos que $H^{p,q}(X) \simeq H^{n-p,n-q}(X)^*$. Em particular $h^{n-p,n-q}(X)=h^{p,q}(X)$.

f. Note que a forma de Kähler e suas potências definem classes não triviais $[\omega^p] \in H^{p,p}(X)$ para $p \leq n$ (veja a demonstração da Proposição \ref{prop:cohom-kahler}). Em particular $H^{p,p}(X)\neq 0$;
\end{proof}

A proposição acima mostra que os números de Hodge \index{números de Hodge} uma variedade de Kähler apresenta diversas simetrias. Uma maneira visual de representar estas simetria é através do chamado \textit{diamante de Hodge} \index{diamante de Hodge}
\begin{equation*}
\begin{matrix}
          &          &         & h^{0,0} &         &         &\\
          &          & h^{1,0} &         & h^{0,1} &         &\\
          &  h^{2,0} &         & h^{1,1} &         & h^{0,2} & \\
          &          &         & \vdots  &         &         & \\
 h^{n,0}  &          &         & \cdots  &         &         &h^{0,n}\\
          &          &         & \vdots  &         &         & \\
          &          &h^{n,n-1}&         &h^{n-1,n}&         &\\
          &          &         & h^{n,n} &         &         &  
 \end{matrix}
\end{equation*}
Note que a soma dos números da $k$-ésima linha é igual ao $k$-ésimo numero de Betti $b_k$.

A simetria $h^{p,q} = h^{q,p}$ dada pela conjugação complexa implica que o diamante é simétrico com respeito a refelxão na reta vertical passando por $h^{0,0}$ e $h^{n,n}$. A dualidade de Serre ($h^{p,q} = h^{n-p,n-q}$) mostra que o diamante é invariante pela rotação de $180^o$. Compondo essas duas simetrias temos que $h^{p,q} = h^{n-p,n-q} = h^{n-q,n-p}$, o que mostra que o diamante de Hodge também é invariante pela relexão na reta horizontal passando por $h^{n,0}$ e $h^{0,n}$.

\begin{example} \label{ex:hodge-supriem}
Seja $X$ uma superfície de Riemann compacta. Em particular $X$ é uma variedade diferenciável compacta e orientável de dimensão $2$. Assim, pela classificação das superfícies, $X$ é homeomorfa a uma esfera com $g$ alças. O número $g$ é chamado \textit{genus} de $X$.

A superfície $X$ pode ser representada como um polígono $\mathcal{P}$ de $4g$ lados $a_i,b_i,a'_i,b'_i~i=1,\ldots,g$ com $a_i$ identificado com $a'_i$ com orientação oposta  e fazemos o mesmo com $b_i$ e $b'_i$. Com isso, a imagem dos lados pela aplicação quociente $\pi:\mathcal{P}\to X$ gera livremente o primeiro grupo de homologia com coeficientes inteiros $H_1(X,\Z)$.

Temos então que $b_1(X) = 2g$ e portanto, pelo item d. da proposição acima, segue que $h^{1,0} (X)=g$, isto é, o espaço das $1$-formas holomorfas globais em $X$ tem dimensão $g$. Em particular, se $X$ é homeomorfa a esfera então $X$ não tem formas holomorfas globais.

Da discussão acima vemos também que $\dim H^1(X,\mathcal{O}_X) = g$.
\end{example}

\begin{example} \label{ex:picard-Pn} \index{grupo de Picard!de $\pr^n$} \index{números de Hodge!de $\pr^n$} \textbf{Números de Hodge  e o grupo de Picard do espaço projetivo.}

Os grupos de cohomologia de $\pr^n$ são $H^{2p}(\pr^n,\C) = \C$ para $p=0,\ldots,n$ e $H^k(\pr^n,\C) =0$ caso contrário. Portanto, como $H^{p,p}(\pr^n) \neq 0$ vemos que na decomposição de Hodge $H^{2p}(\pr^n,\C) = \bigoplus^n_{k=0} H^{k,2p-k}(X)$ o único termo que aparece é $H^{p,p}(\pr^n)$ (e este deve ter dimensão $1$). Concluimos então que
\begin{equation*}
h^{p,q}(\pr^n) = \left \lbrace\begin{split}
 &1~~\text{ se }~p=q=1,\ldots,n\\
 &0~~\text{ se }~p \neq q
\end{split}
\right.
\end{equation*}

Em particular vemos que $h^{p,0}(\pr^n) = \dim H^0(\pr^n,\Omega_{\pr^n}^p) = 0$ se $p\geq 1$, ou seja \textbf{$\pr^n$ não possui formas holomorfas globais}.

O conhecimento dos números de Hodge de $\pr^n$ nos permite calcular o grupo de Picard $\Pic(\pr^n)$.

Da sequência exponencial (veja a seção \ref{subsec:expseq}) obtemos a sequência exata
\begin{equation*}
H^1(\pr^n,\mathcal{O}_{\pr^n}) \longrightarrow H^1(\pr^n,\mathcal{O}^*_{\pr^n}) \longrightarrow H^2(\pr^n,\Z) \longrightarrow H^2(\pr^n,\mathcal{O}_{\pr^n})
\end{equation*}
Usando o isomorfismo de Dolbeaut e o fato de que $h^{0,q}(\pr^n) = 0$ temos que $H^q(\pr^n,\mathcal{O}_{\pr^n}) \simeq H^{0,q}(\pr^n) = 0$ e portanto a sequência acima induz um isomorfismo $\Pic(\pr^n)=H^1(\pr^n,\mathcal{O}^*_{\pr^n}) \simeq  H^2(X,\Z) = \Z $, ou seja, \textbf{o grupo de Picard do espaço projetivo é isomorfo a $\Z$}.
Note que o isomorfismo acima é a aplicação $c_1:\Pic(X) \to H^2(X,\R)$ (veja a definição \ref{def:chernclass}) que associa a cada $L$ sua classe de Chern $c_1(L)$. Isso mostra que a classe de Chern $c_1(L)$ determina a estrutura holomorfa de $L$.
\end{example}

\section{Toros complexos associados a variedades de Kähler} \label{sec:jacobi-albanese}

A existência da decomposição de Hodge nos permite associar alguns invariantes contínuos a variedades de Kähler compactas. Nesta seção vamos estudar os exemplos das variedade de Jacobi e de Albanese.

\subsection{A variedade Jacobiana}

\begin{definition} \index{variedade!Jacobiana}
O \textbf{Jacobiano} de uma variedade complexa $X$, denotado por $\Pic^0(X)$, é o núcleo da aplicação $c_1:\Pic(X) \to H^2(X,\Z)$, ou seja, é o espaço dos fibrados de linha com primeira classe de Chern nula. 
\end{definition}

Vimos na seção \ref{subsec:expseq} que quando $X$ é compacta a aplicação $H^1(X,\Z) \to H^1(X,\mathcal{O}_X)$ é injetora, de modo que, da sequência exponencial, temos a seguinte sequência exata
\begin{equation*}
0 \longrightarrow H^1(X,\Z) \longrightarrow H^1(X,\mathcal{O}_X) \longrightarrow \Pic(X) \stackrel{c_1}{\longrightarrow} H^2(X,\Z),
\end{equation*}
de onde vemos que
\begin{equation*}
\Pic^0(X) = \ker c_1 =  \im (H^1(X,\mathcal{O}_X) \to \Pic(X)) \simeq \frac{H^1(X,\mathcal{O}_X)}{H^1(X,\Z)}.
\end{equation*}

Quando $X$ é uma variedade de Kähler compacta podemos dizer ainda mais sobre a estrutura de $\Pic^0(X)$.\\

Antes de seguir vamos relembrar alguns fatos de topologia algébrica.
\begin{remark} \label{rmk:tensor-product}
Se $G$ é um grupo abeliano podemos considerar o produto tensorial $G \otimes_{\Z} \R$, onde vemos aqui $G$ e $\R$ como $\Z$-módulos. Para um grupo finito de ordem $m$ esse produto é trivial, pois $g \otimes x = g \otimes (m m^{-1}x) = (m\cdot g) \otimes (m^{-1}x) = 0\otimes (m^{-1}x) = 0$.

Se $G$ é finitamente gerado temos, da classificação de grupos abelianos finitamente gerados, que $G \simeq \Z^r \times \widetilde{G}$, com $\widetilde{G}$ finito e portanto $G \otimes \R \simeq \Z^r \otimes \R \simeq \R^r$. Temos uma aplicação natural $G \to G \otimes \R$ dada por $g \mapsto g \otimes 1$, cuja imagem em $G \otimes \R \simeq \R$ é o reticulado $\Z^r$.

Podemos aplicar essa construção para os grupos de cohomologia de uma variedade compacta. Obtemos assim uma aplicação
\begin{equation*}
H^k(X,\Z) \longrightarrow H^k(X,\Z) \otimes \R \simeq H^k(X,\R)
\end{equation*}
cuja imagem é um reticulado em $H^k(X,\R)$.
\end{remark}

\begin{remark} \label{rmk:cohomology-coeficients}
O isomorfismo $H^k(X,\Z) \otimes \R \simeq H^k(X,\R)$ exige uma explicação. No nível de co-cadeias de \v{C}ech, se $\mathcal{U}$ é uma cobertura finita, o grupo $\check{C}^k(\mathcal{U},\Z)$ consiste de um produto finito de cópias de $\Z$, enquanto $\check{C}^k(\mathcal{U},\R)$ é um produto do mesmo número de cópias de  $\R$. Temos portanto um isomorfismo $\check{C}^k(\mathcal{U},\Z) \otimes \R \simeq \check{C}^k(\mathcal{U},\R)$, que induz um isomorfismo $\check{H}^k(\mathcal{U},\Z) \otimes \R \simeq \check{H}^k(\mathcal{U},\R)$. Tomando uma cobertura finita que é acíclica para ambos os feixes (por exemplo uma cobertura formada por abertos cujas interseções são contráteis) temos $H^k(X,\Z) \otimes \R \simeq \check{H}^k(\mathcal{U},\Z) \otimes \R \simeq \check{H}^k(\mathcal{U},\R) \simeq H^k(X,\R)$.

O mesmo argumento acima mostra que $H^k(X,\Z) \otimes \mathbf k \simeq H^k(X,\mathbf k)$ se $\mathbf k$ corpo de característica zero
\end{remark}

\begin{proposition} \label{prop:jacobian-torus}
Se $X$ é uma variedade de Kähler compacta então o Jacobiano $\Pic^0(X) = H^1(X,\mathcal{O}_X)/H^1(X,\Z)$ é um toro complexo de dimensão $h^{0,1}(X) = \frac{1}{2}b_1(X)$.
\end{proposition}

\begin{proof}
Note primeiro que a composição
\begin{equation*}
H^1(X,\R) \longrightarrow H^1(X,\C) \longrightarrow H^1(X,\mathcal{O}_X) \simeq H^{0,1}(X)
\end{equation*}
induzida pelas inclusões de feixes $\R \subset \C \subset \mathcal{O}_X$ é um isomorfismo.

Para ver isso considere um elemento $\alpha \in H^1(X,\R)$. Sua imagem em $H^1(X,\C)$, também denotada por $\alpha$, satisfaz $\overline{\alpha} = \alpha$ e portanto, se escrevermos $\alpha = \alpha^{1,0} + \alpha^{0,1}$, devemos ter $\alpha^{0,1} = \overline{\alpha^{1,0}}$.

\begin{lemma} \label{lemma:jacobian-torus}
A aplicação $H^k(X,\C) \to H^k(X,\mathcal{O}_X)$ induzida pela inclusão $\C \subset \mathcal O_X$ coincide, segundo o isomorfismo de Dolbeaut $H^{0,k}(X) \simeq H^k(X,\mathcal{O}_X)$, com a projeção $\Pi^{0,k}: H^k(X,\C) \to H^{0,k}(X)$ dada pela decomposição de Hodge.
\end{lemma}

\begin{proof}
A aplicação natural $\varphi: H^k(X,\C) \to H^k(X,\mathcal{O}_X)$ é obtida do seguinte diagrama comutativo de feixes (veja a seção \ref{sec:derived-cohomology}).
\begin{equation*}
\begin{matrix}
 0 & \longrightarrow  & \C & \longrightarrow & \mathcal{A}^0_{X,\C} &  \stackrel{d}{\longrightarrow}  & \mathcal{A}^1_{X,\C} &  \stackrel{d}{\longrightarrow} & \mathcal{A}^2_{X,\C} & \stackrel{d}{\longrightarrow} \cdots\\
   & &\cap & & \parallel &  &\downarrow & & \downarrow&  \\
0 & \longrightarrow  & \mathcal O_X & \longrightarrow &  \mathcal{A}^0_{X,\C} &  \stackrel{\delbar}{\longrightarrow}  &  \mathcal{A}^{0,1}_X &  \stackrel{\delbar}{\longrightarrow} & \mathcal{A}^{0,2}_X& \stackrel{\delbar}{\longrightarrow} \cdots
 \end{matrix}
\end{equation*}
onde as flechas verticais são as projeções $\Pi^{0,k}:\mathcal A^k_{X,\C} \to \mathcal A^{0,k}_X$.

Para ver que $\varphi$ coincide com a projeção $H^k(X,\C) \to H^{0,k}(X)$ seja $c \in H^k(X,\C)$ e considere $\alpha \in c$ sua representante harmônica com relação a uma métrica de Kähler fixada. Da comutatividade do diagrama temos que $\alpha^{0,k}$ é $\delbar$-fechada e por definição $\varphi(c)$ é sua classe de Dolbeaut, que é justamente a componente $(0,k)$ de $c$ (veja a demonstração do Teorema \ref{thm:hodge-decomposition}). 
\end{proof}

Do lema acima, imagem de $\alpha$ pela composição acima é $\Pi^{0,1}\alpha = \alpha^{0,1}$. Se $\alpha^{0,1} = 0$ temos também que $\alpha^{1,0} = 0$ e portanto $\alpha = 0$, o que mostra que $H^1(X,\R) \to H^{0,1}(X)$ é injetora. Como $\dim_{\R} H^1(X,\R) = b_1(X) = \dim_{\R} H^{0,1}(X)$ segue $H^1(X,\R) \to H^{0,1}(X)$ é um isomorfismo.

Agora, da observação \ref{rmk:tensor-product}, temos que $H^1(X,\Z)$ é um reticulado em $H^1(X,\R)$ e portanto, pelo isomorfismo acima corresponde a um reticulado em $H^1(X,\mathcal{O}_X)$, o que mostra que $\Pic^0(X) = H^1(X,\mathcal{O}_X)/H^1(X,\Z)$ é um toro complexo, e sua dimensão é $\dim H^1(X,\mathcal{O}_X) = h^{0,1}(X)$.
\end{proof}

\begin{remark}
Em geral, se não exigirmos que $X$ seja de Kähler, não podemos garantir que $\Pic^0(X)$ é um toro. Se $X$ é uma superfície de Hopf por exemplo (ver Exemplo \ref{ex:hopf}) temos que $H^1(X,\Z) \simeq \Z$ enquanto $H^1(X,\mathcal{O}_X) \simeq \C$.
\end{remark}

\begin{example}
Se $X$ é uma superfície de Riemann compacta de \textit{genus} $g$ vimos no exemplo \ref{ex:hodge-supriem} que $h^{0,1}(X) = g$ e portanto a variedade Jacobiana de $X$ é um toro complexo de dimensão $g$.

Em particular, se $g=1$ , isto é se $X$ é um toro, então $\text{Pic}^0(X)$ também é um toro com a mesma dimensão de $X$. Esse toros na verdade são isomorfos, isto é, $\Pic^0(X) \simeq X$. Uma demonstração desse fato será dada na próxima seção.
\end{example}

\begin{example} \textbf{Estruturas holomorfas em fibrados de linha.} \label{ex:hol-structures-lb}
Decidir se um fibrado vetorial complexo admite uma estrutura holomorfa não é, em geral, uma tarefa fácil. Pode ocorrer por exemplo que um fibrado admita uma, nenhuma ou até mesmo muitas estruturas holomorfas não isomorfas. Já no caso dos fibrados de linha sobre variedades de Kähler podemos observar alguns exemplos desse fenômeno.

Em geral, a primeira classe de Chern de um fibrado de linha complexo o determina a menos de isomorfismos diferenciáveis. Isso pode ser visto da seguinte maneira.

Considere a sequência exata 
\begin{equation*}
0 \longrightarrow \Z \longrightarrow \mathcal{C}_X^{\infty} \longrightarrow (\mathcal{C}_X^{\infty})^* \longrightarrow 0
\end{equation*}
induzida pela exponencial de funções suaves a valores complexos.

Os fibrados de linha complexos são classificados, a menos de isomorfismos diferenciávies, pelo grupo de cohomologia $H^1(X,(\mathcal{C}_X^{\infty})^*)$ (a demonstração é análoga a da proposição \ref{prop:isopic}). Olhando para a sequencia exata longa associada a sequência exponencial acima temos
\begin{equation*}
H^1(X,\mathcal{C}_X^{\infty}) \longrightarrow H^1(X,(\mathcal{C}_X^{\infty})^*) \longrightarrow H^2(X,\Z) \longrightarrow H^2(X,\mathcal{C}_X^{\infty})
\end{equation*}
onde a segunda aplicação é, por definição, a primeira classe de Chern\footnote{
Note que essa definição é compatível com a definição \ref{def:chernclass} para fibrados de linha holomorfos, pois as inclusões
de feixes $\mathcal O_X \subset \mathcal C_X ^ \infty$ e $\mathcal O_X^* \subset \mathcal (C_X ^ \infty)^*$ induzem um diagrama comutativo de feixes
\begin{equation*}
\begin{matrix}
0 & \longrightarrow  & \Z & \longrightarrow & \mathcal O_X & \longrightarrow  & \mathcal O_X^* & \longrightarrow & 0 \\
   & & \parallel & & \downarrow & & \downarrow & &  \\
0 & \longrightarrow  & \Z & \longrightarrow & \mathcal{C}_X^{\infty} & \longrightarrow  & (\mathcal{C}_X^{\infty})^* & \longrightarrow & 0
 \end{matrix}
\end{equation*}
e portanto as sequências induzidas na cohomologia também comutam, mostrando que a composição da aplicação de cobordo $c_1:H^1(X,(\mathcal{C}_X^{\infty})^*) \to H^2(X,\Z)$ com a aplicação natural $H^1(X,\mathcal O_X^*) \to H^1(X,(\mathcal{C}_X^{\infty})^*)$ é a aplicação de cobordo $c_1:H^1(X,\mathcal O_X^*) \to H^2(X,\Z)$.}. \index{primeira classe de Chern!de um fibrado de linha}

Como o feixe $\mathcal{C}_X^{\infty}$ é acíclico (veja o exemplo \ref{ex:acyclic}) temos que $H^1(X,\mathcal{C}_X^{\infty}) = H^2(X,\mathcal{C}_X^{\infty}) = 0$ e portanto temos que $c_1:H^1(X,(\mathcal{C}_X^{\infty})^*) \to H^2(X,\Z)$ é um isomorfismo, ou seja, a primeira classe de Chern de $L$ determina sua classe de isomorfismo diferenciável.\\

O exemplo \ref{ex:picard-Pn} mostra que todo fibrado de linha sobre $\pr^n$ admite uma única estrutura holomorfa. De fato, como $c_1: \Pic(\pr^n) \to H^2(\pr^n,\Z)$ é um isomorfismo, se $L$ é um fibrado de linha complexo existe um único fibrado holomorfo $L' \in \Pic(X)$ com $c_1(L)=c_1(L')$. Vemos então que $L$ e $L'$ são diferenciavelmente isomorfos e através desse isomorfismo podemos transferir a estrutura holomorfa de $L'$ a $L$.\\

Considere agora $X$ uma variedade de Kähler compacta com $b_1 > 0$. Da proposição \ref{prop:jacobian-torus}, a variedade de Jacobi $\Pic^0(X)$ é um toro complexo de dimensão $b_1$. Sejam $L$ e $L'$ dois pontos distintos em $\Pic^0(X)$. Temos então que $L$ e $L'$ são fibrados de linha que não são holomorficamente isomorfos. No entanto, como $L$ e $L'$ tem a primeira classe de Chern nula existe, pela observação acima, um isomorfismo diferenciável entre eles. Logo $L$ é um exemplo de fibrado de linha que admite mais de uma estrutura holomorfa. Na verdade, vemos que $L$ admite uma família a $b_1$ parâmetros de estruturas holomorfas, uma para cada ponto do jacobiano de $X$.\\

Para obter um exemplo de fibrado de linha que não admite estrutura holomorfa note que a imagem de $\Pic(X) \to H^2(X,\Z)$ é formada pelas classes de cohomologia que são a primeira classe de Chern de algum fibrado de linha holomorfo. Assim, se esta aplicação não for sobrejetora, qualquer elemento $c \in H^2(X,\Z)$ que não esteja na imagem será a classe de Chern de um fibrado que não admite estrutura holomorfa.

Considere a sequencia exata $\Pic(X) \to H^2(X,\Z) \to H^2(X,\mathcal{O}_X)$ induzida pela sequência exponencial. Então  a aplicação $\Pic(X) \to H^2(X,\Z)$ não será sobrejetora se e somente se a aplicação $H^2(X,\Z) \to H^2(X,\mathcal{O}_X)$ é não nula. Agora a aplicação $H^2(X,\Z) \to H^2(X,\mathcal{O}_X) \simeq H^{0,2}(X)$ é a composição da aplicação $H^2(X,\Z) \to H^2(X,\C)$ induzida pela inclusão de feixes $\Z \subset \C$ com a projeção $H^2(X,\C) \to H^{0,2}(X)$ e portanto só será nula se $H^{0,2}(X,\Z) = 0$. Nesse caso teremos também $H^{2,0}(X,\Z) = 0$ e portanto $H^2(X,\Z) = H^{1,1}(X,\Z)$.

Seja então $X$ uma variedade de Kähler com $H^{2,0}(X,\Z) \neq 0$. Podemos tomar por exemplo $X$ como sendo um toro complexo de dimensão maior ou igual a $2$ pois nesse caso, como os grupos de cohomologia são livres, temos que  $H^{2,0}(X,\Z)\otimes \C =  H^{2,0}(X,\C) = H^0(X,\Omega_X^2) \neq 0$ e portanto $H^{2,0}(X,\Z) \neq 0$. Da discussão acima, $X$ possuirá um fibrado de linha que não adimte nenhuma estrutura holomorfa. 

\end{example}

\subsection{A variedade de Albanese}

Existe um outro toro associado a uma variedade de Kähler compacta, o chamado toro de Albanese.

Comecemos com uma observação.
\begin{lemma}
Se $X$ é uma variedade de Kähler compacta então toda forma holomorfa em $X$ é fechada.
\end{lemma}

\begin{proof}
Seja $\alpha = H^0(X,\Omega^p_X)$ uma $p$-forma holomorfa. Então, da proposição \ref{prop:hol-forms}, temos que $\bar{\partial} \alpha = 0$. Por outro lado, como $\alpha$ é uma forma de grau $(p,0)$, segue que $\bar{\partial}^* \alpha = 0$ e portanto $\Delta_{\bar{\partial}} \alpha = 0$. Como $X$ é de Kähler, temos, pelo corolário \ref{cor:laplacian}, que $\Delta \alpha = 2 \Delta_{\bar{\partial}} \alpha = 0$, ou seja, $\alpha$ é harmônica. Em particular $\alpha$ é fechada.
\end{proof}

Devido a essa observação, temos uma aplicação bem definida
\begin{equation*}
\begin{split}
\Phi: H_1(X,\Z) &\longrightarrow H^0(X,\Omega_X)^* \\
[\gamma] &\longmapsto \left( \alpha \mapsto \int_{\gamma} \alpha \right).
\end{split}
\end{equation*}

Denote por $\Lambda$ a imagem de $H_1(X,\Z)$ pela aplicação $\Phi$.
\begin{definition} \index{variedade!de Albanese}
A variedade de albanese de $X$ é o quociente
\begin{equation*}
\text{Alb}(X) = \frac{H^0(X,\Omega_X)^*}{\Lambda}.
\end{equation*}
\end{definition}

\begin{proposition} \label{prop:albanese-torus}
Quando $X$ é uma variedade de Kähler compacta $\text{Alb}(X)$ é um toro complexo de dimensão $h^{1,0}(X)$.
\end{proposition}

\begin{proof}
Primeiramente note que, pela dualidade de Poincaré, temos um isomorfismo $H_1(X,\Z) \simeq H^{2n-1}(X,\Z)$, segundo o qual um elemento $[\gamma] \in H_1(X,\Z)$ corresponde o seu dual de Poincaré $[\eta_{\gamma}]$, caracterizado pela equação $
\int_{\gamma} \alpha = \int_{X} \alpha \wedge \eta_{\gamma}$, para toda 1-forma fechada $\alpha$ (mais detalhes no final da seção \ref{sec:lefschetz-1,1}).

Por um argumento similar ao da demonstração da proposição \ref{prop:jacobian-torus}, vemos que a composição $H^{2n-1}(X,\Z) \to H^{2n-1}(X,\C) \to H^{n-1,n}(X)$ é injetora e a imagem de $H^{2n-1}(X,\Z)$ é um reticulado $\Gamma$ em $H^{n-1,n}(X)$.

Da dualidade de Serre (Proposição \ref{prop:serre-duality}) temos  que $H^0(X,\Omega_X)^*=H^{1,0}(X)^* \simeq H^{n-1,n}(X)$ e esse isomorfismo é compatível com a dualidade de Poincaré, isto é, a inclusão $\Lambda \subset H^0(X,\Omega_X)^*$ corresponde à inclusão $\Gamma \subset H^{n-1,n}(X)$. Temos portanto que
\begin{equation*}
\text{Alb}(X) = \frac{H^0(X,\Omega_X)^*}{\Lambda} \simeq \frac{H^{n-1,n}(X)}{\Gamma},
\end{equation*}
que é um toro complexo de dimensão $\dim H^0(X,\Omega_X)^* = h^{1,0}(X)$.
\end{proof}

Há uma maneira natural de relacionar $X$ e $\text{Alb}(X)$, através da chamada aplicação de Albanese.

Fixado um ponto base $x_0 \in X$ definimos
\begin{equation*}
\begin{split}
\text{alb}:X &\longrightarrow \text{Alb}(X) \\
x & \longmapsto \left( \alpha \mapsto \int_{x_0}^x \alpha \right)
\end{split}
\end{equation*}
onde a integral é ao longo algum caminhho ligando $x_0$ a $x$.

Note que a integral depende do caminho escolhido, mas os funcionais obtidos escolhendo-se dois caminhos distintos diferem pela integral em um caminho fechado, e portanto obtemos um elemento bem definido do quociente $H^0(X,\Omega_X)^*/\Lambda$.

\begin{proposition}
A aplicação de Albanese é holomorfa e o pullback ~$\text{alb}^*:H^0(\text{Alb}(X),\Omega_{\text{Alb}(X)}) \to H^0(X,\Omega_X)$ é um isomorfismo.
\end{proposition}

\begin{proof}
Como a escolha de um outro ponto base só muda a aplicação de albanese por uma translação, basta verificarmos que $\text{alb}$ é holomorfa em uma vizinhança de $x_0$, pois daí seguirá que é holomorfa em toda $X$.

Agora, para $x$ em uma vizinhança convexa $U$ de $x_0$, a aplicação $A:U \to H^0(X,\Omega_X)^*$ dada por $A(x) = (\alpha \mapsto \int_{x_0}^x \alpha)$ está bem definida e $\text{alb}$ é a composição de $A$ com a projeção $ H^0(X,\Omega_X)^* \to \text{Alb}(X)$.

A verificação do seguinte lema é imediata
\begin{lemma}
Se $Y$ é uma variedade complexa e $V$ é um $\C$-espaço vetorial então uma aplicação $f:Y \to V$ é holomorfa se e somente se $\varphi \circ f : Y \to \C$ é holomorfa para toda $\varphi \in V^*$. 
\end{lemma}

Tomando $Y=U$, $V=H^0(X,\Omega_X)^*$ e $\varphi_{\alpha} = \text{aval}_{\alpha}$ temos que $\varphi_{\alpha} \circ A (x) = \int_{x_0}^x \alpha$, que é holomorfa com respeito a $x$. Como todo funcional linear em $H^0(X,\Omega_X)^*$ é da forma $\varphi_{\alpha}$ segue do lema que $A$ é holomorfa em $U$ e consequentemente $\text{alb}$ é holomorfa em $U$.\\

Para a segunda parte note primeiro que, como consequência do teorema fundamental do cálculo, a diferencial $d(\text{alb})_{x_0} : T_{x_0}X \to T_0(\text{Alb}(X)) \simeq H^0(X,\Omega_X)^*$ é a avaliação em $x_0$, i.e., $(d(\text{alb})_{x_0}\cdot v)(\alpha) = \alpha_{x_0}(v)$, ou em uma notação mais sucinta, $d(\text{alb})_{x_0} = \text{aval}_{x_0}$. Como a escolha de um outro ponto base muda $\text{alb}$ pela adição de uma constante segue que $d(\text{alb})_x = \text{aval}_x$ para todo $x \in X$. \footnote{Isso mostra em particular que $d(\text{alb})_x$ é $\C$-linear para todo $x \in X$, dando uma maneira alternativa de mostrar que $\text{alb}$ é holomorfa.}

Observe que se $T = V/\Gamma$ é um toro complexo então $H^0(T,\Omega_T) \simeq V^*$. Uma maneira de ver isso é a seguinte: equipando $V$ com um produto hermitiano e $T$ com a métrica induzida, as $1$-formas harmônicas em $T$ serão combinações da forma $\sum h_j \varphi_j$ com cada $h_j$ harmônica e $\varphi_j \in V^*$, mas como $T$ é compacta, segue que cada $h_j$ é constante.

Temos portanto que $H^0(\text{Alb}(X),\Omega_{\text{Alb}(X)}) \simeq (H^0(X,\Omega)^*)^* \simeq H^0(X,\Omega_X)$ onde o último isomorfismo é o funcional de avaliação $H^0(X,\Omega_X) \ni \alpha \mapsto \text{aval}_{\alpha} \in (H^0(X,\Omega)^*)^*$.

A demonstração fica concluída notando que o isomorfismo obtido desta maneira nada mais é que o pullback $\text{alb}^*$. De fato, se $\varphi = \text{aval}_{\alpha} \in (H^0(X,\Omega_X)^*)^*$ e $v \in T_xX$ então
\begin{equation*}
(\text{alb}^*\varphi)_x \cdot v = \varphi (d \text{alb}_x \cdot v) = (d(\text{alb})_x\cdot v)(\alpha)= \alpha_x (v).
\end{equation*}
\end{proof}

\section{Teoremas de Lefschetz}
Nesta seção demonstraremos dois teoremas sobre a cohomologia de variedades de Kähler compactas devidos a S. Lefschetz. O primeiro deles, chamado de Teorema ``Difícil'' de Lefschetz (\textit{Hard Lefschetz Theorem} em inglês), diz que, em uma variedade de Kähler compacta, a decomposição de Lefschetz para formas diferenciais passa para a cohomologia. O segundo teorema, chamado Teorema das $(1,1)$-classes diz que toda classe de cohomologia integral e de tipo $(1,1)$ é a classe de Chern de um fibrado de linha holomorfo. Veremos que esse resultado está relacionado com a Conjectura de Hodge, um dos problemas em aberto mais importantes da Geometria Algébrica Complexa.

\subsection{O Teorema ``Difícil'' de Lefschetz}

Na seção \ref{subsec:lefschetz} vimos que o espaço das $k$-formas diferenciais em uma variedade hermitiana $X$ admite uma decomposição
\begin{equation*}
\mathcal{A}^k (X) = \displaystyle \bigoplus_{j \geq 0} L^j(\mathcal{P}^{k-2j}(X)),
\end{equation*}
onde $\mathcal{P}^k(X)$ é o espaço das $k$-formas primitivas e $L: \alpha \mapsto \omega \wedge \alpha$ é o operador de Lefschetz. Se $X$ é uma variedade de Kähler compacta, a condição de Kähler permite que essa decomposição passe para a cohomologia. Este é o conteúdo do chamado Teorema ``Difícil'' de Lefschetz, que demonstraremos a seguir.

Seja $\omega \in \mathcal{A}^{1,1}(X)$ uma forma de Kähler em $X$. Como $\omega$ é fechada podemos definir o operador de Lefschetz
\begin{equation*}
\begin{split}
L: H^{p,q}(X) &\longrightarrow H^{p+1,q+1}(X) \\
[\alpha] &\longmapsto [\omega \wedge \alpha]
\end{split}
\end{equation*}

Para definir o adjunto de $L$ primeiramente note que o operador de Hodge age nas formas harmônicas em $X$. De fato, vimos na proposição \ref{prop:*-harm} que o operador de Hodge establece um isomorfismo $\mathcal{H}^{p,q}_{\delbar}(X) \simeq  \mathcal{H}^{n-q,n-p}_{\del}(X)$, mas como $X$ é de Kähler, temos que $\mathcal{H}^*_{\delbar}(X) = \mathcal{H}^*(X) = \mathcal{H}^*_{\del}(X)$ e portanto temos que $*:\mathcal{H}^{p,q}(X) \simeq  \mathcal{H}^{n-q,n-p}(X)$.

Compondo agora o operador de Hodge com o isomorfismo $H^{p,q}(X) \simeq \mathcal{H}^{p,q}(X)$ obtemos isomorfismos no nível de cohomologia $*:H^{p,q}(X) \simeq  H^{n-q,n-p}(X)$. Definimos então
\begin{equation*}
\Lambda: H^{p,q}(X) \longrightarrow H^{p-1,q-1}(X),~~~ \Lambda = *^{-1} \circ L \circ *
\end{equation*}

Note que $L$ e $\Lambda$ são extensões de operadores reais, de modo que podemos restringir $L:H^k(X,\R) \to H^{k+2}(X,\R)$ e $\Lambda:H^k(X,\R) \to H^{k-2}(X,\R)$.

Definimos os espaços de cohomologia primitiva por

\begin{equation} \label{eq:prim-cohomology}
H^k(X,\R)_p = \{c \in H^k(X,\R) : \Lambda c = 0\}
\end{equation}

Note que $\Lambda[\alpha] = [\Lambda(\alpha)]$ onde no membro direito consideramos o dual de Lefschetz agindo em formas. Com isso vemos que $H^k(X,\R)_p$ é o espaço das classes de cohomologia que são representáveis por $k$-formas que são primitivas no sentido usual (veja o Teorema \ref{thm:lefschetz-forms} e a Observação \ref{rmk:lefschetz-forms}).

\begin{theorem} \label{thm:hard-lefschetz} \textbf{Teorema ``Difícil'' de Lefschetz.} \index{Teorema!``Difícil'' de Lefschetz}
Se $X$ é uma variedade de Kähler compacta então
\begin{equation*}
L^{n-k}:H^k(X,\R) \longrightarrow H^{2n-k}(X,\R) 
\end{equation*}
é um isomorfismo para todo $k \leq n$ e existe uma decomposição
\begin{equation} \label{eq:hard-lefschetz}
H^k(X,\R) = \bigoplus_{i \geq 0} L^i H^{k-2i}(X,\R)_p.
\end{equation}
\end{theorem}

\begin{proof}
Seja $\mathcal{H}^k_p(X)=\{\alpha \in \mathcal{H}^k(X) : \Lambda \alpha = 0\}$ o espaço das $k$-formas harmônicas primitivas. Do fato de $H^k(X,\R)_p$ ser o espaço das classes representáveis por formas primitivas e do isomorfismo de Hodge $H^k(X,\R) \simeq \mathcal{H}^k(X)$, a decomposição (\ref{eq:hard-lefschetz}) segue da decomposição
\begin{equation}
\mathcal{H}^k(X) = \bigoplus_{i \geq 0} L^i \mathcal{H}^{k-2i}(X)_p.
\end{equation}
Para demonstrar a decomposição acima note primeiro que o laplaciano comuta com o operador de Lefschetz, i.e., $[L,\Delta]=0$. De fato, como $X$ é de Kähler temos que  $\Delta = 2\Delta_{\delbar}$ e portanto $[L,\Delta]=0$ se e somente se $[\Delta_{\delbar},L] = 0$. Das identidades de Kähler (Teorema \ref{thm:kahler-identities}) temos que $[L,\delbar] = 0$ e $[\delbar^*,L] = \smo \del$ e portanto
\begin{equation*}
\begin{split}
[\Delta_{\delbar},L] &= [\delbar \delbar^*,L] + [\delbar^* \delbar,L] \\
&= \delbar \delbar^*L - L\delbar \delbar^* + \delbar^* \delbar L - L\delbar^* \delbar \\
&= \delbar \delbar^*L - \delbar L \delbar^* + \delbar^* L \delbar - L\delbar^* \delbar \\
&= \delbar[\delbar^*,L] + [\delbar^*,L] \delbar \\
&= \smo (\delbar \del + \del \delbar) \\
&= 0.
\end{split}
\end{equation*}

Usando este fato vemos que o isomorfismo de Hodge é compatível com a decomposição de Lefschetz para formas (cf. Observação \ref{rmk:lefschetz-forms}), de modo que
\begin{equation*}
\mathcal{H}^k(X) = \mathcal{A}^k (X) \cap \ker \Delta = \displaystyle \bigoplus_{i \geq 0} L^i(\mathcal{P}^{k-2i}(X) \cap \ker \Delta ) = \bigoplus_{i \geq 0} L^i \mathcal{H}^{k-2i}(X)_p,
\end{equation*}
resultando na decomposição desejada.

Da decomposição de Lefschetz para formas temos que $L^{n-k}:\mathcal{A}^k(X) \to \mathcal{A}^{2n-k}(X)$ é um isomorfismo para $k \leq n$. Como $L^{n-k} \Delta = \Delta L^{n-k}$ podemos considerar a restrição $L^{n-k}:\mathcal{H}^k(X) \to \mathcal{H}^{2n-k}(X)$. A restrição é claramente injetora. Agora, se $\beta \in \mathcal{H}^{2n-k}(X)$ sabemos que $\beta = L^{n-k} \alpha$ para alguma $\alpha \in \mathcal{A}^k(X)$. Temos então que $L^{n-k} \Delta \alpha = \Delta L^{n-k} \alpha = \Delta \beta =0$ e como $L^{n-k}$ é um isomorfismo vemos que $\Delta \alpha =0$, isto é, $\alpha \in \mathcal{H}^k(X)$, o que mostra que $L^{n-k}:\mathcal{H}^k(X) \to \mathcal{H}^{2n-k}(X)$ é um isomorfismo. Compondo com o isomorfismo de Hodge concluimos que $L^{n-k}:H^k(X,\R) \to H^{2n-k}(X,\R)$ é um isomorfismo. 
\end{proof}

\begin{example}
Seja $X$ uma variedade de Kähler compacta e $Y \subset X$ uma hipersuperfície fechada. Denote por $[\eta_Y] \in H^{1,1}(X,\R)$ a classe fundamental de $Y$ (cf. eq. \ref{eq:fundamental-class} abaixo) e suponha que $\eta_Y$ seja positiva. Vamos mostrar que a aplicação
\begin{equation*}
\iota^*: H^k(X,\R) \longrightarrow H^k(Y,\R)
\end{equation*}
induzida pela inclusão $\iota: Y \to X$ é injetora se $k \leq n-1$.

Como $\eta_Y$ é uma $(1,1)$-forma real positiva e fechada, ela é a forma fundamental de uma métrica de Kähler em $X$ (cf. proposição \ref{prop:positive-form}). O operador de Lefschetz associado é dado por $L_Y: [\alpha] \mapsto [\eta_Y \wedge \alpha]$ e pelo teorema ``difícil'' de Lefschetz $L_Y^{n-k}:H^k(X,\R) \to H^{2n-k}(X,\R)$ é um isomorfismo para $k \leq n$.

Denote por $(\cdot,\cdot):H^k(X,\R) \times H^{2n-k}(X,\R) \to \R$ o pareamento de Poincaré (veja o corólário \ref{cor:poincare-duality}).

Se $[\alpha] \in H^k(X,\R)$, $k\leq n-1$ é tal que $i^*\alpha = \alpha|_Y = 0$  então, para toda $[\beta] \in H^{2n-k}(X,\R)$ temos que 
\begin{equation*}
\begin{split}
([\beta],L^{n-k}[\alpha] ) &= ([\beta],[\alpha \wedge \eta_Y^{n-k}]) = \int_X \beta \wedge \eta_Y^{n-k} \wedge \alpha \\
 &= \int_X \beta \wedge \eta_Y^{n-k-1} \wedge \alpha \wedge \eta_Y \\
 &= \int_Y (\beta \wedge \eta_Y^{n-k-1} \wedge \alpha)\big|_Y \\
 &=0,
\end{split}
\end{equation*}
e portanto, como o pareamento de Poincaré é não degenerado, concluímos que $L^{n-k}[\alpha] = 0$. Como $L^{n-k}$ é um isomorfismo segue que $[\alpha]=0$, mostrando que $i^*$ é injetora.

Esse resultado pode ser melhorado, resultando no chamado Teorema de Hiperplanos de Lefschetz, que será demonstrado no capítulo \ref{ch:stein-lesfchetz}.

\end{example}

\subsection{O Teorema das $(1,1)$-classes e a Conjectura de Hodge} \label{sec:lefschetz-1,1}

Durante toda esta seção $X$ será uma variedade de Kähler compacta.

Dizemos que uma classe de cohomologia $c \in H^k(X,\C)$ é integral se ela pertence à imagem da aplicação natural $H^k(X,\Z) \to H^k(X,\C)$ induzida pela inclusão de feixes $\Z \subset \C$. Denotamos por $H^{p,q}(X,\Z)$ o grupo das $(p,q)$-formas integrais, isto é,
\begin{equation*}
H^{p,q}(X,\Z) = \im (H^{p+q}(X,\Z) \to H^{p+q}(X,\C)) \cap H^{p,q}(X).
\end{equation*}

\begin{lemma}
Se $L$ é um fibrado de linha holomorfo sobre $X$ e $c \in H^2(X,\C)$ denota a imagem de $c_1(L)$ pela aplicação natural $H^2(X,\Z) \to H^2(X,\C)$ então $c \in H^{1,1}(X,\Z)$.
\end{lemma}

\begin{proof}
Considere o seguinte diagrama comutativo
\begin{equation} \label{eq:lefschetz-1,1}
\begin{matrix}
 \Pic(X) & \stackrel{c_1}{\longrightarrow} & H^2(X,\Z)&\longrightarrow& H^2(X,\mathcal O_X)\\
 & & & \searrow ~~~~~~~~~\nearrow &  \\
 & & & H^2(X,\C)&
\end{matrix}
\end{equation}
onde as aplicações na primeira linha são as induzidas pela sequência exponencial e as flechas diagonais são induzidas pelas inclusões $\Z \subset \C \subset \mathcal O_X$.

Dado $L \in \Pic(X)$ temos, da exatidão da primeira linha do diagrama, que a imagem de $c_1(L)$ em $H^2(X,\mathcal O_X)$ é zero. Denote por $c$ a imagem de $c_1(L)$ em $H^2(X,\C)$. Da comutatividade do diagrama vemos que a imagem de $c$ também é zero em $H^2(X,\mathcal O_X)$ e, lembrando que a aplicação $H^2(X,\C) \to H^2(X,\mathcal O_X) \simeq H^{0,2}(X)$ coincide com a projeção no fator $H^{0,2}(X)$ da decomposição de Hodge (veja o lema \ref{lemma:jacobian-torus}), vemos que $c^{0,2} = 0$. Como $c$ é real, i.e., está contida na imagem de $H^2(X,\R) = H^2(X,\C)$ temos que $\bar c = c$ e portanto $c^{2,0} = \overline {c^{0,2}} = 0$, mostrando que $c \in H^{1,1}(X)$, e portanto $c_1(L) \in H^{1,1}(X,\Z)$.
\end{proof}

O lema cima nos diz que há uma aplicação natural $\Pic(X) \to H^{1,1}(X,\Z) \subset H^2(X,\C)$, que associa a cada fibrado de linha holmorfo sua classe de Chern, vista como um elemento de $H^2(X,\C)$. É natural portanto perguntarmos quais elementos de $H^{1,1}(X,\Z)$ são a primeira classe de Chern de um fibrado de linha holomorfo. A resposta é dada pelo Teorema das $(1,1)$-classes de Lefschetz.

\begin{theorem} \label{thm:lefschetz-1,1} \index{Teorema!das $(1,1)$-classes de Lefschetz} \textbf{Teorema das $(1,1)$-classes de Lefschetz.} Se $X$ é uma variedade de Kähler compacta então a aplicação $\Pic(X) \to H^{1,1}(X,\Z)$ é sobrejetora, isto é, todo elemento de $H^{1,1}(X,\Z)$ é a primeira classe de Chern de um fibrado de linha holomorfo sobre $X$.
\end{theorem}

\begin{proof}
Considere novamente o diagrama (\ref{eq:lefschetz-1,1}) e denote por $\rho: H^2(X,\Z) \to H^2(X,\C)$ a aplicação induzida pela inclusão de feixes. Dado um elemento $c \in H^{1,1}(X,\Z)$ temos que $c = \rho(\alpha)$ para alguma $\alpha \in H^2(X,\Z)$. Como $c \in H^{1,1}(X)$ sua componente $(0,2)$ segundo a decomposição de Hodge é zero e portanto, do lema \ref{lemma:jacobian-torus}, sua imagem em $H^2(X,\mathcal O_X)$ é zero. Da comutatividade do diagrama vemos que a imagem de $\alpha$ em $H^2(X,\mathcal O_X)$ é zero e como a linha é exata segue que $\alpha = c_1(L)$ para algum $L \in \Pic(X)$. Temos portanto que $c = \rho(c_1(L)) \in \im \Pic(X) \to H^{1,1}(X,\Z)$.
\end{proof}

\subsubsection{A Conjectura de Hodge}

O Teorema das $(1,1)$-classes de Lefschetz responde parcialmente à famosa Conjectura de Hodge, um dos sete problemas do milênio do ``Clay Mathematics Institute''.

Antes de enunciarmos a Conjectura de Hodge vamos relembrar alguns fatos da topologia algébrica de variedades diferenciáveis.

Seja $M$ uma variedade diferenciável compacta e orientada de dimensão $n$ e $V \subset M$ uma subvariedade fechada de codimensão $k$. O dual de Poincaré de $V$ é a classe de cohomologia $[\eta_V] \in H^k(X,\R)$ da $k$-forma fechada definida pela equação
\begin{equation} \label{eq:fundamental-class}
\int_M \alpha \wedge \eta_V = \int_V \alpha|_V,
\end{equation}
para toda $(n-k)$-forma fechada $\alpha$ (veja por exemplo \cite{bott-tu}, cap 1, \S 5). A classe $[\eta_V]$ é chamada de \textbf{classe fundamental} de $V$.

É possível ver que a classe fundamental de uma subvariedade é integral, isto é, pertence a imagem da aplicação natural $H^k(X,\Z) \to H^k(X,\R)$. Uma maneira de ver isso é definindo a classe fundamental de uma outra maneira, através do chamado pareamento de intersecção (veja por exemplo \cite{g-h}, cap. 0.).\\

Suponha agora que $X$ seja um variedade de Kähler compacta de dimensão $n$ e $V \subset X$ uma subvariedade complexa e compacta de codimensão $p$. Note que $X$ tem dimensão real $2n$ e $V$ dimensão real $2n-2p$.

\begin{lemma} \label{lemma:homology-kahler}
A classe de homologia $[V] \in H_{2n-2p}(X,\Z)$ é não trivial, isto é, $V$ não é um bordo em $X$.
\end{lemma}

\begin{proof}
Fixe uma métrica de Kähler em $X$ e seja $\omega$ a forma de Kähler. Se $V=\del \Omega$ fosse um bordo em $X$ teríamos, pelo teorema de Stokes
\begin{equation*}
\int_V \omega^{n-p} = \int_\Omega d \omega^{n-p} = 0,
\end{equation*}
o que é absurdo pois $\omega^{n-p} = (n-p)! \text{vol}_V$.
\end{proof}

 Da discussão acima, podemos associar a $V$ sua classe fundamental, que será um elemento $[\eta_V] \in H^{2p}(X,\Z)$.
\begin{lemma}
Segundo a decomposição de Hodge, a classe fundamental de $V$ é de tipo $(p,p)$, isto é, $[\eta_V] \in H^{p,p}(X,\Z)$.
\end{lemma}
\begin{proof}
Da decomposição de bigrau para o fibrado de formas (equação \ref{eq:decomp-ext-power}) vemos que $\bigwedge^{2n-2p}_\C V = \bigwedge^{n-p,n-p} V$, pois $\Lambda^{2n-2p}_\C V$ tem posto $1$ e a forma volume de $V$ é um elemento não nulo de $\bigwedge^{n-p,n-p} V$. Para formas isso significa que  $\mathcal A^{2n-2p}_{\C}(V) = \mathcal A^{n-p,n-p}(V)$, isto é, toda forma de grau máximo em $V$ tem bigrau $(n-p,n-p)$. Sendo assim, a restrição de uma forma $\alpha \in \mathcal A^{2n-2p}_{\C}(X)$ a $V$ só não será zero se sua componente em $\mathcal A^{n-p,n-p}(X)$ for não nula.

Analogamente todo elemento de $\mathcal A^{2n}_{\C}(X)$ é de tipo $(n,n)$ e portanto, da equação (\ref{eq:fundamental-class}) e do fato de $V$ não ser um ciclo trivial (Lema \ref{lemma:homology-kahler}) segue que $\eta_V$ é de tipo $(p,p)$ e portanto $[\eta_V] \in H^{p,p}(X,\Z)$.
\end{proof}

A conjectura de Hodge procura descrever, para uma variedade projetiva, que tipos de classe de cohomologia em $H^{p,p}(X,\Z)$ provém de classes fundamentais de subvariedades $V \subset X$.

A primeira conjectura formulada por Hodge dizia que se $X$ é uma variedade projetiva então toda classe em $H^{p,p}(X,\Z)$ é combinação linear com coeficientes inteiros de classes fundamentais de subvariedades de $X$. Colocada dessa maneira a conjectura é falsa. Um contra-exemplo foi dado por Michael Atiyah e Friedrich Hirzebruch, em seu artigo ``Vector bundles and homogeneous spaces'', no qual construiram um exemplo de uma classe de torção em $H^{p,p}(X,\Z)$ que, pelo Lema \ref{lemma:homology-kahler} por exemplo, não pode ser escrita como combinação linear de classes fundamentais.

Após o aparecimento desse contra-exemplo a conjectura foi modificada, e passou a dizer que toda classe em $H^{p,p}(X,\Z)$ que não é de torção é combinação linear com coeficientes inteiros de classes fundamentais de subvariedades de $X$. Com essa formulação a conjectura ainda é falsa. Um contra exemplo foi dado pelo matemático húngaro János Kollár em 1992.

Tendo em vista esses exemplos a formulação atual da conjectura é dada em termos da cohomologia com coeficientes racionais. Denote por $H^{p,p}(X,\mathbb Q) = H^{p,p}(X) \cap H^{2p}(X,\mathbb Q)$. Dizemos que uma classe em $H^{p,p}(X,\mathbb Q)$ é \emph{analítica} se pode ser escrita como combinação linear com coeficientes racionais de classes fundamentais de subvariedades de $X$.\\

\textbf{Conjectura de Hodge.} \index{Conjectura!de Hodge} Se $X$ é uma variedade projetiva então toda classe de cohomologia em $H^{p,p}(X,\mathbb Q)$ é analítica.\\

A solução da Conjectura de Hodge ainda é desconhecida, mas há evidências de que a resposta é afirmativa. Para $p=1$ por exemplo temos, como consequência do Teorema das $(1,1)$-classes de Lefschetz.

\begin{theorem} \label{thm:hodge-p=1}
Seja $X$ uma variedade projetiva então todo elemento de $H^{1,1}(X,\Z)$ é a classe fundamental de algum divisor\footnote{Se $D = \sum a_i [Y_i]$ é um divisor em $X$ definimos a classe fundamental de $D$ por $[\eta_D]=\sum a_i [\eta_{Y_i}]$.} em $X$.
\end{theorem}

A demonstração deste resultado segue quase imediatamente do Teorema das $(1,1)$-classes e dos seguintes fatos, cujas demonstrações podem ser econtradas no capítulo 1 de \cite{g-h}: 1) para variedades projetivas, a aplicação natural $\Div(X) \to \Pic(X)$ é sobrejetora, isto é, todo fibrado de linha holomorfo sobre $X$ é da forma $\mathcal O (D)$ para algum divisor $D$. 2) A primeira classe de Chern do fibrado $\mathcal O(D)$ coincide com a classe fundamental de $D$, isto é, $c_1(\mathcal O(D)) = [\eta_D]$.

Aceitando esses fatos, a demonstração do teorema acima fica fácil. Dada uma classe em $c \in H^{1,1}(X,\Z)$ temos, do Teorema das $(1,1)$-classes de Lefschetz, que $c = c_1(L)$ para algum $L \in \Pic(X)$. Da afirmação 1 acima temos que $L = \mathcal O (D)$ para algum divisor $D$ e usando a afirmação 2 concluímos que $c = c_1(\mathcal O (D)) = [\eta_D]$.\\

Do teorema \ref{thm:hodge-p=1} e do fato que $H^k(X,\mathbb Q) = H^k(X,\Z)\otimes \mathbb Q$ (veja a observação \ref{rmk:cohomology-coeficients}), vemos facilmente que a conjectura de Hodge é válida para o caso $p=1$. Usando o Teorema ``Difícil'' de Lefschetz não é difícil ver que o resultado vale também para $p=n-1$.

Ainda não existem resultados gerais para outros valores de $p$. É sabido que a conjectura vale para uma classe grande de variedades abelianas (i.e. toros projetivos), mas mesmo para essas variedades não há resultados completos. Houve também tentativas de generalização da Conjectura de Hodge para variedades de Kähler, dizendo que, em uma variedade de Kähler compacta, toda classe em $H^{p,p}(X,\mathbb Q)$ é combinação linear de classes de Chern de fibrados vetoriais sobre $X$, mas essa conjectura foi revelada falsa, como mostraram alguns contra-exemplos devidos a Claire Voisin.
\chapter{Geometria Complexa dos Fibrados Vetorais} \label{ch:vb}
Neste capítulo estudaremos com mais profundidade os fibrados vetoriais complexos sobre uma variedade complexa $X$, introduzidos no capítulo \ref{ch:div-lb}.

Um pensamento recorrente na Geometria Diferencial é a introdução de estruturas geométricas nos fibrados vetoriais sobre uma variedade $X$. Nesta seção vamos estudar algumas delas, como métricas hermitianas, conexões e curvatura. Analisaremos suas interrelações e como elas podem ser usadas para entender a geometria complexa de $X$.\\

\subsubsection{Formas diferenciais com valores em um fibrado vetorial} \index{formas diferenciais!com valores em um fibrado vetorial}

Dado um fibrado vetorial complexo $E$ sobre uma variedade $M$, o fibrado de $k$-formas com valores em $E$ é o fibrado vetorial $\bigwedge^k E = \bigwedge^k(TM)^* \otimes E$ e seu feixe de seções é denotado por $\mathcal{A}^k(E)$. Podemos também considerar formas com valores complexos, definindo $\bigwedge_{\C}^k E = \bigwedge^k(T_{\C}M)^* \otimes E$ e o feixe associado $\mathcal{A}_{\C}^k(E)$.

Um elemento típico da fibra $(\bigwedge^k E)_x$ é da forma $\alpha = \sum \alpha_i \otimes s_i$ onde $\alpha_i$ são $k$-formas em $T_x M$ e $s_i$ são elementos da fibra $E_x$. Sendo assim, uma seção $\alpha \in \mathcal{A}^k(E)(U)$ será da forma $\alpha = \sum \alpha_i \otimes s_i$ onde $\alpha_i$ são $k$-formas diferenciáveis sobre $U$ e $s_i$ seções diferenciáveis de $E$. Note que $\mathcal{A}^0(E)$ é o feixe de seções de $E$.

Dada uma $k$-forma complexa usual $\alpha \in \bigwedge^k(T_{\C}M)^*$ e $k$ vetores $v_1,\ldots,v_k$ podemos calcular $\alpha(v_1,\ldots,v_k)$, obtendo assim um número complexo. Agora se $\alpha = \sum \alpha_i \otimes s_i$ é uma $k$-forma com valores em $E$, quando calculamos $\alpha(v_1,\ldots,v_k)$ obtemos um elemento de $E$, dado  por $\alpha (v_1,\ldots,v_k)= \sum \alpha_i (v_1,\ldots,v_k) s_i$.\\

Se $X$ é uma variedade complexa e $E$ um fibrado vetorial sobre $X$, a decomposição $\bigwedge^k_{\C}X =  \bigoplus_{p+q=k} \bigwedge^{p,q} X$ induz as decomposições
\begin{equation}
\textstyle \bigwedge^k_{\C} E  = \displaystyle \bigoplus_{p+q=k} \textstyle \bigwedge^{p,q} E~~~\text{ e } ~~~ \mathcal{A}_{\C}^k(E) = \displaystyle \bigoplus_{p+q=k} \mathcal{A}^{p,q}(E).
\end{equation}
onde $\bigwedge^{p,q} E = \bigwedge^{p,q} X \otimes E$ e $\mathcal{A}^{p,q}(E)$ é seu feixe de seções.

Durante esse capítulo vamos considerar apenas formas diferenciais complexas. Sendo assim vamos abandonar o subscrito $\C$ na notação acima, denotando por $\mathcal{A}^k(E)$ os feixes de formas complexas com valores em $E$.\\

Existem algumas operações que podemos fazer com formas a valores em um fibrado vetorial:

1) Podemos multiplicar a esquerda uma forma diferencial usual por uma forma com valores em $E$, obtendo uma nova forma com valores $E$, isto é, existe uma aplicação
\begin{equation*}
\begin{split}
\wedge: \mathcal{A}_X^k \times \mathcal{A}^l(E) &\longrightarrow \mathcal{A}^{k+l}(E) \\
(\eta, \alpha \otimes s) &\longmapsto (\eta \wedge \alpha) \otimes s.
\end{split}
\end{equation*}

2) Em geral, se $E$ não possui nenhuma estrutura adicional (como por exemplo uma métrica hermitiana) não há uma maneira natural de definir a multiplicação duas formas com valores em $E$. No entanto, podemos multiplicar uma forma com valores em $E$ e uma forma com valores em $E^*$, isto é, existe um pareamento

\begin{equation*}
\begin{split}
\wedge: \mathcal{A}^k(E) \times \mathcal{A}^l(E^*) &\longrightarrow \mathcal{A}_X^{k+l} \\
( \alpha \otimes s, \beta \otimes \varphi ) &\longmapsto \varphi (s) \alpha \wedge \beta
\end{split}
\end{equation*}

3) Se $\varphi: E \to F$ é um morfismo de fibrados vetoriais, temos morfismos induzidos nos fibrados de formas
\begin{equation}
\begin{split}
\varphi: \mathcal{A}^k(E) &\longrightarrow \mathcal{A}^k(F) \\
 \varphi(\alpha \otimes s)&= \alpha \otimes \varphi(s).
\end{split}
\end{equation}

\subsubsection*{O operador $\delbar_E$}

Se $E=X \times \C^r$ é o fibrado trivial de posto $r$, um seção de $E$ é simplesmente uma $r$-upla de funções complexas. Consequentemente, uma seção de $\mathcal{A}^k_{\C}(E)(U)$ é uma $r$-upla de $k$-formas diferenciais complexas $\alpha = (\alpha_1,\ldots,\alpha_r) = \sum \alpha_i \otimes e_i$, onde $e_i$ é seção sobre $U$ que é constante igual ao $i$-ésimo vetor coordenado. Podemos então definir a derivada exterior $d \alpha = (d \alpha_1,\ldots,d\alpha_r) = \sum d\alpha_i \otimes e_i$.

No entanto, para um fibrado vetorial qualquer não podemos adotar tal definição, pois se $\alpha = \sum \alpha_i \otimes s_i \in \mathcal{A}^k_{\C}(E)(U)$, para qualquer função $\lambda:U \to \C^*$ temos $\alpha = \sum \alpha_i \lambda \otimes \lambda^{-1} s_i $ enquanto
\begin{equation*}
\begin{split}
\sum d(\alpha_i \lambda) \otimes \lambda^{-1} s_i &= \sum ( d(\alpha_i) \lambda \otimes \lambda^{-1} s_i + \alpha_i d\lambda \otimes \lambda^{-1} s_i)\\
&= \sum ( d(\alpha_i) \otimes s_i + \alpha_i d\lambda \otimes \lambda^{-1} s_i) \\
&\neq \sum d(\alpha_i) \otimes s_i.
\end{split}
\end{equation*}

Essa discussão sugere que não há uma maneira natural de definir um operador diferencial $d: \mathcal{A}_{\C}^k(E) \to \mathcal{A}_{\C}^{k+1}(E)$ que estenda a definição de diferencial exterior usual. No entanto, quando $E$ é um fibrado holomorfo, podemos fazer esta extensão para a parte $(0,1)$ do operador $d$, isto é, o operador $\delbar$.

\begin{proposition} \label{prop:delbar-E} \index{operador!$\delbar_E$}
Em um fibrado vetorial holomorfo $E$ existe um operador naturalmente definido $\delbar_E:\mathcal{A}^{p,q}(E) \to \mathcal{A}^{p,q+1}(E)$ que satistfaz $\delbar_E^2 = 0$, a regra de Leibniz $\delbar_E(f \alpha) = \delbar (f) \wedge \alpha + f \delbar_E(\alpha)$ e coincide com o operador $\delbar$ quando $E$ é o fibrado trivial.
\end{proposition}

\begin{proof}
A definição de $\delbar_E$ é feita localmente. Seja $U \subset X$ um aberto e $s_1,\ldots,s_r \in H^0(U,E)$ um referencial holomorfo de $E$ sobre $U$. Dada $\alpha \in \mathcal{A}^{p,q}(E)(U)$,  como $\{s_i\}$ é um referencial, ela se escreve de modo único como $\alpha = \sum_i \alpha_i \otimes s_i$ com $\alpha_i \in \mathcal{A}_X^{p,q}(U)$. \footnote{Mais precisamente: sabemos que $\alpha$ é da forma $\alpha = \sum_j \beta_j \otimes t_j$ para $\beta_j \in \mathcal{A}^{p,q}_X(U)$ e $t_j \in \mathcal{A}^0(E)$. Como $s_1,\ldots,s_r$ é um referencial existem funções univocamente determindas $f^i_j$ tal que $t_j = \sum_i f^i_j s_i$ e portanto $\alpha = \sum_{ij}f^i_j \beta_j \otimes s_i$. Denotando $\alpha_i = \sum_j f^i_j \beta_j$}

Defina
\begin{equation} \label{eq:delbar-vb}
\delbar_E \alpha = \sum_i \delbar(\alpha_i) \otimes s_i.
\end{equation}

Essa definição independe da escolha do referencial holomorfo $\{s_i\}$. De fato, se $\{s'_i\}$ é um outro referencial temos que $s_i = \sum_j \psi_{ij} s'_j$, onde $\psi_{ij}: U \to \GL(r,\C)$ é holomorfa e $\alpha$ se escreve nesse referencial como $\alpha = \sum_j (\sum_i \alpha_i \psi_{ij}) \otimes s'_j$. Assim, o operador $\delbar'_E$ definido usando esse novo referencial é
\begin{equation*}
\begin{split}
\delbar'_E \alpha &= \sum_j \delbar (\sum_i \alpha_i \psi_{ij}) \otimes s'_j = \sum_{ij} (\delbar \alpha_i \cdot \psi_{ij} + \alpha_i \delbar \psi_{ij}) \otimes s'_j = \sum_{ij} \delbar \alpha_i \cdot \psi_{ij} \otimes s'_j \\ &= \sum_{ij} \delbar \alpha_i \cdot \otimes  \psi_{ij} s'_j = \sum_i \delbar(\alpha_i) \otimes s_i = \delbar_E \alpha.
\end{split}
\end{equation*}

Podemos assim definir $\delbar_E$ globalmente: cobrimos $X$ por abertos $U_i$ de modo que existam referenciais holomorfos sobre cada $U_i$ e definimos $\delbar_E|_{U_i}$ pela equação (\ref{eq:delbar-vb}). Da discussão acima as definições concordam em todas as interseções $U_i \cap U_j$ e portanto $\delbar_E$ está bem definido em todo $X$.

A propriedade $\delbar_E^2 = 0$ e a regra de Leibniz para $\delbar_E$ seguem diretamente da expressão local (\ref{eq:delbar-vb}) e das  propriedades correspondentes do operador $\delbar$.
\end{proof}

A definição de $\delbar_E$ também pode ser vista de outra maneira. Note que, como $\delbar$ se anula nas funções holomorfas, o operador $\delbar: \mathcal A^{p,q}_X \to A^{p,q+1}_X$ é uma aplicação $\mathcal O_X$-linear e assim podemos definir
\begin{equation*}
\delbar_E = \delbar \otimes \text{id}: \underbrace{\mathcal A^{p,q}_X \otimes_{\mathcal O_X} E}_{\mathcal A^{p,q}(E)} \to \underbrace{A^{p,q+1}_X \otimes_{\mathcal O_X} E}_{\mathcal A^{p,q+1}(E)}.
\end{equation*}

Escrevendo essa definição de $\delbar_E$ em termos de trivializações locais obtemos a expressão (\ref{eq:delbar-vb}) acima.

A discussão feita na seção \ref{sec:dolbeaut} pode ser estendida a esta situação. Para cada $p \geq 0$ fixado, o operador $\delbar_E$ fornece um complexo de feixes $\mathcal{A}^{p,0}(E) \to \mathcal{A}^{p,1}(E) \to \mathcal{A}^{p,2}(E) \to \cdots$. Os grupos cohomologia de Dolbeaut de $E$ são definidos como a cohomologia desse complexo no nível das seções globais
\begin{equation*}
H^{p,q}(X,E) \doteq H^q(\mathcal{A}^{p,\bullet}(X,E),\delbar_E) = \frac{\ker(\delbar_E:\mathcal{A}^{p,q}(E) \to \mathcal{A}^{p,q+1}(E))}{\im(\delbar_E:\mathcal{A}^{p,q-1}(E) \to \mathcal{A}^{p,q}(E))}
\end{equation*}

Como consequência do $\delbar$-Lema de Poincaré (Proposição \ref{prop:poincare-lemma}) podemos ver que a sequência de feixes $\mathcal{A}^{p,0}(E) \to \mathcal{A}^{p,1}(E) \to \mathcal{A}^{p,2}(E) \to \cdots$ é exata e, pela definição de $\delbar_E$, temos que $\ker(\delbar_E:\mathcal{A}^{p,0}(E) \to \mathcal{A}^{p,1}(E)) = \Omega_X^p \otimes E$ é o feixe de $p$-formas holomorfas com valores em $E$. Vemos portanto que o complexo $(\mathcal{A}^{p,\bullet}(E),\delbar_E)$ é uma resolução do feixe $\Omega_X^p \otimes E$. Como os feixes $\mathcal{A}^{p,q}(E)$ são acíclicos, obtemos o isomorfismo de Dolbeaut para fibrados vetoriais:
\begin{equation}
H^{p,q}(X,E) \simeq H^q(X,\Omega_X^p \otimes E).
\end{equation}

\begin{remark} \label{rmk:serre-duality-vb}
Um teorema de dualidade devido a Jean-Pierre Serre \index{Dualidade de Serre} diz que o dual topológico do espaço $H^q(X,\Omega_X^p \otimes E)$ é naturalmente isomorfo a $H^{n-q}_c(X,\Omega_X^{n-p} \otimes E^*)$, o espaço de cohomologia de $X$ com coeficientes em $\Omega^{n-p}_X \otimes E$ e suporte compacto.

Note que quando $X$ é compacta e $E=\mathcal O_X$ é o fibrado trivial recuperamos o resultado do corolário \ref{cor:serre-dolbeaut-duality}. Diferentemente da demonstração apresentada na seção \ref{sec:hodge-decomp}, a demonstração da Dualidade de Serre pode ser feita, como em \cite{serre-dualite}, de maneira puramente analítica, sem necessitar da introdução de uma métrica hermitiana em $X$.
\end{remark}

\section{Estruturas hermitianas}

Dado um fibrado vetorial complexo $E$ sobre uma variedade $M$, uma \textit{estrutura hermitiana em $E$} \index{métrica!hermitiana!em um fibrado complexo} é um produto hermitiano $h_x$ em cada fibra $E_x$ que varia diferenciavelmente no seguinte sentido: se $\psi:E|_U \to U \times \C^r$ é uma trivialização então as funções $h_{ij}(x) = h_x(\psi_x^{-1}(e_i),\psi_x^{-1}(e_j))$ são diferenciáveis em $U$.

Se $h$ é uma estrutura hermitiana em $E$ diremos que o par $(E,h)$ (ou simplesmente $E$) é um \textit{fibrado vetorial hermitiano}.

Dado um fibrado hermitiano $(E,h)$ de posto $r$, dizemos que um conjunto de seções locais $s_1,\ldots,s_r:U \to E$ é um \textit{referencial unitário} se $h_x(s_i(x),s_j(x)) = \delta_{ij}$ para todo $x \in U$. É fácil ver, usando o processo de Gram-Schmidt, que referenciais unitários locais sempre existem.\\

\begin{remark}
Assim como uma métrica hermitiana determina referenciais unitários locais, os referenciais locais determinam a estrutura hermitiana. De fato: se $r,t \in E_x$ e $s_1,\ldots,s_r$ é um referencial local em torno de $x$ podemos escrever $r = \sum a_i s_i(x)$ e $t = \sum b_j s_j(x)$ e portanto teremos que $h_x(r,t) = \sum a_i \bar{b}_i$.
\end{remark}

\begin{example}
Seja $X$ uma variedade complexa e $J:TX\to TX$ a estrutura complexa induzida. Vimos no capítulo \ref{ch:kahler} que um produto hermitiano em $TX$ é equivalente a uma métrica riemanniana em $X$ compatível com $J$ (veja a observação \ref{rmk:herm-riem}). Sendo assim uma métrica hermitiana no fibrado complexo $(TX,J)$ é simplesmente uma métrica hermitiana em $X$.
\end{example}
\begin{example}
Se $E = U \times \C^r$ então uma estrutura hermitiana em $E$ é dada por $h_x(u,v) = u^t H(x) \bar{v}$, onde $H$ é uma função diferenciável $H:U \to \GL(r,\C)$ tal que $H(x)$ é uma matriz hermitiana positiva definida para todo $x \in U$.
\end{example}

\begin{example} \label{ex:herm-local}
Do exemplo anterior vemos que se $E$ é um fibrado hermitiano qualquer, a estrutura hermitiana é dada, localmente, por funções diferenciáveis com valores no espaço das  matrizes hermitianas positivas definidas. Mais precisamente, se $\psi:E|_U \to U \times \C^r$ é uma trivialização então a métrica em $U$ é dada por $h_x(u,v) = \psi_x(u)^t H(x) \overline{\psi_x (v)}$.

Uma outra trivialização sobre $U$ será da forma $\psi' = F \circ \psi$ com $F(x,v) = (x,\varphi(x)\cdot v)$ e $\varphi:U \to GL(r,\C)$. Com relação a $\psi'$ a métrica é dada por $h_x(u,v) = \psi'_x(u)^t H'(x) \overline{\psi'_x (v)}$, de onde vemos que
\begin{equation*}
H = \varphi^t H' \bar{\varphi}.
\end{equation*}
Vemos então que uma maneira alternativa de definir uma estrutura hermitiana em $E$ é dar, para cada aberto de uma cobertura trivializante $\{U_i\}$, um função $H_i: U_i \to \GL(r,\C)$ cujos valores são matrizes hermitianas positivas definidas e satisfazem a condição de compatibilidade
\begin{equation} \label{eq:herm-cocycles}
H_j = \varphi_{ij}^t H_i \bar{\varphi}_{ij}
\end{equation}
\end{example}

\begin{example} \textbf{Métricas hermitianas em fibrados de linha.}
Seja $L$ um fibrado de linha holomorfo sobre $X$. Pela equação (\ref{eq:herm-cocycles}), uma métrica hermitiana em $L$ é dada por uma coleção de funções $h_i:U_i \to \C^*$ satisfazendo $h_j = |\varphi_{ij}|^2 h_i$, onde $\varphi_{ij}$ são os cociclos de $L$ com relação a cobertura $\{U_i\}$.

Dizemos que um conjunto de seções globais $s_1,\ldots,s_N \in H^0(X,L)$ geram $L$ se para todo $x \in X$ algum $s_j(x) \neq 0$. Dado tal conjunto considere  a função
\begin{equation*}
h_i(x) = \frac{1}{\displaystyle \sum_{\alpha=1}^N \big |\varphi_i(s_{\alpha}(x)) \big |^2}
\end{equation*}
onde $\varphi_i$ é uma trivialização local em $U_i$.

É claro que as funções acima satisfazem $h_j = |\varphi_{ij}|^2 h_i$ e portanto definem uma métrica hermitiana $h$ em $L$. Dizemos que $h$ é a métrica induzida pelas seções $s_1,\ldots,s_N$.\\

Um exemplo importante é o do fibrado $\mathcal{O}(1)$ sobre $\pr^n$. Vimos no exemplo \ref{ex:sec-Ok} que os polinômios lineares $z_0,\ldots,z_n$ podem ser vistos como seções holomorfas de $\mathcal{O}(1)$. Explicitamente, sobre uma reta $\ell \in \pr^n$, a seção $z_j(\ell) \in \ell^*$ associa a cada $v \in \ell$ a sua $j$-ésima coordenada. É claro que essas seções não se anulam simultaneamente em nenhum ponto de $\pr^n$, de modo que $\{z_0,\ldots,z_n\} \subset H^0(\pr^n, \mathcal{O}(1))$ gera $\mathcal{O}(1)$. Temos então uma métrica hermitiana associada, dada por
\begin{equation} \label{eq:metric-O1}
h_i = \frac{1}{\displaystyle \sum_{\alpha=0}^n \bigg | \frac{z_{\alpha}}{z_i} \bigg|^2} = \frac{1}{1 + \displaystyle \sum_{\alpha \neq i} \bigg | \frac{z_{\alpha}}{z_i} \bigg|^2} ~~ \text{ em } ~~ U_i = \{z_i \neq 0\}
\end{equation} 
\end{example}

\subsubsection{Métricas hermitianas induzidas}
Uma métrica hermitiana $h$ em um fibrado complexo $E$ induz um isomorfismo $\C$-antilinear $\Phi: E \to E^*$, dado por $\Phi(s) = h(\cdot,s) \doteq s^*$. Com isso podemos definir uma métrica $h^*$ em $E^*$ pela fórmula
\begin{equation*}
h^*(s^*,t^*) = \overline{h(s,t)}~~s^*,t^* \in E^*
\end{equation*}

Existem também métricas induzidas na soma direta e no produto tensorial de fibrados.

Se $(E_1,h_1)$ e $(E_2,h_2)$ são fibrados hermitianos, definimos a métrica $h_1 \oplus h_2$ em $E_1 \oplus E_2$ por
\begin{equation*}
(h_1 \oplus h_2)(s_1 + s_2, t_1 + t_2) = h_1(s_1,t_1) + h_2(s_2,t_2),~~s_1,t_1 \in E_1,~ s_2,t_2 \in E_2
\end{equation*}
e a métrica $h_1 \otimes h_2$ em $E_1 \otimes E_2$ pela por
\begin{equation*}
(h_1 \otimes h_2) (s_1\otimes s_2,t_1\otimes t_2) = h_1(s_1,t_1)h_2(s_2,t_2),~~s_1,t_1 \in E_1,~ s_2,t_2 \in E_2
\end{equation*}

Considere agora uma sequência exata de fibrados complexos
\begin{equation*}
0 \longrightarrow F \stackrel{\iota}{\longrightarrow} E \stackrel{\varphi}{\longrightarrow} G \longrightarrow 0.
\end{equation*}
e suponha que $E$ possui uma métrica hermitiana $h$.

Identificando $F$ com sua imagem por $\iota$, podemos vê-lo como um subfibrado de $E$ e portanto a restrição $h|_F$ define uma métrica hermitiana em $F$.

Denote por $F^\perp = \{s \in E : h(s,t)=0, \text{ para todo } t \in F\}$ o complemento ortogonal de $F$ em $E$. É fácil ver que $F^\perp$ é um subfibrado de $E$ e que $E = F \oplus F^\perp$. Sendo assim, o morfismo $\varphi$ se restringe a um isomorfismo $F^{\perp} \simeq G$ e portanto a restrição $h|_{F^\perp}$ define uma métrica em $G$.

\begin{example}
Considere a sequência de Euler em $\pr^n$ (veja a proposição \ref{prop:euler-seq})
\begin{equation*}
0 \longrightarrow \mathcal O_{\pr^n} \longrightarrow \mathcal O(1)^{\oplus(n+1)} \longrightarrow \mathcal T_{\pr^n} \to 0.
\end{equation*}
Tomando o produto tensorial com o fibrado tautológico $\mathcal O(-1)$ obtemos a sequência exata
\begin{equation*}
0 \longrightarrow \mathcal O (-1) \longrightarrow \mathcal O_{\pr^n}^{\oplus(n+1)} \longrightarrow \mathcal T_{\pr^n} \otimes \mathcal O(-1) \to 0.
\end{equation*}

O fibrado trivial $\mathcal O_{\pr^n}^{\oplus(n+1)} = \pr^n \times \C^{n+1}$ possui uma métrica canônica constante dada por $h_x(v,w) = (v,w)$, onde $(\cdot,\cdot)$ é o produto hermitiano padrão em $\C^{n+1}$. Da discussão acima, $h$ se restringe a uma métrica em\footnote{Da demonstração da proposição \ref{prop:euler-seq}, vemos que a aplicação induzida $\mathcal O (-1) \to \mathcal O_{\pr^n}^{\oplus(n+1)}$ coincide com a inclusão $\mathcal O (-1) \subset \pr^n \times \C^{n+1}$} $\mathcal O (-1) \subset \pr^n \times \C^{n+1}$ e também induz uma métrica em $\mathcal T_{\pr^n} \otimes \mathcal O (-1) \simeq \mathcal O (-1)^\perp$.

A aplicação $\varphi:\mathcal O_{\pr^n}^{\oplus(n+1)} \to \mathcal T_{\pr^n} \otimes \mathcal O(-1)$ é obtida tomando o produto tensorial da aplicação $\mathcal O(1)^{\oplus(n+1)} \to \mathcal T_{\pr^n}$ com $\id_{\mathcal O(-1)}$ e portanto é dada, sobre um ponto $x \in \pr^n$, por $v \mapsto d \pi_z(v) \otimes z$, para algum $z \in x$.

Podemos supor, sem perda de generalidade, que $z \in S^{2n+1}$ e portanto, usando a proposição \ref{prop:FS-submersion}, vemos que, para $v \in \mathcal O (-1)^\perp$,
\begin{equation*}
|v|^2 = \pi \cdot h_{FS}(d\pi_z(v)) = \pi \cdot h_{FS}(d\pi_z(w)) |z| = \pi \cdot (h_{FS} \otimes h)(d \pi_z(v) \otimes z),
\end{equation*}
o que mostra que a métrica induzida pelo isomorfismo $\varphi:\mathcal O (-1)^\perp \to \mathcal T_{\pr^n} \otimes \mathcal O (-1)$ é $\pi \cdot h_{FS} \otimes h$.
\end{example}

\section{Conexões} \label{sec:connections}
No começo do capítulo vimos que não há uma maneira natural de derivar seções de um fibrado. A ideia de uma conexão é prover uma maneira de definirmos tais derivadas. Veremos que podemos impor condições de compatibilidade de uma conexão com as demais estruturas do fibrado e assim, em alguns casos, veremos que existem conexões privilegidas.

\begin{definition} \label{def:connection} \index{conexão!em um fibrado complexo}
Uma conexão em um fibrado vetorial $E$ é um morfismo de feixes $\nabla: \mathcal{A}^0(E) \to \mathcal{A}^1(E)$ que é $\C$-linear e satisfaz a regra de Leibniz
\begin{equation} \label{eq:leibniz}
\nabla(f s) = df \otimes s + f \nabla(s),
\end{equation}
onde $f$ é um função e $s$ é uma seção de $E$, localmente definidas\footnote{
Observação: Alguns autores definem uma conexão agindo apenas sobre as seções globais de $E$, isto é, como uma aplicação $\C$-linear $\nabla:\mathcal{A}^0(E)(X) \to \mathcal{A}^1(E)(X)$ satisfazendo a regra de Leibniz. Essa definição é equivalente à definição \ref{def:connection}, como veremos a seguir.

Se $\nabla$ é uma conexão em $E$ temos, em particular, uma derivação no espaço das seções globais $\nabla:\mathcal{A}^0(E)(X) \to \mathcal{A}^1(E)(X)$. Reciprocamente, dada uma aplicação $\widetilde{\nabla}:\mathcal{A}^0(E)(X) \to \mathcal{A}^1(E)(X)$ que é $\C$-linear e satisfaz a regra de Leibniz (\ref{eq:leibniz}), podemos definir uma conexão no sentido da definição acima da seguinte maneira.

Seja $U \subset M$ um aberto e considere $s \in \mathcal{A}^0(E)(V)$, onde $V$ um aberto tal que $\overline V \subset U$. Seja $\widetilde s \in \mathcal{A}^0(E)(X)$ uma extensão de $s$ que é trivial fora de $U$. Defina então a aplicação $\nabla^V: \mathcal{A}^0(E)(V) \to \mathcal{A}^1(E)(V)$ por $\nabla^V(s) = \widetilde \nabla (\widetilde s)|_V$. Uma maneira de construir tal extensão é a seguinte: considere $f$ uma função diferenciável em $X$ tal que $f = 1$ em $V$ e $f=0$ em $M \setminus U$ e defina $\widetilde s = f s\in \mathcal{A}^0(E)(X)$.

Note que a definição de $\nabla^V$ independe da extensão de $s$. De fato, por linearidade basta mostrar que se $s \in \mathcal A^0(E)$  e $s|_U = 0$ então $\widetilde \nabla (s)|_V = 0$ e isso vale pois, se $f$ é como acima então $fs = 0$ em $X$  e portanto $0 = \widetilde \nabla(fs) = df \otimes s + f \widetilde \nabla (s)$. Usando o fato de $f$ ser constante igual a $1$ em $V$ obtemos $0 = \widetilde \nabla (s)|_V$.

Temos assim derivações  $\nabla^V: \mathcal{A}^0(E)(V) \to \mathcal{A}^1(E)(V)$ para cada aberto com $\overline V \subset U$. Cobrindo $U$ por tais abertos podemos definir $\nabla:  \mathcal{A}^0(E)(U) \to \mathcal{A}^1(E)(U)$ e por um argumento análogo ao do parágrafo anterior vemos que a definição independe da cobertura escolhida. Definimos assim um morfismo de feixes $\nabla: \mathcal{A}^0(E) \to \mathcal{A}^1(E)$, que é uma conexão no sentido da definição \ref{def:connection}
} .
\end{definition}

Dado um campo de vetores localmente definido $X$ podemos avaliar a $1$-forma $\nabla(s)$ em $X$, obtendo assim uma seção $\nabla(s)(X) \doteq \nabla_X (s) \in \mathcal{A}^0(E)$. Pensamos em $\nabla_X(s)$ como sendo a derivada direcional, ou derivada covariante da seção $s$ na direção do vetor $X$. Note que, por (\ref{eq:leibniz}), temos a regra de derivação
\begin{equation} \label{eq:leibniz2}
\nabla_X(f s) = X(f) \cdot s + f \nabla_X(s).
\end{equation}

\begin{example}
Seja $E = M \times \C^r$ o fibrado trivial. Uma seção de $E$ é da forma $s = (f_1,\ldots,f_r)$, onde cada $f_j$ é uma função diferenciável localmente definida em $M$. Pela regra de Leibniz usual temos que $\nabla s = (d f_1,\ldots,d f_r)$ define uma conexão em $E$. 
\end{example}

Se $\nabla$ e $\nabla'$ são duas conexões em $E$ temos, para toda seção $s$ e função $f$, que
\begin{equation*}
(\nabla - \nabla')(fs) = df \otimes s + f \nabla(s) -  df \otimes s - f \nabla'(s) = f(\nabla - \nabla')(s),
\end{equation*}
e portanto podemos ver $\nabla - \nabla'$ como uma $1$-forma com valores em $\End(E)$. Mais precisamente, $\nabla - \nabla	 \in \mathcal{A}^1(\End(E))$ é a $1$-forma que, avaliada em um vetor tangente $X$, fornece o endomorfismo $s \mapsto \nabla_X s - \nabla'_X s$.

Reciprocamente, se $\nabla$ é uma conexão em $E$ e $a \in \mathcal{A}^1(\End(E))(M)$ então $\nabla + a$ é uma conexão em $E$. Isso mostra que o espaços das conexões em $E$ é um espaço afim modelado no espaço vetorial $\mathcal{A}^1(\End(E))(M)$.

\begin{example} \label{ex:trivial-connection}
No caso do fibrado trivial $E = X \times \C^r$ temos que $\End(E) = X \times \mathfrak{gl}(r,\C)$ onde $\mathfrak{gl}(r,\C)$ é o espaço das matrizes $r \times r$ com coeficientes complexos. Sendo assim o espaço \-$\mathcal{A}^1(\End(E))(M)$ é formado por matrizes cujas entradas são $1$-formas em $M$.

Como a diferencial exterior $d$ define uma conexão em $E$, vemos que qualquer outra conexão será da forma $\nabla s = ds + A\cdot s$, onde $A$ é uma matriz de $1$-formas que age em $s = (s_1,\ldots,s_r)$ por multiplicação. Em suma, vemos que qualquer conexão em $E$ é da forma $\nabla = d + A$.
\end{example}

Utilizando o exemplo acima e as trivializações locais podemos descrever as conexões localmente. Seja $\nabla$ uma conexão em $E$. Uma trivialização de $E$ sobre um aberto $U_i$ fornece um isomorfismo $\psi_i:E_{U_i} \simeq U_i \times \C^r$ e portanto temos uma conexão $\nabla_i$ induzida no fibrado trivial $U_i \times \C^r$ dada explicitamente por $\nabla_i(s) = \psi_i \nabla(\psi_i^{-1}(s))$. Do exemplo \ref{ex:trivial-connection} vemos então que $\nabla_i = d + A_i$, onde $A_i$ é uma matriz de $1$-formas definidas em $U_i$. Agora, se $s \in \mathcal{A}^0(E)(U_i \cap U_j)$ temos que $\nabla (s) = \psi_j^{-1} \circ (d + A_j) \circ \psi_j(s)$ e por outro lado
\begin{equation*}
\begin{split}
\nabla (s) &= \psi_i^{-1} \circ (d + A_i) \circ \psi_i(s)\\
& = \psi_j^{-1} \psi_{ij}^{-1} \circ (d + A_i) \psi_{ij} \circ \psi_j(s) \\
&=  \psi_j^{-1} \psi_{ij}^{-1} \circ d (\psi_{ij} \circ \psi_j(s)) + \psi_j^{-1} \psi_{ij}^{-1} \circ (A_i \psi_{ij}) \circ \psi_j(s)\\
& = \psi_j^{-1} \psi_{ij}^{-1} \circ (d \psi_{ij} + \psi_{ij}d) \circ \psi_j(s) +  \psi_j^{-1} \psi_{ij}^{-1} \circ (A_i \psi_{ij}) \circ \psi_j(s) \\
&=  \psi_j^{-1} \circ (d + \psi_{ij}^{-1}d \psi_{ij} + \psi_{ij}^{-1}A_i \psi_{ij} )  \circ \psi_j(s)
\end{split}
\end{equation*}
e portanto vemos que as matrizes $A_i$ devem satisfazer as condições de compatibilidade
\begin{equation} \label{eq:local-connection}
A_j = \psi_{ij}^{-1}d \psi_{ij} + \psi_{ij}^{-1}A_i \psi_{ij}
\end{equation}
em $U_i \cap U_j$.

Reciprocamente, dadas matrizes de $1$-formas $A_i$ definidas em $U_i$ e satisfazendo as condições (\ref{eq:local-connection}) então podemos definir uma conexão em $E$ por $\nabla|_{U_i} = \psi_i^{-1} \circ (d + A_i) \circ \psi_i$.

Essa descrição da conexão em coordenadas é bastante utilizada na física. No jargão dos físcos, a escolha de uma trivialização é chamada de \textit{gauge}, e a matriz de um formas $A_i$ é chamada de potencial vetor.

\subsubsection{Conexões induzidas}

Se $E_1$ e $E_2$ são fibrados vetoriais com conexões $\nabla_1$ e $\nabla_2$ temos conexões naturalamente induzidas na soma direta $E_1 \oplus E_2$:
\begin{equation*}
\nabla(s_1 \oplus s_2) = \nabla_1(s_1) + \nabla_2 (s_2) \in \mathcal{A}^1(E_1) \oplus \mathcal{A}^1(E_2) = \mathcal{A}^1(E_1 \oplus E_2)
\end{equation*}
e no produto tensorial $E_1 \otimes E_2$:
\begin{equation*}
\nabla(s_1 \otimes s_2) = \nabla_1(s_1) \otimes s_2 + s_1 \otimes \nabla_2 (s_2) \in \mathcal{A}^1(E_1) \otimes \mathcal{A}^1(E_2) = \mathcal{A}^1(E_1 \otimes E_2).
\end{equation*}

Dado um fibrado vetorial $E$ com uma conexão $\nabla$, existem também conexões naturalmente induzidas em $E^*$ e $\End(E)$. Para justificar a naturalidade dessas conexões vamos considerar um contexto um pouco mais geral.

Definimos os fibrados
\begin{equation*}
\mathcal{T}^{(r,s)}_E = \overbrace{E \otimes \cdots \otimes E}^{r \text{ vezes}} \otimes \overbrace{E^* \otimes \cdots \otimes E^*}^{s \text{ vezes}} ~~~ \text{ e } ~~~ \mathcal{T}_E = \bigoplus_{r,s\geq0} \mathcal{T}^{(r,s)}_E,
\end{equation*}
onde adotamos a convenção de que $\mathcal{T}^{(0,0)} = M \times \C$ é fibrado trivial. Os casos que nos interessarão mais são $\mathcal{T}^{(0,1)}_E = E^*$ e $\mathcal{T}^{(1,1)}_E = E \otimes E^* \simeq \End(E)$.

Para cada par de índices $(i,j)$ com $1\leq i \leq r$ e $1 \leq j \leq s$ existe uma aplicação
\begin{equation*}
\begin{split}
C^{i,j}:\mathcal{T}^{(r,s)}_E &\longrightarrow \mathcal{T}^{(r-1,s-1)}_E\\
s_1 \otimes \cdots \otimes s_r \otimes t^1 \otimes \cdots t^s &\longmapsto t^j(s_i)~s_1 \otimes \cdots \hat{s_i} \cdots \otimes s_r \otimes t^1 \otimes \cdots \hat{t^i} \cdots \otimes t^s
\end{split}
\end{equation*}
chamada contração.

Note que, no caso $r=s=1$ só há uma contração, que nada mais é que a avaliação $s \otimes t \mapsto t(s)$.\\

A proposição a seguir mostra que uma conexão em $E$ se estende naturalmente a conexões nos produtos $\mathcal T^{(r,s)}_E $. Uma demonstraçao para o caso em que $E=TM$ pode ser encontrada no capítulo 2 de \cite{k-n} e a demonstração lá apresentada pode ser imediatamente adaptada para o caso geral.

\begin{proposition} \label{prop:tensor-connection}
Seja $\nabla$ uma conexão em $E$. Então $\nabla$ induz uma conexão em cada $\mathcal{T}^{(r,s)}_E$ (também denotada por $\nabla$), que é determinada pelas seguintes propriedades.
\begin{itemize}
\item[1.]  Em $\mathcal{T}_E^{(0,0)} = M \times \C$, $\nabla = d$ é a diferencial exterior.
\end{itemize}
Para todo vetor tangente $X \in TM$ temos
\begin{itemize}
\item[2.] $\nabla_X$ é uma derivação em $\mathcal{T}_E$, isto é, $\nabla_X (s \otimes t) = \nabla_X(s)\otimes t + s \otimes \nabla_X(t)$ para todos $s,t \in \mathcal{T}_E$,
\item[3.] $\nabla_X$ comuta com todas as contrações.
\end{itemize}
\end{proposition}

Com essa proposição podemos descrever as conexões induzidas no fibrado dual e no fibrado dos endomorfismos.

\begin{corollary} \label{cor:dual-connection}
1. A conexão naturalmente induzida em $E^*$ é dada por
\begin{equation*}
(\nabla \varphi)(s) = d (\varphi(s)) - \varphi(\nabla(s)),~~~ \varphi \in \mathcal{A}^0(E^*),~ s \in \mathcal{A}^0(E).
\end{equation*}
2. Segundo o isomorfismo $E \otimes E^* \simeq \End(E)$, a conexão induzida em $\End(E)$ é dada por
\begin{equation*}
(\nabla \psi)(s) = \nabla (\psi(s)) - \varphi(\psi(s)),~~~ \psi \in \mathcal{A}^0(\End(E)),~ s \in \mathcal{A}^0(E).
\end{equation*}
\end{corollary}

\begin{proof}
1. Seja $C:E \otimes E^* \to \mathcal{T}_E^{(0,0)}$ a contração $C(\varphi \otimes s) = \varphi(s)$. Usando as propriedades da proposição acima temos, para todo vetor tangente $X$,
\begin{equation*}
\begin{split}
d (\varphi(s)) (X) &= X (\varphi(s)) = X(C(\varphi \otimes s)) \\
&= C(\nabla_X(\varphi \otimes s)) \\
&= C(\nabla_X(\varphi) \otimes s) + C(\varphi \otimes \nabla_X(s)) \\
&= (\nabla_X \varphi)(s) + \varphi(\nabla_X(s)),
\end{split}
\end{equation*}
e como $X$ é arbitrário segue que $d (\varphi(s)) = (\nabla \varphi)(s) + \varphi(\nabla(s))$.\\

2. Segundo o isomorfismo $E \otimes E^* \simeq \End(E)$ um elemento $s \otimes \phi$ é levado no endomorfismo $(t \mapsto \phi(t) s)$ e portanto a contração $C:E \otimes E^* \otimes E \to E$ dada por $s \otimes \phi \otimes t \mapsto \phi(t)s$ corresponde a aplicação de avaliação $\End(E) \otimes E \to E$. Sendo assim, da propriedade 3. da prosição acima, vemos que
\begin{equation*}
\begin{split}
\nabla_X (\psi(s)) &= \nabla_X(C(\psi \otimes s)) = C \nabla_X(\psi \otimes s)\\
&=  C(\nabla_X(\psi) \otimes s) + C(\psi \otimes \nabla_X(s))\\
&= (\nabla_X \psi)(s) + \psi(\nabla_X(s))
\end{split}
\end{equation*}
para todo $X \in TM$. Como $X$ é arbitrário obtemos $\nabla(\psi(s)) = (\nabla \psi)(s) + \psi(\nabla (s))$.
\end{proof}

\begin{example} \label{ex:connection-trivial-end}
Se $E = M \times \C^r$ e $\nabla = d + A$, a conexão induzida em $\End(E) = M \times \mathfrak{gl}(r,\C)$ é dada por $\nabla T = dT + [A,T]$, de fato, usando o item 2. do corolário acima temos:
\begin{equation*}
\begin{split}
(\nabla T)(s) &= \nabla (T\cdot s)) - T\cdot(\nabla(s)) = d(T\cdot s) + A\cdot(T\cdot s) - T\cdot ds - T\cdot(As) \\
 &=dT \cdot s + T \cdot ds + A \cdot(T \cdot s) - T\cdot ds - T\cdot(A\cdot s) \\
 &=(dT + AT-TA)\cdot s.
\end{split}
\end{equation*}
\end{example}

Existem ainda conexões induzidas por \textit{pullbacks}. Possivemente a melhor maneira de definir tais conexões seja usando a descrição local da conexão. Seja $E \to M$ um fibrado e $f:N \to M$ uma aplicação diferenciável. Lembre que os cociclos do fibrado  $f^*E \to N$ com relação a uma cobertura $\{U_i\}$ são $\varphi_{ij} = \psi_{ij} \circ f$ em $f^{-1}(U_i~\cap~U_j)$, onde $\psi_{ij}$ são os cociclos de $E$. Se $\nabla$ é uma conexão $E$ temos matrizes de $1$-formas, $A_i$ definidads em $U_i$ que satisfazem a condição de compatibilidade (\ref{eq:local-connection}). Sendo assim, as matrizes $f^*A_i$ são matrizes de $1$-formas em $f^{-1}(U_i)$ e satisfazem
\begin{equation*}
f^*A_j = \varphi_{ij}^{-1}d \varphi_{ij} + \varphi_{ij}^{-1}f^*A_i \varphi_{ij}~~ \text{ em }~f^{-1}(U_i \cap U_j)
\end{equation*}
e portanto definem uma conexão em $f^*E$, denotada por $f^* \nabla$.
\subsubsection{Extensão para formas com valores em $E$.}

Assim como a diferencial exterior de funções $d:\mathcal{A}^0_M \to \mathcal{A}^1_M$ se estende para formas diferenciais de grau mais alto, uma conexão $\nabla$ em $E$ pode ser estendida a uma aplicação $d^\nabla: \mathcal{A}^k(E) \to \mathcal{A}^{k+1}(E)$ pela fórmula
\begin{equation} \label{eq:nabla-extension} \index{dnabla@$d^\nabla$}
d^\nabla(\alpha \otimes s) = d \alpha \otimes s + (-1)^k \alpha \wedge \nabla(s),~~~ \alpha \in \mathcal{A}_X^k,~s \in \mathcal{A}^0(E).
\end{equation}
Note que, quando $k=0$, $d^\nabla: \mathcal{A}^0(E) \to \mathcal{A}^1(E)$ coincide com a conexão original.
\begin{example}
Considere o fibrado trivial $E=M \times \C^r$ com a conexão $\nabla = d$. Neste caso o fibrado de $k$-formas com valores em $E$ é simplesmente a soma direta de $r$ cópias de $\mathcal{A}^k(M)$ e a extensão $d^\nabla:\mathcal{A}^k(E) \to \mathcal{A}^{k+1}(E)$ é a diferencial exterior em cada componente.

Para uma outra conexão $\nabla = d + A$, a extensão $d^\nabla:\mathcal{A}^k(E) \to \mathcal{A}^{k+1}(E)$ é dada por $d^\nabla(\alpha) = d \alpha + A \wedge \alpha$, onde $\alpha = (\alpha_1,\ldots,\alpha_r)$ é uma $r$-upla de $k$ formas e $A \wedge \alpha$ denota a multiplicação da matriz de $1$-formas $A$ pelo vetor $\alpha$ e a multiplicação entre coeficientes é a multiplicação exterior.
\end{example}

Da regra de Leibniz usual obtemos uma regra para a derivação do produto de uma $k$-forma, $\beta \in \mathcal{A}^k_X$ por uma $l$-forma com valores em $E$, $t \in \mathcal{A}^l(E)$:
\begin{equation} \label{eq:vb-leibniz}
d^\nabla(\beta \wedge t) = d \beta \wedge t + (-1)^k \beta \wedge d^\nabla t.
\end{equation}

\subsection{Curvatura}
A partira da extensão da conexão para formas diferenciais (eq. \ref{eq:nabla-extension}) podemos definir a curvartura de uma conexão.

\begin{definition} \label{def:curvature} \index{curvatura!de uma conexão}
Seja $\nabla$ uma conexão em $E$. A \textbf{curvatura} de $\nabla$ é a aplicação $F_{\nabla}=d^\nabla \circ d^\nabla: \mathcal{A}^0(E) \to \mathcal{A}^2(E)$.
\end{definition}

A curvatura é linear no seguinte sentido: se $s \in \mathcal{A}^0(E)$ e $f \in \mathcal{A}^0$ então $F_{\nabla}(fs) = fF_{\nabla}(s)$. De fato, usando a definição de $d^\nabla: \mathcal{A}^1(E) \to \mathcal{A}^2(E)$ e a fórmula (\ref{eq:vb-leibniz}) temos
\begin{equation*}
\begin{split}
F_{\nabla}(fs) &= d^\nabla (\nabla (fs)) = d^\nabla(df \otimes s) + d^\nabla(f \nabla(s)) \\
 &=(d^\nabla)^2 f \otimes s - df \wedge \nabla(s) + df \wedge \nabla(s) + f d^\nabla(\nabla(s)) = f d^\nabla(\nabla(s)) \\
 &=f F_{\nabla}(s).
\end{split}
\end{equation*}

Podemos então ver a curvatura como uma $2$-forma com valores no fibrado $\End(E)$: $F_{\nabla} \in \mathcal{A}^2(\End(E))(M)$ é a $2$-forma que avaliada em um par de vetores $(X,Y)$ fornece o endomorfismo $s \mapsto F_{\nabla}(s)(X,Y)$.

\begin{example} \label{ex:curvature-trivial}
No fibrado trivial com a conexão $\nabla = d$ a curvatura é $F_{\nabla} = d^2 = 0$.

Já para a conexão $\nabla = d + A$, vimos que $d^\nabla: \mathcal{A}^1(E) \to \mathcal{A}^2(E)$ é dada por $\beta \mapsto d \beta + A \wedge \beta$, e portanto
\begin{equation*}
\begin{split}
F_{\nabla}(s)&= d^\nabla(ds) + d^\nabla(As) = d^2 s + A \wedge (ds) + d(As) + A \wedge (As) \\
&= A \wedge ds + d(A)s - A\wedge ds + A \wedge (As) \\
&= (dA)s + (A\wedge A)s,
\end{split}
\end{equation*}
onde usamos a regra de Leibniz e o fato de $A$ ser uma matriz de $1$-formas.

Vemos então que a curvatura de $\nabla$ é a matriz de $2$-formas
\begin{equation} \label{eq:curvature-trivial}
F_{\nabla} = dA + A \wedge A.
\end{equation}
\end{example}

Uma conexão em $E$ induz uma conexão em $\End(E)$, que se estende a $d^\nabla:\mathcal{A}^2(\End(E))$ $\to \mathcal{A}^3(\End(E))$ usando a fórmula (\ref{eq:nabla-extension}). Em particular podemos aplicar $d^\nabla$ à forma de curvatura $F_{\nabla} \in \mathcal{A}^2(\End(E))$
\begin{proposition} \label{prop:bianchi}
\textbf{Identidade de Bianchi.} A curvatura satisfaz $d^\nabla (F_{\nabla}) = 0$.
\end{proposition}
\begin{proof}
O resultado é local, isto é, se provarmos que $d^\nabla(F_{\nabla}|_U) = 0$ para abertos $U$ cobrindo $M$ seguirá que $d^\nabla (F_{\nabla}) = 0$.

Seja então $U$ um aberto de uma trivialização de $E$. Sendo assim $E|_U \simeq U \times \C^r$ e segundo esse isomorfismo $\nabla = d + A$. Vimos no exemplo \ref{ex:connection-trivial-end} que a conexão induzida em $\End(E)|_U = U \times \mathfrak{gl}(r,\C)$ é dada por $\nabla T = dT + AT - TA$ e portanto a extensão em $d^\nabla: \mathcal{A}^k(\End(E)|_U) \to \mathcal{A}^{k+1}(\End(E)|_U) $ é dada por
\begin{equation*}
d^\nabla \Theta = d \Theta + A \wedge \Theta - (-1)^k \Theta \wedge A,
\end{equation*}
onde $\Theta$ é uma matriz de $k$-formas.

Aplicando a fórmula acima para a curvatura $F_{\nabla}|_U = dA + A\wedge A$ (veja o exemplo \ref{ex:curvature-trivial}) obtemos
\begin{equation*}
\begin{split}
\nabla(F_{\nabla}|_U) &= d F_{\nabla} + A \wedge F_{\nabla} -  F_{\nabla} \wedge A \\
&=d^2 A + d(A \wedge A) + A \wedge dA + A \wedge A \wedge A - dA \wedge A -  A \wedge A \wedge A\\
&=d(A \wedge A) + A \wedge dA - dA \wedge A \\
&=dA \wedge A - A \wedge dA + A \wedge dA - dA \wedge A \\
&=0.
\end{split}
\end{equation*}
\end{proof}

\begin{remark} \label{rmk:bianchi-lb}
Note que no caso de fibrados de linha, o fibrado de endomorfismos é trivial, pois $\End(L) \simeq L^* \otimes L \simeq \mathcal O_X$. Sendo assim, a forma de curvatura de uma conexão $\nabla$ é uma $2$-forma usual $F_\nabla \in \mathcal{A}^2(M)$ e como a conexão induzida em $\mathcal{A}^k(L^*\otimes L)$ é trivial, a identidade de Bianchi diz que $d F_\nabla =0$.
\end{remark}

\begin{proposition} \label{prop:induced-curvatures}
Se $\nabla$ é uma conexão em $E$ e $F_{\nabla}$ é sua curvatura então
\begin{itemize}
\item[1.] A curvatura $F^*$ da conexão induzida em $E^*$ é $F^* = - F_{\nabla}^t$.
\item[2.] Se $f:N \to M$ é uma aplicação diferenciável, a curvatura da conexão pullback $f^*\nabla$ é $F_{f^* \nabla} = f^* F_{\nabla}$.
\end{itemize}
Se $E_1$ e $E_2$ são fibrados complexos munidos de conexões $\nabla_1$ e $\nabla_2$ e se $F_1$ e $F_2$ são as respectivas formas de curvatura então
\begin{itemize}
\item[3.] A curvatura da conexão induzida em $E_1 \oplus E_2$ é dada por $F = F_1 \oplus F_2$.
\item[4.] A curvatura da conexão induzida em $E_1 \otimes E_2$ é dada por $F = F_1 \otimes \id_{E_2} + \id_{E_1} \otimes F_2$.
\end{itemize}
\end{proposition}

\begin{proof}
1. Sejam $\varphi \in \mathcal{A}^0(E^*)$ e $s \in \mathcal{A}^0(E)$. Usando o corolário \ref{cor:dual-connection} obtemos
\begin{equation*}
0 = d^2(\varphi(s)) = d(\nabla^* \varphi(s)) + d(\varphi (\nabla s)).
\end{equation*}
Da regra de Leibniz e da definição do produto $\wedge$ de uma forma com valores em $E$ e uma forma com valores em $E^*$ obtemos as seguintes fórmulas
\begin{equation*}
\begin{split}
d \alpha(s) = d^{\nabla^*} \alpha (s) - \alpha \wedge \nabla s,\text{ para } ~ \alpha \in \mathcal A^1(E^*),~ s \in \mathcal A^0(E) \\
d \varphi(\gamma) = \nabla^* \varphi \wedge \gamma + \varphi(d^\nabla \gamma),\text{ para } ~ \gamma \in \mathcal A^1(E),~ \varphi \in \mathcal A^0(E^*)
\end{split} 
\end{equation*}

Vemos então que
\begin{equation*}
d(\nabla^* \varphi(s)) = {d^{\nabla^*}}^2 \varphi(s) - \nabla^* \varphi \wedge \nabla s ~~~ \text{ e }~~~ d(\varphi (\nabla s)) = \nabla^* \varphi \wedge \nabla s + \varphi((d^\nabla)^2 s), 
\end{equation*}
e portanto, substituindo na primeira equação, obtemos
\begin{equation*}
0 = {d^{\nabla^*}}^2 \varphi(s) - \nabla^* \varphi \wedge \nabla s + \nabla^* \varphi \wedge \nabla s + \varphi((d^\nabla)^2 s) = F^*(\varphi)(s) + \varphi(F_\nabla (s)),
\end{equation*}
ou seja, $F^*(\varphi)(s) = - \varphi(F_\nabla (s)) = - F_{\nabla}^t(\varphi)(s)$.\\

2. Por definição, as matrizes de $1$-formas associadas a $f^*\nabla$ são  $f^*A_i$, onde $A_i$ são as matrizes associadas a $\nabla$. Assim, a expressão local de $F_{f^*\nabla}$ é
\begin{equation*}
F_{f^*\nabla}|_{U_i} = d f^*A_i + f^*A_i \wedge f^*A_i = f^*(dA_i + A_i \wedge A_i) = f^*(F_{\nabla}|_{U_i}),
\end{equation*}
de onde segue que $F_{f^*\nabla} = f^* F_{\nabla}$.\\

3. A conexão em $E_1 \oplus E_2$ é dada por $\nabla = \nabla_1 \oplus \nabla_2$. Note que se $s_1 \in \mathcal{A}^0(E_1) \subset \mathcal{A}^0(E)$ então $\nabla(s_1) = \nabla(s_1+0) = \nabla_1(s_1)$. Consequentemente, a extensão de $\nabla$ a formas com valores em $E$ satisfaz $d^\nabla(\alpha_1) = d^{\nabla_1}(\alpha_1)$ se $\alpha_1 \in \mathcal{A}^1(E_1)$ e analogamente para $E_2$.

Sendo assim, a curvatura de $\nabla$ é dada por
\begin{equation*}
\begin{split}
F(s_1 + s_2) &= d^\nabla(\nabla(s_1+s_2)) = d^\nabla(\nabla_1 s_1 + \nabla_2 s_2) = d^\nabla (\nabla_1 s_1) + d^\nabla(\nabla_2 s_2)\\
&= d^{\nabla_1}(\nabla_1 s_1) + d^{\nabla_2}(\nabla_2 s_2) =F_1(s_1) + F_2(s_2)\\
&=(F_1 \oplus F_2)(s_1 + s_2).\\
\end{split}
\end{equation*}

4. Lembrando que a conexão induzida em $E_1 \otimes E_2$ é $\nabla = \nabla_1 \otimes \id_2 + \id_1 \otimes \nabla_2$, temos o seguinte resultado.
\begin{lemma}
Se $t_1 \in \mathcal{A}^k(E_1)$ e $s_2 \in \mathcal{A}^0(E_2)$ então $d^\nabla(t_1 \otimes s_2) = d^{\nabla_1} t_1 \otimes s_2 + (-1)^k t_1 \wedge \nabla_2 s_2$ e analogamente se $s_1 \in \mathcal{A}^0(E_1)$ e $t_2 \in \mathcal{A}^k(E_2)$ então $d^\nabla(s_1 \otimes t_2) = (-1)^{k+1}\nabla_1 s_1 \wedge t_2 + s_1 \otimes d^{\nabla_2} t_2$.
\end{lemma}
\begin{proof}
Podemos supor que $t_1 = \alpha \otimes s_1$ com $\alpha \in \mathcal{A}^k_M$ e $s_1 \in \mathcal{A}^0(E_1)$. Nesse caso temos
\begin{equation*}
\begin{split}
d^\nabla(t_1 \otimes s_2) &= d^\nabla(\alpha \otimes (s_1 \otimes s_2)) = d \alpha \otimes (s_1 \otimes s_2) + (-1)^k \alpha \wedge \nabla(s_1 \otimes s_2)\\
&=d \alpha \otimes (s_1 \otimes s_2) + (-1)^k \alpha \wedge (\nabla_1 s_1 \otimes s_2) + (-1)^k \alpha \wedge (s_1 \otimes \nabla_2 s_2)\\
&=d^\nabla_1(\alpha \otimes s_1) \otimes s_2 +  (-1)^k \alpha \wedge (s_1 \otimes \nabla_2 s_2)\\
&=d^\nabla_1(\alpha \otimes s_1) \otimes s_2 + (-1)^k \alpha \otimes s_1 \wedge \nabla_2 s_2 \\
&=d^\nabla_1 t_1 \otimes s_2 + (-1)^k t_1 \wedge \nabla_2 s_2.
\end{split}
\end{equation*}
Para demonstrar a segunda afirmação suponha que $t_2 = \alpha \otimes s_2$ com $\alpha \in \mathcal{A}^k_M$ e $s_2 \in \mathcal{A}^0(E_2)$. Temos então que
\begin{equation*}
\begin{split}
d^\nabla(s_1 \otimes t_2) &= d^\nabla(\alpha \otimes (s_1 \otimes s_2)) = d \alpha \otimes (s_1 \otimes s_2) + (-1)^k \alpha \wedge \nabla(s_1 \otimes s_2)\\
&=d \alpha \otimes (s_1 \otimes s_2) + (-1)^k \alpha \wedge (\nabla_1 s_1 \otimes s_2) + (-1)^k \alpha \wedge (s_1 \otimes \nabla_2 s_2)\\
&= (-1)^{k+1} (\nabla_1 s_1 \wedge \alpha) \otimes s_2 + s_1 \otimes (d \alpha \otimes s_2 + (-1)^k \alpha \wedge \nabla_2 s_2)\\
&= (-1)^{k+1} \nabla_1 s_1 \wedge (\alpha \otimes s_2) + s_1 \otimes d^{\nabla_2}(\alpha \otimes s_2)\\
&= (-1)^{k+1}\nabla_1 s_1 \wedge t_2 + s_1 \otimes d^{\nabla_2} t_2.
\end{split}
\end{equation*}
\end{proof}

Usando o lema acima, vemos que a curvatura de $\nabla$ é dada por
\begin{equation*}
\begin{split}
F(s_1 \otimes s_2) &= d^\nabla(\nabla_1 s_1 \otimes s_2) + d^\nabla(s_1 \otimes \nabla_2 s_2) \\
&= d^{\nabla_1}(\nabla_1 s_1) \otimes s_2 - \nabla_1 s_1 \wedge \nabla_2 s_2  + \nabla_1 s_1 \wedge \nabla_2 s_2 + s_1 \otimes d^{\nabla_2}(\nabla_2 s_2) \\
&=d^{\nabla_1}(\nabla_1 s_1) \otimes s_2 + s_1 \otimes d^{\nabla_2}(\nabla_2 s_2) \\
& = F_1(s_1) \otimes s_2 + s_2 \otimes F_2(s_2),
\end{split}
\end{equation*}
ou seja, $F = F_1 \otimes \id_{E_2} + \id_{E_1} \otimes F_2$.
\end{proof}

\subsection{A conexão de Chern} \label{subsec:chern-connection}
É um fato conhecido da Geometria Riemanniana que no fibrado tangente de uma variede riemanniana $M$ existe uma única conexão, chamada conexão de Levi-Civita, que é compatível com a métrica e é livre de torção\footnote{Uma conexão em $TM$ é livre de torção se satisfaz $\nabla_X Y - \nabla_Y X = [X,Y]$, para todos $X,Y \in TM$.}. Em outras palavras existe uma conexão privilegiada em $TM$.

Veremos que para um fibrado holomorfo $E$ com uma estrutura hermitiana há um fenômeno semelhante, a saber, existe uma conexão privilegiada em $E$, chamada conexão de Chern, que é determinada pela métrica e pela estrutura holomorfa.

\subsubsection{Compatibilidade com a métrica}

Considere $(E,h)$ um fibrado hermitiano. A métrica $h$ nos permite definir uma multiplicação de formas com valores em $E$:
\begin{equation} \label{eq:h-forms}
\begin{split}
h:\mathcal{A}^k(E) \times \mathcal{A}^l(E) &\longrightarrow \mathcal{A}_M^{k+l} \\
h(\alpha \otimes s, \beta \otimes t) &= \alpha \wedge \overline{\beta}~h(s,t)
\end{split}
\end{equation}

Note que $h$ é sesquilinear e  satisfaz $h(s,t)=(-1)^{kl} \overline{h(t,s)}$ para $s \in \mathcal{A}^k(E)$ e $t \in \mathcal{A}^l(E)$.

\begin{definition}
Uma conexão $\nabla$ em um fibrado hermitiano $(E,h)$ é compatível com $h$ se
\begin{equation} \label{eq:h-compatibility}
dh(s,t) = h(\nabla s,t) + h(s,\nabla t)
\end{equation}
para todas seções, $s,t \in \mathcal{A}^0(E)$. É comum dizermos também que $\nabla$ é uma conexão hermitiana.
\end{definition}

Note que a condição (\ref{eq:h-compatibility}) é equivalente a $Xh(s,t) = h(\nabla_X s,t) + h(s,\nabla_X t)$ para todo vetor tangente $X$.

\begin{example} \label{ex:riemann-curvaure} \textbf{Curvatura riemanniana.}
Seja $X$ uma variedade complexa e $g$ uma métrica hermitiana em $X$. O fibrado tangente real $TX$ admite uma única conexão real\footnote{Uma conexão real em $TX$ é uma aplicação $\R$-linear $\nabla: \mathcal A_ \R^{0}(TX) \to \mathcal A_\R^1(TX)$ satisfazendo a regra de Leibniz.} livre de torção $\nabla$ que é compatível com $g$, i.e., que satisfaz $Zg(X,Y) = g(\nabla_Z X,Y) + g(X, \nabla_Z Y)$. Essa conexão é chamada de conexão de Levi-Civita e é determinada pela fórmula de Koszul
\begin{equation*}
g(\nabla_X Y, Z) = \frac{1}{2} \left\lbrace \begin{split} &X g(Y,Z) + Y g(X,Z) - Z g(X,Y) \\&+ g([X,Y],Z) - g([X,Z],Y) - g([Y,Z],X)\end{split}\right \rbrace.
\end{equation*}

O tensor de curvatura da variedade riemanniana $(X,g)$ é definido por\footnote{Veja por exemplo \cite{k-n} cap. 4.}
\begin{equation*}
R(X,Y) = \nabla_X \nabla_Y  - \nabla_Y \nabla_X - \nabla_{[X,Y]}.
\end{equation*}

Podemos estender $\nabla$ a $T_\C X = TX \otimes \C$ de modo $\C$-linear. O tensor de curvatura também se estende e essa extensão coincide com a curvatura de $\nabla$, como veremos a seguir.

Seja $Z \in \mathcal{A}^0(T_\C X)$ e considere $\nabla Z \in \mathcal{A}^1(T_\C X)$. Por simplicidade podemos supor\footnote{Em geral, $\nabla Z$ será uma soma de elementos do tipo $\alpha \otimes W$, mas como a conexão é linear não há perda de generalidade em supor que $\nabla Z$ é indecomponível.} que $\nabla Z = \alpha \otimes W$ onde $\alpha$ é uma $1$-forma e $W$ uma seção de $T_\C X$. Assim, a derivada covariante é dada por $\nabla_X Z = \alpha(X)W$.

Da fórmula intrínseca para a diferencial exterior temos que $d \alpha (X,Y)= X \alpha(Y) - Y\alpha(X) - \alpha([X,Y])$. Sendo assim temos que
\begin{equation*}
\begin{split}
F_\nabla(X,Y)(Z) &= d^\nabla(\alpha \otimes W)(X,Y) = (d \alpha \otimes W - \alpha \wedge \nabla W)(X,Y) \\
&=(d \alpha)(X,Y)W - \alpha(X) \nabla_Y W + \alpha(Y) \nabla_X W\\
&= [X \alpha(Y) - Y\alpha(X) - \alpha([X,Y])]W - \alpha(X) \nabla_Y W + \alpha(Y) \nabla_X W \\
&= (X \alpha(Y))W + \alpha(Y)\nabla_X W - (Y \alpha(X))W - \alpha(X)\nabla_Y W - \alpha([X,Y])W \\
&= \nabla_X (\alpha(Y) W) - \nabla_Y(\alpha(X) W) - \alpha([X,Y])W\\
&=\nabla_X \nabla_Y Z - \nabla_Y \nabla_X Z- \nabla_{[X,Y]}Z \\
&= R(X,Y)(Z).
\end{split}
\end{equation*}

Poderíamos considerar a extensão de $g$ a um produto hermitiano $h$ no fibrado complexo $(TX,J)$ via $h = g - \smo \omega$, e pensarmos na conexão de Levi-Civita como uma conexão em $(TX,J)$. No entanto, essa tentativa não produz em geral uma conexão no sentido da definição \ref{def:connection}, pois $\nabla$ não será $\C$-linear. Na verdade, como a estrutura complexa de $TX$ é dada pela multiplicação por $J$, $\nabla$ será $\C$-linear se e somente se $\nabla JY = J \nabla Y$ para todo campo $Y$, ou seja, se e somente se $\nabla J = 0$. Nesse caso é fácil ver que $\nabla$ preserva $h$ (pois preserva $g$ e $J$), e portanto a conexão de Levi-Civita é uma conexão hermitiana em $(TX,J,h)$ se e somente se a métrica $g$ é de Kähler.

Falaremos com mais detalhes da conexão de Levi-Civita em uma variedade de Kähler no exemplo \ref{ex:chern-levi-civita}.
\end{example}

\subsubsection{Compatibilidade com a estrutura holomorfa}

Consideremos agora um um fibrado holomorfo $E$ sobre uma variedade complexa $X$.

Vimos que o feixe de $1$-formas com valores em $E$ se decompõe como $\mathcal{A}^1(E) = \mathcal{A}^{1,0}(E) \oplus \mathcal{A}^{0,1}(E)$. Sendo assim, uma conexão $\nabla:\mathcal{A}^0(E) \to \mathcal{A}^1(E)$ se escreve como $\nabla = \nabla^{1,0} + \nabla^{0,1}$, onde $\nabla^{1,0}:\mathcal{A}^0(E) \to \mathcal{A}^{1,0}(E)$ e $\nabla^{0,1}:\mathcal{A}^0(E) \to \mathcal{A}^{0,1}(E)$.

Tomando a parte $(0,1)$ da equação (\ref{eq:leibniz}) vemos que $\nabla^{0,1}$ satisfaz~
\begin{equation*}
\nabla^{0,1}(f s) = \delbar f \otimes s + f \nabla^{0,1}(s),~f \in \mathcal{A}^0_X,~ s \in \mathcal{A}^0(E),
\end{equation*}
que é a mesma equação satisfeita pelo operador $\delbar_E: \mathcal{A}^0(E) \to \mathcal{A}^{0,1}(E)$ (veja a Proposição \ref{prop:delbar-E}).

Essa discussão motiva a seguinte definição.

\begin{definition}
Dizemos que uma conexão $\nabla$ em um fibrado holomorfo $E$ é \emph{compatível com a estrutura holomorfa} se $\nabla^{0,1}=\delbar_E$.
\end{definition}

Se além de uma estrutura holomorfa o fibrado possuir uma estrutura hermitiana podemos nos perguntar se existe uma conexão que é compatível com as duas estruturas simultaneamente.

\begin{proposition} \label{prop:chern-connecion} \index{conexão!de Chern}
Seja $E$ um fibrado holomorfo com uma estrutura hermitiana. Então existe uma única conexão hermitiana em $E$ que também é a compatível com a estrutura holomorfa. Essa conexão recebe o nome de \textbf{conexão de Chern}.
\end{proposition}

\begin{proof}
Suponha que exista uma tal conexão $\nabla$. Tomando a componente $(1,0)$ da equação (\ref{eq:h-compatibility}) e usando que $h$ é conjugado-linear na segunda entrada vemos que
\begin{equation*}
\del h(s,t) = h(\nabla^{1,0}s,t) + h(s, \nabla^{0,1}t) = h(\nabla^{1,0}s,t) + h(s, \delbar_E t),
\end{equation*}
para todas as seções de $E$.

Seja $\sigma = \{\sigma_i\} \subset H^0(U,E)$ um referencial holomorfo local. Como as seções $\sigma_i$ são holomorfas temos que $\delbar_E (\sigma_i) = 0$ e portanto da equação acima temos que
\begin{equation} \label{eq:chern-conn1}
\del h(\sigma_i,\sigma_j) = h(\nabla^{1,0}\sigma_i,\sigma_j),
\end{equation}
o que mostra que $\nabla^{1,0}$ é determinada por $h$ em $U$.

Consequentemente, a conexão $\nabla = \nabla^{1,0} + \delbar_E$ é determinada, em $U$  por $h$ e pela estrutura holomorfa, mostrando a unicidade em $U$ e consequentemente em $X$.

Para ver a existência, defina $\nabla^{1,0}$ em $U$ pela equação (\ref{eq:chern-conn1}). Essa definição não depende do referencial holomorfo escolhido. Para ver isso considere outro referencial holomorfo $\sigma'$. Denote por $\psi$ a trivialização definida por $\{\sigma_i\}$ e por $\psi'$ a trivialização definida por $\{\sigma'_i\}$. Se $\psi$ e $\psi'$ estão relacionadas por $\psi' = F \circ \psi$ com $F(x,v) = (x,\varphi(x)\cdot v)$ e $\varphi:U \to GL(r,\C)$ então os referenciais se relacionam por $(\sigma')^t \varphi = \sigma$, onde vemos $\sigma$ e $\sigma'$ como vetores linha.

Denote por $H$ a matriz $H_{ij}=h(\sigma_i,\sigma_j)$, por $S$ a matriz $S_{ij}= h(\nabla^{1,0}\sigma_i,\sigma_j)$ e por $H'$ e $S'$ as matrizes correspondentes usando o referencial $\sigma'$. Note que a equação (\ref{eq:chern-conn1}) em forma matricial fica $\del H = S$.

Tomando a parte $(1,0)$ da regra de Leibniz (\ref{eq:leibniz}) temos que $\nabla^{1,0} \sigma_k = \sum_i (\nabla^{1,0} \sigma_i' \varphi_{ik} + \sigma_i' \del \varphi_{ik})$ e portanto
\begin{equation*}
S_{kl} = \sum_{ij}\left( \varphi_{ik} h(\nabla^{1,0}\sigma'_i, \sigma'_j) \bar \varphi_{jl} + \del \varphi_{ik} h(\sigma'_i,\sigma'_j) \bar \varphi _{jl} \right),
\end{equation*}
ou em forma matricial $S =  \varphi^t S' \bar \varphi + \del \varphi^t H' \bar \varphi$. Como $H =  \varphi^t H' \bar \varphi$ (veja o exemplo  \ref{ex:herm-local}) vemos que se $S' = \del H'$ então
\begin{equation*}
S = \varphi^t S' \bar \varphi + \del \varphi^t H' \bar \varphi = \varphi^t (\del H') \bar \varphi + \del \varphi^t H' \bar \varphi = \del (\varphi^t H' \bar \varphi) = \del H,
\end{equation*}
o que mostra que a equação (\ref{eq:chern-conn1}) é independente do referencial holomorfo.

Com isso podemos definir $\nabla^{1,0}$ globalmente: usamos a equação (\ref{eq:chern-conn1}) em cada aberto de uma cobertura trivializante e as definições concordam nas intersecções. Tendo definido $\nabla^{1,0}$, definimos a conexão de Chern por $\nabla = \nabla^{1,0} + \delbar_E$ que, por construção, será compatível com a métrica e com  a estrutura holomorfa.
\end{proof}

Podemos usar a fórmula (\ref{eq:chern-conn1}) para obter uma expressão local para a conexão de Chern. Para isso considere $\psi:E|_U \to U \times \C^r$ uma trivialização sobre um aberto $U$.

A trivialização $\psi$ leva a conexão de Chern em uma conexão $\nabla=d + A$ em $U \times \C^r$  e o operador $\delbar_E$ em $E|_U$ corresponde ao operador $\delbar$ usual em $U \times \C^r$. Tomando a parte $(0,1)$ de $\nabla$ e usando a compatibilidade com a estrutura holomorfa vemos que $\delbar = \nabla^{0,1} = \delbar + A^{0,1}$  e portanto $A^{0,1}=0$, ou seja, $A$ é uma matriz de $(1,0)$-formas.

Temos portanto que $\nabla^{1,0} = \del + A$ em $U \times \C^r$. Como $\del e_i = 0$ temos que $\nabla^{1,0} e_i = A e_i$ e portanto a equação (\ref{eq:chern-conn1}) para o referencial $\{e_i\}$ fica $\del h(e_i,e_j) = h(Ae_i,e_j)$. Escrevendo em termos da matriz $H = (h_{ij})$ temos que $\del H = A H$, de onde obtemos a expressão local para a matriz $A$
\begin{equation} \label{eq:local-chern}
A = \del H \cdot H^{-1}.
\end{equation}

\begin{example} \label{ex:chern-levi-civita} \textbf{Conexão de Levi-Civita vs. Conexão de Chern}

Seja $(X,g)$ uma variedade de Kähler, $J:TX \to TX$ a estrutura complexa induzida e $\nabla$ a conexão de Levi-Civita em $TX$. Podemos ver $\nabla$ como  uma aplicação $\R$-linear $\nabla:TX \to \mathcal{A}^1(TX)$ que satisfaz a regra de Leibniz, é compatível com $g$ e é livre de torção.

Podemos estender $\nabla$ de forma $\C$-bilinear a uma aplicação $\nabla:T_\C X \to \mathcal{A}^1(T_\C X)$. Como $g$ é de Kähler temos que $\nabla$ comuta com $J$ e portanto $\nabla$ preserva os auto-espaços $T^{1,0}X$ e $T^{0,1}X$ de $J$. Em particular podemos considerar a restrição $\nabla: T^{1,0}X \to \mathcal A^1(T^{1,0}M)$. É fácil ver que $\nabla$ define uma conexão em $T^{1,0}X$ (o ponto mais delicado é a $\C$-linearidade, que segue do fato de $\nabla$ comutar com $J$ e de $J|_{T^{1,0}X} = \smo \id$).

Considere o isomorfismo $TX \simeq T^{1,0}X$ dado por $Z \mapsto Z^{1,0} = \frac{1}{2}(Z - \smo JZ)$. Vimos que o produto $h = g - \smo \omega$ é $J$ sesquilinear e portanto, segundo esse isomorfismo, define um produto hermitiano $h$ em $T^{1,0}X$. Como $g$ e $J$ são paralelos com relação a conexão de Levi-Civita temos que a conexão $\nabla: T^{1,0}X \to \mathcal A^1(T^{1,0}M)$ preserva $h$.

Vimos também que o fibrado $T^{1,0}X$ admite uma estrutura holomorfa (exemplo \ref{ex:hol-tangent-bundle}) e portanto podemos falar da conexão de Chern de $(T^{1,0}X,h)$.

\begin{proposition}
A conexão $\nabla: T^{1,0}X \to \mathcal A^1(T^{1,0}M)$ induzida pela conexão de Levi-Civita coincide com a conexão de Chern de $(T^{1,0}X,h)$. 
\end{proposition}
\begin{proof}
Já vimos que $\nabla$ é compatível com $h$ e portanto, para verificar que $\nabla$ é a conexão de Chern só falta verifcar que $\nabla^{0,1}:T^{1,0}X \to \mathcal A^{0,1}(T^{1,0}M)$ coincide com o operador $\delbar$ de $T^{1,0}X$, o que é equivalente a mostrar que $\nabla^{0,1}Z^{1,0} = 0$ se e somente se $Z^{1,0}$ é uma seção holomorfa de $T^{1,0}X$.

Vamos precisar do seguinte resultado.
\begin{lemma}
Uma seção local $Z^{1,0}$ de $T^{1,0}X$ é holomorfa se e somente se $[Z,JW] = J[Z,W]$ para todo campo $W \in TX$.\footnote{Se $\mathcal L$ denota a derivada de Lie temos que $(\mathcal L_Z J)W =  \mathcal L_Z(JW) - J \mathcal L_Z W = [Z,JW] - J[Z,W]$  e portanto o enunciado do lema pode ser reescrito como: $Z^{1,0}$ é holomorfo se e somente se $\mathcal L_Z J = 0$.}
\end{lemma}
\begin{proof}
Seja $Z \in TX$ e escreva $Z = Z^{1,0} + Z^{0,1}$ onde $Z^{1,0} = \frac{1}{2}(Z - \smo JZ)$ e $Z^{0,1} = \overline{Z^{1,0}}$. Analogamente, dado $W \in TX$ decomponha $W \in TX$ como $W = W^{1,0} + W^{0,1}$.

Como $J$ age como a multiplicação por $\smo$ em $T^{1,0}X$ e como a multiplicação por $-\smo$ em $T^{0,1}X$ temos que $JW = \smo (W^{1,0}-W^{0,1})$ e portanto 
\begin{equation*}
\begin{split}
[Z,JW] &= \smo [Z^{1,0} + Z^{0,1},W^{1,0}-W^{0,1}] \\
 &= \smo \left \lbrace [Z^{1,0},W^{1,0}] - [Z^{1,0},W^{0,1}] + [Z^{0,1},W^{1,0}] -[Z^{0,1},W^{0,1}] \right \rbrace.
\end{split}
\end{equation*}

Escolha um sistema de coordenadas holomorfas locais e escreva $Z^{1,0} = \sum_{j=0}^n a_j \frac{\del}{\del z_j}$ e $W^{1,0} = \sum_{j=0}^n b_j \frac{\del}{\del z_j}$ onde $a_j$ e $b_j$ são funções suaves. Note que $Z^{1,0}$ será holomorfo se e somente se os coeficientes $a_j$ forem holomorfos.

Da expressão do colchete em coordenadas temos que\footnote{Para mostrar que $[Z^{1,0},W^{1,0}] \in T^{1,0}X$ poderíamos também usar a integrabilidade da estrutura complexa (veja a observação \ref{rmk:nijenhuis}), que implica que  $[T^{1,0}X,T^{1,0}X] \subset T^{1,0}X$.}
\begin{equation*}
[Z^{1,0},W^{1,0}] = \sum_{i=0}^n \left( \sum_{j=0}^n a_j \frac{\del b_i}{\del z_j} -  b_j \frac{\del a_i}{\del z_j} \right ) \frac{\del}{\del z_i} \in T^{1,0}X
\end{equation*}
e portanto $J[Z^{1,0},W^{1,0}]=\smo [Z^{1,0},W^{1,0}]$. Como $[Z^{0,1},W^{0,1}] = \overline{[Z^{1,0},W^{1,0}]} \in T^{0,1}X$ vemos também que  $J[Z^{0,1},W^{0,1}]=-\smo [Z^{0,1},W^{0,1}]$. Sendo assim, vemos que
\begin{equation*}
\begin{split}
J[Z,W] &= J [Z^{1,0} + Z^{0,1},W^{1,0}+W^{0,1}]\\
 &= \smo [Z^{1,0},W^{1,0}] + J [Z^{1,0},W^{0,1}] + J[Z^{0,1},W^{1,0}] - \smo [Z^{0,1},W^{0,1}],
\end{split}
\end{equation*}
e portanto $[Z,JW] = J[Z,W]$ se e somente se $J ([Z^{1,0},W^{0,1}] + [Z^{0,1},W^{1,0}]) = \smo (-[Z^{1,0},W^{0,1}] + [Z^{0,1},W^{1,0}])$. Chame esta última equação de $(\star)$.

Escrevendo em coordenadas vemos que
\begin{equation*}
[Z^{1,0},W^{0,1}] + [Z^{0,1},W^{1,0}] = \sum_{k=0}^n \left( \sum_{i=0}^n \overline{a_i} \frac{\del b_k}{\del \overline z_i} - \overline{b_i} \frac{\del a_k}{\del \overline z_i} \right) \frac{\del}{\del z_k} + \sum_{l=0}^n \left( \sum_{j=0}^n a_j \frac{\del \overline{b_l}}{\del z _j} - b_j \frac{\del \overline{a_l}}{\del z _j} \right) \frac{\del}{\del \overline{z_l}},
\end{equation*}
e usando novamente o fato de que $J$ age como multiplicação por $\smo$ em $T^{1,0}X$ e a igualdade $\frac{\del \overline{f}}{\del \overline z_i} = \overline{\frac{\del f}{\del z_i}}$, vemos que a equação $(\star)$ é equivalente ao sistema
\begin{equation*}
\sum_{i=0}^n \left(- \overline{a_i} \frac{\del b_k}{\del \overline z_i} + \overline{b_i} \frac{\del a_k}{\del \overline z_i} \right) = \sum_{i=0}^n \left( -\overline{a_i} \frac{\del b_k}{\del \overline z_i} - \overline{b_i} \frac{\del a_k}{\del \overline z_i} \right)~\forall~ k=1,\ldots,n \Leftrightarrow \sum_{i=0}^n \overline{b_i} \frac{\del a_k}{\del \overline z_i} = 0 ~~ \forall ~k=1,\ldots,n.
\end{equation*}

Como $W$ é arbitrário vemos que vale $(\star)$ para todo $W$ se e somente se $\frac{\del a_k}{ \del \overline z_i} =0$ para todo $i$ e todo $k$, ou seja, se e somente se os coecientes de $Z^{1,0}$ são holomorfos.
\end{proof}

Do fato  da conexão de Levi-Civita ser livre de torção, o lema acima diz que $Z^{1,0}$ é holomorfo se e somente se $\nabla_{JW}Z = J\nabla_W Z$ para todo $W \in TX$.

Agora note que $\nabla^{0,1}Z^{1,0} = 0$ se e somente se $\nabla_{W^{0,1}}Z^{1,0} = 0$ para todo $W^{0,1} \in T^{0,1}X$. Escrevendo $Z^{1,0} = \frac{1}{2}(Z - \smo JZ)$ e $W^{0,1} = \frac{1}{2}(W + \smo JW)$ temos
\begin{equation*}
\begin{split}
\nabla_{W^{0,1}}Z^{1,0} &= \frac{1}{4}\nabla_{W + \smo JW}((Z - \smo JZ)\\
&=\frac{1}{4} \left \lbrace (\nabla_W Z + \nabla_{JW}JZ) + \smo(\nabla_{JW}Z - \nabla_W JZ)) \right \rbrace \\
&=\frac{1}{4} \left \lbrace (\nabla_W Z + J\nabla_{JW}Z) + \smo(\nabla_{JW}Z - J\nabla_W Z)) \right \rbrace,
\end{split}
\end{equation*}
e portanto $\nabla_{W^{0,1}}Z^{1,0} = 0$ para todo  $W^{0,1} \in T^{0,1}X$ se e somente se $\nabla_{JW}Z = J\nabla_W Z$ para todo $W \in TX$, ou seja, se e somente se $Z^{1,0}$ é uma seção holomorfa de $T^{1,0}X$, mostrando que $\nabla^{0,1}=\delbar$.
\end{proof}
\end{example}

As condições de compatibilidade de uma conexão têm consequências na sua curvatura, que passa a assumir valores em subfibrados de $\mathcal{A}^2(\End(E))$.\\

\textbf{Curvatura de uma conexão hermitiana.} Um endomorfismo $T$ de um fibrado hermitiano $(E,h)$ é dito \emph{anti-hermitiano} se $h(Ts_1,s_2) + h(s_1, T s_2) = 0$ para todas as seções $s_1,s_2 \in \mathcal{A}^0(E)$. O conjunto dos endomorfismos anti-hermitianos é denotado por $\End(E,h)$. Note que se $T \in \mathcal{A}^0(\End(E,h))$ e $\lambda$ é uma função real então $\lambda T \in \mathcal{A}^0(\End(E,h))$, o que mostra que $\End(E,h)$ é um subfibrado real de $\End(E)$.

Se $\nabla$ é uma conexão hermitiana em $(E,h)$ então sua curvatura é uma $2$-forma com valores no fibrado dos endomorfismos anti-hermitianos, isto é, $F_ \nabla \in \mathcal{A}^2(\End(E,h))$. Para verificar essa afirmação note primeiro que da condição de compatibilidade (\ref{eq:h-compatibility}) e da definição de $h$ agindo em formas com valores em $E$ (eq. \ref{eq:h-forms}) temos que
\begin{equation*}
dh(s_1,s_2) = h(\nabla(s_1),s_2) + (-1)^{k_1}h(s_1,\nabla(s_2))
\end{equation*}
para $s_1 \in \mathcal{A}^{k_1}(E)$ e $s_2 \in \mathcal{A}^{k_2}(E)$.

Sendo assim, para $s_1,s_2 \in \mathcal{A}^0(E)$ temos que
\begin{equation*}
\begin{split}
dh(\nabla (s_1),s_2) &= h(F_\nabla (s_1),s_2) - h(\nabla (s_1), \nabla (s_2)) \\
dh(s_1,\nabla (s_2)) &= h(\nabla (s_1), \nabla (s_2)) + h(s_1,F_\nabla (s_2)).
\end{split}
\end{equation*}
Somando as equações acima e usando novamente a compatibilidade com a métrica obtemos
\begin{equation*}
h(F_\nabla (s_1),s_2) + h(s_1,F_\nabla (s_2)) = dh(\nabla (s_1),s_2) + dh(s_1,\nabla (s_2)) = d(dh(s_1,s_2)) = 0,
\end{equation*}
mostrando que $F_ \nabla \in \mathcal A^2(\End(E,h))$.\\

\textbf{Curvatura de uma conexão compatível com a estrutura holomorfa.} Seja agora $E$ um fibrado holomorfo e $\nabla$ uma conexão compatível com a estrutura holomorfa. Restringindo a extensão $d^\nabla:\mathcal A^k(E) \to \mathcal A ^{k+1}(E)$ a um subespaço $\mathcal A^{p,q}(E)$ com $p+q=k$ obtemos $d^\nabla: \mathcal A^{p,q}(E) \to \mathcal A^{p+1,q}(E) \oplus \mathcal A^{p,q+1}(E)$. Como $\nabla^{0,1}= \delbar_E$, a componente em $\mathcal A^{p,q+1}(E)$ é o operador $\delbar_E: \mathcal A^{p,q}(E) \to \mathcal A^{p,q+1}(E)$. 

Denotando por $(d^\nabla)^{1,0}:\mathcal A^{p,q}(E) \to \mathcal A^{p+1,q}(E)$ a outra componente e usando o fato que $\delbar_E^2 = 0$ (veja a Proposição \ref{prop:delbar-E}) temos que
\begin{equation*}
d^\nabla \circ d^\nabla = ((d^\nabla)^{1,0}+\delbar_E) \circ ((d^\nabla)^{1,0}+\delbar_E) =(d^\nabla)^{1,0} \circ (d^\nabla)^{1,0} + (d^\nabla)^{1,0} \circ \delbar_E + \delbar_E \circ (d^\nabla)^{1,0}.
\end{equation*}
Sendo assim vemos que $F_\nabla:\mathcal{A}^0(E) \to \mathcal A^{2,0} (E)\oplus \mathcal A^{1,1} (E)$, ou seja $F_ \nabla \in (A^{2,0} \oplus \mathcal A^{1,1})(\End(E))$.\\

\textbf{Curvatura da conexão de Chern.} Seja $(E,h)$ um fibrado hermitiano holomorfo e $\nabla$ sua conexão de Chern. Do fato de $\nabla$ ser compatível com a métrica vimos que sua curvatura satisfaz $h(F_ \nabla (s_1),s_2) + h(s_1,F_ \nabla (s_2))= 0$. Como $\nabla$ é compatível com a estrutura holomorfa, vimos que $F_\nabla$ não tem componente $(0,2)$. Sendo assim, tomando a componente $(2,0)$ da equação acima vemos que
\begin{equation*}
0 = h(F_ \nabla^{2,0} (s_1),s_2) + h(s_1,F_ \nabla^{0,2} (s_2)) = h(F_ \nabla^{2,0} (s_1),s_2),
\end{equation*}
para todas as seções $s_1,s_2$ e portanto $F_\nabla^{2,0} = 0$, ou seja, $F_\nabla \in \mathcal{A}^{1,1}(\End(E,h))$.\\

Da discussão acima podemos exibir uma expressão local para a curvatura da conexão de Chern.

Escolha uma trivialização de modo que $E|_U \simeq U \times \C^r$. De acordo com esse isomorfismo, a métrica hermitiana corresponde a uma matriz $H:U \to \GL(r,\C)$ e a conexão de Chern é dada por $\nabla = d + A$ onde $A = \del H \cdot H^{-1}$ é uma matriz de $(1,0)$-formas.

A curvatura dessa conexão é dada pela matriz  $F = dA + A\wedge A$ que, da discussão acima, deve ser uma matriz de $(1,1)$-formas. Em particular o termo $A \wedge A$ deve se anular e no termo $d A = (\del + \delbar)A$ só aparece a derivada $\delbar A$. Concluimos então que a curvatura da conexão de Chern é dada, localmente, por
\begin{equation} \label{eq:local-curv-chern} \index{curvatura!da conexão de Chern}
F = \delbar(\del H \cdot H^{-1})
\end{equation}

\begin{example}
\textbf{Curvatura de fibrados de linha.} Seja $L$ um fibrado de linha holomorfo sobre $X$. Note que o fibrado de endomorfismos de $L$ é trivial, pois $\End(L) \simeq L^* \otimes L \simeq \mathcal O_X$. Sendo assim, a forma de curvatura de uma conexão é uma $2$-forma usual em $X$.

Seja $h$ uma métrica hermitiana em $L$. Um endomorfismo de $L$ é dado pela multiplicação por uma função diferenciável $\lambda$ e esse endomorfismo anti-hermitiano se e só se $\bar \lambda = -\lambda$, isto é, se $\lambda$ assume valores imaginários puros. Desta forma, a curvatura $F$ da conexão de Chern será uma $(1,1)$-forma imaginária pura.

Com relação a uma trivialização $L|_U \simeq U \times \C$, $h$ corresponde a uma função real positiva, e da fórmula (\ref{eq:local-curv-chern}) acima vemos que a curvatura é dada por
\begin{equation} \label{eq:curvature-lb}
F = \delbar(\del h \cdot h^{-1}) = \delbar \del \log (h).
\end{equation}

Um exemplo importante é o do fibrado $\mathcal O(1)$ com a métrica induzida pelas seções $z_0,\ldots,z_n$. Da expressão local desta métrica (veja a equação \ref{eq:metric-O1}), vemos que a curvatura $F$ da conexão de Chern é dada, em $U_i=\{z_i \neq 0 \}$ por
\begin{equation*}
F_{\mathcal O (1)} = \delbar \del \log \frac{1}{1 +  \sum_{j \neq i} \left| \frac{z_j}{z_i} \right|^2} = \del \delbar \log \bigg( 1 + \sum_{j \neq i} \left | \frac{z_j}{z_i} \right|^2 \bigg),
\end{equation*}
que é a mesma expressão da forma de Fubini-Study a menos de um múltiplo (veja a equação \ref{eq:omegaFS-local}). Mais precisamente temos que
\begin{equation} \label{eq:curv-O1}
\frac{\smo}{2 \pi} F_{\mathcal O (1)} =\omega_{FS}.
\end{equation}
\end{example}

O exemplo acima sugere um método para buscarmos métricas de Kähler em uma variedade a partir de fibrados de linha sobre ela.

Seja $X$ uma variedade complexa e $L$ um fibrado de linha holomorfo sobre $X$ com uma métrica hermitiana. A curvatura $F$ da conexão de Chern de $L$ é uma $(1,1)$-forma imaginária pura e portanto $\omega = \smo F$ é uma $(1,1)$-forma real em $X$. Além disso, da identidade de Bianchi, vemos que $\omega$ é fechada (veja a observação \ref{rmk:bianchi-lb}). Sendo assim, para que $\omega$ defina uma métrica de Kähler em $X$, basta verificarmos a positividade de $\omega$ (veja a Proposição \ref{prop:positive-form}).

Note que $\omega = \smo F$ será positiva no sentido da definição \ref{def:positive-form} se e somente se $F(v,\bar v) > 0$ para todo $v \in T^{1,0}X$ não nulo (veja observação \ref{prop:positive-form}). Esse fato motiva a seguinte definição

\begin{definition}
Uma conexão $\nabla$ em um fibrado de linha tem curvatura semi-positiva se $F_\nabla(v , \bar v) \geq 0$ para todo $0 \neq v \in T^{1,0}X$ e  se a desigualdade é estrita dizemos que $\nabla$ tem curvatura positiva. De maneira análoga definimos curvtatura semi-negativa e negativa.
\end{definition}

Da discussão acima temos o seguinte resultado.

\begin{proposition}
Seja $X$ uma variedade complexa e $L$ um fibrado de linha holomorfo hermitiano sobre $X$. Se a conexão de Chern de $L$ tem curvatura positiva então $\omega = \smo F$ define uma métrica de Kähler em $X$. 
\end{proposition}

\section{Classes de Chern}
Nesta seção vamos definir classes características de um fibrado vetorial complexo e suas principais propriedades. A ideia é associar a cada fibrado vetorial complexo $E$ sobre uma variedade $M$, invariantes cohomológicos, isto é, elementos em $H^*(M,\C)$.

 Nosso foco principal será nas chamdadas classes de Chern. A beleza e a utilidade das classes de Chern reside no fato destas poderem ser definidas de diversas maneiras equivalentes. Para fibrados de linha por exemplo já definimos $c_1(L) \in H^2(X,\Z)$ como sendo a imagem de $L \in H^1(X,\mathcal O_X^*)$ pela a aplicação de cobordo $H^1(X,\mathcal O_X^*) \to H^2(X,\Z)$.

A abordagem apresentada aqui se enquadra na chamada Teoria de Chern-Weil, segundo a qual as classes características de $E$ são calculadas a partir de polinômios invariantes aplicados à curvatura de uma conexão em $E$.

\subsubsection{Polinômios invariantes}

Seja $V$ um espaço vetorial sobre $\C$ de dimensão finita. Segundo o isomorfismo $\Hom(V\times \cdots \times V,\C) = V^* \otimes \cdots \otimes V^*$, podemos ver uma aplicação $k$-multilinear simétrica $P:V \times \cdots V \to \C$ como um elemento da álgebra simétrica $S^k(V^*)$.

A um elemento $P \in S^k(V^*)$ corresponde um polinômio homogêneo de grau $k$ em $V$, $\widetilde P:V \to \C$ dado por
\begin{equation*}
\widetilde P (B) = P(B,\ldots,B),~B \in V.
\end{equation*}

Reciprocamente, dada uma aplicação $\widetilde P:V \to \C$ homogênea de grau $k$, podemos obter $P \in S^k(V^*)$ de modo que $P(B,\ldots,B) = \widetilde P(B)$ definindo
\begin{equation*}
P(B_1,\ldots,B_k)=\frac{1}{k!}\frac{\del}{\del \lambda_1} \cdots\frac{\del}{\del\lambda_k} \widetilde P(\lambda_1 B_1 + \cdots +\lambda_k B_k)\bigg|_{\lambda=0}.
\end{equation*}

Para nossos objetivos vamos considerar o caso em que $V = \gl(r,\C)$ é o espaço das matrizes $r\times r$ com coeficientes complexos.
\begin{definition}
Uma forma simétrica $P \in S^k(\gl(r,\C)^*)$ é dito \textbf{invariante} se
\begin{equation*}
P(C B_1 C^{-1},\ldots,C B_k C^{-1}) = P(B_1,\ldots,B_k),~C \in GL(r,\C),~B_j \in \gl(r,\C),
\end{equation*}
ou equivalentemente, se o polinômio associado $\widetilde P$ satisfaz
\begin{equation*}
\widetilde P (C B C^{-1}) = \widetilde P(B), ~C \in GL(r,\C),~B \in \gl(r,\C).
\end{equation*}
\end{definition}

\begin{lemma} \label{lemma:g-invariant}
Se $\widetilde P$ é um polinômio invariante então
\begin{equation} \label{eq:g-invariant}
\sum_{j=1}^k P(B_1,\ldots,B_{j-1},[B,B_j],B_{j+1},\ldots,B_k) = 0
\end{equation}
para todos $B,B_1,\ldots,B_k \in \gl(r,\C)$.
\end{lemma}
\begin{proof}
Considere a aplicação exponencial $\gl(r,\C) \ni B \mapsto e^B \in GL(r,\C)$. Como $\widetilde P$ é invariante, temos, para $C = e^{tB}$, que $P(e^{tB}B_1e^{-tB},\ldots,e^{tB}B_ke^{-tB}) = P(B_1,\ldots,B_k)$ para todo $t$. Da linearidade de $P$ e do fato que $\frac{d}{dt}(e^{tB}B_je^{-tB})|_{t=0} = BB_j - B_jB = [B,B_j]$ obtemos
\begin{equation*}
0 = \frac{d}{dt} P(e^{tB}B_1e^{-tB},\ldots,e^{tB}B_ke^{-tB}) \bigg|_{t=0} = \sum_{j=1}^k P(B_1,\ldots,B_{j-1},[B,B_j],B_{j+1},\ldots,B_k).
\end{equation*}
\end{proof}
\begin{remark}
A definição e o lema acima podem ser enunciados em um contexto mais geral. Seja $G$ um grupo de Lie, $\mathfrak {g}$ sua álgebra de Lie e considere $\text{Ad}:G \to \GL(\mathfrak {g})$ a representação adjunta de $G$. Essa representação induz uma representação em $\mathfrak {g}^*$ via $(\text{Ad}(g) \cdot f)(X) = f(\text{Ad}(g^{-1})X)$ e se estende por derivação a uma ação em $\mathfrak {g} ^* \otimes \cdots \otimes \mathfrak {g} ^*$. Essa ação se restringe ao espaço $S^k(\mathfrak {g}^*)$ dos tensores simétricos. Neste caso, dizemos que $P \in S^k(\mathfrak {g}^*)$ é invariante se $\text{Ad}(g)\cdot P = P$ para todo $g \in G$ e $P$ será invariante se e somente se $\sum_{i=1}^k P(X_1,\ldots,\text{ad}_X(X_i),\ldots, X_k) = 0$ para todos $X,X_1,\ldots,X_k \in \mathfrak{g}$.

No caso particular em que $G = \GL(r,\C)$ e $\mathfrak{g}=\gl(r,\C)$ temos $\text{Ad}(C)\cdot B = CBC^{-1}$ e $\text{ad}_A (B) = [A,B]$, e portanto recuperamos a invariância acima bem como a sua caracterização em termos do colchete.\\
\end{remark}

Seja um fibrado vetorial complexo de posto $r$ sobre uma variedade $M$.

A escolha de uma trivialização $\psi:E|_U \to U \times \C^r$ induz por conjugação um isomorfismo $\Phi: \End(E)|_U \simeq U \times \gl(r,\C)$, isto é, $\Phi(t) = \psi \circ t \circ \psi^{-1}$. Usando esse isomorfismo podemos calcular uma $k$-forma simétrica $P \in S^k(\gl(r,\C))$ em uma $k$-upla de endomorfismos $t_1,\ldots,t_k \in \End(E)(U)$, fazendo
\begin{equation*}
P(t_1,\ldots,t_k) = P(\Phi(t_1),\ldots,\Phi(t_k)),
\end{equation*}
e obtemos assim uma função $P:\End(E)|_U \times \cdots \times \End(E)|_U \to \C$.

Note que a definição acima independe da trivialização, pois uma outra trivialização será da forma $\psi'= F \circ \psi$ onde $F(x,v) = (x,\varphi(x)\cdot v)$ com $\varphi:U \to \GL(r,\C)$ de modo que o isomorfismo induzido em $\End(E)|_U$ é dado por $\Phi' = \varphi \cdot \Phi \cdot \varphi^{-1}$. Como $P$ é invariante concluimos que $P(\Phi(t_1),\ldots,\Phi(t_k)) = P(\Phi'(t_1),\ldots,\Phi'(t_k))$.

O fato da definção independer da trivialização também nos permite definir $P$ globalmente e obtemos uma função bem definida $P:\End(E) \times \cdots \times \End(E) \to \C$.

Podemos estender $P$ a formas com valores em $\End(E)$ e obtemos assim uma aplicação multilinear
\begin{equation*}
\begin{split}
P: \textstyle \bigwedge^{i_1} \End(E) \times \cdots \times \textstyle \bigwedge^{i_k} \End(E) &\longrightarrow \textstyle \bigwedge^{i_1+\cdots +i_k}(T_{\C}M)^* \\
P(\alpha_1 \otimes t_1,\ldots,\alpha_k \otimes t_k) &= (\alpha_1 \wedge \cdots \wedge \alpha_k) P(t_1,\ldots,t_k),
\end{split}
\end{equation*}
que induz uma aplicação nos feixes de formas com valores em $\End(E)$
\begin{equation*}
P:\mathcal A^{i_1} (\End(E)) \times \cdots \times \mathcal A^{i_k} (\End(E)) \longrightarrow \textstyle \mathcal A_{M,\C}^{i_1+\cdots +i_k}.
\end{equation*}

Note que, como $P \in S^k(\gl(r,\C))$ é simétrica, a aplicação definida acima é graduado simétrica e além disso
\begin{equation} \label{eq:wedge-P}
\beta \wedge P(\gamma_1,\ldots, \gamma_k) = (-1)^{l(i_1+\cdots+i_{j-1})}P(\gamma_1,\ldots,\beta \wedge \gamma_j,\ldots,\gamma_k),~\beta \in \mathcal A^l_{M,\C}.
\end{equation}

Antes de continuarmos é conveniente definirmos a operação
\begin{equation*}
\begin{split}
\star: \mathcal{A}^k(\End(E)) \times \mathcal{A}^l(\End(E)) &\longrightarrow \mathcal{A}^{k+l}(\End(E)) \\
( \alpha \otimes s, \beta \otimes t ) &\longmapsto (\alpha \wedge \beta) \otimes (s \circ t)
\end{split}
\end{equation*}
e o colchete de dois elementos $A \in \mathcal{A}^k(\End(E))$ e $B \in \mathcal{A}^l(\End(E))$ por
\begin{equation*}
[A,B] = A \star B - (-1)^{kl} B \star A.
\end{equation*}

Com essa convenção temos, do lema \ref{lemma:g-invariant}, que se $P$ é uma forma invariante então a aplicação induzida nas formas com valores em $\End(E)$ satisfaz
\begin{equation} \label{eq:g-invariant-forms}
\sum_{j=1}^k (-1)^{l(i_1 + \cdots + i_{j-1})}P(\gamma_1,\ldots,\gamma_{j-1},[a,\gamma_j],\gamma_{j+1},\ldots,\gamma_k) = 0,~ \gamma_j \in \mathcal A^{i_j}(\End(E)), a \in \mathcal{A}^l(\End(E)).
\end{equation}
Para verificar a fórmula acima podemos supor que $\gamma_j = \alpha_j \otimes t_j$ e $a = \beta \otimes t$. Fixada uma trivialização local sobre $U$ seja $\Phi: \End(E)|_U \simeq U \times \gl(r,\C)$ o isomorfismo induzido e denote $B_i = \Phi(t_i)$, $B=\Phi(t)$, $\eta_i =\Phi(\gamma_i) = \alpha_i \otimes B_i$ e $ A = \Phi(a)=\beta \otimes B$.

Temos então que
\begin{equation*}
\begin{split}
P(\gamma_1,\ldots,[a,\gamma_j]),\ldots, \gamma_k)|_U &\stackrel{\text{def}}{=} P(\eta_1,\ldots,[A,\eta_j],\ldots, \eta_k)\\
&= P(\eta_1,\ldots, \beta \wedge \alpha_j B B_j - (-1)^{l i_j}  \alpha_j \wedge \beta B_j B,\ldots, \eta_k) \\
&=P(\eta_1,\ldots, \beta \wedge \alpha_j B B_j -  \beta \wedge \alpha_j B_j B,\ldots, \eta_k) \\
&=P(\eta_1,\ldots, \beta \wedge \alpha_j B B_j,\ldots, \eta_k) - P(\eta_1,\ldots, \beta \wedge \alpha_j B_j B,\ldots, \eta_k)\\
&=\alpha_1 \wedge \cdots \wedge \beta \wedge \alpha_j \wedge \cdots \wedge \alpha_k [(P(B_1,\ldots,BB_j,\ldots,B_k) - P(B_1,\ldots,B_j B,\ldots,B_k)] \\
&=(-1)^{l(i_1+\cdots+i_{j-1})}\beta \wedge \alpha_1 \wedge \cdots \wedge \alpha_k P(B_1,\ldots,[B,B_j],\ldots,B_k),
\end{split}
\end{equation*}
e portanto somando e usando o lema \ref{lemma:g-invariant} obtemos
\begin{equation*}
\sum_{j=0}^k (-1)^{l(i_1+\cdots+i_{j-1})} P(\gamma_1,\ldots,[a,\gamma_j]),\ldots, \gamma_k)|_U = \beta \wedge \alpha_1 \wedge \cdots \wedge \alpha_k \sum_{j=0}^k P(B_1,\ldots,[B,B_j],\ldots,B_k) = 0.
\end{equation*}

\begin{lemma} \label{lemma:dP}
Se $P$ é uma forma simétrica invariante e $\gamma_j \in \mathcal{A}^{i_j}(\End(E))$ então, para toda conexão $\nabla$ em $E$, temos que
\begin{equation*}
dP(\gamma_1,\ldots,\gamma_k) = \sum_{j=1}^k (-1)^{i_1 + \cdots + i_{j-1}} P(\gamma_1,\ldots,d^\nabla(\gamma_j),\ldots,\gamma_k),
\end{equation*}
onde $d^\nabla$ denota a extensão de $\nabla$ a $\mathcal A^\bullet(\End(E))$.
\end{lemma} 

\begin{proof}
Por linearidade podemos supor que $\gamma_i = \alpha_i \otimes t_i$ onde $\alpha_i$ é uma forma diferencial e $t_i$ um endomorfismo de $E$. Escolha uma trivialização de $E$ sobre $U$ e denote por $s_i = \Phi(t_i)$ o elemento de $U \times \gl(r,\C)$ obtido de $t_i$ por conjugação. Denote $\beta_i =\alpha_i \otimes s_i = \Phi(\gamma_i)$.

Da definição de $P$ temos que $P(\gamma_1,\ldots,\gamma_k)|_U = (\alpha_1 \wedge \cdots \wedge \alpha_k) P(s_1,\ldots,s_k)$ e portanto, da regra de Leibniz para formas diferenciais temos que
\begin{equation*}
\begin{split}
&dP(\gamma_1,\ldots,\gamma_k)|_U =\\
&\sum_{j=1}^k (-1)^{i_1 + \cdots + i_{j-1}} \alpha_1 \wedge \cdots \wedge d \alpha_j \wedge \cdots \wedge \alpha_k P(s_1,\ldots,s_k) + (-1)^{i_1 + \cdots + i_k} \alpha_1 \wedge \cdots \wedge \alpha_k dP(s_1,\ldots,s_k)\\
&=\sum_{j=1}^k (-1)^{i_1 + \cdots + i_{j-1}} P(\beta_1,\ldots,d\alpha_j \otimes s_j,\ldots,\beta_k) + (-1)^{i_1 + \cdots + i_k} \alpha_1 \wedge \cdots \wedge \alpha_k \sum_{j=1}^k P(s_1,\ldots,d s_j,\ldots s_k) \\
&=\sum_{j=1}^k (-1)^{i_1 + \cdots + i_{j-1}} P(\beta_1,\ldots,d\alpha_j \otimes s_j,\ldots,\beta_k) +(-1)^{i_1 + \cdots + i_k} \sum_{j=1}^k (-1)^{i_{j+1}+\cdots+i_k} P(\beta_1,\ldots,\alpha_j \wedge d s_j,\ldots \beta_k)\\
&= \sum_{j=1}^k (-1)^{i_1 + \cdots + i_{j-1}} P(\beta_1,\ldots,d\alpha_j \otimes s_j + (-1)^{i_j} \alpha_j \wedge ds_j,\ldots,\beta_k) \\
&= \sum_{j=1}^k (-1)^{i_1 + \cdots + i_{j-1}} P(\beta_1,\ldots,d\beta_j,\ldots,\beta_k),
\end{split}
\end{equation*}
onde, na segunda igualdade usamos o fato de $P$ ser multilinear e na terceira igualdade aplicamos a fórmula (\ref{eq:wedge-P}).

A conexão induzida por $\Phi$ em $U \times \gl(r,\C)$ será dada por $\nabla T = dT + AT - TA$ onde $A$ é uma matriz de $1$- formas (veja o exemplo \ref{ex:connection-trivial-end}) e a sua extensão para $k$-formas com valores em $U \times \gl(r,\C)$ é dada por $d^\nabla(\Theta) = d \Theta + A\wedge \Theta - (-1)^k \Theta \wedge A = d \Theta + [A,\Theta]$. Assim, usando a equação (\ref{eq:g-invariant-forms}) e o cálculo acima concluimos que
\begin{equation*}
\begin{split}
dP(\gamma_1,\ldots,\gamma_k)|_U &= \sum_{j=1}^k (-1)^{i_1 + \cdots + i_{j-1}} P(\beta_1,\ldots,d\beta_j,\ldots,\beta_k)\\
&= \sum_{j=1}^k (-1)^{i_1 + \cdots + i_{j-1}} P(\beta_1,\ldots,d\beta_j+[A,\beta_j],\ldots,\beta_k) \\
&= \sum_{j=1}^k (-1)^{i_1 + \cdots + i_{j-1}} P(\beta_1,\ldots,d^\nabla(\beta_j),\ldots,\beta_k)\\
&= \sum_{j=1}^k (-1)^{i_1 + \cdots + i_{j-1}} P(\gamma_1,\ldots,d^\nabla(\gamma_j),\ldots,\gamma_k)|_U,
\end{split}
\end{equation*}
terminando a demonstração.
\end{proof}

\begin{corollary}
Seja $E$ um fibrado complexo de posto $r$ sobre $M$, $\nabla$ uma conexão em $E$ e $F_\nabla \in \mathcal A^2(\End(E))$ sua curvatura. Então, para qualquer $k$-forma simétrica invariante $P$ em $\gl(r,\C)$, a forma $\widetilde P(F_\nabla) = P(F_\nabla,\ldots,F_\nabla) \in \mathcal A_{\C}^{2k}(M)$ é fechada. 
\end{corollary}

\begin{proof}
Da identidade de Bianchi temos que $d^\nabla F_\nabla=0$ e portanto, usando o lema anterior, concluímos que $d \widetilde P (F_\nabla) = P(d^\nabla F_\nabla, F_\nabla, \ldots, F_\nabla) + \cdots + P(F_\nabla, F_\nabla, \ldots, d^\nabla F_\nabla) = 0$.
\end{proof}

Vemos então que, dada uma conexão $\nabla$ em $E$ e um polinômio invariante $\widetilde P$ de grau $k$ obtemos uma classe de cohomologia $[\widetilde P (F_\nabla)] \in H^{2k}(M,\C)$. O lema a seguir mostra que esta classe independe da conexão escolhida.

\begin{proposition} \label{prop:nabla-independent}
Se $\nabla$ e $\nabla'$ são duas conexões em $E$ e $\widetilde P$ é um polinômio invariante então $[\widetilde P(F_\nabla)] = [\widetilde P(F_{\nabla'})]$.
\end{proposition}

\begin{proof}
Precisamos mostrar que $\widetilde P(F_\nabla)$ e $\widetilde P(F_{\nabla'})$ diferem por uma forma diferencial exata.

Vimos na seção \ref{sec:connections} que $\nabla - \nabla' = A \in \mathcal{A}^1(\End(E))$. Considere a família de conexões dada por $\nabla^t = \nabla + tA$. Note que $\nabla^0 = \nabla$ e $\nabla^1 = \nabla'$, ou seja, $\nabla^t$ é uma curva no espaço afim das conexões ligando $\nabla$ a $\nabla'$.

A curvatura da conexão $\nabla^t$ é dada por $F_t = F + d^\nabla(tA) + tA \star tA = F + t d^\nabla A + t^2 A \star A$, onde $F$ é a curvatura de $\nabla$. Lembrando que um elemento $\Theta \in \mathcal A^k(\End(E))$ age em $B \in \mathcal A^l(\End(E))$ via $B \mapsto [\Theta,B] = \Theta \star B - (-1)^{kl} B \star \Theta$ temos que
\begin{equation*}
d^{\nabla^t} A = d^\nabla A + [tA,A] = d^\nabla A + t A \star A + t A \star A = d^\nabla A +2t A \star  A = \frac{d}{dt}F_t.
\end{equation*}

Da identidade de Bianchi para cada $\nabla^t$ temos que $d^{\nabla^t}(F_t) = 0$ e portanto, do lema \ref{lemma:dP} e da linearidade de $P$ temos que
\begin{equation*}
kdP(A,F_t,\ldots,F_t) = kP(d^{\nabla^t}A,F_t,\ldots,F_t) = kP\bigg(\frac{d}{dt}F_t,F_t,\ldots,F_t\bigg) = \frac{d}{dt} P(F_t,\ldots,F_t) = \frac{d}{dt} \widetilde P(F_t),
\end{equation*}
de onde concluimos que
\begin{equation*}
\widetilde P(F_\nabla)-\widetilde P(F_{\nabla'}) = \widetilde P(F_1) - \widetilde P(F_0) = \int_0^1 \frac{d}{dt} \widetilde P(F_t) = kd \int_0^1 P(A,F_t,\ldots,F_t)
\end{equation*}
e portanto $\widetilde P(F_\nabla)$ e $\widetilde P(F_{\nabla'})$ definem a mesma classe em $H^{2k}(M,\C)$.
\end{proof}

Chegamos agora ao ponto que queríamos. Dado um fibrado $E \to M$ de posto $r$ podemos associar a cada polinômio homogêneo invariante $\widetilde P$ em $\gl(r,\C)$ um elemento em $H^{2k}(M,\C)$. Em outras palavras, obtivemos uma aplicação linear
\begin{equation*}
(S^k\gl(r,\C)^*)^{\GL(r,\C)} \longrightarrow H^{2k}(M,\C),
\end{equation*}
onde $(S^k\gl(r,\C)^*)^{\GL(r,\C)}$ denota o espaço dos polinômios invariantes de grau $k$.

Note ainda que essa associação é um homomorfismo de álgebras, isto é, a classe associado ao produto de dois polinômios $\widetilde P$ e $\widetilde Q$ é o produto da classe associada a $\widetilde P$ com a classe associada a $\widetilde Q$. Obtemos assim um homomorfismo de álgebras
\begin{equation*} \index{homomorfismo de Chern-Weil}
(S^*\gl(r,\C)^*)^{\GL(r,\C)} \longrightarrow H^{2*}(M,\C),
\end{equation*}
chamado homomorfismo de \emph{Chern-Weil}.

As classes de cohomologia obtidas dessa forma são chamadas \textbf{classes características} de $E$. \index{classes!características} Um dos exemplos mais importantes de classes características são as chamadas Classes de Chern.

\subsubsection{Classes de Chern}
As classes de Chern são definidas tomando como polinômios invariantes os polinômios simétricos elementares nos autovalores de uma matriz.

Um polinômio $P \in \C[x_0,\ldots,x_r]$ é dito simétrico se $P(x_{\sigma(1)},\ldots,x_{\sigma(r)}) = P(x_1,\ldots,x_n)$ para toda permutação $\sigma:\{1,\ldots,r\} \to \{1,\ldots,r\}$. Os exemplos básicos de polinômios simétricos são os chamados \emph{polinômios simétricos elementares}, definidos por
\begin{equation*}
\begin{split}
e_0(x_1,\ldots,x_r) &= 1\\
e_1(x_1,\ldots,x_r) &= \sum_{1\leq i \leq r} x_i\\
e_2(x_1,\ldots,x_r) &= \sum_{1\leq i < j \leq r} x_i x_j \\
e_3(x_1,\ldots,x_r) &= \sum_{1\leq i < j < k \leq r} x_i x_j x_k \\
 &~~\vdots \\
e_r(x_1,\ldots,x_r) &= x_1\cdots x_r.
\end{split}
\end{equation*}

É um resultado clássico da teoria de polinômios que qualquer polinômio simétrico $P$ pode ser escrito como $P(x_1,\ldots,x_r)=Q(e_1(x_1,\ldots,x_r),\ldots,e_r(x_1,\ldots,x_r))$, onde $Q \in \C[x_1,\ldots,x_r]$, isto é, $P$ é um polinômio nos $e_1,\ldots,e_r$.

Usando os polinômios simétricos elementares podemos definir polinômios invariantes em $\gl(r,\C)$ da seguinte maneira. Seja $B \in \gl(r,\C)$ uma matriz e sejam $\lambda_1,\ldots,\lambda_r$ seus autovalores. Defina os polinômios 
\begin{equation*}
\widetilde P_k (B) = e_k(\lambda_1,\cdots,\lambda_r),~k=0,\ldots,r,
\end{equation*}
e $\widetilde P_k = 0$ se $k > r$.

Note que a definição de $\widetilde P_k (B)$ independe de como ordenamos os autovalores justamente por que os $e_k$ são simétricos.

Como os autovalores de uma matriz são invariantes por conjugação, vemos que os polinômios $\widetilde P_k (B)$ são invariantes.

Note que os $\widetilde P_k$ podem ser caracterizados pela equação
\begin{equation} \label{eq:det-I+B}
\det(I + B) = 1 + \widetilde P_1(B) + \widetilde P_2(B) + \cdots + \widetilde P_r(B).
\end{equation}
Para verificar a igualdade acima note primeiro que se $B=\text{diag}(\lambda_1,\ldots,\lambda_r)$ então
\begin{equation*}
\det(I+B) = \prod_{j=1}^r(1+\lambda_j) = 1 + e_1(\lambda_1,\ldots,\lambda_r) + \cdots + e_r(\lambda_1,\ldots,\lambda_r)
\end{equation*}
e portanto a fórmula (\ref{eq:det-I+B}) vale para matrizes diagonais. Como ambos os membros são invariantes por conjugação vemos que (\ref{eq:det-I+B}) vale também no espaço $\mathcal D \subset \gl(r,\C)$ das matrizes diagonalizáveis e como  $\mathcal D$ é denso em $\gl(r,\C)$ obtemos a igualdade em toda parte.

Note que $e_0(B) = 1$, $e_1(B) = \text{tr} B$ e $e_r(B) = \det B$.

\begin{definition} \label{def:chern-class} \index{classes!de Chern}
Seja $E$ um fibrado complexo de posto $r$ sobre $M$ e $\nabla$ uma conexão em $E$. A $k$-ésima forma de Chern de $E$ com respeito a $\nabla$ é a forma diferencial
\begin{equation*}
c_k(E,\nabla) = \widetilde P_k \bigg( \frac{\smo}{2 \pi} F_\nabla \bigg) \in \mathcal A^{2k}_\C (M),
\end{equation*}
e a \textbf{$k$-ésima classe de Chern} de $E$ é classe de cohomologia
\begin{equation*} \index{primeira classe de Chern}
c_k(E) = [c_k(E,\nabla)] \in H^{2k}(M,\C),
\end{equation*}
que pela proposição \ref{prop:nabla-independent} independe de $\nabla$.

A \textbf{classe de Chern total} de $E$ é o elemento
\begin{equation*}
c(E) = c_0(E) + c_1(E) + \ldots + c_r(E) \in H^{2*}(M,\C).
\end{equation*}
\end{definition}

A proposição a seguir resume as principais propriedades das classes de Chern.
\begin{proposition}
As classes de Chern satisfazem as seguintes propriedades:
\begin{itemize}
\item[1.] Naturalidade - Se $E$ é um fibrado complexo sobre $M$ e $f: M \to N$ é uma aplicação diferenciável então $c(f^*E) = f^*c(E)$, onde $f^*:H^*(M,\C) \to H^*(N,\C)$ denota o pullback na cohomologia.

\item[2.] Fórmula de Whitney - Se $E$ e $F$ são fibrados sobre $M$ então $c(E \oplus F) = c(E)\cdot c(F)$ onde $\cdot$ é o produto na cohomologia induzido pelo produto exterior de formas.

\item[3.] Normalização - A primeira classe de Chern do fibrado tautológico sobre $\pr^1$ é dada por $c_1(\mathcal O(-1)) = - [\omega_{FS}]$.
\end{itemize}
\end{proposition}

\begin{proof}
1. Seja $\nabla$ uma conexão em $E$ e considere a conexão pullback em $f^*E$. Na proposição \ref{prop:induced-curvatures}, vimos que sua curvatura é dada por $F_{f^*\nabla} = f^*F_\nabla$ e portanto
\begin{equation*}
c_k(f^*E) = \left[ \widetilde P_k \left( \frac{\smo}{2\pi} F_{f^*\nabla}\right)\right] = \left[ \widetilde P_k \left( \frac{\smo}{2\pi}  f^*F_\nabla \right)\right] = f^*\left[ \widetilde P_k \left( \frac{\smo}{2\pi} F_{\nabla}\right)\right] = f^*c_k(E).
\end{equation*}

2. Sejam $\nabla_E$, $\nabla_F$ conexões em $E$ e $F$ respectivamente e considere a conexão $\nabla = \nabla_E + \nabla_F$. Da proposição \ref{prop:induced-curvatures} sua curvatura é $F_\nabla = F_{\nabla_E} \oplus F_{\nabla_F}$. Temos então que
\begin{equation*}
\det(I + \textstyle \frac{\smo}{2\pi}F_\nabla) = \det \left( \begin{split} I+\textstyle \frac{\smo}{2\pi} F_{\nabla_E} ~ 0 \\ 0 ~ I+\textstyle\frac{\smo}{2\pi} F_{\nabla_F} \end{split} \right) = \det(I + \textstyle \frac{\smo}{2\pi}F_{\nabla_E}) \cdot \det(I + \textstyle \frac{\smo}{2\pi}F_{\nabla_F}),
\end{equation*}
e portanto $c(E \oplus F) = [\det(I + \textstyle \frac{\smo}{2\pi}F_\nabla)] = [\det(I + \textstyle \frac{\smo}{2\pi}F_{\nabla_E})]\cdot[\det(I + \textstyle \frac{\smo}{2\pi}F_{\nabla_F})] = c(E)\cdot c(F)$.

3. Seja $\nabla$ a conexão de Chern em $\mathcal O(1)$. Usando a equação (\ref{eq:curv-O1}) temos que $c_1(\mathcal O(1)) = [\frac{\smo}{2\pi}F] = [\omega_{FS}]$. Considerando a conexão dual em $\mathcal O(-1)$ e lembrando que $F_{\nabla^*} = -F_\nabla$ (proposição \ref{prop:induced-curvatures}) concluimos que
\begin{equation*}
c_1(\mathcal O(-1)) = \left[ \frac{\smo}{2\pi} F_{\nabla^*}\right] = - \left[ \frac{\smo}{2\pi} F_{\nabla}\right] = - [\omega_{FS}].
\end{equation*}

\end{proof}

\begin{remark}
As propriedades $1-3$ da proposição acima caracterizam as classes de Chern e podem ser usadas como uma definição axiomática. Veja por exemplo \cite{k-n2}, cap. XII.
\end{remark}

\begin{remark} Seja $E$ um fibrado holomorfo sobre uma variedade complexa $X$. Nesse caso, as formas de Chern terão tipo $(k,k)$. Esse fato pode ser visto da seguinte maneira.
Considere uma métrica hermitiana $h$ e seja $\nabla$ a conexão de Chern associada. Na seção \ref{subsec:chern-connection} vimos que $F_\nabla \in \mathcal{A}^{1,1}(\End(E))$, de modo que $\widetilde P_k(F_\nabla) \in \mathcal A^{k,k}(X)$.

Se além disso, $X$ é uma variedade de Kähler, de acordo com a decomposição de Hodge (\ref{eq:kahler-hodge}), as classes de Chern serão de tipo $(k,k)$, isto é, $c_k(E) \in H^{k,k}(X)$.\\
\end{remark}

Para fibrados de linha holomorfos já havíamos encontrado, na seção \ref{subsec:expseq}, uma definição da primeira classe de Chern como sendo a imagem de $L$ pela a aplicação de cobordo $\delta: \Pic(X) \to H^2(X,\Z)$ induzida pela sequência exponencial. Vimos  também que essa definição pode ser feita na categoria diferenciável (veja o exemplo \ref{ex:hol-structures-lb}) e que nesse caso a aplicação induzida $\delta: H^1(M,(\mathcal{C}_X^{\infty})^*) \to H^2(M,\Z)$ é um isomorfismo.

Considerando a aplicação $H^2(M,\Z) \to H^2(M,\C)$ induzida pela inclusão de feixes $\Z \subset \C$ podemos comparar as duas definições. Para uma demonstração do resultado a seguir consulte \cite{huybrechts}.

\begin{proposition}
Seja $L$ um fibrado de linha complexo sobre $M$ visto como um elemento de $H^1(M,(\mathcal{C}_X^{\infty})^*)$  e considere a aplicação de cobordo $\delta: H^1(M,(\mathcal{C}_X^{\infty})^*) \to H^2(M,\Z)$. Então a imagem de $\delta(L)$ pela aplicação natural $H^2(M,\Z) \to H^2(M,\C)$ é $-c_1(L)$.
\end{proposition}

\begin{corollary}
Seja $\omega_{FS}$ a forma de Fubini-Study em $\pr^n$. Então a classe de cohomologia $c=[\omega_{FS}]$ é integral, isto é, $c$ pertence a imagem da aplicação natural $H^2(\pr^n,\Z) \to H^2(\pr^n,\C)$.
\end{corollary}
\begin{proof}
Considere o fibrado $\mathcal O(1)$ sobre $\pr^n$. Usando o resultado da proposição acima temos que
\begin{equation*}
c=[\omega_{FS}] = c_1(\mathcal O (1)) = -\delta(\mathcal O(1)) \in \im H^2(\pr^n,\Z) \to H^2(\pr^n,\C).
\end{equation*}
\end{proof}

\begin{example} \label{ex:lb-pn} \index{grupo de Picard!de $\pr^n$} \textbf{Fibrados de Linha sobre $\pr^n$.}
No exemplo \ref{ex:picard-Pn} vimos que o grupo de Picard de $\pr^n$ é isomorfo a $\Z$, onde o isomorfismo é dado por $\delta=-c_1:\Pic(\pr^n) \to H^2(X,\Z) \simeq \Z$. Sendo assim, $\Pic(\pr^n)$ é gerado por uma pré-imagem de um gerador de $H^2(\pr^n,\omega)$.

\begin{lemma}
A classe $c = [\omega_{FS}]$ gera $H^2(\pr^n,\Z)$.
\end{lemma}

\begin{proof}
O grupo de homologia $H_2(\pr^n,\Z)$ tem posto $1$ e é gerado pela classe de $\pr^1 \subset \pr^n$. Do Teorema dos Coeficientes Universais para a cohomologia\footnote{O Teorema dos Coeficientes Universais para a cohomologia diz que para todo grupo abeliano de coeficientes $G$ existe uma sequência exata $0 \to \text{Ext}(H_{i-1}(X,\Z),G) \to H^i(X,G) \to \Hom(H_i(X,\Z),G) \to 0$. Em particular, se $H_{i-1}(X,\Z)$ é livre então $H^i(X,G) \simeq \Hom(H_i(X,\Z),G)$. Para uma demonstração consulte \cite{hatcher}.} temos que $H^2(\pr^n,\Z) \simeq \Hom(H_2(\pr^n,\Z),\Z)$ e o isomorfismo é dado pela integração $[\alpha] \mapsto \left(V \mapsto \int_V \alpha \right)$. Assim, para mostrar que $[\omega_{FS}]$ gera $H^2(\pr^n,\Z)$ basta mostrar que $\int_{\pr^1} \omega_{FS} = 1$, o que pode ser visto pelo cálculo
\begin{equation*}
\begin{split}
\int_{\pr^1} \omega_{FS} &= \frac{\smo}{2 \pi} \int_{\C} \frac{1}{(1 + |w|^2)^2} dw \wedge d \bar w \\
&= \frac{1}{\pi} \int_{\R^2} \frac{1}{(1 + x^2 + y^2)^2} dx \wedge dy \\
&= 2 \int_0^{\infty} \frac{r dr}{(1+r^2)^2} = 1,
\end{split}
\end{equation*}
onde usamos as substituições $w = x + \smo y$ e $(x,y) = (r\cos \theta,r~ \text{sen}~ \theta )$.

\end{proof}

Como $c_1(\mathcal O(1)) = [\omega_{FS}] $, vemos então que $\Pic(\pr^n)$ é gerado por $\mathcal O(1)$. Em outras palavras, \textbf{todo fibrado de linha sobre $\pr^n$ é isomorfo a algum $\mathcal O (k)$}.
\end{example}

\begin{example} \textbf{A Classe de Chern de $\pr^n$}

Definimos as classes de Chern de uma variedade complexa $X$ por $c_k(X) = c_k (\mathcal T_x)$. Vamos calcular a classe de Chern do espaço projetivo.

Da sequência de Euler $0 \to \mathcal{O}_{\pr^n} \to \mathcal O(1)_{\pr^n}^{\oplus n+1} \to \mathcal{T}_{\pr^n} \longrightarrow 0$ temos que $\mathcal O(1)_{\pr^n}^{\oplus n+1} \simeq \mathcal{O}_{\pr^n} \oplus \mathcal{T}_{\pr^n}$ como fibrados vetoriais complexos e portanto $c(\mathcal O(1)_{\pr^n}^{\oplus n+1}) = c(\mathcal{O}_{\pr^n} \oplus \mathcal{T}_{\pr^n})$.

Da fórmula de Whitney temos que $c(\mathcal{O}_{\pr^n} \oplus \mathcal{T}_{\pr^n}) = c(\mathcal{O}_{\pr^n})\cdot c( \mathcal{T}_{\pr^n}) = c( \mathcal{T}_{\pr^n})$ pois $c(\mathcal{O}_{\pr^n}) = 1$ e como $c(\mathcal O(1)) = 1 + c_1(\mathcal O(1)) = 1 + \omega_{FS}$ temos, usando novamente a fórmula de Whitney, que $c(\mathcal O(1)_{\pr^n}^{\oplus n+1}) = (1+\omega_{FS})^{n+1}$.

Concluímos portanto que a classe de Chern de $\pr^n$ é dada por
\begin{equation*}
c(\pr^n) = (1 + \omega_{FS})^{n+1}.
\end{equation*}
\end{example}

\subsection{A Conjectura de Calabi e métricas de Kähler-Einstein}

Dada uma variedade diferenciável $X$, uma pergunta comum da Geometria Riemanniana é se existem ``boas'' métricas em $X$, como por exemplo métricas com curvatura seccional constante ou métricas de Einstein.

No caso de variedades de Kähler, o mesmo tipo de pergunta pode ser feita: quando uma variedade de Kähler admite ``boas'' métricas? Nesta seção iremos discutir dois resultados nesta direção. O primeiro deles é a chamada Conjectura de Calabi, que diz que, dada uma variedade de Kähler $(X,g)$ com forma de Kähler $\omega$, podemos encontrar uma outra métrica $g'$ de modo que $[\omega']=[\omega]$ e $g'$ tem curvatura de Ricci especificada. Mencionaremos também uma segunda conjectura importante, relacionada à existência das chamadas métricas de Kähler-Einstein.\\

Vamos começar relembrando algumas definições da Geometria Riemanniana. Se $(X,g)$ é uma variedade riemanniana e $D$ é a conexão de Levi-Civita em $TX$, o \emph{tensor de curvatura} de $(X,g)$ é definido por
\begin{equation*}
R(X,Y)Z = D_X D_Y Z  - D_Y D_X Z - D_{[X,Y]} Z,
\end{equation*}
e o \emph{tensor de Ricci} é o tensor
\begin{equation*}
r(X,Y) = \text{tr} (Z \mapsto R(Z,X)Y)).
\end{equation*}

Em um ponto $x \in X$ o tensor de Ricci pode ser escrito, em termos de uma base ortonormal $\{e_1,\ldots,e_n\}$ de $T_x X$, como
\begin{equation*}
r(X,Y) = \sum_{i=1}^n g(R(e_i,X)Y,e_i)
\end{equation*}

O tensor de curvatura satisfaz as seguintes identidades
\begin{equation} \label{eq:R-identities}
\begin{split}
g(R(X,Y)Z,W) = g(R(Z,W)X,Y)& \\
R(X,Y)Z + R(Y,Z)X + R(Z,X)Y = 0& \\
g(R(X,Y)Z,W) + g(Z,R(X,Y)W) = 0&.
\end{split}
\end{equation}

Dessas identidades é fácil ver que $r$ é um tensor simétrico, e portanto tem o mesmo tipo que $g$. Dizemos que a métrica $g$ é uma \emph{métrica de Einstein} se existe $\lambda \in \R$ tal que $r = \lambda g$.\\

Suponha agora que $(X,g)$ seja uma variedade de Kähler.

\begin{definition}
A \emph{forma de Ricci} de $(X,g)$, denotada por $\rho$, é a $2$-forma associada ao tensor de Ricci $r$, isto é,
\begin{equation*} 
\rho = r(J \cdot, \cdot).
\end{equation*}
\end{definition}

Em termos de uma base ortonormal $\{e_1,\ldots,e_{2n} \} \subset T_x X$ a forma de Ricci em $x$ é dada por
\begin{equation} \label{eq:ricci-basis}
\rho(X,Y) = \sum_{i=1}^{2n} g(R(e_i,JX)Y,e_i) = \sum_{i=1}^{2n} g(R(Y,e_i)e_i,JX).
\end{equation}

Tanto o tensor de Ricci quanto a forma de Ricci podem ser estendidos de forma $\C$-bilinear a $TX \otimes \C$. Em particular podemos calcular $r$ e $\rho$ em vetores de $T^{1,0}X$.

No Exemplo \ref{ex:chern-levi-civita} vimos que a conexão de Levi-Civita em $TX$ induz naturalmente uma conexão $\nabla$ em $T^{1,0}X$ que coincide com a conexão de Chern. Além disso, pelos cálculos feitos no Exemplo \ref{ex:riemann-curvaure}, o tensor de curvatura $R$ corresponde à curvatura $F_\nabla$ de $\nabla$.

\begin{lemma}
Se $u,v \in T^{1,0}X$, a forma de Ricci de $(X,g)$ é dada por $\rho(u,v) = \smo \text{tr}_\C (F_\nabla (u,v))$.
\end{lemma}

\begin{proof}
Denote por $g_\C$ a extensão de $g$ a $T_\C X$ pela fórmula $g_\C(u \otimes \lambda,v \otimes \mu) = \lambda \bar \mu g(u,v)$. É fácil ver que, segundo o isomorfismo $\xi:TX \to T^{1,0}X$, $u \mapsto \frac{1}{2}(u - \smo Ju)$, a restrição de $g_\C$ a $T^{1,0}X$ corresponde a $\frac{1}{2}h = \frac{1}{2}(g - \smo \omega)$.

Fixe $x \in X$ e escolha uma base ortonormal de $TX$ da forma $\{x_1,\ldots,x_n,y_1,\ldots,y_n\}$ com $y_i~=~Jx_i$. Considere os vetores $z_i = \xi(x_i)=\frac{1}{2}(x_i - \smo Jx_i) \in T^{1,0}_x X$. Da observação acima temos que os $z_i$'s são dois a dois ortogonais e $g_{\C}(z_i,z_i) = 2$. Assim, os vetores $\zeta_i = \sqrt 2 z_i = \frac{1}{\sqrt 2}(x_i - \smo Jx_i)$ formam uma base $g_\C$-ortonormal de $T^{1,0}_x X$.

Da discussão feita no Exemplo \ref{ex:riemann-curvaure}, o endomorfismo $F_\nabla (u,v): T^{1,0}X \to T^{1,0}X$ é obtido  primeiro estendendo $R(u,v):TX \to TX$ a $T_\C X$ e depois o restringindo a $T^{1,0}X$. Vemos portanto que
\begin{equation*}
\begin{split}
\text{tr} F_\nabla(u,v) &= \sum_{i=1}^n g_\C (F_\nabla (u,v)\zeta_i,\zeta_i) = 2 \sum_{i=1}^n g_\C (F_\nabla (u,v)z_i,z_i) \\
 &= \sum_{i=1}^n h(R(u,v)x_i,x_i) = \sum_{i=1}^n \left[ g(R(u,v)x_i,x_i) - \smo \omega(R(u,v)x_i,x_i) \right] \\
 &= \sum_{i=1}^n \left[ g(R(u,v)x_i,x_i) + \smo g(R(u,v)x_i,y_i) \right] \\
 &= \smo \sum_{i=1}^n g(R(u,v)x_i,y_i),
\end{split}
\end{equation*}
onde na última igualdade usamos o fato de $R(u,v)$ ser antisimétrico (veja eq.~\ref{eq:R-identities}).

Agora, usando a expressão (\ref{eq:ricci-basis}), as identidades (\ref{eq:R-identities}) e o fato de $R(\cdot,\cdot)$ comutar com $J$ obtemos
\begin{equation*}
\begin{split}
\rho(u,v) &= \sum_{i=1}^{n} \left[ g(R(v,x_i)x_i,Ju) + g(R(v,y_i)y_i,Ju) \right] \\
 &= \sum_{i=1}^{n} \left[- g(R(v,x_i)y_i,u) + g(R(v,y_i)x_i,u) \right] \\
 &= - \sum_{i=1}^{n} g(R(y_i,x_i)v,u) \\
 &= - \sum_{i=1}^{n} g(R(u,v)x_i,y_i),
\end{split}
\end{equation*}
de onde concluimos que $\rho(u,v) = \smo \text{tr}_\C (F_\nabla (u,v))$.
\end{proof}

Lembrando que a primeira classe de Chern de um fibrado complexo é dada por $c_1(E) = [\frac{\smo}{2\pi}\text{tr} F_\nabla]$ e que $c_1(X) = c_1(T^{1,0}X)$ vemos que 

\begin{corollary} \label{cor:Ric=c1}
A forma de Ricci $\rho$ é uma $(1,1)$-forma real e fechada e $2 \pi \rho$ representa  a primeira classe de Chern de $X$, isto é,  $2\pi \rho \in c_1(X)$. 
\end{corollary}

Como consequência interessante observamos que a classe de cohomologia da forma de Ricci depende apenas da estrutura complexa de $X$, e não da particular escolha de métrica. 

Do corolário acima é natural perguntar quais $(1,1)$-formas reais e fechadas em $2\pi c_1(X)$ podem ser a forma de de Ricci de uma métrica de Kähler em $X$. A seguinte conjectura foi proposta pelo matemático italiano Eugenio Calabi, nos anos 1950.\\

\textbf{Conjectura de Calabi.} \index{Conjectura!de Calabi} Sejam $X$ uma variedade de Kähler compacta, $g$ uma métrica de Kähler em $X$ e $\omega$ sua forma fundamental. Dada uma $(1,1)$-forma real e fechada $\rho' \in 2\pi c_1(X)$ existe uma única métrica de Kähler $g'$ em $X$ com forma fundamental $\omega'$ tal que $[\omega']=[\omega]$ e cuja forma de Ricci é $\rho'$.\\

A Conjectura de Calabi foi respondida afirmativamente pelo matemático Shing-Tung Yau no final dos anos 1970, o que, entre outros feitos, lhe rendeu a Medalha Fields em 1983. Hoje o resultado é conhecido como \emph{Teorema de Calabi-Yau}\index{Teorema!de Calabi-Yau}. A ideia da demonstração é usar o $\del \delbar$-lema para transformar o problema em uma equação diferencial parcial de segunda ordem não linear, chamada equação de Monge-Ampère.\\

Fixe uma métrica de Kähler $g$ em $X$ e sejam $\omega$ sua forma fundamental e $\rho$ a sua forma de Ricci. Dada uma $(1,1)$-forma real e fechada $\rho' \in 2\pi c_1(X)$ queremos encontrar uma métrica $g'$ em $X$ com forma de Ricci $\rho'$ e forma fundamental $\omega'$ na mesma classe de $\omega$.

Como $[\rho] = 2\pi c_1(X) = [\rho']$, vemos, do $\del \delbar$-lema (cf. Exemplo \ref{ex:del-delbar}), que existe uma função real $f$ em $X$ tal que
\begin{equation*}
\rho' = \rho + \smo \del \delbar f.
\end{equation*}

Pode-se mostrar neste caso que as formas fundamentais satisfazem a relação
\begin{equation*}
(\omega')^n = A e^f \omega^n,
\end{equation*}
onde $n$ é a dimensão de $X$ e $A$ é uma constante positiva.

Como $\omega$ e $\omega'$ são cohomólogas temos, pelo teorema de Stokes, que $\int_X (\omega')^n = \int_X \omega^n$ e portanto, lemabrando que $n! \text{vol}_g = \omega^n$, a equação acima diz que
\begin{equation*}
A \int_X e^f \text{vol}_g = \text{Vol}(X).
\end{equation*}

Esta discussão nos leva a uma segunda versão da Conjecura de Calabi.\\

\textbf{Conjectura de Calabi (II)}. Sejam $X$ uma variedade de Kähler compacta, $g$ uma métrica de Kähler em $X$ e $\omega$ sua forma fundamental. Seja $f \in C^{\infty}(X;\R)$ e defina $A>0$ pela equação $A \int_X e^f \text{vol}_g = \text{Vol}(X)$. Então existe uma única métrica de Kähler $g'$ em $X$ com forma fundamental $\omega'$ tal que $[\omega']=[\omega]$ e $(\omega')^n=Ae^f \omega^n$.\\

Note que a versão acima já não faz menção à forma de Ricci. Em vez de prescrever $\rho'$, podemos pensar no problema de Calabi como o de encontrar métricas de Kähler $g'$ prescrevendo a forma volume $\text{vol}_{g'}$. De fato, a forma volume de $X$ pode ser escrita como $F \text{vol}_g$, onde $F$ é uma função real e a versão acima da Conjectura de Calabi diz que se $F$ for positiva e $F \text{vol}_g$ tiver o mesmo volume total que a métrica original então podemos encontrar uma única $g'$ de Kähler com $\text{vol}_{g'} = F \text{vol}_g$.\\

Podemos simplificar o problema um pouco mais. Da condição $[\omega'] = [\omega]$ obtemos, aplicando novamente o $\del \delbar$-lema, uma função real $\varphi$ em $X$ tal que
\begin{equation*}
\omega' = \omega + \smo \del \delbar \varphi,
\end{equation*}
que é única a menos de uma constante. Assim, se impusermos a condição $\int_X \varphi \text{vol}_g = 0$, determinamos $\varphi$ univocamente. Obtemos assim uma terceira versão da conjectura.\\

\textbf{Conjectura de Calabi (III)}. Sejam $X$ uma variedade de Kähler compacta, $g$ uma métrica de Kähler em $X$ e $\omega$ sua forma fundamental. Seja $f \in C^{\infty}(X;\R)$ e defina $A>0$ pela equação $A \int_X e^f \text{vol}_g = \text{Vol}(X)$. Então existe uma única função $\varphi \in C^{\infty}(X;\R)$ tal que
\begin{itemize}
\item[(i)] $\omega + \smo \del \delbar \varphi$ é uma $(1,1)$-forma positiva,
 
\item[(ii)] $\int_X \varphi \text{vol}_g = 0$,

\item[(iii)] $(\omega + \smo \del \delbar \varphi)^n = A e^f \omega^n$ em $X$.
\end{itemize}

Em termos de um sistema de coordenadas holomorfas $z_1,\ldots,z_n$ a equação $(iii)$ acima pode ser reescrita como
\begin{equation} \label{eq:monge-ampere}
\det \left(h_{ij} + \frac{\del^2 \varphi}{\del z_i \del \bar z _j} \right) = A e^f \det(h_{ij}),
\end{equation}
onde $h_{ij} = h(\frac{\del}{\del z_i}, \frac{\del}{\del \bar z_j})$ são os coeficientes da métrica nessas coordenadas.

A equação (\ref{eq:monge-ampere}) é chamada \emph{equação de Monge-Ampère} e a prova da Conjectura de Calabi dada por Yau consiste mostrar que ela admite uma única solução suave.\\

Uma segunda conjectura importante da Geometria Diferencial Complexa diz respeito à existência de métricas de Kähler-Einstein.

\begin{definition} \index{métrica!de Kähler-Einstein}
Seja $X$ uma variedade complexa. Uma métrica $g$ em $X$ é dita uma \textbf{métrica de Kähler-Einstein} se $g$ é ao mesmo tempo uma métrica de Kähler e uma métrica de Einstein.

Equivalentemente, uma métrica $g$ é de Kähler-Einstein se $g$ é de Kähler e sua forma de Ricci $\rho$ satisfaz
\begin{equation} \label{eq:kahler-einstein-condition}
\rho = \lambda \omega,
\end{equation} 
onde $\omega$ é a forma fundamental de $g$ e $\lambda \in \R$ é uma constante.

Dizemos que $X$ é uma variedade de Kähler-Einstein se $X$ admite uma métrica de Kähler-Einstein.
\end{definition}

Observe que quando $\lambda = 0$ temos que $\rho = 0$ e portanto $r=0$. Neste caso dizemos que a métrica $g$ é \textbf{\textit{Ricci-flat}}.\\

Como $\omega$ é uma $(1,1)$-forma positiva (cf. Definição \ref{def:positive-form}), a condição (\ref{eq:kahler-einstein-condition}) mostra que se $g$ é uma métrica de Kähler-Einstein então vale uma das seguintes condições: a) $\rho = 0$,  b) $\rho$ é positiva ou c) $\rho$ é negativa (i.e., $-\rho$ é positiva).

Essas condições se refletem, pelo corolário \ref{cor:Ric=c1}, em condições na primeira classe de Chern de $X$. Dizemos que uma classe de cohomologia $c \in H^2(X,\R)$ é positiva, e escrevemos $c > 0$, se $c$ pode ser representada por uma forma positiva. Dizemos que $c$ é negativa, e escrevemos $c < 0$, se $-c$ é positiva. Como $[\rho] = 2\pi c_1(X)$, a equação (\ref{eq:kahler-einstein-condition}) implica que $c_1(X) = \frac{\lambda}{2\pi} [\omega]$, de onde obtemos o seguinte resultado.

\begin{proposition}
Se $X$ uma variedade de Kähler-Einstein então vale uma das seguintes condições:
\begin{enumerate}
\item $c_1(X) = 0$
\item $c_1(X) > 0$
\item $c_1(X) < 0$.
\end{enumerate}
\end{proposition}

Uma pergunta natural é se a validade de uma das condições acima é suficiente para a existência de métricas de Kähler-Einstein. No caso $c_1(X) = 0$ a resposta é afirmativa, o que segue do Teorema de Calabi-Yau.

\begin{theorem}
Seja $X$ uma variedade de Kähler compacta com $c_1(X) = 0$. Então existe uma métrica de Kähler \textit{Ricci-flat} $g$ em $X$. Além disso, se $\omega$ denota a forma fundamental de $g$, então $g$ é a única métrica \textit{Ricci-flat} em $X$ com forma fundamental na classe de $\omega$.
\end{theorem}

\begin{proof}
Tomando $\rho' = 0 \in 2\pi c_1(X)$, o Teorema de Calabi-Yau diz que existe uma única métrica $g$ em $X$ com forma de Ricci $\rho = 0$ e com a classe $[\omega]$ especificada.
\end{proof}

\begin{remark}
As variedades complexas compactas com $c_1(X)=0$ são chamadas variedades de Calabi-Yau e são importante na Física. Em uma das possíveis formulações da Teoria de Supercordas o espaço é dado localmente como um produto $M^{3,1} \times X$, onde $M^{3,1}$ é o espaço de Minkowski e $X$ é uma variedade de Calabi-Yau de dimensão (complexa) $3$. A variedade $X$ é ``pequena'' com relação a $M^{3,1}$, o que explica por que alguns fenômenos não podem ser observados nos experimentos atuais.
\end{remark}

A condição $c_1(X) < 0$ também é suficiente para a existência de métricas de Kähler-Einstein. O seguinte resultado foi mostrado independentemente por Thierry Aubin e Shing-Tung Yau.

\begin{theorem} \emph{(Teorema de Aubin-Yau)}
Se $X$ é uma variedade de Kähler compacta com $c_1(X) < 0$ então $X$ admite, a menos de uma constante multiplicativa, uma única métrica de Kähler-Einstein. 
\end{theorem}

No caso positivo, a condição $c_1(X) > 0$ nãó é suficiente para existência de métricas de Kähler-Einstein. Nos anos 1980 descobriu-se uma outra obstrução, que é o anulamento do chamado invariante de Futaki, $\text{Fut}_X$. Em 1990, Gang Tian mostrou que, quando $\dim X = 2$, o anulamento de $\text{Fut}_X$ é uma condição suficiente para a existência de métricas de Kähler-Einstein e, em 2004, foi provado por X.J Wang e Z.H.Zhu que ela também é suficiente no caso de variedades tóricas. No entanto, em 1997, o próprio Tian construiu um contraexemplo para a conjectura, exibindo uma variedade compacta com $c_1(X)>0$, $\text{Fut}_X = 0$ mas que não admite nenhuma métrica de Kähler-Einstein.

Atualmente acredita-se que a existência de métricas de Kähler-Einstein em variedades com classe de Chern positiva esteja relacionada com algumas condições de estabilidade no contexto da Teoria de Invariantes Geométrica. Uma discussão sobre o assunto pode ser encontrada em \cite{phong-sturm}.
\chapter{Topologia de Variedades Complexas} \label{ch:stein-lesfchetz}

Neste capítulo vamos estudar alguns aspectos topológicos de variedades complexas. Demonstraremos dois resultados principais. O primeiro diz respeito à homologia de subvariedades complexas fechadas de $\C^N$, as chamadas variedades de Stein, e como consequência dele obteremos o Teorema de Hiperplanos de Lefschetz, que permite estudar a topologia de uma variedade algébrica projetiva por meio das suas seções de hiperplanos. A apresentação dada aqui é inspirada no artigo de Aldo Andreotti e Theodore Frankel \cite{andreotti-frankel}.

A demonstração do primeiro teorema é um dos inúmeros exemplos de aplicação da Teoria de Morse, que recordaremos a seguir. A referência clássica é o livro de J. Milnor \cite{milnor}. Para um tratamento mais moderno consulte o livro de R. Palais e C. Terng \cite{palais-terng}.

\section{Teoria de Morse}
A ideia da Teoria de Morse é estudar a topologia de uma variedade diferencável $M$ olhando para os pontos críticos de uma função suave $f:M \to \R$.\\

Vamos recordar algumas definições básicas.

Seja $M$ uma variedade diferenciável e fixemos $f:M \to \R$ uma função diferenciável em $M$. Um ponto $p \in M$ é um \textbf{ponto crítico} de $f$ se $df_p=0$. É fácil ver que se $p$ é um ponto de máximo ou mínimo local de $f$ então $p$ é um ponto crítico. A imagem $c=f(p)$ de um ponto crítico é chamado \textbf{valor crítico} de $f$. Um \textbf{ponto regular} é um ponto que não é crítico e um número $c \in \R$ é um \textbf{valor regular} se $f^{-1}(c)$ não contém nenhum ponto crítico \footnote{Em paricular todo número $c \in \R \setminus f(M)$ é um valor regular.}.

Dado $c \in \R$ defina os seguintes subconjuntos de $M$
\begin{equation*}
M^c = \{x \in M : f(x) \leq c\}~~\text{ e }~~ M^{c-} = \{x \in M : f(x) < c\}.
\end{equation*}

Do teorema da função implícita é fácil ver que se $c$ é um valor regular então $M^c$ é uma variedade com bordo $\del M = f^{-1}(c)$. O resultado central da Teoria de Morse descreve como a topologia de $M^c$ varia conforma aumentamos $c$ (Teorema \ref{thm:morse-fund}).

Se $p$ é um ponto crítico de $f$ definimos o \textbf{hessiano} de $f$ em $p$ como sendo a forma bilinear
\begin{equation*}
\begin{split}
\Hess_p(f) : T_pM \times T_pM &\longrightarrow \R \\
\Hess_p(f)(X,Y)&= X(\widetilde Yf)(p),
\end{split}
\end{equation*}
onde $\widetilde Y$ é uma extensão de $Y$ a um campo local em $M$.

Note que se $\widetilde X$ é uma extensão de $X$ a um campo local temos que
\begin{equation*}
\widetilde X(\widetilde Yf)(p) - \widetilde Y (\widetilde Xf)(p) = [X,Y]_p(f) = df_p([X,Y]) = 0.
\end{equation*}
Em particular isso mostra que $X(\widetilde Yf)(p) = Y(\widetilde Xf)(p)$ independe da extensão de $Y$ escolhida e que $\Hess_p(f)$ é uma forma bilinear simétrica em $T_pM$.

\begin{remark}
Note que a definição acima só se aplica ao pontos críticos de $f$. Podemos definir o Hessiano de maneira global com a ajuda de uma conexão.

Seja $\nabla$ uma conexão em $TM$ e considere sua extensão ao fibrado de tensores. Definimos então o hessiano de $f$ com respeito a $\nabla$ como sendo o $2$-tensor $\Hess^{\nabla}(f) = \nabla df$.

Do fato de $\nabla_X$ agir como derivação e comutar com as contrações temos que $\widetilde X(\widetilde Yf) = df(\nabla_{\widetilde X} \widetilde Y) + \Hess^\nabla (f)(\widetilde X,\widetilde Y)$ e portanto, se $p$ é um ponto crítico de $f$, recuperamos a definição acima: $\Hess^\nabla(f)(\widetilde X,\widetilde Y)(p) = \widetilde X(\widetilde Yf)(p) = \Hess_p(f)$. 
\end{remark}

\begin{remark}
Se $M$ possuir uma métrica riemanniana $g$ podemos representar $\Hess_p(f)$ por um operador linear autoadjunto em $T_pM$, isto é, existe $\hess_p(f):T_pM \to T_pM$ tal que
\begin{equation*}
\Hess_p(f)(X,Y) = g_p(\hess_p(f)X,Y),~~X,Y \in T_pM. 
\end{equation*}
\end{remark}

\begin{definition}
Um ponto crítico $p$ é \textbf{não degenerado} se a aplicação bilinear $\Hess_p(f)$ é não degenerada, ou equivalentemente se o operador linear $\hess_p(f)$ é um isomorfismo.

O \textbf{índice} de um ponto crítico não degenerado é o indice da forma bilinear $\Hess_p(f)$, isto é, é a máxima dimensão de um subespaço $W \subset T_pM$ tal que $\Hess_p(f)|_{W \times W}$ é negativa definida. Equivalentemente o índice de $p$ é o número de autovalores negativos de $\hess_p(f)$.

Uma função $f:M\to \R$ é dita uma \textbf{função de Morse} se todos os seus pontos críticos são não degenerados.
\end{definition}

\begin{example}
Considere $M = \R^n$ e $f:\R^n\to \R$ a função $f (x_1,\ldots,x_n) = - x_1^2 -x_2^2 - \cdots - x_\lambda^2 + x_{\lambda+1}^2 + \cdots + x_n^2$. Neste caso temos que o único ponto crítico de $f$ é a origem e como $\hess_0(f) = \text{diag}(-2,\ldots,-2,2,\ldots,2)$ (com $-2$ aparecendo $\lambda$ vezes) vemos que o índice de $f$ na origem é igual a $\lambda$.
\end{example}

O lema a seguir mostra que exemplo acima serve de modelo local para o comportamento de uma função em torno de um ponto crítico não degenerado de índice $\lambda$. Para uma demonstração consulte \cite{milnor}.

\begin{lemma} \textbf{Lema de Morse.} Seja $p$ um ponto crítico não degenerado de $f$ de índice $\lambda$. Então existe um  difeomorfismo $\varphi:U \to \varphi(U) \subset \R^n$ definido em uma vizinhança $U$ de $p$ tal que $\varphi(p)=0$ e 
\begin{equation*}
f \circ \varphi^{-1}(x_1,\ldots,x_n) = f(p) - x_1^2 -x_2^2 - \cdots - x_\lambda^2 + x_{\lambda+1}^2 + \cdots + x_n^2.
\end{equation*}
para $(x_1,\ldots,x_n) \in \varphi(U)$.
\end{lemma}

A forma local acima mostra, em particular, que $p$ é o único ponto crítico de $f$ em $U$. 
\begin{corollary}
O conjunto dos pontos críticos não degenerados de $f$ é discreto em $M$.
\end{corollary}

\subsubsection{Existência de funções de Morse para subvariedades do $\R^N$}

Seja $M$ uma subvariedade de $\R^N$ e $q \in \R^n \setminus M$. Vamos denotar por $L_q$ o quadrado da função distância a $q$ em $M$, isto é
\begin{equation*}
\begin{split}
L_q: M &\longrightarrow \R \\
L_q(x) &= \frac{1}{2}|| x-q ||^2.
\end{split}
\end{equation*}
O fator $\frac{1}{2}$ é uma convenção e como veremos simplifica alguns cálculos.

O objetivo é mostrar que $L_q$ é uma função de Morse em $M$ para quase todo $q \in \R^n \setminus M$.\\

Antes de continuar vamos recordar algumas definições e resultados básicos da teoria de subvariedades do $\R^N$.

Dado um ponto $p \in M$, o espaço normal a $M$ é $p$ é o espaço $\nu_pM = T_pM^\perp$ e definimos o fibrado normal de $M$ como sendo $\nu(M)=\bigcup_{p \in M}\nu_pM$. Não é difícil ver que $\nu(M)$ é um fibrado vetorial sobre $M$ cujo posto é igual a codimensão de $M$ em $\R^N$.

Um campo de vetores $\xi$ definido em um aberto $U \subset M$ é um \textit{campo normal} se $\xi(p) \perp T_p M$ para todo $p \in U$. Em outras palavras, um campo normal é uma seção local de $\nu(M)$. Dado um campo normal, definimos o \textit{operador de Weingarten}
\begin{equation*}
\begin{split}
A_\xi: T_pM &\longrightarrow T_pM \\
A_\xi(v) &= - (\nabla_v \widetilde \xi)^\top(p),
\end{split}
\end{equation*}
para $p \in U$, onde $u \mapsto u^\top$ denota a projeção ortogonal de $\R^N$ em $T_p M$, $\nabla$ é a conexão de Levi-Civita em $\R^N$ e $\widetilde \xi$ é uma extensão de $\xi$ a um campo local em $\R^N$.

O operador $A_\xi$ é auto-adjunto e portanto podemos associá-lo a uma forma bilinear simétrica, chamada \textit{segunda forma fundamental} com relação a $\xi$ no ponto $p$
\begin{equation*}
II_\xi(u,v) = \langle A_\xi(u),v\rangle,~~ u,v \in T_pM.
\end{equation*}
 
É fácil ver que a diferencial de $L_q$ em um ponto $p \in M$ é dada por
\begin{equation} \label{eq:der-Lq}
(dL_q)_p \cdot v = \langle p-q,v \rangle,~~v \in T_p M,
\end{equation}
e portanto $p$ é um ponto crítico de $L_q$ se e somente se o vetor $p-q$ é normal a $M$ em $p$.
\begin{lemma} \label{lemma:hess-Lq}
O hessiano da função $L_q$ em um ponto crítico $p \in M$ é dado por $\hess_p(L_q) = I - A_\xi$, onde $\xi = q-p \in \nu_p M.$
\end{lemma}

\begin{proof}
Sejam $X,Y \in T_pM$ e seja $\widetilde Y$ uma extensão de $Y$ a um campo local em $M$. Por definição temos que $\Hess_p(L_q)(X,Y) = X(\widetilde Y L_q)(p)$. Da fórmula (\ref{eq:der-Lq}) temos que $\widetilde Y(L_q) = (dL_q) \cdot \widetilde Y = \langle I - q,\widetilde Y \rangle$ e portanto
\begin{equation*}
X(\widetilde Y(L_q)) = X \langle I - q,\widetilde Y \rangle = \langle X,\widetilde Y\rangle + \langle I-q, X \widetilde Y\rangle.
\end{equation*}
Seja $\widetilde \xi$ uma extensão de $\xi = q-p \in \nu_p M$ a uma seção local de $\nu(M)$. Como $\langle \widetilde \xi, \widetilde Y\rangle = 0$ temos que $0 = X\langle \widetilde \xi, \widetilde Y\rangle =\langle X \widetilde \xi, \widetilde Y\rangle + \langle \widetilde \xi, X\widetilde Y\rangle$ e portanto
\begin{equation*}
\langle I-q, X \widetilde Y\rangle(p) = \langle -\widetilde \xi, X \widetilde Y\rangle = \langle X \widetilde \xi,\widetilde Y\rangle = - \langle A_\xi(X),Y\rangle,
\end{equation*}
onde usamos que $X \widetilde \xi = \nabla_X \widetilde \xi$ e que  $\widetilde Y$ é tangente a $M$.

Combinando as duas fórmulas acima obtemos
\begin{equation*}
\Hess_p(L_q)(X,Y) = X(\widetilde Y L_q)(p) = \langle X,Y\rangle - \langle A_\xi X,Y\rangle = \langle (I - A_\xi) X,Y\rangle,
\end{equation*}
e portanto $\hess_p(L_q) = I - A_\xi$.
\end{proof}

A \textit{exponencial normal} de $M$ é a restrição $\exp^\perp:\nu(M) \to \R^N$ da exponencial Riemannina $\exp:T\R^N \to \R^N$ ao fibrado normal de $M$. Note que $\exp^\perp$ nada mais é que a translação $\nu_xM \ni v \mapsto x+v \in \R^N$.

\begin{definition}
Um ponto $q=\exp^\perp(p,v)$ na imagem da exponencial normal é dito um \textbf{ponto focal} de $M$ com relação a $x$ se $q$ é um valor crítico de $\exp^\perp$, ou seja, se a diferencial $d \exp^\perp_{(x,v)}:T_{(x,v)}\nu(M) \to \R^N$ não é um isomorfismo.
\end{definition}

O lema a seguir mostra que os pontos focais estão intimamente relacionados com os pontos críticos da função distância.

Antes de continuarmos observe que a conexão $\nabla$ induz uma decomposição natural do espaço tangente a $\nu(M)$ em um ponto $(x,v)$
\begin{equation*}
T_{(x,v)}\nu(M) = T_xM \oplus \nu_xM,
\end{equation*}

O isomorfismo é dado por $T \nu(M) \ni \frac{d}{dt} \xi(t) \mapsto (\frac{d}{dt} \pi (\xi(t)),\frac{\nabla^{\perp}}{dt} \xi(t))$, onde $\pi: \nu(M) \to M$ é a projeção natural e $\frac{\nabla^{\perp}}{dt}$ é a derivada covariante induzida pela conexão normal em $\nu(M)$.

Dado um vetor $z \in T_{(x,v)} \nu(M)$ denotaremos por $z^\top$ sua componente em $T_xM$ e por $z^\perp$ sua componente em $\nu_xM$.
\begin{lemma}
Um ponto $q = \exp^\perp(p,v)$ é focal se e somente se $p$ é um ponto crítico degenerado de $L_q$.
\end{lemma}
\begin{proof}
Primeiramente note que se $q$ está na imagem da exponencial normal, temos que $p-q$ é normal a $M$ em $p$ e portanto $p$ é um ponto crítico de $L_q$. Reciprocamente, se $p$ é um ponto crítico de $L_q$ então $p-q$ é normal a $M$ e portanto $q = p + (q-p) = \exp^\perp_p(q-p)$ está na imagem da exponencial normal.

Seja $\gamma(t)=(X(t),V(t))$ uma curva em $\nu(M)$ onde $X$ é uma curva em $M$ com $X(0)=p$ e $X'(0)=u$ e $V$ é um campo normal ao longo de $X$ com $V(0)=v$.

Temos então que
\begin{equation*}
\begin{split}
d \exp^\perp_{(x,v)}\cdot \gamma'(0) &= \frac{d}{dt} \bigg(\exp^\perp(\gamma(t))\bigg) \bigg|_{t=0} = \frac{d}{dt} \big(X(t) + V(t) \big) \big|_{t=0}\\
&= X'(0) + V'(0) = X'(0) + V'(0)^\top + V'(0)^\perp.
\end{split}
\end{equation*}

Agora, em $\R^N$, a derivada covariante de um campo ao longo de uma curva coincide com a derivada usual, e portanto temos que $V'(t) = \frac{D}{dt} V(t) = \nabla_{X'} \widetilde V (t)$, onde $\frac{D}{dt}$ denota a derivada covariante ao longo de $X$ e $\widetilde V$ é uma extensão local de $V$.

Vemos assim que $V'(0)^\top = (\nabla_{X'} \widetilde V)^\top(0) = - (A_V (X')) (0) = - A_v (X'(0))$, e portanto
\begin{equation*}
d \exp^\perp_{(x,v)}\cdot \gamma'(0) = X'(0) - A_v(X'(0)) +V'(0)^\perp = (I-A_v)(X'(0)) + V'(0)^\perp.
\end{equation*}

Tomando agora $X(t)=p$ constante e $V(t) = v +tz$ com $z \in \nu_xM$ temos que $\gamma'(0)=(0,z)$ e portanto vemos que $d \exp^\perp_{(x,v)}\cdot(0,z) = z$, isto é, $d \exp^\perp_{(x,v)}$ se restringe a identidade em $\nu_pM$. Sendo assim, vemos da equação acima que $d \exp^\perp_{(x,v)}$ não será um isomorfismo se e somente se $I-A_\xi = \hess_p(L_q)$ for degenerada, ou seja, se e somente se $p$ é um ponto crítico degenerado de $L_q$.
\end{proof}

\begin{corollary}
Se $q$ não é um ponto focal de $M$ então $L_q:M \to \R$ é uma função de Morse.
\end{corollary}

\begin{corollary} \label{cor:Lq-morse}
A função $L_q$ é de Morse para quase todo $q \in \R^N \setminus M$.
\end{corollary}
\begin{proof}
Do lema acima vemos que os pontos $q \in  \R^N \setminus M$ tais que $L_q$ não é de Morse são exatamente os valores críticos da exponecial normal $\exp^\perp: \nu(M) \to \R^N$ e, pelo teorema de Sard\footnote{O Teorema de Sard diz que se $f:M\to N$ é uma função suave entre duas variedades então o conjunto dos valores críticos de $f$ tem medida nula em $M$.}, esse conjunto tem medida nula.
\end{proof}

\subsection{O Teorema Fundamental e as Desigualdades de Morse}

O resultado central da Teoria de Morse descreve como a topologia dos subníveis $M^a$ muda conforme $a$ aumenta.

Não é difícil ver, usando o fluxo do gradiente de $f$, que  se $[a,b]$ não contém nenhum valor crítico então a topologia  não muda quando passamos de $a$ para $b$, pois nesse caso o campo gradiente não tem singularidades em $f^{-1}([a,b])$.

A situação é mais complicada se $[a,b]$ contiver um valor crítico. O que ocorre nesse caso é que $M^a$ é obtido de $M^b$ pela junção de alças, cujo tipo depende dos índices dos pontos críticos em questão.
Se $P \subset N$ são variedades com bordo de mesma dimensão $n$, dizemos que \textit{$N$ é obtida de $P$ pela junção de uma alça de tipo $\lambda$} se existe um fechado $H \subset N$ e uma aplicação $\alpha:\bar D^{n-\lambda} \times \bar D^\lambda \to H$ (onde $\bar D^k$ denota o disco unitário fechado em $\R^k$) satisfazendo as seguintes condições
\begin{itemize}
\item[1.] $N = P \cup H$,
\item[2.] A restrição de $\alpha$ a $\bar D^{n-\lambda} \times S^{\lambda-1}$ define um homeomorfismo $\bar D^{n-\lambda} \times S^{\lambda-1} \simeq H \cap \del P$
\item[3.] A restrição de $\alpha$ a $\bar D^{n-\lambda} \times D^\lambda$ define um homeomorfismo $\bar D^{n-\lambda} \times D^\lambda \simeq N \setminus P$.
\end{itemize}
Note que nesse caso $N$ tem o mesmo tipo de homotopia que $P$ unido com um $\lambda$-célula, pois podemos contrair a imagem do disco fechado $\bar D^{n-\lambda}$ em $N$.

\begin{theorem} \label{thm:morse-fund} \textbf{Teorema Fundamental da Teoria de Morse}
Seja $f$ uma função de Morse limitada inferiormente em $M$ e suponha que $M^a$ é compacto para todo $a \in \R$. Então
\begin{itemize}
\item[A.] Se $f$ não tem nenhum valor crítico no intervalo $[a,b]$ então $M^b$ e $M^a$ são difeomorfos.
\item[B.] Se $c$ é o único valor crítico no intervalo $(a,b)$ e $p_1,\ldots,p_k \in f^{-1}(c)$ são os pontos críticos no nível $c$ com $\text{ind}_{p_i} = \lambda_i$ então $M^b$ é obtido de $M^a$ pela junção de uma $\lambda_i$-alça para cada $p_i$. Em particular $M^b$ tem o mesmo tipo de homotopia que $M^a$ unido com uma $\lambda_i$-célula para cada $p_i$.
\end{itemize}
\end{theorem}

Para uma demonstação consulte \cite{palais-terng}, capítulo 9.

As desigualdades de Morse exprimem a diferença dos polinômios
\begin{equation*}
\mathcal M_a(t) = \sum_{i=0}^n \mu_i(a) \cdot t^i~~\text{ e }~~ \mathcal P_a(t) =  \sum_{i=0}^n b_i(a) \cdot t^i,
\end{equation*}
onde $\mu_i(a)$ é o número de pontos críticos de $f$ de índice $i$ cujos valores críticos não excedem $a$ e $b_i(a) = \dim H_i(M^a,\mathbf k)$ é o $i$-ésimo número de Betti do subnível $M^a$ com relação a um corpo de coeficientes fixado $\mathbf k$.

\begin{theorem} \label{thm:morse-inequalities} \textbf{Desigualdades de Morse} \index{desigualdades de Morse}
Nas condições do teorema anterior, se $a$ não é um valor crítico de $f$, a diferença dos polinômios $\mathcal M_a(t)$ e $\mathcal P_a (t)$ satisfaz
\begin{equation} \label{eq:morse-inequalities}
\mathcal M_a (t) - \mathcal P_a (t) = Q_a(t) (1+t),
\end{equation}
onde $Q_a(t)$ é um polinômio com coeficientes inteiros não negativos.
Consequentemente temos que
\begin{equation} \label{eq:morse-inequalities2}
\begin{split}
b_0(a) &\leq \mu_0(a) \\
b_1(a) - b_0(a) &\leq \mu_1(a) - \mu_0(a)\\
 &\vdots \\
b_k(a) - b_{k-1}(a) + \cdots +(-1)^k b_0(a) &\leq \mu_k(a) - \mu_{k-1}(a) + \cdots (-1)^k \mu_0(a) 
\end{split}
\end{equation}
e finalmente
\begin{equation} \label{eq:weak-morse}
b_k(a) \leq \mu_k(a).
\end{equation}
\end{theorem}

Usando o Teorema \ref{thm:morse-fund} podemos esboçar a demonstração do teorema acima. Para tanto vamos relembrar alguns fatos de topologia algébrica. Para mais detalhes consulte \cite{hatcher}.\\

Seja $X$ um espaço topológico e $A \subset X$ um subespaço. Denote por $C_ \bullet(X;G)$ e por $C_\bullet (A;G)$ o espaço das cadeias singulares de $X$ e $A$ com coeficientes no grupo abeliano $G$ e seja $C_i(X,A) = C_i(X)\slash C_i(A)$. Como o operador de bordo $\del:C_i(A;G) \to C_{i-1}(A;G)$ é a restrição de $\del:C_i(X;G) \to C_{i-1}(X;G)$, obtemos um operador de bordo bem definido no quociente $\del:C_i(X,A;G)$ $\to C_{i-1}(X,A;G)$. Os grupo de homologia do par $(X,A)$ ou homologia relativa de $X$ e $A$ com coeficientes em $G$ são definidos por
\begin{equation*}
H_i(X,A;G) = \frac{\ker \del :C_i(X,A;G) \to C_{i-1}(X,A;G)}{\im \del :C_{i-1}(X,A;G) \to C_i(X,A;G)}.
\end{equation*}

Da própria definição de $C_\bullet (X,A;G)$ e do operador de bordo obtemos uma sequência exata de complexos
\begin{equation*}
0 \longrightarrow C_\bullet(A;G) \longrightarrow C_\bullet(X;G)\longrightarrow C_\bullet(X,A;G) \longrightarrow 0,
\end{equation*}
de onde obtemos uma sequência exata longa relacionando a homologia do par $(X,A)$ com as homologias de $X$ e $A$:
\begin{equation} \label{eq:relative-homology-seq}
\cdots \to H_i(A;G) \to H_i(X;G) \to H_i(X,A;G) \to H_{i-1}(A;G) \to H_{i-1}(X;G) \to  H_{i-1}(X,A;G) \to \cdots.
\end{equation}
No decorrer do argumento vamos supor que $G=\mathbf k$ é um corpo fixado e denotaremos simplesmente por $H_i(X),H_i(A)$ e $H_i(X,A)$ os grupos de cohomologia com coeficientes em $k$.\\

Para demonstrar as desigualdades de Morse vamos analisar como a diferença $\Delta^a(t)$ dos polinômios em questão se comportam nos subconjuntos $M^a$ conforme $a$ passa por um ponto crítico.

Como $f$ é limitada inferiormente temos que $\Delta^a(t) = 0$ para $a<<0$, pois nesse caso $M^a= \emptyset$.

Sejam agora $a < b$ e suponha que nenhum deles é um valor crítico de $f$. Nosso objetivo é comparar $\Delta^a(t)$ e $\Delta^b(t)$.

\textbf{$1^o$ caso.} Suponha que não existe nenhum ponto crítico em $[a,b]$. Neste caso temos que $\mathcal M_a(t) = \mathcal M_b(t)$. Pelo Teorema \ref{thm:morse-fund} (A) temos que $M^a \simeq M^b$ e portanto $\mathcal P_a(t)= \mathcal P_b(t)$. Sendo assim vemos que $\Delta^a(t) = \Delta^b(t)$, isto é, a diferença não muda.

\textbf{$2^o$ caso.} Suponha que exista um único ponto crítico em $M$ com valor crítico $c \in (a,b)$ e seja $\lambda$ seu índice. Neste caso temos que $\mathcal M_b(t) = \mathcal M_a(t) + t^\lambda$.

Pelo Teorema \ref{thm:morse-fund} (B) temos que $M_b$ tem o mesmo tipo de homotopia de $M_a$ unida com uma $\lambda$-célula: $M^b \simeq M^a \cup e^\lambda$. O Teorema de Excisão para a homologia diz que se $Z \subset A \subset X$ e o fecho de $Z$ está contido no interior de $A$ então existem isomorfismos $H_i(X\setminus Z, A\setminus Z) \simeq H_i(X,A)$. Tomando $X=M^b$, $A=M^a$ e $Z$ o complementar em $M^a$ de uma vizinhança tubular de $\del e^\lambda \simeq S^{\lambda-1}$ temos que
\begin{equation*}
H_k(M^b,M^a) = H_k(D^\lambda,S^{\lambda-1}) = H_{k-1}(S^{\lambda-1}) = \left \lbrace \begin{split} &\mathbf k~ \text{ se }  k=\lambda \\ &0~ \text{ caso contrário } \end{split} \right.
\end{equation*}
onde a primeira igualdade segue do Teorema de Excisão e do fato de $M^b \setminus Z$ ter o mesmo tipo de homotopia de $D^\lambda$ e $M^a \setminus Z$ ter o mesmo tipo de homotopia de $S^{\lambda-1}$. A segunda igualdade segue facilmente aplicando a sequência (\ref{eq:relative-homology-seq}) para o par $(D^\lambda,S^{\lambda-1})$.

Olhando para a sequência longa associada ao par $(M^b,M^b)$ obtemos
\begin{equation*}
0 \to H_\lambda (M^a) \to H_\lambda (M^b) \to \overbrace{H_\lambda (M^b,M^a)}^{ \mathbf k} \stackrel{\delta}{\to}  H_{\lambda-1} (M^a) \to H_{\lambda-1} (M^b) \to 0.
\end{equation*}

Há duas possibilidades para $\delta$. Se $\delta$ é a aplicação nula temos a sequência $0 \to H_\lambda (M^a) \to H_\lambda (M^b) \to H_\lambda (M^b,M^a) \to 0$ e portanto $b_\lambda(b) = b_\lambda(a) + 1$. Temos  então $\mathcal P_a = \mathcal P_a + t^\lambda$ de onde vemos que
\begin{equation*}
\Delta^b(t) = \mathcal M_b(t) - P_b(t) = (\mathcal M_a(t) + t^\lambda) - ( P_a(t) + t^\lambda) =  \Delta^a(t),
\end{equation*}
e portanto a diferença não muda.

Agora, se $\delta \neq 0$, temos uma aplicação sobrejetora $H_{\lambda-1} (M^a) \to H_{\lambda-1} (M^b)$ com núcleo unidimensional. Sendo assim vemos que $b_{\lambda-1}(b) = b_{\lambda-1}(a) - 1$ e portanto $\mathcal P_b(t) = \mathcal P_a(t) - t^{\lambda-1}$. Sendo assim, a nova diferença fica
\begin{equation*}
\begin{split}
\Delta^b(t) &= \mathcal M_b(t) - P_b(t) = (\mathcal M_a(t) + t^\lambda) - ( P_a(t) - t^{\lambda-1})\\
&=  \Delta^a(t) + t^{\lambda-1}+ t^\lambda\\
&= \Delta^a(t) + t^{\lambda-1}(1+t).
\end{split}
\end{equation*}

Em geral, $f$ possui um número finito de pontos críticos no nível $f^{-1}(c)$, para cada um deles, uma das alternativas acima se aplica.\\

Recapitulando o argumento, vimos que a diferença $\Delta^a(t)$ é zero para $a<<0$ e, conforme aumentamos $a$, somamos à diferença anterior um polinômio da forma $Q(t)(1+t)$ onde $Q(t)$ tem coecientes inteiros não negativos, de onde obtemos a estimativa do Teorema \ref{thm:morse-inequalities}.

Para obtermos as desigualdades (\ref{eq:weak-morse}) basta compararmos os termos de mesmo grau em (\ref{eq:morse-inequalities}). Se $q_0,\ldots,q_{n-1}$ denotam os coeficientes do polinômio $Q$ temos que $\mu_0(a) - b_0(a) = q_0 \geq 0$ e $\mu_i(a) - b_i(a) = q_i + q_{i-1} \geq 0$ para $i \geq 1$. Para obter as desigualdades (\ref{eq:morse-inequalities2}) note que
\begin{equation*}
\begin{split}
\mu_k(a) - &\mu_{k-1}(a) + \cdots + \mu_1(a) + (-1)^k \mu_0(a) = \\
&= (b_k(a) + q_k + q_{k-1}) - (b_{k-1}(a) + q_{k-1} + q_{k-2}) + \cdots + (b_1(a) + q_1 + q_0) + (-1)^k(b_0(a) + q_0) \\
&=b_k(a) - b_{k-1}(a) + \cdots + b_1(a) + (-1)^k b_0(a) + q_k + q_0 + (-1)^k q_0 \\
&\geq b_k(a) - b_{k-1}(a) + \cdots + b_1(a) + (-1)^k b_0(a).
\end{split}
\end{equation*}

\section{Homologia de variedades de Stein}

Uma variedade de Stein é uma subvariedade complexa fechada de $\C^N$. \index{variedade!de Stein}

Em contraste com as subvariedades do espaço projetivo $\pr^N$, as variedades de Stein nunca são compactas (veja o Corolário \ref{cor:compact-submanifold}). Se pensarmos em analogia com  a Geometria Algébrica, as variedaes de Stein correspondem às variedades algébricas afins.  

\begin{theorem} \label{thm:stein-homology}
Seja $X \subset \C^N$ uma variedade de Stein de dimensão complexa $n$. Então os grupos de homologia de $X$ com coeficientes em $\Z$ satisfazem
\begin{equation*}
H_i(X,\Z) = 0 ~~ \text{ para } i > n 
\end{equation*}
e
\begin{equation*}
H_n(X,\Z) \text{ é livre de torção.} 
\end{equation*}
\end{theorem}
\begin{proof}
A ideia é usar a Teoria de Morse para calcular $H_i(X,\mathbf k)$ para um corpo $\mathbf k$ e depois aplicar o Teorema dos Coeficientes Universais para provarmos as afirmações sobre $H_i(X,\Z)$.\\

Do Corolário \ref{cor:Lq-morse} podemos escolher $q \in \C^N \setminus X$ de modo que $L_q:X \to \R$ seja uma função de Morse. Além disso, como $X$ é fechada, os subníveis $X^a$ são compactos e portanto podemos aplicar o Teorema \ref{thm:morse-inequalities}.\\

\textbf{Afirmação 1}: $L_q$ não tem pontos críticos de índice maior que $n$ e portanto $\mu_i(a)=0$ para todo $i > n$ e todo $a \in \R$.
\begin{proof}
Seja $p \in X$ um ponto crítico de $L_q$. Temos que mostrar que $\text{ind}_p(L_q)\leq n$. No lema \ref{lemma:hess-Lq} vimos que $\hess_p(L_q) = I - A_\xi$, onde $\xi = q-p$ e portanto o índice de $L_q$ será o número de autovalores negativos de $I-A_\xi$.

Denote por $J:TX \to TX$ a estrutura complexa induzida pela estrutura complexa padrão de $\C^N$.\\

\textbf{Afirmação 2}: O operador de Weingarten anti-comuta com $J$, isto é, $JA_\xi = -A_\xi J$.
\begin{proof}
Sejam $u,v \in T_X$. Note que, como $J$ é uma isometria em cada espaço tangente, o vetor $J\xi$ também é normal a $X$ em $p$. As segundas formas fundamentais com respeito a $\xi$ e $J\xi$ estão relacionadas por
\begin{equation*}
II_{J\xi}(u,v) = -\langle \nabla_u J\xi,v \rangle = -\langle J\nabla_u \xi,v \rangle =  \langle \nabla_u \xi,J v \rangle = -II_\xi(u,Jv),
\end{equation*}
onde usamos que $\nabla J =0$, pois a métrica padrão em $\C^N$ é de Kähler.

Aplicando esta identidade duas vezes otemos $II_\xi(Ju,Jv) = - II_{J\xi}(Ju,v) = II_{-\xi}(u,v) = -II_\xi(u,v)$. Em particular temos que $II_\xi(Ju,v) = II_\xi(u,Jv)$, e  portanto
\begin{equation*}
\begin{split}
\langle A_\xi Ju,v\rangle = II_\xi(Ju,v) = II_\xi(u,Jv) \\
\langle J A_\xi u,v\rangle = -\langle A_\xi u,Jv\rangle = -II_\xi(u,Jv),
\end{split}
\end{equation*}
mostrando que $JA_\xi = - A_\xi J$.
\end{proof}

Da afirmação $2$ vemos que se $v$ é um autovetor de $A_\xi$ com autovalor $a$ então $Jv$ é um autovetor de $A_\xi$ com autovalor $-a$, isto é, os autovalores de $A_\xi$ aparecem em pares com sinais opostos. Sendo assim, os autovalores de $I - A_\xi$ aparecem em pares da forma $1 \pm a$. Em particular, no máximo $n$ deles são negativos, de onde vemos que $\text{ind}_p(L_q)) \leq n$, demonstrando a afirmação $1$.
\end{proof}

A afirmação $1$ mostra que $H_i(X,\mathbf k )=0$ para $i>n$, onde $\mathbf k$ é um corpo de coeficientes fixado. De fato, se $H_k(X,\mathbf k)$ fosse não trivial para algum $k > n$ existiria um $k$-ciclo não trivial $K$ em $M$. Como $K$ é compacto, $K$ estaria contido em algum subnível $X^a$ e portanto definiria um elemento não nulo  em $H_k(X^a,\mathbf k)$. No entanto, das desigualdades de Morse, isso implicaria que $\mu_k(a) \geq b_k(a) > 0$, contrariando a afirmação $1$.

O Teorema dos Coeficientes Universais para a homologia (veja por exemplo \cite{hatcher}, p.264) diz que existe uma sequência exata
\begin{equation*}
0 \longrightarrow H_i(X,\Z) \otimes \mathbf k \longrightarrow H_i(X,\mathbf k) \longrightarrow \text{Tor}(H_{i-1}(X,\Z),\mathbf k) \longrightarrow 0
\end{equation*}
para $i > 0$ e portanto, como o termo do meio se anula para $i > n$, vemos que

a) $H_i(X,\Z) \otimes \mathbf k = 0$ se $i > n$,

b) $H_i(X,\Z)$ não tem torção se $i\geq n$.

A segunda conclusão do enunciado segue do item b) fazendo $i=n$.

Para demonstrar a primeira conclusão note que de a) vemos que se $i>n$ então $H_i(X,\Z)$ só tem elementos de torção\footnote{Um elemento de ordem infinita $0 \neq x \in H_i(X,\Z)$ geraria um subespaço não trivial $\{x \otimes \lambda: \lambda \in \mathbf k \} \subset H_i(X,\Z) \otimes \mathbf k$.}. Mas, do item b), $H_i(X,\Z)$ não tem torção, e portanto devemos ter $H_i(X,\Z) = 0$ para $i>n$.
\end{proof}

\begin{remark}
A afirmação $2$ na demonstração acima tem como consequência o fato de que \textit{toda subvariedade complexa de $\C^N$ é mínima}.

Dada uma subvariedade $M \subset \R^n$ e um vetor normal $\xi \in \nu_pM$, a curvatura média de $M$ em $p$ na direção $\xi$ é, por definição, o número $H_\xi(p) = \text{tr} A_\xi$ e dizemos que uma subvariedade é \textit{mínima} se $H_\xi(p)$ é zero para todo $p$ e toda direção normal $\xi$.

Na demonstração acima vimos que se $X \subset \C^N$ é uma subvariedade complexa e $\xi$ é uma direção normal então os autovalores de $A_\xi$ aparecem em pares com sinais opostos. Em particular $\text{tr} A_\xi = 0$ e portanto $X$ é uma subvariedade mínima.
\end{remark}
\begin{remark}
Existem algumas definições equivalentes de uma variedade de Stein. Uma delas por exemplo é a seguinte: uma variedade complexa $X$ é uma variedade de Stein se $H^q(X,\mathcal F) = 0$ para $q>0$ e todo feixe coerente e analítico\footnote{Um feixe analítico é um feixe de $\mathcal O_X$-módulos e um feixe $\mathcal F$ é dito coerente se todo $x \in X$ admite uma vizinhança $U$ tal que existe uma sequência exata $\mathcal O_X^{\oplus p}|_U \to \mathcal O_X^{\oplus q}|_U \to \mathcal F|_U \to 0$.} $\mathcal F$ sobre $X$.

Uma outra possível definição, mais analítica, é a seguinte: uma variedade complexa $X$ é uma variedade de Stein se $X$ é holomorficamente separável (i.e., para todos $p \neq q$ em $X$ existe uma função holomorfa global em $X$ tal que $f(p)\neq f(q)$), holomorficamente convexa (i.e., para todo compacto $K$ o conjunto $\bar K = \{x \in X: |f(x)|\leq \sup_K |f|, \forall f \in \mathcal O(X)\}$ é de novo compacto) e se para todo ponto $x_0 \in X$ existe um aberto $U$ contendo $x_0$ e funções holomorfas globais $f_1,\ldots,f_n \in \mathcal O(X)$ de modo que $(f_1|_U,\ldots,f_n|_U)$ formam um sistema de coordenadas em $U$.

Um estudo completo sobre as variedades de Stein, inclusive a demonstração da equivalência das definições acima podem ser econtradas no livro \cite{grauert-remmert}.
\end{remark}

\section{O Teorema de Hiperplanos de Lefschetz}

Uma importante consequência do Teorema \ref{thm:stein-homology} é o chamado Teorema de Hiperplanos de Lefschetz, que permite calcular alguns grupos de cohomologia de uma variedade algébrica $X$ a partir dos respectivos grupos de um corte de $X$ por uma hipersuperfície.

\begin{theorem} \index{Teorema!de Hiperplanos de Lefschetz}
Seja $X \subset \pr^m$ uma variedade algébrica de dimensão $n$. Seja $Y = X \cap W$ a intersecção de $X$ com uma hipersuperfície algébrica $W$ que contém os pontos singulares de $X$ mas não contém $X$.

Nessas condições, o homomorfismo
\begin{equation*}
H^i(X,\Z) \longrightarrow H^i(Y,\Z)
\end{equation*}
induzido pela inclusão $Y \subset X$ é um isomorfismo para $i<n-1$ e é injetor se $i=n-1$. Além disso o quociente $H^{n-1}(Y,\Z) \slash H^{n-1}(X,\Z)$ é livre de torção.
\end{theorem}

\begin{proof}
Primeiramente note que basta provarmos o resultado no caso em que $W \subset \pr^m$ é um hiperplano. De fato, se  $W$ é uma hipersuperfície de grau $d$, o mergulho de Veronese $\varphi_{\mathcal O(d)}:\pr^m \to \pr^N$ transformará $W$ em uma seção de hiperplano $W'$ de $P=\varphi_{\mathcal O(d)}(\pr^m)$, isto é $W = H \cap P$ onde $H$ é um hiperplano de $\pr^N$ (veja o exemplo \ref{ex:veronese}). Denotando por $X'$ e $Y'$ as imagens de $X$ e $Y$ respectivamente temos que $Y' = X' \cap H$, isto é, $Y'$ é uma seção de hiperplano de $X'$ e como $\varphi_{\mathcal O(d)}$ é um mergulho a aplicação $H^i(X',\Z) \longrightarrow H^i(Y',\Z)$ é conjugada de $H^i(X,\Z) \longrightarrow H^i(Y,\Z)$ pela aplicação induzida por $\varphi_{\mathcal O(d)}$.

Suponha portanto que $W=H$ é um hiperplano e considere o aberto $U = \pr^m \setminus H$. Como $W$ contém os pontos singulares de $X$ temos que $X \setminus Y \subset U$ é uma subvariedade complexa fechada de $U \simeq \C^m$, ou seja, $X\setminus Y$ é uma variedade de Stein. Do Teorema \ref{thm:stein-homology} vemos então que
\begin{equation} \label{eq:homology-XY}
H_i(X \setminus Y, \Z) = 0 ~~\text{ para } i > n~~ \text{ e }~~ H_n(X \setminus Y, \Z)~~\text{ é livre de torção}.
\end{equation}

Dualizando o complexo de cadeias relativas $C_\bullet(X,Y;\Z)$ obtemos um complexo de cocadeias relativas $C^\bullet(X,Y;\Z)$ e uma sequência de complexos $0 \to C^\bullet(X,Y;\Z) \to C^\bullet(X;\Z)\to C^\bullet(Y;\Z) \to 0$ de onde vemos que existe uma sequência exata para a cohomologia relativa
\begin{equation*}
\cdots \to H^i(X,Y;\Z) \to H^i(X;\Z) \to H^i(Y;\Z) \to H^{i+1}(X,Y;\Z) \to H^{i+1}(X;\Z) \to  H^{i+1}(Y;\Z) \to \cdots.
\end{equation*}
Da dualidade de Lefschetz (veja \cite{munkres}, capítulo 8) temos que $H^k(X,Y;\Z) \simeq H_{2n-k}(X-Y;\Z)$ e portanto, de (\ref{eq:homology-XY}), vemos que $H^k(X,Y;\Z) = 0$ para $k < n$ e $H^n(X,Y;\Z)$ é livre de torção.

A sequência exata acima nos fornece então as sequências $0 \to H^i(X;\Z) \to H^i(Y;\Z) \to 0$ para $i<n-1$, mostrando que a aplicação induzida $H^i(X;\Z) \to H^i(Y;\Z)$ é um isomorfismo.

Para $i=n-1$ temos a sequência $0  \to H^{n-1}(X;\Z) \to H^{n-1}(Y;\Z) \to H^n(X,Y;\Z)$, que mostra que $H^{n-1}(X;\Z) \to H^{n-1}(Y;\Z)$ é injetora. Além disso vemos que o quociente $H^{n-1}(Y,\Z) \slash H^{n-1}(X,\Z)$ é isomorfo a um subgrupo de $H^n(X,Y;\Z)$ e portanto é livre de torção. 
\end{proof}

\begin{example} \textbf{Cohomologia de hipersuperfícies projetivas.}
Como aplicação do Teorema de Hiperplanos de Lefschetz podemos calcular os grupos de cohomologia de uma hipersuperfície suave $Y \subset \pr^{n+1}$, com a única exceção do grupo do meio $H^n(X,\Z)$.

Seja $Y \subset \pr^{n+1}$ uma hipersuperfície suave de dimensão $n$. Aplicando o Teorema de Hiperplanos de Lefschetz para $X= \pr^{n+1}$ e $W=Y$ vemos que a aplicação $H^i(\pr^{n+1},\Z) \to H^i(Y,\Z)$ é um isomorfismo para $i<n$. Como $H^i(\pr^{n+1},\Z)$ é igual a $\Z$ em grau par e igual a 0 em grau  ímpar concluímos que  $H^i(Y,\Z)$ é igual $\Z$ se $i<n$ é par e é igual $0$ se $i<n$ é ímpar.

Da dualidade de Poincaré temos que $H_i(Y,\Z) \simeq H^{2n-i}(Y,\Z)$ é igual $\Z$ se $i>n$ é par e é igual $0$ se $i>n$ é ímpar e como esses grupos são todos livres temos, do teorema dos coeficientes universais para a cohomologia, que $H^i(Y,\Z) \simeq \Hom(H_i(Y,\Z),\Z) \simeq \Z$ para $i>n$ par e $H^i(Y,\Z) = 0$ para $i>n$ ímpar. Concluímos assim que a cohomologia de $Y$ é dada por
\begin{equation*}
H^i(Y,\Z) = \left \lbrace \begin{split} \Z ~ &\text{ se } i \text{ é par} \\ 0 ~ &\text{ se } i \text{ é ímpar} \end{split} \right. ~~~~(i\leq 2n, i\neq n).
\end{equation*}

Vemos portanto que todas as hipersuperfícies $Y \subset \pr^{n+1}$ possuem a mesma cohomologia em grau $i \neq n$, ou seja, as únicas informações topologicamente interessantes sobre $Y$ (do ponto de vista da cohomologia) estão concentradas no grupo do meio $H^n(Y,\Z)$.

Em geral, é de se esperar que o grupo $H^n(Y,\Z)$ dependa do grau da hipersuperfície. No caso $n=1$ por exemplo, a fórmula de Plücker para curvas planas (veja por exemplo \cite{g-h}, cap. 2) diz que se $Y \subset \pr^2$ é uma curva plana suave de grau $d\geq 2$ então seu gênero é $g = \frac{(d-1)(d-2)}{2}$ e portanto o grupo $H^1(X,\Z)$ é um grupo abeliano livre de posto $2g=(d-1)(d-2)$.\\

O resultado acima pode ser generalizado para intersecções completas. Se  $Y = X_1 \cap X_2 \subset \pr^{n+2}$ é uma intersecção completa de dimensão $n$ então, pelo teorema de Hiperplanos de Lefschetz, temos que $H^i(Y,\Z) \simeq H^i(X_1,\Z)$ para $i < \dim X_1 -1 = n$.  Como $X_1$ é uma hipersuperfície em $\pr^{n+2}$, a discussão acima mostra que $H^i(X_1,\Z) \simeq H^i(\pr^{n+2},\Z)$ para $i<n+1$ e por argumento de dualidade análogo ao usado acima concluímos que a cohomologia de $Y$ é a mesma de $\pr^{n+2}$ para em grau $i \neq n$.

O mesmo argumento pode ser aplicado para uma intersecção completa qualquer, de onde obtemos o seguinte resultado.

\begin{proposition} \label{prop:cohomology-complete-intersection}
Seja $Y = X_1 \cap \cdots \cap X_k \subset \pr^{n+k}$ uma intersecção completa de dimensão $n$. Então os grupos de cohomologia de $Y$ são dados por
\begin{equation*}
H^i(Y,\Z) = \left \lbrace \begin{split} \Z ~ &\text{ se } i \text{ é par} \\ 0 ~ &\text{ se } i \text{ é ímpar} \end{split} \right. ~~~~(i\leq 2n, i\neq n).
\end{equation*}
\end{proposition}

Como consequência podemos calular o grupo de Picard das intersecções completas.
\begin{proposition}
Seja $X \subset \pr^N$ uma intersecção completa de dimensão maior ou igual $3$. Então grupo de Picard de $X$ é isomorfo a $\Z$. Mais precisamente, todo fibrado de linha sobre $X$ é isomorfo a $\mathcal O(k)|_X$ para algum $k \in \Z$.
\end{proposition}
\begin{proof}
Considere a sequência exata longa associada a sequência exponencial
\begin{equation*}
H^1(X,\mathcal{O}_X) \longrightarrow \Pic(X) \stackrel{c_1}{\longrightarrow} H^2(X,\Z) \longrightarrow H^2(X,\mathcal{O}_X)
\end{equation*}

Da Proposição \ref{prop:cohomology-complete-intersection} vemos que $H^1(X,\Z) = 0$ e portanto $0 = H^1(X,\C) = H^{1,0}(X) \oplus H^{0,1}(X)$. Em particular $H^1(X,\mathcal{O}_X) \simeq H^{0,1}(X) = 0$.

Como $n\geq 3$ vemos também que $H^2(X,\Z) = \Z$ e portanto $H^2(X,\C)$ é unidimensional. Como $H^2(X,\C) = H^{2,0}(X) \oplus H^{1,1}(X) \oplus H^{0,2}(X)$ e $H^{1,1}(X) \neq 0$ concluimos que $H^2(X,\mathcal{O}_X) \simeq H^{0,2}(X) = 0$.

Sendo asssim a sequência acima mostra que $c_1:\Pic(X) \to H^2(X,\Z) \simeq \Z$ é um isomorfismo. Para ver que todo fibrado de linha é isomorfo a algum $\mathcal O (k)$ argumentamos como no exemplo \ref{ex:lb-pn}: como $c_1(\mathcal O(1)) = [\omega_{FS}]$ gera $H^2(X,\Z)$, o grupo de Picard é gerado por $\mathcal O (1)$.
\end{proof}
Uma outra consequência interessante da discussão acima é a seguinte.
\begin{corollary}
Nenhum toro complexo de dimensão maior que $1$ pode ser mergulhado como uma intersecção completa em $\pr^N$.
\end{corollary}

De fato, um toro complexo $X$ tem $H^1(X,\Z) \simeq \Z^{2n}$ mas, se existisse um mergulho $X \to \pr^N$ tal que sua imagem fosse uma intersecção completa teríamos, da Proposição \ref{prop:cohomology-complete-intersection} e do fato que $n \geq 2$, que $H^1(X,\Z) = 0$.
\end{example}

\bibliography{bibliografia}
\bibliographystyle{plain}

\printindex
\end{document}